\documentclass[11pt,reqno,a4paper]{amsart}

\usepackage{amsmath, amsthm, amssymb}
\usepackage[T1]{fontenc}
\usepackage{setspace}
\usepackage[english]{babel}
\usepackage{mathrsfs}
\usepackage[all]{xy}
\usepackage{paralist}
\usepackage{graphicx}
\usepackage{accents}
\usepackage{enumerate}
\usepackage{mathtools}

\usepackage[dvipsnames]{xcolor}
\usepackage[colorlinks=true,linkcolor=NavyBlue,urlcolor=MidnightBlue,citecolor=PineGreen]{hyperref}

\makeatletter
\def\l@subsection{\@tocline{2}{0pt}{1.8pc}{3pc}{}}
\makeatother

\usepackage[margin=2.5cm,top=3.0cm,bottom=3.0cm]{geometry}

\makeatletter
\newtheorem*{rep@theorem}{\rep@title}
\newcommand{\newreptheorem}[2]{%
\newenvironment{rep#1}[1]{%
 \def\rep@title{#2~\ref{##1}}%
 \begin{rep@theorem}}%
 {\end{rep@theorem}}}
 
\makeatother

\theoremstyle{plain}

\newtheorem{theorem}{\normalfont \scshape Theorem}

\newtheorem{lemma}{\normalfont \scshape Lemma}[section]
\newtheorem{proposition}[lemma]{\normalfont \scshape Proposition}
\newtheorem{corollary}[lemma]{\normalfont \scshape Corollary}

\newreptheorem{theorem}{Theorem}
\newreptheorem{proposition}{Proposition}

\theoremstyle{definition}

\theoremstyle{remark}
\newtheorem{remark}[lemma]{\normalfont \scshape Remark}

\renewenvironment{proof}[1][\proofname]{\noindent{\scshape #1}.\quad}{\qed}

\numberwithin{equation}{section}

\newcommand{\bfm}{{\bf m}}

\newcommand{\bfr}{{\bf r}}
\newcommand{\bfs}{{\bf s}}

\newcommand{\bfw}{{\bf w}}
\newcommand{\bfx}{{\bf x}}

\newcommand{\bfrho}{{\boldsymbol{\rho}}}

\newcommand{\bfA}{{\bf A}}
\newcommand{\bfB}{{\bf B}}

\newcommand{\bbN}{\mathbb N}

\newcommand{\calC}{\mathcal C}
\newcommand{\calD}{\mathcal D}

\newcommand{\calH}{\mathcal H}

\newcommand{\calO}{\mathcal O}

\newcommand{\calR}{\mathcal R}

\newcommand{\calU}{\mathcal U}
\newcommand{\calV}{\mathcal V}
\newcommand{\calW}{\mathcal W}
\newcommand{\calX}{\mathcal X}
\newcommand{\calY}{\mathcal Y}

\newcommand{\frakB}{\mathfrak B}

\newcommand{\frakK}{\mathfrak K}

\newcommand{\frakb}{\mathfrak b}

\newcommand{\frakd}{\mathfrak d}

\newcommand{\frakg}{\mathfrak g}
\newcommand{\frakh}{\mathfrak h}

\newcommand{\frakm}{\mathfrak m}

\newcommand{\frakp}{\mathfrak p}
\newcommand{\frakq}{\mathfrak q}

\newcommand{\scrB}{\mathscr B}
\newcommand{\scrC}{\mathscr C}

\newcommand{\scrF}{\mathscr F}

\newcommand{\scrL}{\mathscr L}

\newcommand{\scrP}{\mathscr P}

\newcommand{\scrV}{\mathscr V}

\newcommand{\sfC}{\mathsf C}

\newcommand{\sfF}{\mathsf F}
\newcommand{\sfG}{\mathsf G}
\newcommand{\sfH}{\mathsf H}
\newcommand{\sfI}{\mathsf I}

\newcommand{\sfL}{\mathsf L}

\newcommand{\sfP}{\mathsf P}

\newcommand{\sfU}{\mathsf U}

\newcommand{\bp}{\begin{pmatrix}}
\newcommand{\ep}{\end{pmatrix}}
\newcommand{\p}{\partial}

\newcommand{\R}{\mathbb{R}}

\newcommand{\T}{\mathbb{T}}

\newcommand{\ve}{\varepsilon}
\newcommand{\va}{\varsigma}

\newcommand{\dx}{\textnormal{d}x}
\newcommand{\dy}{\textnormal{d}y}
\newcommand{\dv}{\textnormal{d}v}
\newcommand{\dw}{\textnormal{d}w}
\newcommand{\dt}{\textnormal{d}t}
\newcommand{\ds}{\textnormal{d}s}

\newcommand{\ddt}{\frac{\textnormal{d}}{\textnormal{d}t}}
\newcommand{\tnd}{\textnormal{d}}

\newcommand{\weakto}{\rightharpoonup}

\newcommand{\lt}{\left}
\newcommand{\rt}{\right}
\newcommand{\la}{\langle}
\newcommand{\ra}{\rangle}

\newcommand{\ov}{\overline}
\newcommand{\wt}{\widetilde}

\newcommand{\intr}{\int_{\R^d}}

\def\XXint#1#2#3{{\setbox0=\hbox{$#1{#2#3}{\int}$ }
\vcenter{\hbox{$#2#3$ }}\kern-.6\wd0}}

\newcommand{\doubleparantheses}[2]{(\!(#1, #2)\!)}
\newcommand{\doublebrackets}[1]{\langle\!\langle #1 \rangle\!\rangle}
\newcommand{\NORM}[1]{\lVert #1 \rVert}
\newcommand{\NORMMM}[1]{{\left\vert\kern-0.25ex\left\vert\kern-0.25ex\left\vert #1 
    \right\vert\kern-0.25ex\right\vert\kern-0.25ex\right\vert}}

\newcommand{\leftperp}[1]{{}^\perp\!\!\!\:#1}

\allowdisplaybreaks

\date{\today}

\begin{document}

\title[Weakly inhomogeneous solutions to nonlinear VFP in a domain]{Weakly inhomogeneous solutions to a nonlinear Vlasov--Fokker--Planck equation in a domain}

\author{Sihyun Song}
\address{Department of Mathematics, Yonsei University, Seoul 03722, Republic of Korea}
\email{ssong@yonsei.ac.kr}

\date{\today}

\makeatletter
\def\blfootnote{\gdef\@thefnmark{}\@footnotetext}
\makeatother

\blfootnote{MSC No. 35B40, 35Q70, 35Q84}

\begin{abstract}
A nonlinear Vlasov--Fokker--Planck equation is considered in a bounded domain, with specular reflection imposed on the boundary. It is proven that if the initial datum is sufficiently close, in a weighted $L^\infty$-space, to a radial and spatially homogeneous function that is regular enough, then a corresponding global-in-time weak solution exists, and this one decays with exponential rate to the global Maxwellian. The proof is split into two steps. In the first, the stability of the manifold of spatially homogeneous solutions to this equation is established. In the second part, the close-to-Maxwellian regime is considered. A careful regularization scheme is designed so that this mollified equation can be well-approximated by a linear system. Inspired by the recent preprint [Carrapatoso and Mischler, arxiv:2407.09031], for this linear system, $L^2$-hypocoercivity along with ultracontractivity properties are derived in order to change the space of functional decay to the $L^\infty$-framework. The properties of the original model are then recovered by means of fixed point arguments and carefully passing to the limit of the regularization parameter.
\end{abstract}

\maketitle
\setcounter{tocdepth}{2}
\tableofcontents

\singlespacing

\section{Introduction}

\subsection{The model and problem}

We consider a nonlinear Vlasov--Fokker--Planck equation with specular reflection imposed at the boundary:
\begin{equation}\label{eq: main}
\begin{cases}
    (\p_t + v\cdot \nabla_x) f = \scrF[f] &\text{in }\R_+\times\Omega\times\R^d,\\
    \scrF[f] := \rho[f] \nabla_v\cdot \Big( \theta[f] \nabla_v f + (v-u[f])f) \Big), \\
    f|_{t=0} = f_0 &\text{on }\Omega\times\R^d, \\
    \gamma_- f(t,x,v) = \gamma_+ (t,x, \scrV_x v) &\text{on } [0,\infty)\times \Sigma_-.
\end{cases}
\end{equation}
In the above, the coefficients of the equation are the macroscopic quantities associated to the solution, defined by
\begin{equation*}
    u[f] = \begin{cases} j[f]/\rho[f] & \rho[f] > 0, \\ 0 & \rho[f] = 0, \end{cases} \quad \rho[f] := \intr f \,\dv, \quad j[f] := \intr vf \,\dv,
\end{equation*}
along with
\begin{equation*}
    \theta[f] := \frac{1}{d\rho[f]}\intr |v-u[f]|^2 f \,\dv = \frac{1}{d\rho[f]} \lt( E[f] - \frac{|j[f]|^2}{\rho[f]} \rt) , \quad E[f] := \intr |v|^2 f \,\dv .
\end{equation*}
The quantities $\rho[f]$, $u[f]$, and $\theta[f]$ are often respectively referred to as the local density, bulk velocity, and temperature of $f$.

The Vlasov--Fokker--Planck equation is a fundamental equation which arises in the study of plasma physics or granular media \cite{Gre52,Kir46}. It can be formally derived from the Landau equation
\begin{equation*}
\begin{split}
    &\p_t f + v\cdot \nabla_x f = Q(f,f), \\
    &Q(f,f) := \p_{v_i} \int_{\R^d} a_{ij}(v-v')(f(v')\p_{v_j} f - f \p_{v_j'} f(v')) \,\dv' 
\end{split}
\end{equation*}
by replacing $f(v')$ with the corresponding local Maxwellian which corresponds to $f$, and by imposing that the matrix $a_{ij}$ corresponds to the Maxwellian potential. This simplification results in that the relaxation towards equilibrium of solutions to \eqref{eq: main} are expressed through quadratic terms which give a good approximation for binary collisions. However, as pointed out by Saint-Raymond in \cite{SR02}, there has been no ``satisfactory'' existence theory for the model \eqref{eq: main}, especially in the framework of weak solutions.

In some simpler cases where the collision frequency is modified, for instance if one considers
\begin{equation*}
    \p_t f + v\cdot \nabla_x f = \nu[f] \nabla_v  \cdot \Big(\theta[f]\nabla_v f + (v - u[f] ) f \Big)
\end{equation*}
with $\nu[f]$ obeying some better bounds compared to $\rho[f]$, then some satisfactory results have appeared in the recent years. For the simplest case of $\nu[f]\equiv 1$ and for $x\in \T^d$, a gradient flow-type approach was implemented in \cite{CarGan04} to prove the existence of (constructive) weak solutions under the assumption of $f_0\in L^1_{\la v \ra^6}\cap L^2 \cap L\log L(\T^d\times\R^d)$. Under the different set of assumptions $f_0\in L^1_{\la v \ra^3}\cap L^\infty\cap L\log L(\T^d\times\R^d)$, the existence of weak solutions was proven also in \cite{ChoiHwaYoo25}, under the same setting of $\nu\equiv 1$ and $x\in \T^d$. These results were subsequently generalized, in various directions, by the author in a joint work with Choi \cite{ChoiSong25}. There, we considered initial datum satisfying only the physically natural assumptions of $L^1_{\la v \ra^2}\cap L\log L(\Omega\times\R^d)$, more general (but bounded) collision frequencies $\nu$, and a general bounded domain $\Omega$. However, in that work, only the cases of fixed inflow or partial absorption-reflection boundary conditions were considered; the cases of pure specular reflection or Maxwell boundary conditions were not encompassed.

At this moment, the existence theory for $\eqref{eq: main}_1$ yet appears to be out of reach for just physically admissible initial datum. Nevertheless, this work aims to provide a step forward in the existence theory by considering the equation in the weakly inhomogeneous regime. Namely, in this work, we prove that if the initial datum is sufficiently close to some spatially homogeneous and radial function $g_0$, then there exists a global-in-time weak solution to \eqref{eq: main} which relaxes exponentially fast to the equilibrium state. The explicit main results will be stated in the next section.

\subsection{Notations and main results}

Throughout this paper, we always impose the following assumptions and notations. The domain $\Omega\subset\R^d$ is assumed to be bounded, and we assume in particular that we can write $\Omega$ as
\begin{equation*}
    \Omega:= \{ x\in \R^d : \mathfrak{d}(x) > 0 \},
\end{equation*}
where $\mathfrak{d}(x) := \textnormal{dist}(x,\p\Omega)$ is the signed distance function. We then define
\begin{equation*}
    n_x := -\frac{\nabla \frakd(x)}{|\nabla \frakd(x)|},
\end{equation*}
which coincides with the unit outwards normal vector to $\Omega$ whenever $x\in \p\Omega$. We assume throughout that $\frakd \in W^{2,\infty}(\ov{\Omega})$ and that there exists some $c_{\Omega}>0$ for which $|\nabla_x \frakd|\ge c_{\Omega}$ on all of $\ov{\Omega}$. The Hausdorff measure on $\p\Omega$ will be denoted $\tnd\sigma_x$.

We remind the reader that the specular reflection boundary condition means that once particles touch the wall, they bounce away with an angle to the normal that is the same as with the angle of entry. Namely, the specular reflection boundary condition says that
\begin{equation}\label{eq: specular}
    \gamma_- f(t,x,v) = \gamma_+ f(t,x, v - 2(n_x\cdot v) n_x) \quad \forall (t,x,v)\in (0,\infty)\times \Sigma_-,
\end{equation}
and we used the shorthand notation $\scrV_x v := v - 2(n_x\cdot v) n_x$ in $\eqref{eq: main}_2$. The following notation is always adopted: for each $x\in \p\Omega$,
\begin{align*}
    \Sigma_{\pm}^x := \{v\in \R^d : \pm v\cdot n_x > 0\},
\end{align*}
along with
\begin{align*}
    &\Sigma:= \p\Omega\times\R^d, \quad \Sigma_{\pm} := \{(x,v) \in \Sigma: v\in \Sigma_{\pm}^x\}, \\
    &\Sigma^T := (0,T)\times \Sigma, \quad \Sigma^T_{\pm} := (0,T)\times \Sigma_{\pm}, \quad T\in (0,\infty).
\end{align*}
We will also write $\calO := \Omega\times\R^d$ and $\calO^T := (0,T)\times\calO$, $\Omega^T:= (0,T)\times\Omega$ for given $T\in (0,\infty)$. Finally we remark that we often abuse notation and write \eqref{eq: specular} as
\begin{equation*}
    \gamma_- f = \gamma_+ f \circ \scrV_x \quad \text{on } (0,\infty)\times \Sigma_-.
\end{equation*}

Regarding functional spaces, for a weight function $\omega$ and $p\in [1,\infty]$ we will use the notation
\begin{equation*}
    \|f\|_{L^p_\omega} := \|f \omega\|_{L^p}.
\end{equation*}
For Sobolev spaces, when $p\in [1,\infty)$ we define $W^{1,p}_{\omega}$ as those (equivalence class of) measurable functions for which the norm
\begin{equation*}
    \|f\|_{W^{1,p}_{\omega}} := \Big(\|f\|_{L^p_\omega}^p + \|\nabla f\|_{L^p_{\omega}}^p \Big)^{1/p}
\end{equation*}
is finite. Generalizations to $W^{k,p}_{\omega}$, where $k \in \bbN$ or $p=\infty$, are given in the conventional way. The domain or variable of integration is typically clear from the later context and we will not specify it at this point. In the case $p=2$ we use the standard notation $H^k_{\omega} := W^{k,2}_{\omega}$. Also, in integrals we will sometimes omit the differential form notation in order to save space if the context is clear, for instance
\begin{equation*}
    \iint_{\calO} f \,\dx\dv = \iint_{\calO} f.
\end{equation*}

For a Banach space $X$ and a sequence $(f_k)\subset X$, we will write $f_k\weakto f$ weakly, or $f_k \weakto f$ in $X-w$, if $f_k\to f$ in $\sigma(X,X^*)$. Similarly, for a sequence of functionals $(f_k)\subset X^*$, we will write $f_k \weakto f$ weakly$^*$, or $f_k\weakto f$ in $X^*-*$, if $f_k \to f$ in $\sigma_*(X^*,X)$.

In regard of weak solutions, we will say that a function $f$ is a weak solution to \eqref{eq: main} if it satisfies $\eqref{eq: main}_1$ in duality with $\calD'([0,t]\times \ov{\calO})$ for every $t \ge 0$ with trace functions $f_t$ and $\gamma f$, the latter of which satisfies the specular reflection boundary condition $\eqref{eq: main}_2$ pointwisely (at least a.e.).

We also recall that the equation \eqref{eq: main} formally satisfies certain conservation laws. The beginning observation is that
\begin{equation}\label{Eq/ F Conv}
    \int_{\R^d} (1,v,|v|^2) \scrF[f] \,\dv = 0
\end{equation}
holds for any $f$ which has sufficient regularity and decay. Consequently, if $f$ is a smooth and rapidly decaying solution to \eqref{eq: main} we immediately have the mass and energy conservation
\begin{equation}\label{Eq/ Mass Egy}
    \ddt \iint_{\calO} (1,|v|^2) f \,\dx\dv = 0,
\end{equation}
thanks to the Gauss--Ostrogradsky theorem and the specular reflection boundary condition \eqref{eq: specular}. In the case where $\Omega$ has rotational symmetries, then another conservation law arises; namely, one can directly check that for $R(x) := Ux + u$, with $U\in \R^{d\times d}$ and $u\in \R^d$, the solution to \eqref{eq: main} satisfies
\begin{equation}\label{Eq/ Mom}
    \ddt \iint f R(x)\cdot v \,\dx\dv = 0
\end{equation}
if and only if
\begin{equation}\label{M/Eq/ calR}
    \text{$U$ is skew symmetric and } R(x) \cdot n_x = 0 \quad \forall x\in \p\Omega.
\end{equation}
The proof of the above makes crucial use of the fact that $\int v \scrF[f] \,\dv = 0$. We denote by $\calR_\Omega$ the collection of all such $R(x) = Ux + u$ which satisfy \eqref{M/Eq/ calR}, and we point out that $\calR_{\Omega}$ does not reduce to $\{0\}$ in the case where rotational symmetries are present. In conclusion, the conservation laws of \eqref{eq: main} can be summarized as
\begin{equation*}
    \ddt \iint_{\calO}(1, R(x)\cdot v, |v|^2) f \,\dx\dv = 0 \quad \forall R\in \calR_{\Omega}.
\end{equation*}
In view of the above, we impose the following condition on the initial datum. Denoting $\doublebrackets{f}:= \iint_{\calO} f$, we will say that $f$ satisfies the compatibility condition (C) if the following holds:
\begin{equation*}
    \textnormal{(C)} \qquad \doublebrackets{f - \mu} = \doublebrackets{(R(x)\cdot v) (f - \mu)} = \doublebrackets{|v|^2 (f - \mu)} = 0 \quad \forall R\in \calR_\Omega,
\end{equation*}
where $\mu$ is the usual normalized equilibrium $\mu(v) := e^{-|v|^2/2}/(2\pi)^{d/2}$.

Our main result reads as follows.

\begin{theorem}\label{Thm: Main}
    Let some radial function $g_0:\R^d_v\to \R_+$ be given with $\|g_0\|_{L^1_{\la v \ra^2}(\R^d)} < +\infty$ and $\rho[g_0] = 1$. Then there exists a class of admissible weights $\calW$ such that the following holds: for $\omega\in \calW$, if
    \begin{equation*}
        \|g_0\|_{W^{2,\infty}_{\la v \ra \omega }} < +\infty,
    \end{equation*}
    then there exists $\delta>0$ small enough such that if an initial datum $f_0$ obeying the compatibility condition (C) satisfies
    \begin{equation*}
        \|f_0 - g_0\|_{L^\infty_\omega(\calO)} \le \delta,
    \end{equation*}
    there exists a global-in-time weak solution $f \in L^\infty_{\omega}(\R_+\times \calO)$ to \eqref{eq: main} which satisfies (C) for all times $t\ge 0$. Furthermore, the following dissipative estimate holds
    \begin{equation*}
        \forall t\ge 0, \quad \|f_t - \mu\|_{L^\infty_\omega(\calO)} \lesssim e^{- \iota' t} 
    \end{equation*}
    for some constant $\iota' > 0$ independent of time.
\end{theorem}

\begin{remark}
    The admissible class of weights $\calW$ in Theorem \ref{Thm: Main} is defined as
    \begin{equation*}
        \calW := \calW_{g_0} \cap \calW_{\mu^{-1/2}},
    \end{equation*}
    where for each $\frakg :\R^d_v\to \R$ with $\|\frakg \|_{L^1_{\la v \ra^2}(\R^d)} < \infty$, we define $\calW_{\frakg}$ as
\begin{equation*}
    \calW_{\frakg} := \lt\{
        \begin{aligned}
        \la v \ra^m, \quad & m > d + 2, \\
        \la v \ra^m e^{\gamma |v|^\tau}, \quad &\text{either } m \ge 0, \; \gamma > 0, \text{ and } \tau \in (0,2) \\
        &\text{or } m \ge 0, \; \gamma\in \lt(0,\frac{d}{2E[\frakg]}\rt), \text{ and } \tau = 2
        \end{aligned}
        \rt\} .
\end{equation*}
\end{remark}

\begin{remark}
    The uniqueness of the solution as constructed in Theorem \ref{Thm: Main} is unknown. If we impose some Sobolev regularity on the initial datum, then we believe that uniqueness can be established, but we do not pursue this direction in the current work.
\end{remark}

The proof of Theorem \ref{Thm: Main} is split into two parts. In the first part, we consider the weakly-inhomogeneous regime in a finite time horizon setting. Namely, the main theorem of the first section establishes the stability of the manifold of spatially homogeneous solutions up to an arbitrary time:

\begin{theorem}\label{Thm1}
    We fix an admissible weight $\omega\in \calW_{g_0}$. Let any $a \ll 1$, $T>0$ be given, and assume that $g_0\in W^{2,\infty}_{\la v \ra \omega}(\R^d)$. Then there exists $\delta=\delta(a,T)$ such that if $\|f_0 - g_0\|_{L^\infty_\omega}\le \delta(a,T)$, then there exists a weak and renormalized solution $f\in L^\infty_{\omega}((0,T)\times \calO)$ to the Vlasov--Fokker--Planck system \eqref{eq: main}, satisfying 
    \begin{equation}\label{Eq/ Stab}
    \sup_{t\in [0,T]} \|f_t - g_t\|_{L^\infty_{\omega}(\calO)} \le a.
    \end{equation}
    If, furthermore $f_0$ satisfies the compatibility assumption (C), then $f_t$ satisfies (C) also at every $t\in [0,T]$.
\end{theorem}

In the second part, we consider the close-to-Maxwellian regime. That is, we consider the setting where the initial datum is sufficiently close to $\mu$. We prove that a global-in-time weak solution exists and that it decays exponentially to the Maxwellian.

\begin{theorem}\label{Thm2}
    We fix an admissible weight $\omega\in \calW_{\mu^{-1/2}}$. There exists some $a_0>0$ small enough such that if $f_0$ satisfies (C) and $\|f_0 - \mu\|_{L^\infty_{\omega}}\le a_0$, then there exists a weak and renormalized solution $f\in L^\infty_{\omega}(\R_+\times \calO)$ to the Vlasov--Fokker--Planck system \eqref{eq: main}. This one satisfies all of the conservation laws (C) and also the following hypodissipative estimate for some $\iota'>0$:
    \begin{equation}\label{Eq/ hypodissipative}
        \|f_t - \mu\|_{L^\infty_{\omega}(\calO)} \lesssim e^{-\iota' t}.
    \end{equation}
\end{theorem}

The proof of Theorem \ref{Thm: Main} is then a simple combination of Theorems \ref{Thm1}--\ref{Thm2}.

\begin{proof}[Proof of Theorem \ref{Thm: Main}]
Since $\omega\in \calW_{g_0} \cap \calW_{\mu^{-1/2}}$, we observe that both the results of Theorems \ref{Thm1}--\ref{Thm2} are at our disposal. Let $a_0>0$ be fixed as required in Theorem \ref{Thm2}. Thanks to the already known results on the spatially homogeneous Fokker--Planck equation (see Lemma \ref{Lem/ Trend to Equi}), we may also fix appropriate $T>0$ so that
\begin{align*}
    \|g_T - \mu\|_{L^\infty_\omega} \le a_0 /2.
\end{align*}
Next, we apply Theorem \ref{Thm1} to find $\delta=\delta(a_0,T)>0$ such that if
\begin{align*}
    \|f_0 - g_0\|_{L^\infty_\omega} \le \delta,
\end{align*}
then there exists a weak solution $f_t\in L^\infty_{\omega}([0,T]\times \calO)$ to \eqref{eq: main} satisfying
\begin{equation*}
    \|f_T - g_T\|_{L^\infty_\omega} \le a_0 /2.
\end{equation*}
Then the triangle inequality shows
\begin{align*}
    \|f_T - \mu\|_{L^\infty_\omega} \le a_0,
\end{align*}
and thus we may now apply Theorem \ref{Thm2} (applied with initial datum $f_T$) to obtain a solution $f$ which extends to $t\ge T$ and exponentially converges to $\mu$ as $t\to\infty$.
\end{proof}

\subsection{Strategy of proof and review of the literature}

\subsubsection{A brief review}

The Cauchy problem for kinetic equations and the long-time behavior of solutions is an important topic of research, and it has been investigated thoroughly in the past few decades. In regard of the model \eqref{eq: main}, results are relatively scarce but we may first mention \cite{Cho16} where a similar (but simpler) equation is considered, in the framework of high-order Sobolev spaces. Also, there is the more recent work \cite{LiaoYang21} on a similar nonlinear Fokker--Planck equation to \eqref{eq: main} in $\R^3$, and its trend to equilibrium is established. As aforementioned, \cite{CarGan04,ChoiHwaYoo25,ChoiSong25} discussed the existence of weak solutions to a model similar to that of \eqref{eq: main}, but no rigorous results on the convergence $f\to \mu$ were obtained. 

Regarding more general models, the convergence (and its rate) of a solution to kinetic PDE to the equilibrium state has been extensively discussed, and this topic is typically called \textit{hypocoercivity}; we refer to \cite{DolMouSch09,DolMouSch15,GuaMisMou17,MisMou16} and in particular the memoir \cite{Vil09}. In special regard of the Landau and Boltzmann equations, we may mention \cite{CarMis17,DesVil05,Guo02,KimGuoHwang20,StrGuo08}; in regard of existence and trend to equilibrium problems posed on general domains, we can mention \cite{BCMT23,BriGuo16,CarMis24KFP,CarMis24,Guo10,GuoHwangJangOu20,KimLee18A,KimLee18B}.

There are various well-known methods which allow one to establish the hypocoercivity property of solutions. For instance we may mention the celebrated relative entropy method, in the spirit of Bakry--Emery. We refer to \cite{Jun16} for an introduction to entropy methods and \cite{DolLi18} as another reference where entropy and $L^2$-estimates are appropriately interpolated. In the relative entropy framework, the goal is to find a suitable Lyapunov functional, equivalent to the usual entropy, which decays along the flow. This is actually not so simple as in the spatially homogeneous case because typically the dissipation of the entropy only provides relaxation towards a local equilibrium.

We also mention the $L^2$-hypocoercivity method, which is reminiscent of the procedure taken when investigating diffusion limits; there is a huge amount of references, and we just mention \cite{AddDolbLiTay21,BCMT23,DolMouSch15} at this point. It is worth pointing out that the recent work \cite{BCMT23} was able to establish $L^2$-hypocoercivity in a general bounded domain, with Maxwell reflection condition imposed at the boundary, and this has been a source of enlightenment for \cite{CarMis24KFP,CarMis24}, and this work as well. There is also the $H^1$-hypocoercivity method, which incorporates a ``twisted Fisher information,'' popularized by Villani \cite{Vil09} and that of which is in the spirit of H\"ormander's work on hypoellipticity \cite{Hor67}.

Another interesting task in the theory of kinetic equations has been to obtain this hypocoercivity property in larger spaces, for instance spaces that are not Hilbert spaces. There are several motivations as to why such a topic is required. For one, while there have been groundbreaking works by Guo in high-order Sobolev spaces (see for instance \cite{Guo02}), these Sobolev spaces are not very well suited to boundary problems where we expect regularity loss to occur near grazing sets. Also, it is typically the case that the PDE in question is well-posed in weighted Lebesgue spaces (without having to extend to Sobolev regularity), and therefore these appear to be the natural spaces in which to derive decay of the solution. Consequently, many techniques have been developed in order to change the functional space in which decay holds, in particular \cite{GuaMisMou17,KimGuoHwang20,MisMou16}. It is worth emphasizing that a crucial ingredient in the proof of for instance \cite{MisMou16} is the ultracontractivity property of kinetic equations. We remark that such an ultracontractivity property is well-known in the parabolic setting, thanks to the classical works of De Giorgi--Nash--Moser. Recently, the applications of their methods to the kinetic framework has seen light \cite{GIMV19}. We refer in particular to the recent works \cite{ImbMou21,Zhu24} where regularity of some kinetic Fokker--Planck type equations is investigated.

The next two sections briefly describe the story behind the main results (Theorems \ref{Thm1}--\ref{Thm2}, respectively) and their proofs.

\subsubsection{The weakly inhomogeneous regime}

The material of this section is inspired by \cite{ArkEspPul87}, where the Boltzmann equation was considered in such a weakly inhomogeneous regime; this was next generalized in \cite{GuoLiu17} to bounded domains with specular reflection imposed on the boundary. In this regime, the main problem becomes a question of whether we can ``keep small'' perturbations about a general spatially inhomogeneous solution $g$ to \eqref{eq: main}. In the case of the Fokker--Planck equation, a major difficulty is due to the fact that if one investigates the equation satisfied by the perturbation $h:=f-g$, then the principal part of the operator also becomes perturbed. In particular, the diffusion coefficient ends up with a quadratic term in $h$, which we denote $\bfA[h]$. If we are to work in a high-order Sobolev space as was done for instance in \cite{Guo02}, then we could just merely treat this term as a small nonlinear perturbation. However, our goal is to work only with weighted Lebesgue spaces, and so without any Sobolev space assumptions on the initial datum; thus, such an approach is not possible.

This difficulty is overcome through several steps. In \cite{CarMis24} (where the Landau equation was discussed), the coefficients of the PDE were suitably modified so as the analysis reduces to that of a linear one. It turns out that trying to directly apply that approach does not work so well for the model \eqref{eq: main}. Indeed, as aforementioned, there is a quadratic term in the diffusion coefficient which prevents the application of any known fixed point theorems. Towards this end, instead of immediately linearizing this equation, we first mollify the quadratic diffusion coefficient $\bfA[h]$ carefully, so that it now reads $\bfA_\va[h]$ with a regularization parameter $\va>0$. The method of mollification is specifically designed so that $\bfA_\va[\cdot]$ enjoys a strong stability property in suitable Lebesgue spaces, even if its argument only converges for instance weakly$-*$. This specific choice of mollification ensures that a fixed point argument becomes feasible, and we can then turn our attention solely to the linear counterpart of this mollified equation. That is, in the quadratic terms of the nonlinear PDE, we replace one in each term so that it depends on a fixed function $k=k(t,x,v)$ instead of the solution $h$, and consequently to ensure that the resulting PDE becomes linear.

Thanks to preliminary results on linear nonlocal Kolmogorov-type equations, which we write out in Section \ref{Sec/ Pre}, we readily obtain the well-posedness for the linear equation in a weighted $L^2_{\wt{\omega}}$ framework, where $\wt{\omega}$ is extracted from a set $\wt{\calW}$ which consists of polynomial weights. By manipulating the weights in $\wt{\calW}$, in particular choosing two weights $\omega, \wt{\omega}$ suitably so that $\wt{\omega}\omega^{-1}\in L^2(\R^d)$, we extend the $L^2_{\wt{\omega}}$ solution to the $L^\infty_{\omega}$ framework and rigorously establish $L^\infty_{\omega}$-estimates for it by resorting to its backwards dual equation and using a duality formula (Lemma \ref{p: l: dual}). This estimate can be made independent of $\va$, and we prove that if the initial datum $h_0:= f_0 - g_0$ is small enough, then the solution operator leaves a small ball of $L^\infty_{\omega}([0,T]\times\calO)$ invariant. Consequently, resorting to a fixed point theorem of Schauder, we obtain a solution to the nonlinear equation, which yet has mollified coefficient $\bfA_\va[h]$. We then pass to the limit $\va\to 0$ to show that a limit point exists which solves the original equation and satisfies the required stability estimate in \eqref{Eq/ Stab}.

\subsubsection{The close-to-Maxwellian regime}

The second part of this work is inspired by the strategy that was taken in \cite{CarMis24} regarding the Landau equation, see also \cite{CarMis24KFP}. In Section \ref{Sec/ Sketch2} we will provide a more explicit sketch of proof of Theorem \ref{Thm2}, and collect some of the key ingredients that are to be established later in the paper. Here, let us provide the reader with a rough draft of the scheme we will be using.

Similarly as in the weakly inhomogeneous regime, we write $h:= f - \mu$ and consider the (nonlinear) PDE satisfied by $h$. Then, again similarly as before, we mollify the equation suitably so that some strong compactness can be expected in the coefficients of the PDE; as a result, a fixed point argument becomes feasible, and we can therefore turn our attention to the linear counterpart of this equation. Replacing the nonlinear coefficients appropriately, the analysis then reduces to that of a linear evolution equation
\begin{equation*}
    \p_t h = \sfF_\va[k] h,
\end{equation*}
where $k=k(t,x,v)$ is assumed to be given and $\sfF_\va[k]$ is a suitable (time-dependent) operator, see Section \ref{Sec/ Max Setup}.

Then motivated by the so-called factorization of semigroups approach in \cite{GuaMisMou17}, we consider the decomposition $\sfF_\va[k] := \sfG_\va[k] + \sfH_\va[k]$ where $\sfG_\va[k]$ is the regularizing part and $\sfH_\va[k]$ is the dissipative part of the operator. In Section \ref{Sec/ WP sfH}, we prove that indeed the evolution equation corresponding to $\sfH_\va[k]$ is well-posed in suitable weighted $L^2\cap L^\infty$ spaces and that it is dissipative in all of those spaces, provided that $k$ is small in the $L^\infty_{\omega}$-norm and the operator $\sfG_\va[k]$ is chosen appropriately. See Proposition \ref{M:Prop: h diss}.

Now regarding the full operator $\sfF_\va[k]$, we show in Sections \ref{S/ F WP}--\ref{S/ F Hypo} that it is well-posed in the reference space $L^2_{\mu^{-1/2}}$ and that it enjoys the property of hypocoercivity in this space. However, in order to show that the dissipative property of $\sfF_\va[k]$ in $L^2_{\mu^{-1/2}}$ extends to that of $L^\infty_{\omega}$, we aim to argue in the spirit of \cite{MisMou16} and we therefore need to show that $\sfH_\va[k]$ enjoys an ultracontractive property. This is done in Section \ref{S/ ultra}, where we incorporate largely the arguments of those in \cite{CarMis24KFP} in order to make possible estimates in the spirit of De Giorgi--Nash--Moser. Thanks to the fact that we only consider the case of specular reflection boundary here, we are able to simplify some of the arguments there and prove an $L^2_{\mu^{-1/2}}\to L^\infty_{\omega}$-type estimate for the evolution system corresponding to $\sfH_\va[k]$, which is the cornerstone step in extending the functional space of decay. We mention that such an argument is also reminiscent of the framework which has been developed in \cite{GuoHwangJangOu20,KimGuoHwang20}. 

As aforementioned, in Section \ref{S/ F Hypo} we establish the hypocoercivity property of $\sfF_\va[k]$ in $L^2_{\mu^{-1/2}}$, but this is not merely possible by putting together already known results. Indeed, although the ``linearized operator'' corresponding to $\scrF$ is well-behaved and falls into the analysis of \cite{BCMT23}, it turns out that $\sfF_\va[k]$ does not, and we need some more refined analysis in order to obtain a satisfactory coercivity estimate. For this reason, instead of resorting to known results, we go back to some elementary estimates related to the linear Fokker--Planck operator to recover some better energy estimates, see Lemma \ref{lem: L coercive}. At last, using these refined estimates, we incorporate the framework developed in \cite{BCMT23} in order to prove the existence of a twisted scalar product, equivalent to that of the original one on $L^2_{\mu^{-1/2}}$, for which the Dirichlet form associated to $\sfF_\va[k]$ is strictly coercive, see Proposition \ref{prop: L_g hypo}.

Finally, in Section \ref{Sec/ Ext} we combine the results of Sections \ref{S/ ultra}--\ref{S/ F WP}--\ref{S/ F Hypo}, and use an extension trick in the spirit of \cite{GuaMisMou17} in order to prove that the evolution system associated to $\sfF_\va[k]$ extends to the $L^\infty_{\omega}$ framework and that it is dissipative in this space. The remaining steps are then the same as with the weakly inhomogeneous regime: we perform a fixed point argument to recover a solution to the nonlinear (but mollified) problem, and then we pass the regularization parameter $\va\to 0$ in order to obtain a solution to the original problem. The properties of the solution $f$ then follow by carefully passing everything that we established for the linear system to the nonlinear one.

\section{Preliminaries} \label{Sec/ Pre}

\subsection{Well-posedness of linear nonlocal Kolmogorov equations}

We provide a general well-posedness result for Kolmogorov-type equations, in which the right-hand side depends also on the macroscopic moments of the solution. Similar results have already been mentioned in the recent preprint \cite{CarMis24}, see also \cite{Car98}. We present the material here in order to keep the presentation rather self-contained.

We shall consider a general linear operator of the form below, where differentiation acts only in the velocity variable:
\begin{equation*}
    \scrL f := \sigma_{ij}\p^2_{v_iv_j} f + \zeta_i \p_{v_i} f + \eta f,
\end{equation*}
along with a general nonlocal-in-$v$ linear operator $\scrP$. For a typical example one can think of $\scrP f := \rho[f] \mu$. Also, for convenience, we make use of the notation (for $p\in [1,\infty]$)
\begin{equation}\label{p: e: Gamma}
\begin{split}
    \Gamma_{\scrL,\omega,p}(v) &:= 2\lt(1-\frac{1}{p}\rt) \sigma_{ij} \frac{\p_{v_i}\omega}{\omega} \frac{\p_{v_j}\omega}{\omega} + \lt(\frac{2}{p}-1\rt) \sigma_{ij} \frac{\p^2_{v_iv_j}\omega}{\omega} + \frac{2}{p} \p_{v_j} \sigma_{ij} \frac{\p_{v_i}\omega}{\omega} \\
        &\quad - \zeta_i \frac{\p_{v_i}\omega}{\omega} + \frac{1}{p}\p^2_{v_iv_j}\sigma_{ij} - \frac{1}{p}\p_{v_i} \zeta_i + \eta.
\end{split}
\end{equation}
The motivation behind this defect quantity lies in that it holds, at least formally, that [Lemma 2.5, CM24]
\begin{equation}\label{p: e: formal}
        \int_{\R^d} (\scrL f)|f|^{p-2} f \, \omega^p \,\dv = - \frac{4(p-1)}{p^2}\int_{\R^d} \sigma_{ij} \p_{v_i} \wt{f} \p_{v_j} \wt{f} \,\dv + \int_{\R^d} \Gamma_{\scrL,\omega,p} |f|^p\, \omega^p \,\dv,
    \end{equation}
    where $\wt{f} := f|f|^{\frac{p}{2}-1} \omega^{\frac{p}{2}}$. 

We then assume the following conditions: for the weight, we impose
\begin{enumerate}
    \item[(W1)] $\omega \in C^2(\R^d;(0,\infty))$. Moreover, $|\la v \ra^{-1} \nabla_v \omega| \lesssim \omega$ and $|\la v \ra^{-2} \nabla_v^2 \omega| \lesssim \omega$.
    \item[(W2)] $\omega$ is radial, \textit{i.e.} there exists $\omega_r:\R_+ \to (0,\infty)$ such that $\omega(v) = \omega_r(|v|)$.
\end{enumerate}

Regarding the coefficients of $\scrL$, we assume that
\begin{enumerate}
    \item[(L1)] The matrix $\sigma(t,x,v)$ is symmetric and there exists $\underline{\sigma}>0$ such that $\sigma_{ij} \bfx_i \bfx_j \ge \underline{\sigma}|\bfx|^2$ for all $\bfx = (\bfx_1,\ldots,\bfx_d)\in \R^d$.

    \item[(L2)] There exist $\lambda_{\scrL} > 0$ and $\Lambda : \R^d\to [0,\infty)$ such that the weight $\omega$ satisfies
    \begin{equation*}
        \forall v\in \R^d, \quad \Gamma_{\scrL,\omega,2}(v) \le \lambda_{\scrL} - \Lambda(v).
    \end{equation*}

    \item[(L3)] There also holds
    \begin{equation*}
        \la v \ra^{-1} \Big( |\p^2_{v_iv_j} \sigma_{ij}| + |\p_{v_i} \sigma_{ij}| + |\sigma_{ij}| + |\p_{v_i} \zeta_i| + |\zeta_i|  + |\eta|  \Big) \in L^\infty(\calO^T).
    \end{equation*}
\end{enumerate}

The system in concern is then a nonlocal Kolmogorov equation complemented with the specular reflection boundary condition \eqref{eq: specular}:
    \begin{equation}\label{p: e: kolmogorov}
    \begin{cases}
        \p_t f + v\cdot \nabla_x f = \scrL f + \scrP f + S, \\
        \gamma_- f(t,x,v) = \gamma_+ f (t,x,\scrV_x v), 
    \end{cases}
    \end{equation}
The solution class is denoted as $\calH_{\omega,T}$, the Hilbert space consisting of the (equivalence class of) functions for which the following norm
    \begin{equation*}
        \|h\|_{\calH_{\omega,T}}^2 := \|h\|_{L^2_{\omega}(\calO^T)}^2 + \|h\|_{ H^{1,\dagger}_{\omega,T} }^2, \quad \|h\|_{H^{1,\dagger}_{\omega,T}}^2 := \iiint_{\calO^T} \Big( \sigma_{ij}\p_{v_i} (\omega h) \p_{v_j} (\omega h) + \Lambda \omega^2 h^2 \Big) \,\dx\dv\dt 
    \end{equation*}
    is finite.
    
For later convenience, we set the notation
\begin{align*}
    \tnd\nu_1 := \omega^2 |v\cdot n_x| \,\tnd\sigma_x \dv.
\end{align*}
The nonlocal operator $\scrP$ is assumed to satisfy

\begin{enumerate}
    \item[(P1)] $\|\scrP\|_{ \scrB(L^2_\omega(\calO^T))} =: \lambda_{\scrP} < \infty$.
\end{enumerate}

Finally, the source term will be assumed to satisfy
\begin{enumerate}
    \item[(S1)] $S\in L^2_{t,x} H^{-1}_{v,{\rm loc}}$. Namely, $S = S_1 + \nabla_v\cdot S_2$ where $S_1,S_2\in L^2_{\rm loc}([0,T]\times\ov{\calO})$. Also, there exists $\lambda_S > 0$ such that
    \begin{equation*}
        \|S_1 \omega\|_{L^2(\calO^T)}^2 + \|S_2 \nabla_v \omega\|_{L^2(\calO^T)}^2 + \frac{1}{2 \underline{\sigma}} \|S_2 \omega\|_{L^2(\calO^T)}^2 \le \lambda_S.
    \end{equation*}
\end{enumerate}

\begin{proposition}\label{prop: WP}
    Let $\scrL$, $\scrP$, and $S$ satisfy the conditions above. Then for each prescribed initial datum $f_0 \in L^2_\omega(\calO)$, there exists a corresponding unique weak solution $f\in C([0,T];L^2_\omega(\calO)) \cap \calH_{\omega,T}$ to \eqref{p: e: kolmogorov}. The solution satisfies the equation $\eqref{p: e: kolmogorov}_1$ in duality with $\calD(\calO^T)$, and furthermore it is a renormalized solution.
\end{proposition}

\begin{remark}
    By a renormalized solution to $\eqref{p: e: kolmogorov}_1$, we mean that for any $t\in (0,T]$, $\beta\in W^{2,\infty}(\R)$, and $\chi \in \calD([0,t]\times \ov{\calO})$, the following Green's formula holds
    \begin{equation}\label{P: E: Renorm Green Formula}
    \begin{split}
        &\iint_{\calO} \beta(f_t) \chi(t) \,\dx\dv - \iint_{\calO} \beta(f_0) \chi(0)\,\dx\dv + \iiint_{\Sigma^t} \beta(\gamma f) \, \chi \, (n_x \cdot v) \,\ds\tnd\sigma_x\dv \\
        &\quad + \iiint_{\calO^t} \lt[ \beta(f) \Big( -\p_t \chi - v\cdot \nabla_x \chi - \p^2_{v_i v_j}(\sigma_{ij}\chi) - \p_{v_i}(\zeta_i \chi) \Big) + \beta''(f) \sigma_{ij} \p_{v_i}f \p_{v_j}f \,\chi \rt] \,\ds\dx\dv\\
        &= \iiint_{\calO^t} \beta'(f) (\eta f + \scrP f) \, \chi \,\ds\dx\dv + \iiint_{\calO^t} (S_1 - S_2 \cdot \nabla_v) (\beta'(f) \chi) \,\ds\dx\dv.
    \end{split}
    \end{equation}
\end{remark}

\begin{proof}[Proof of Proposition \ref{prop: WP}]
    As aforementioned, similar arguments are available in \cite{CarMis24}; we provide all details for the sake of completeness. 

    (Step 1: The inflow problem.) First, we consider the problem with fixed inflow
    \begin{equation}\label{p: e: inflow}
        \begin{cases}
            \p_t f + v\cdot \nabla_x f = \scrL f + \scrP f + S, \\
            \gamma_- f = k \quad \text{on }\Sigma_-^T, \\
            f|_{t=0} = f_0,
        \end{cases}
    \end{equation}
    where $k \in L^2$. Then we dampen the equation enough to obtain coercivity: we set $\lambda > \lambda_{\scrL} + \lambda_{\scrP}$ and $f_\lambda:= f e^{-\lambda t}$, in order to instead solve for
    \begin{equation}\label{p: e: dmp inflow}
        \begin{cases}
            \p_t f_\lambda + v\cdot \nabla_x f_\lambda = \scrL f_\lambda + \scrP f_\lambda - \lambda f_\lambda + S_\lambda, \\
            S_\lambda := S e^{-\lambda t}, \\
            \gamma_- f_\lambda = k_\lambda := k e^{-\lambda t} \quad \text{on }\Sigma_-^T, \\
            f_\lambda|_{t=0} = f_{\lambda 0} := f_0 e^{-\lambda t}.
        \end{cases}
    \end{equation}
    Next, we define on $\calH_{\omega,T} \times C_c^1(\calO^T \cup \Sigma_-^T)$ the bilinear and real-valued form
    \begin{equation*}
        \bfB(h,\chi) := \iiint_{\calO^T} \Big( \lambda h \chi \omega^2 - h \scrL^* (\chi \omega^2) - \chi\omega^2 \scrP h - h (\p_t\chi + v\cdot \nabla_x \chi) \omega^2 \Big).
    \end{equation*}
    In this way, to find a distributional solution to \eqref{p: e: dmp inflow}, it is enough to search for $f_\lambda \in \calH_{\omega,T}$ which satisfies the variational problem
    \begin{equation*}
    \begin{split}
        &\forall \chi \in C_c^1(\calO^T \cup \Sigma_-^T),\\
        &\bfB(f_\lambda, \chi) = \iiint_{\calO^T} S_\lambda \chi \omega^2 \,\dx\dv\dt + \iiint_{\Sigma_-^T} k_\lambda \chi \omega^2 (n(x)\cdot v)_- \,\tnd\sigma_x \dv \dt + \iint_{\calO} f_{\lambda 0} \chi(0) \omega^2 \,\dx\dv.
    \end{split}
    \end{equation*}
    Thanks to the conditions (L2)--(P1)--(S1), we note that $\bfB$ is coercive, since indeed
    \begin{align*}
        \bfB(\chi, \chi) &= \iiint_{\calO^T} \Big( \lambda \chi^2 \omega^2 - \chi \scrL^* (\chi \omega^2) - \chi \omega^2 \scrP \chi \Big) \,\dx\dv\dt  \\
        &\ge \iiint_{\calO^T} (\lambda - \Gamma_{\scrL,\omega,p} -  \lambda_{\scrP} ) \chi^2 \omega^2 \,\dx\dv\dt \\
        &\ge \iiint_{\calO^T} (\lambda - \lambda_{\scrL} -  \lambda_{\scrP} ) \chi^2 \omega^2 \,\dx\dv\dt.
    \end{align*}
    Then owing to Lions' version of the Lax--Milgram theorem \cite[Chap III, \S 1]{Lions61}, we deduce the existence of a solution $f_\lambda$ to \eqref{p: e: dmp inflow}. Consequently $f := f_\lambda e^{\lambda t}$ is a distributional solution to \eqref{p: e: inflow}. By virtue of (straightforward adaptations of) trace results in \cite{Mi10}, we then obtain that $f \in C([0,T];L^2_\omega(\calO))$ and that $f$ satisfies the renormalized Green's formula as in \eqref{P: E: Renorm Green Formula}.
    
    (Step 2: Energy estimates.) Taking $\beta(z)\nearrow |z|^2$ and $\chi(v) \nearrow 1$, we obtain the energy estimate
    \begin{equation}\label{p: e: egy start}
    \begin{split}
        &\iint_{\calO} f_t^2 \omega^2 \,\dx\dv - \iint_{\calO} f_0^2 \omega^2 \,\dx\dv + 2\|f\|_{H^{1,\dagger}_{\omega,t}}^2+ \|\gamma f\|_{L^2(\Gamma^t ,\tnd\nu_1)}^2 \\
        &\le 2(\lambda_\scrL + \lambda_\scrP ) \iiint_{\calO^t} f^2 \omega^2 \,\dx\dv\ds + 2 \iiint_{\calO^t} S f\omega^2 \,\dx\dv\ds. 
    \end{split}
    \end{equation}
    For the integral containing the source term, we can estimate
    \begin{align*}
        &\iiint_{\calO^t} S f\omega^2 \,\dx\dv\ds \\
        &= \iiint_{\calO^t} S_1 f \omega^2 \,\dx\dv\ds - \iiint_{\calO^t} S_2 \cdot \nabla_v (f\omega^2) \,\dx\dv\ds \\
        &= \iiint_{\calO^t} S_1 f \omega^2 \,\dx\dv\ds - \iiint_{\calO^t} \Big( S_2 \omega \cdot \nabla_v (f \omega) + S_2 f \omega \cdot \nabla_v \omega \Big) \,\dx\dv\ds \\
        &\le \|S_1\|_{L^2_{\omega}(\calO^t)} \|f\|_{L^2_{\omega}(\calO^t)} + \|S_2 \omega\|_{L^2(\calO^t)} \|\nabla_v(f\omega)\|_{L^2(\calO^t)}  + \|S_2 \nabla_v\omega\|_{L^2(\calO^t)} \|f\|_{L^2_{\omega}(\calO^t)} \\
        &\le \|f\|_{L^2_\omega(\calO^t)}^2 + \frac{1}{2}\underline{\sigma}\|\nabla_v (\omega h)\|_{L^2}^2 + \lambda_S \\
        &\le \|f\|_{L^2_{\omega}(\calO^t)}^2 + \frac{1}{2} \|f\|_{H^{1,\dagger}_{\omega,t}}^2 + \lambda_S .
    \end{align*}
    We point out that the second term concerning $\|f\|_{H_{\omega,t}^{1,\dagger}}^2$ can be absorbed into the left-hand side of \eqref{p: e: egy start}. Consequently, we obtain with $\wt{\lambda} := 2(\lambda_\scrL + \lambda_\scrP + 1)$ that
    \begin{equation}\label{p: e: egy est}
    \begin{split}
        &\|f_t\|_{L^2_{\omega}(\calO)}^2 + \int_0^t \lt( \|\gamma_- f \|_{L^2(\Gamma_-, \tnd\nu_1) }^2 + \|f\|_{H^{1,\dagger}_{\omega,s}}^2 \rt) e^{\wt{\lambda}(t-s)} \,\ds \\
        &\le \|f_0\|^2_{L^2_{\omega}(\calO)} e^{\wt{\lambda}t} + \frac{2\lambda_S}{\wt{\lambda}} (e^{\wt{\lambda}t} - 1 ) + \int_0^t \|k(s)\|^2_{L^2(\Gamma_-, \tnd\nu_1)} e^{\wt{\lambda}(t-s)}\,\ds,
    \end{split}
    \end{equation}
    which also implies uniqueness of the solution to \eqref{p: e: inflow}. Indeed, if $f_1, f_2$ are two solutions to \eqref{p: e: inflow}, then $f:= f_2 - f_1$ would solve \eqref{p: e: inflow} with inflow datum $k \equiv 0$, initial datum $f_0 \equiv 0$, and source $S\equiv 0$. Thus we could apply \eqref{p: e: egy est} to conclude (noting that we can set $\lambda_S = 0$).

    (Step 3: Approximating the specular reflection boundary condition.) We let $\alpha\in (0,1)$ and $k\in \calH_{\omega,T} \cap C([0,T];L^2_\omega(\calO))$ be given, with trace $\gamma k\in L^2(\Gamma,\tnd\nu_1)$. Then we consider the partial absorption-reflection problem 
    \begin{equation}\label{p: e: partial}
        \begin{cases}
            \p_t f + v\cdot \nabla_x f = \scrL f + \scrP f + S, \\
            \gamma_- f(t,x,v) = \alpha \gamma_+ k(t,x,\scrV_x v), \\
            f|_{t=0} = f_{0}.
        \end{cases}
    \end{equation}
    Thanks to the results in Steps 1--2, we deduce the existence and uniqueness of $f$ the solution to the above problem. Similarly as with the end of Step 2, let $k = k_1, k_2$ in $\eqref{p: e: partial}_2$, and denote by $f_1,f_2$ the respective solutions to \eqref{p: e: partial}. Then $\mathfrak{f} := f_2 - f_1$ is a solution to \eqref{p: e: partial} with source term $S\equiv 0$, initial datum $\mathfrak{f}_0 \equiv 0$, and boundary condition
    \begin{equation*}
        \gamma_- \mathfrak{f}(t,x,v) = \alpha \gamma_+ (k_2 - k_1)(t,x,\scrV_x v).
    \end{equation*}
    By using the energy estimate in \eqref{p: e: egy est} (with zero initial datum and source term), we get
    \begin{equation*}
        \|\mathfrak{f}_t\|_{L^2_\omega}^2 + \int_0^t \|\gamma_- \mathfrak{f}\|_{L^2(\Gamma_-, \tnd\nu_1)}^2 \,\ds \le \alpha \int_0^t \|k_2 - k_1\|_{L^2(\Gamma_-, \tnd\nu_1)}^2 e^{\wt{\lambda}(t-s)} \,\ds,
    \end{equation*}
    and therefore the map $k \mapsto f$ is Lipschitz, with Lipschitz constant $\alpha$, with respect to the norm
    \begin{equation*}
        \NORMMM{f}^2 := \sup_{t\in [0,T]} \lt( \|f_t\|_{L^2_{\omega}(\calO)}^2 + \int_0^t \|\gamma f_s\|_{L^2(\Gamma_-,\tnd\nu_1)}^2 e^{\wt{\lambda}(t-s)} \,\ds \rt).
    \end{equation*}
    Since $\alpha<1$, we apply the contraction mapping theorem to deduce the existence and uniqueness of a fixed point of this map. This one is the unique distributional and renormalized solution to
    \begin{equation}\label{p: e: alpha}
        \begin{cases}
            \p_t f + v\cdot \nabla_x f = \scrL f + \scrP f + S, \\
            \gamma_- f(t,x,v) = \alpha \gamma_+ f(t,x,\scrV_x v), \\
            f|_{t=0} = f_0.
        \end{cases}
    \end{equation}
    For this solution, the same energy estimate as in \eqref{p: e: egy est} provides
    \begin{equation}\label{eq: energy again} 
    \begin{split}
        &\|f_t\|_{L_{\wt{\omega}(\calO)}^2}^2 + \int_0^t \Big( (1-\alpha) \|\gamma_- f_t\|_{L^2(\Gamma_-, \tnd\nu_1)}^2 + \|f\|_{H^{1,\dagger}_{\omega,t}}^2 \Big) e^{\wt{\lambda}(t-s)}\,\ds \\
        &\le \|f_0\|_{L^2_{\omega}(\calO)}^2 e^{\wt{\lambda} t} + \frac{2\lambda_S}{\wt{\lambda}}(e^{\wt{\lambda}t} - 1).
    \end{split}
    \end{equation}

    We now set $\alpha = \alpha_k\in (0,1)$ in $\eqref{p: e: alpha}_2$, with $\alpha_k\nearrow 1$. Denote by $f^k$ the solution to \eqref{p: e: alpha} corresponding to $\alpha_k$. Thanks to the energy estimate in \eqref{eq: energy again}, we observe that we may admit $\beta(z)=z^2$ and $\chi(x,v) = (n_x \cdot v)\omega^2 \lt<v\rt>^{-4}$ as an admissible pair into the renormalized formulation satisfied by the $f^k$. Owing to (W2), we have $\chi(x,\scrV_x v) = - \chi(x,v)$, thus 
    \begin{equation}\label{p: e: alphak}
    \begin{split}
        &\iint_{\calO} |f^k_t \omega|^2 (n_x\cdot v) \la v \ra^{-4} \,\dx\dv - \iint_{\calO} |f^k_0 \omega|^2 (n_x\cdot v) \la v \ra^{-4} \,\dx\dv \\
        &\quad + (1 + \alpha_k) \iiint_{\Sigma_+^t} |\gamma_+ f^k|^2 (n_x\cdot v)^2 \omega^2 \la v \ra^{-4} \,\tnd\sigma_x \dv \ds \\
        &= \iiint_{\calO^t} |f^k \omega|^2 \Big( v \cdot D_x n_x \, v \, \la v \ra^{-4} + \omega^{-2} \p^2_{v_i v_j} (\sigma_{ij} \chi ) + \omega^{-2} \p_{v_i} (\zeta_i \chi)  \Big) \,\dx\dv\ds \\
        &\quad + \iiint_{\calO^t} 2 \sigma_{ij} \p_{v_i} f^k \p_{v_j} f^k \,\chi \,\dx\dv\ds + \iiint_{\calO^t} 2 f^k (\eta f^k + \scrP f^k) \chi \,\dx\dv\ds \\
        &\quad + \iiint_{\calO^t} (S_1 - S_2 \cdot \nabla_v ) (2f^k \chi) \,\dx\dv\ds.
    \end{split}
    \end{equation}
    Using the energy bounds in \eqref{eq: energy again}, along with (W1)--(L3)--(P1)--(S1), we can readily check that every term involving $f^k$ is bounded uniformly in $k$. For instance,
    \begin{align*}
        \iiint_{\calO^T} 2\sigma_{ij} \p_{v_i} f^k \p_{v_j} f^k |n(x)\cdot v|\omega^2 \lt<v\rt>^{-4}  &\lesssim \iiint_{\calO^T} |\nabla_v f^k|^2 \omega^2 \la v \ra^{-2} \,\dx\dv\ds \\
        &= \iiint_{\calO^T} |\omega \nabla_v f^k|^2 \la v \ra^{-2} \,\dx\dv\ds \\
        &= \int_{\calO^T} \lt|\nabla_v (f^k\omega) - f^k \nabla_v \omega \rt|^2 \la v \ra^{-2} \,\dx\dv\ds \\
        &\lesssim \int_{\calO^T} \Big(|\nabla_v (f^k\omega)|^2 + (f^k)^2 \lt|\la v \ra^{-1} \nabla_v \omega \rt|^2 \Big) \,\dx\dv\ds \\
        &\lesssim \int_{\calO^T} \Big(|\nabla_v (f^k \omega)|^2 + |f^k|^2 \omega^2 \Big) \,\dx\dv\ds \\
        &\le C .
    \end{align*}
    From \eqref{p: e: alphak} we can therefore deduce
    \begin{equation*}
        \int_{\Sigma_+^T} |\gamma_+ f^k|^2 (n(x)\cdot v)^2 \omega^2 \lt<v\rt>^{-4} \,\dt\tnd\sigma_x \dv \le C.
    \end{equation*}
    Since $\gamma_- f^k = \alpha_k (\gamma_+ f^k \circ \scrV_x)$, this also implies
    \begin{equation*}
        \int_{\Sigma_-^T} |\gamma_- f^k|^2 (n_x\cdot v)^2 \omega^2 \la v \ra^{-4} \,\dt\tnd\sigma_x \dv \le C \alpha_k^2 \le C.
    \end{equation*}
    Therefore, there are limits $\xi_\pm$ such that modulo a subsequence
    \begin{equation}\label{p: e: traces}
        \gamma_\pm f^k \weakto \xi_\pm \quad \text{in}\quad L^2(\Sigma^T, (n(x)\cdot v)^2 \omega^2 \lt<v\rt>^{-4}\dt\tnd\sigma_x\dv) - w.
    \end{equation}
    To see that $\xi_- = \xi_+ \circ \scrV_x$, we write
    \begin{align*}
        &\xi_- - \xi_+ \circ \scrV_x \\
        &= (\xi_- - \gamma_- f^k) + (\gamma_- f^k - \gamma_+ f^k \circ \scrV_x ) + (\gamma_+ f^k \circ \scrV_x  - \xi_+ \circ \scrV_x) \\
        &= (\xi_- - \gamma_- f^k) + (\alpha_k - 1) \gamma_+ f^k \circ \scrV_x + (\gamma_+ f^k \circ \scrV_x  - \xi_+ \circ \scrV_x).
    \end{align*}
    Then, using the weak convergences in \eqref{p: e: traces} and the fact that $\alpha_k \nearrow 1$, we can prove that $\xi_- = \xi_+ \circ \scrV_x$ in the sense of distributions and thus a.e.

    Now passing to the limit $\alpha_k\nearrow 1$ in all equations of \eqref{p: e: alpha} (with the help of (P1)), we obtain that the limit $f$ satisfies the following Green's formula in duality with $\chi\in \calD((0,T)\times \ov{\calO})$:
    \begin{equation} \label{P: E: xi weak}
        \begin{split}
            &\iiint_{\Sigma^T} \xi \chi (n_x\cdot v) \,\tnd\sigma_x \dv \dt \\  
            &= \iiint_{\calO^T} \Big( f (\p_t + v\cdot \nabla_x + \scrL^*) \chi + (\scrP f  + S) \chi  \Big) \,\dx\dv\dt.
        \end{split}
    \end{equation}
    In particular $f$ is a solution to $\eqref{p: e: kolmogorov}_1$ in $\calD'((0,T)\times \ov{\calO})$. Owing to the facts that $f \in L^2_{t,x}H^1_{v, {\rm loc}}$ and $S\in L^2_{t,x}H^{-1}_{v,{\rm loc}}$, we may straightforwardly adapt \cite[Theorem 4.2]{Mi10} once again, to deduce that $f$ is in fact a renormalized solution with unique traces $f_t$ and $\gamma f$ which satisfy \eqref{P: E: Renorm Green Formula} for every $\chi\in \calD([0,T]\times \ov{\calO})$. Choosing $\chi\in \calD((0,T)\times \ov{\calO})$, we in particular obtain from \eqref{P: E: Renorm Green Formula} that
    \begin{align*}
        &\iiint_{\Sigma^T} \beta(\gamma f) \chi (n_x\cdot v) \,\tnd\sigma_x \dv\dt \\
        &= \iiint_{\calO^T} \lt[ \beta(f) \Big(\p_t \chi + v\cdot \nabla_x \chi + \p^2_{v_i v_j}(\sigma_{ij}\chi) + \p_{v_i}(\zeta_i \chi) \Big) - \beta''(f) \sigma_{ij} \p_{v_i}f \p_{v_j}f \,\chi \rt] \,\dt\dx\dv\\
        &\quad + \iiint_{\calO^T} \beta'(f) (\eta f + \scrP f) \, \chi \,\dt\dx\dv + \iiint_{\calO^t} (S_1 - S_2 \cdot \nabla_v) (\beta'(f) \chi) \,\dt\dx\dv ,
    \end{align*}
    for every $\beta\in C^2\cap W^{2,\infty}$. Since $f$ is already a weak solution, and since $\gamma f\in L^1_{\rm loc}(\Sigma^T,|n_x\cdot v|\tnd\sigma_x\dv\dt)$, we can choose an approximating sequence $\beta(z)\nearrow z$ in order to deduce that
    \begin{equation}\label{P: E: xi compare}
    \begin{split}
        &\iiint_{\Sigma^T} \gamma f \, \chi \, (n_x\cdot v) \,\tnd\sigma_x \dv \dt \\
        &\quad + \iiint_{\calO^T} \Big( f (\p_t + v\cdot \nabla_x + \scrL^*) \chi + (\scrP f + S) \chi \Big) \,\dx\dv\dt.
    \end{split}
    \end{equation}
    Comparing \eqref{P: E: xi compare} with \eqref{P: E: xi weak}, we deduce that $\xi (n_x\cdot v) = \gamma f (n_x\cdot v)$ in the sense of $\calD'((0,T)\times \ov{\calO})$, and consequently a.e. in $\Sigma^T_{\pm}$. Recalling that $\xi_- = \xi_+\circ \scrV_x$ we immediately deduce that $\gamma_- f = \gamma_+ f \circ \scrV_x$ also. In particular, we have proven that $f$ is a weak and renormalized solution to
    \begin{equation}\label{p: e: original reflection}
    \begin{cases}
        \p_t f + v\cdot \nabla_x f = \scrL f + \scrP f + S, \\
        \gamma_- f = \gamma_+ f \circ \scrV_x , \\
        f|_{t=0} = f_{in}.
    \end{cases}
    \end{equation}
    As before we obtain also that $f\in C([0,T];L^2_\omega(\calO))$ by the arguments in \cite{Mi10}.

    (Step 4: Uniqueness.)
    We consider two solutions to the system \eqref{p: e: original reflection} with the same initial datum $f_0$; then their difference, which we denote $\mathfrak{f}$, is again a solution to \eqref{p: e: original reflection}, now with both initial datum and source term zero. Then, we can obtain an energy estimate for $\mathfrak{f}$ that is similar to that of \eqref{eq: energy again}, but with $\alpha = 1$, $\mathfrak{f}_0 \equiv 0$, and $\lambda_S \equiv 0$. We deduce uniqueness in the class $C([0,T];L^2_\omega(\calO))$, as desired.
\end{proof}


For later use, we also mention a useful fact on the following backwards adjoint equation, complemented with the dual specular reflection boundary condition:
    \begin{equation}\label{p: e: dual kolmogorov}
        \begin{cases}
            -\p_t b - v\cdot \nabla_x b = \scrL^* b + \scrP^* b \quad \text{in }\calO^T,\\
            \gamma_+ b = \gamma_- \circ \scrV_x \quad \text{on }\Sigma_+^T.
        \end{cases}
    \end{equation}
In the above, we denote by $\scrP^*$ the formal adjoint operator of $\scrP$. We assume $\scrP^*\in \scrB(L^2_{\omega^{-1}}(\calO^T))$ and that
\begin{equation*}
    \iint_{\calO} (\scrP f) b \,\dx\dv = \iint_{\calO} f \,\scrP^* b \,\dx\dv  
\end{equation*}
holds for any $f\in L^2_{\omega}(\calO)$ and $b\in L^2_{\omega^{-1}}(\calO)$.

The result we mention is a duality formula.
    
\begin{lemma}\label{p: l: dual}
    We assume that $f\in C([0,T];L^2_{\omega}(\calO)) \cap L^2((0,T)\times\Omega; H^1_{\omega}(\R^d))$ is a solution to \eqref{p: e: kolmogorov}, corresponding to initial datum $f_0\in L^2_\omega(\calO)$. Analogously, we suppose that $b\in C([0,T];L^2_{\omega^{-1}}(\calO)) \cap  L^2((0,T)\times\Omega; H^1_{\omega^{-1}}(\R^d))$ is a solution to the backwards equation \eqref{p: e: dual kolmogorov} corresponding to terminal datum $b_T\in L^2_{\omega^{-1}}(\calO)$. On the weight, let us assume
    \begin{enumerate}
        \item[(W1')] $\omega\in C^2(\R^d;(0,\infty))$, $\omega\gtrsim 1$, and $|\nabla_v \omega| \lesssim \omega$.  
    \end{enumerate}
    Then under the assumptions (W1')--(W2), (L1)--(L2)--(L3), and (P1), it holds that
    \begin{equation*}
        \iint_{\calO} f_T b_T \,\dx\dv = \iint_{\calO} f_0 b_0 \,\dx\dv + \iiint_{\calO^T} S b \,\dx\dv\dt,
    \end{equation*}
    whenever the last integral of the right-hand side is finite.
\end{lemma}

\begin{remark}
    In practice, we will usually consider a polynomial weight $\omega = \la v \ra^m$ so that (W1') is satisfied. Also, typically $m$ will be taken large enough so that $L^2_{\la v \ra^m} \hookrightarrow L^1_{\la v \ra^2}$. This inclusion is important later on for the weight to be compatible with the nonlocal operator $\scrP$, in particular to ensure that the condition (P1) is satisfied.
\end{remark}

\begin{proof}[Proof of Lemma \ref{p: l: dual}]
    (Step 1: Regularizing the problem.)
    Consider mollified versions $f_\ve$ and $b_\ve$ of $f$ and $b$, respectively. Namely, let
    \begin{equation*}
        (f)_\ve = f_\ve := f \star_\ve \bfr_\ve *_v \bfr_\ve,
    \end{equation*}
    where $(\bfr_\ve)$ is the usual family of mollifiers supported in the ball $B_\ve$, and the convolution-translation operation $\star_\ve$ is defined in the spatial variable by (see for instance \cite{BloLeD01})
    \begin{align*}
        \forall x\in \ov{\Omega}, \quad f \star_\ve \bfr_\ve (x) := \int_\Omega f(y) \bfr_\ve(x - 2\ve n_x - y) \,\dy.
    \end{align*}
    For each fixed $(x,v)\in \calO$, we observe that
    \begin{equation*}
        \bfr_\ve(x-2\ve n_x - y) \bfr_\ve(v-w) \in \calD(\calO_{y,w}),
    \end{equation*}
    and thus, using this as test function in the distributional formulation of $f$, we obtain that $f_\ve$ is a solution to
    \begin{equation}\label{P/Eq/ reg rough}
    \begin{split}
        \p_t f_\ve + (v\cdot \nabla_x f)_\ve &= (f \p^2_{v_iv_j}\sigma_{ij})_\ve - 2 (f \p_{v_j}\sigma_{ij}) \star_\ve \bfr_\ve *_v \p_{v_i} \bfr_\ve + (f \sigma_{ij}) \star_\ve \bfr_\ve *_v \p_{v_i v_j}^2 \bfr_\ve \\
        &\quad + (f \nabla_v \cdot \zeta) \star_\ve \bfr_\ve *_v \bfr_\ve - (f \zeta) \star_\ve \bfr_\ve *_v \nabla_v \cdot \bfr_\ve \\
        &\quad + (\eta f)_\ve + (\scrP f)_\ve + (S)_\ve \quad \text{in} \quad \calD'((0,T)\times \ov{\calO}).
    \end{split}
    \end{equation}
    From this very equation it follows that $f_\ve \in W^{1,1}((0,T);W^{1,\infty}_{\rm loc}(\ov{\calO}))$. We may therefore assume without loss of generality that $f_\ve$ is continuous on $[0,T]\times \ov{\calO}$, and thus $f_\ve$ admits traces $\gamma f_\ve$ and $f_{\ve t}$. Next, we observe that $f_\ve \to f$ in $L^2_{t,x} H^1_{v, \rm loc}$ by standard property of the mollifier. We claim that this implies $\gamma f_\ve \to \gamma f$ strongly in a suitable norm. Indeed, observe that the PDE for $f_\ve$ may be rewritten as
    \begin{align*}
        \p_t f_\ve + v\cdot \nabla_x f_\ve = \scrL f_\ve + (\scrP f)_\ve + (S)_\ve + R_\ve^1,
    \end{align*}
    where $R_\ve^1$ is the remainder term (commutator)
    \begin{align*}
        R_\ve^1 &:= v\cdot \nabla_x f_\ve - (v\cdot \nabla_x f)_\ve \\
        &\quad - 2 \p_{v_i} \Big(  (f \p_{v_j}\sigma_{ij})_\ve - f_\ve \p_{v_j} \sigma_{ij} \Big) + \p^2_{v_i v_j} \Big( (f\sigma_{ij})_\ve - f_\ve \sigma_{ij} \Big) \\
        &\quad + \Big( (f \nabla_v\cdot \zeta)_\ve - f_\ve \nabla_v\cdot \zeta \Big) + \nabla_v\cdot \Big( (f\zeta)_\ve - f_\ve \zeta \Big) + \Big( (\eta f)_\ve - \eta f_\ve \Big) .
    \end{align*}
    
    Thanks to the bounds of (L3), the arguments of \cite{DiPLion89,Mi10} show that $R_\ve^1 \to 0$ at least in $L^1_{\rm loc}(\calO^T) + L^2_{t,x} H^{-1}_{v, {\rm loc}}$. Now using this PDE formulation, and noting that the chain rule applies to $f_\ve$ (since $f_\ve\in W^{1,1}_t W^{1,\infty}_{x,v,{\rm loc}}$), we may obtain estimates in the manner of \eqref{p: e: alphak}. Namely, using the test function $\varphi_R(x,v) := (n_x\cdot v) \omega^2 \chi_R(v)$ where $\chi_R\in C_c^\infty(B_{R+1})$ and $\mathbf{1}_{|v|\le R} \le \chi_R \le \mathbf{1}_{|v|\le R+1}$, we obtain (using the notation $\calO^T_R := (0,T)\times \Omega \times B_R$)
    \begin{equation}\label{P: E: Trace Egy}
    \begin{split}
        &\iiint_{\calO^T_R} |\gamma f_\ve|^2 (n_x\cdot v)^2 \omega^2 \,\dx\dv\dt \\
        &\lesssim_{R,\omega} \iint_{\Omega\times B_{R+1}} (|f_{\ve T}|^2 + |f_{\ve 0}|^2) \,\dx\dv  \\
        &\quad + \iiint_{ \calO^T_{R+1} } |f_\ve \omega|^2 \Big| v \cdot D_x n_x \, v + \omega^{-2} \p^2_{v_i v_j} (\sigma_{ij} \varphi_R ) + \omega^{-2} \p_{v_i} (\zeta_i \varphi_R )  \Big|  \,\dx\dv\dt \\
        &\quad + 2 \lt| \iiint_{ \calO^T_{R+1} }  \Big( \sigma_{ij} \p_{v_i} f_\ve \p_{v_j} f_\ve + \eta f_\ve^2  + 2f_\ve (\scrP f)_\ve \Big) \omega^2 \,\dx\dv\dt \rt|  \\
        &\quad + \lt| \iiint_{ \calO^T_{R+1} } ( (S)_\ve + R_\ve^1 ) (2f_\ve \varphi_R ) \,\dx\dv\dt \rt| .
    \end{split}
    \end{equation}
    Using the bounds in (W1)--(L3)--(P1), and noting that the domain of integration is bounded, we can crudely bound the terms in the right-hand side of \eqref{P: E: Trace Egy} in the following manner. First,
    \begin{align*}
        &\iiint_{\calO^T_{R+1}} |f_\ve|^2 \Big| v \cdot D_x n_x \, v +  \p^2_{v_i v_j} (\sigma_{ij} \varphi_R ) +  \p_{v_i} (\zeta_i \varphi_R)  \Big| \,\dx\dv\dt  \le C_{R,\omega} \iiint_{\calO^T_{R+1}} |f_\ve|^2 \,\dx\dv\dt, \\
        &\lt|\iiint_{\calO^T_{R+1}} \sigma_{ij} \p_{v_i} f_\ve \p_{v_j} f_\ve \omega^2 \,\dx\dv\dt \rt|  \le C_{R,\omega} \iiint_{\calO^T_{R+1}} |\nabla_v f_\ve|^2 \,\dx\dv\dt .
    \end{align*}
    Next, 
    \begin{align*}
        \iiint_{\calO^T_{R+1}} |\eta f_\ve^2  + f_\ve (\scrP f)_\ve| \omega^2 &\le C_R \|f_\ve\|_{L^2(\calO^T_{R+1})}^2 + \lt(\iiint_{\calO^T_{R+1}} |f_\ve|^2 \omega^2 \rt)^{\frac{1}{2}} \lt( \iiint_{\calO^T_{R+1}} |(\scrP f)_\ve|^2 \omega^2  \rt)^{\frac{1}{2}} \\
        &\le C_R \|f_\ve\|_{L^2(\calO^T_{R+1})}^2 + C_R \|f_\ve\|_{L^2(\calO^T_{R+1})} \lt(\iiint_{\calO_{R+1}^T} |(\scrP f)_\ve|^2 \rt)^{\frac{1}{2}} \\
        &\le C_R \|f_\ve\|_{L^2(\calO^T_{R+1})}^2 + C_R \|f_\ve\|_{L^2(\calO^T_{R+1})} \lt(\iiint_{\calO^T_{R+1+\ve}} |\scrP f|^2 \rt)^{\frac{1}{2}} \\
        &\le C_R \|f_\ve\|_{L^2(\calO^T_{R+1})}^2 + C_{R,\omega} \|f_\ve\|_{L^2(\calO^T_{R+1})} \lt(\iiint_{\calO^T_{R+1+\ve}} |\scrP f|^2 \omega^2 \rt)^{\frac{1}{2}}  \\
        &\le C_R \|f_\ve\|_{L^2(\calO^T_{R+1})}^2 + \lambda_{\scrP} \, C_{R,\omega} \|f_\ve\|_{L^2(\calO^T_{R+1})} \|f\|_{L^2_\omega(\calO^T)}.
    \end{align*}
    In the above, the first inequality is just Cauchy--Schwarz; the second is because $1\lesssim \omega \lesssim 1$ on $B_{R+1}$; the third inequality is a standard computation which follows from the definition of the mollifier; the fourth inequality uses again that $1\lesssim \omega$, and the last inequality owes to (P1). For the final term in \eqref{P: E: Trace Egy}, we have:
    \begin{align*}
        \lt| \iiint_{\calO^T_{R+1}} ( (S)_\ve + R_\ve^1 ) (2f_\ve \varphi_R)  \rt| &\le C_R  \|S\|_{L^2((0,T)\times \Omega; H^{-1}(B_{R+1}))} \|f_\ve\|_{L^2((0,T)\times \Omega; H^1(B_{R+1}))} \\
        &\quad + C_R \|R_\ve^1\|_{L^2((0,T)\times\Omega;H^{-1}(B_{R+1}))} \|f_\ve\|_{L^2((0,T)\times\Omega;H^1(B_{R+1}))}.
    \end{align*}
    by using Young's convolution inequality. Consequently,
    \begin{align*}
        \iiint_{(0,T)\times \p\Omega\times B_{R}} |\gamma f_\ve|^2 (n_x\cdot v)^2 \omega^2 &\lesssim_R \|f_{\ve T}\|_{L^2(\Omega\times B_{R+1})}^2 + \|f_{\ve 0}\|_{L^2(\Omega\times B_{R+1})}^2 \\
        &\quad + \|f_\ve\|_{L^2((0,T)\times\Omega; H^1(B_{R+1}))}^2 + \|f_\ve\|_{L^2(\calO^T_{R+1})} \|f \|_{L^2_\omega(\calO^T)} \\
        &\quad + \|S\|_{L^2((0,T)\times\Omega; H^{-1}(B_{R+1}))} \|f_\ve\|_{L^2((0,T)\times\Omega; H^1(B_{R+1}))}  \\
        &\quad + \|R_\ve^1\|_{L^2((0,T)\times\Omega;H^{-1}(B_{R+1}))} \|f_\ve\|_{L^2((0,T)\times\Omega;H^1(B_{R+1}))} \\
        &\lesssim_{R,T} \|f_{\ve T}\|_{L^2(\Omega\times B_{R+1})}^2 + \|f_{\ve 0}\|_{L^2(\Omega\times B_{R+1})}^2 \\
        &\quad + \|f_\ve\|_{L^2((0,T)\times\Omega;H^1(B_{R+1}))}^2 + \|f_\ve\|_{L^2(\calO^T_{R+1})} \\
        &\quad + \|f_\ve\|_{L^2((0,T)\times\Omega;H^1(B_{R+1}))} + o(\ve),
    \end{align*}
    the last inequality due to the facts that $f\in L^2_\omega(\calO^T)$, $S\in L^2_{t,x} H^{-1}_{v, {\rm loc}}$, and because $(f_\ve)$ is bounded in $L^2_{t,x}H^1_{v, {\rm loc}}$. Now owing to the linear nature of the problem, for any $\ve_1,\ve_2>0$ we may apply this estimate to $f_{\ve_2} - f_{\ve_1}$. Then there holds
    \begin{equation} \label{P: E: TraceCauchy}
    \begin{split}
        &\iiint_{(0,T)\times \p\Omega \times B_R} |\gamma f_{\ve_2} - \gamma f_{\ve_1}|^2 (n_x\cdot v)^2 \omega^2 \\
        &\lesssim \|f_{\ve_2 T} - f_{\ve_1 T} \|_{L^2(\Omega\times B_{R+1})}^2 + \|f_{\ve_2 0} - f_{\ve_1 0}\|_{L^2(\Omega\times B_{R+1})}^2 \\
        &\quad + \|f_{\ve_2} - f_{\ve_1}\|_{L^2((0,T)\times\Omega; H^1(B_{R+1}))}^2 + \|f_{\ve_2} - f_{\ve_1}\|_{L^2(\calO^T_{R+1})} \\
        &\quad + \|f_{\ve_2} - f_{\ve_1}\|_{L^2((0,T)\times\Omega;H^1(B_{R+1}))} + o(\ve_1,\ve_2).
    \end{split}
    \end{equation}
    Using that $(f_{\ve t})$ is Cauchy in $L^2_{\rm loc}(\ov{\calO})$ and $(f_{\ve})$ is Cauchy in $L^2((0,T)\times\Omega;H^1_{\rm loc}(\R^d))$, we deduce that $(\gamma f_\ve)$ is Cauchy in $L^2((0,T)\times \p\Omega \times B_R, (n_x\cdot v)^2 \omega^2 \tnd\sigma_x\dv\dt)$. Consequently there is a limit $\xi$ for which, up to a diagonal extraction, $\gamma f_\ve \to \xi$ in $L^2_{\rm loc}(\Sigma^T; (n_x\cdot v)^2 \,\tnd\sigma_x\dv\dt)$ and a.e. We can then consider the Green's formula satisfied by $f_\ve$ in duality with $\calD'((0,T)\times\ov{\calO})$, and pass to the limit, to deduce that $f$ satisfies $\eqref{p: e: kolmogorov}_1$ in $\calD'((0,T)\times \ov{\calO})$ with trace $\xi$. Such a trace is unique (an immediate consequence of the Green's formula), and therefore we conclude to $\xi = \gamma f$.

    Because all of the estimates above are local in nature, and since $b\in L^2((0,T)\times\Omega;H^1_{\rm loc}(\R^d))$, we can proceed with the dual backwards equation in the same manner. We first have that $b_\ve$ is a solution to
    \begin{equation*}
        - \p_t b_\ve - v\cdot \nabla_x b_\ve = \scrL^* b_\ve + (\scrP^* b)_\ve + R_\ve^2
    \end{equation*}
    in the sense of $\calD'([0,T]\times \ov{\calO})$. Here the commutator satisfies (at least) that $R_\ve^2\to 0$ in $L^1_{\rm loc} + L^2_{t,x}H^{-1}_{v,{\rm loc}}$; we can also prove, similarly as we did for $f_\ve$, that $\gamma b_\ve \to \gamma b$ strongly in $L^2_{\rm loc}(\Sigma^T;(n_x\cdot v)^2\,\tnd\sigma_x\dv\dt)$ and a.e. as $\ve \searrow 0$.
    
    (Step 2: Renormalizing into an $L^\infty$-framework.) Now we may argue by adapting the framework of \cite{DiPLion89} and \cite{Mis00}. We take an odd renormalizing function
    \begin{equation*}
        \beta_M(z) := \begin{cases}
            z & |z| \le M, \\
            M+1 & |z| \ge M+2,
        \end{cases}
    \end{equation*}
    defined suitably on $(-(M+2),-M)\cup (M,M+2)$ so that $\beta_M \in C^2 \cap W^{2,\infty}(\R)$. Then in the sense of $\calD'((0,T)\times \ov{\calO})$, we have
    \begin{align*}
        \p_t \beta_M(f_\ve) + v\cdot \nabla_x \beta_M(f_\ve) &= \scrL \beta_M(f_\ve) + \beta_M'(f_\ve) ( (\scrP f)_\ve + (S)_\ve + R_\ve^1) \\
        &\quad - \beta_M''(f_\ve) \sigma_{ij} \p_{v_i} f_\ve \p_{v_j} f_\ve + \eta (\beta_M'(f_\ve) f_\ve - \beta_M(f_\ve)) , \\
        -\p_t \beta_M(b_\ve) - v\cdot \nabla_x \beta_M(b_\ve) &= \scrL^* \beta_M(b_\ve) + \beta_M'(b_\ve) ( (\scrP^* b)_\ve + R_\ve^2) \\
        &\quad + (b_\ve \beta_M'(b_\ve) - \beta_M(b_\ve)) \p^2_{v_i v_j}\sigma_{ij}  - \beta_M''(b_\ve) \sigma_{ij} \p_{v_i} g \p_{v_j} g  \\
        &\quad + (b_\ve \beta_M'(b_\ve) - \beta_M(b_\ve)) (\nabla_v\cdot \zeta ) \\
        &\quad + \eta (b_\ve \beta_M'(b_\ve) - \beta_M(b_\ve)).
    \end{align*}
    Furthermore, going back to the formulation in \eqref{P/Eq/ reg rough}, it is clear that $\beta_M\in W^{2,\infty}$ implies
    \begin{equation*}
        \beta_M(f_\ve), \beta_M(b_\ve) \in W^{1,1}((0,T), W^{1,\infty}_{\rm loc}(\ov{\calO})) \cap L^\infty((0,T)\times\ov{\calO}),
    \end{equation*}
    and therefore the product rule applies to $\mathfrak{P}_{M,\ve} := \beta_M(f_\ve) \beta_M(b_\ve)$. We obtain namely that
    \begin{align*}
        &\p_t \mathfrak{P}_{M,\ve} + v\cdot \nabla_x \mathfrak{P}_{M,\ve} \\
        &= \p_{v_i} \Big( \beta_M(b_\ve) \sigma_{ij} \p_{v_j} \beta_M(f_\ve) - \beta_M(f_\ve) \p_{v_j} (\sigma_{ij} \beta_M(b_\ve) ) + \zeta_i \beta_M(f_\ve) \beta_M(b_\ve) \Big) \\
        &\quad + \beta_M(b_\ve) \beta_M'(f_\ve) \Big( (\scrP f)_\ve + (S)_\ve +  R_\ve^1 \Big) + \beta_M(f_\ve) \beta_M'(f_\ve) \Big( (\scrP^* b)_\ve + R_\ve^2 \Big) \\
        &\quad + \beta_M(b_\ve) \Big( - \beta_M''(f_\ve) \sigma_{ij} \p_{v_i} f_\ve \p_{v_j} f_\ve + \eta (\beta_M'(f_\ve) f_\ve - \beta_M(f_\ve)) \Big) \\
        &\quad + \beta_M(f_\ve) \Big( (b_\ve \beta_M'(b_\ve) - \beta_M(b_\ve)) \p^2_{v_i v_j}\sigma_{ij}  - \beta_M''(b_\ve) \sigma_{ij} \p_{v_i} g \p_{v_j} g \Big) \\
        &\quad + \beta_M(f_\ve) \Big( (b_\ve \beta_M'(b_\ve) - \beta_M(b_\ve)) (\nabla_v\cdot \zeta ) + \eta (b_\ve \beta_M'(b_\ve) - \beta_M(b_\ve)) \Big).
    \end{align*}
    

    (Step 3: Passing to the limit $\ve\to 0$.) Writing out explicitly the Green's formula satisfied by $\mathfrak{P}_{M,\ve}$ in duality with $\calD([0,T]\times\ov{\calO})$, let us first aim to pass to the limit $\ve\to 0$. Observe that as $\ve\to 0$,
    \begin{equation}\label{P: E: beta M ve}
    \begin{split}
        &\beta_M(f_{\ve t}) \to \beta_M(f_t) \quad \text{a.e. in $\calO$ for all $t\in [0,T]$}, \\
        &\beta_M(f_\ve) \to \beta_M(f), \quad \beta_M'(f_\ve) \to \beta_M'(f), \quad \beta_M''(f_\ve) \to \beta_M''(f) \quad \text{a.e. in } (0,T)\times \calO,\\
        &\beta_M(b_\ve) \to \beta_M(b), \quad \beta_M'(b_\ve) \to \beta_M'(b), \quad \beta_M''(b_\ve) \to \beta_M''(b) \quad \text{a.e. in } (0,T)\times \calO, \\
        &\beta_M(\gamma f_\ve) \to \beta_M(\gamma f), \quad \beta_M(\gamma b_\ve) \to \beta_M(\gamma b) \quad \text{a.e. in } \Sigma^T,
    \end{split}
    \end{equation}
    at least up to a subsequence. Now, all sequences in \eqref{P: E: beta M ve} are trivially uniformly-in-$\ve$ bounded in $L^\infty$ (by definition of $\beta_M$). Therefore, by the convergence theorem of Vitali, the convergences in \eqref{P: E: beta M ve} are true in $L^p_{\rm loc}([0,T]\times \ov{\calO})$ for any $p\in [1,\infty)$.

    Also, since $\scrP f, \scrP^* b \in L^2_{\rm loc}([0,T]\times \ov{\calO})$ and $S\in L^2_{t,x} H^{-1}_{v,{\rm loc}}$, it is clear that
    \begin{equation*}
        \begin{split}
            &(\scrP f)_\ve \to \scrP f \quad\text{in} \quad L^2_{\rm loc}([0,T]\times \ov{\calO}), \\
            &(\scrP^* b)_\ve \to \scrP^* b \quad \text{in} \quad L^2_{\rm loc}([0,T]\times \ov{\calO}),\\
            &(S)_\ve \to S \quad \text{in} \quad L^2([0,T]\times \Omega; H^{-1}_{\rm loc}(\R^d)).
        \end{split}
    \end{equation*}

    For the trace $\gamma \mathfrak{P}_{M,\ve}$ we may discuss as follows.  Note that $f_\ve$ and $b_\ve$ may be assumed to be (uniformly) continuous in each compact subset of $\ov{\calO}$ (by Sobolev embedding). Therefore $\beta_M(f_\ve)$, $\beta_M(b_\ve)$ are also continuous on $\ov{\calO}$, which means $\gamma\mathfrak{P}_{M,\ve} = \beta_M(\gamma f_\ve) \beta_M(\gamma b_\ve)$. It follows, again merely by definition of $\beta_M$, that $(\gamma \mathfrak{P}_{M,\ve})_\ve$ is uniformly bounded in $L^\infty(\Sigma^T)$. Combining this with $\eqref{P: E: beta M ve}_4$ (which yields that $\gamma \mathfrak{P}_{M,\ve} \to \beta_M(\gamma f) \beta_M(\gamma b)$ a.e.), we obtain again through Vitali's theorem that $\gamma \mathfrak{P}_{M,\ve} \to \beta_M(\gamma f) \beta_M(\gamma b)$ strongly in $L^p_{\rm loc}(\Sigma^T)$ for every $p\in [1,\infty)$.
    
    Using all of these strong convergences, we can now easily pass to the limit $\ve\to 0$ in every term of the Green's formula that is satisfied by $\mathfrak{P}_{M,\ve}$. We deduce that $\mathfrak{P}_M := \beta_M(f) \beta_M(b)$ satisfies 
    \begin{align*}
        &\p_t \mathfrak{P}_{M} + v\cdot \nabla_x \mathfrak{P}_{M} \\
        &= \p_{v_i} \Big( \beta_M(b) \sigma_{ij} \p_{v_j} \beta_M(f) - \beta_M(f) \p_{v_j} (\sigma_{ij} \beta_M(b) ) + \zeta_i \beta_M(f) \beta_M(b) \Big) \\
        &\quad + \beta_M(b) \beta_M'(f) (\scrP f  + S) - \beta_M(f) \beta_M'(b) \scrP^* b  \\
        &\quad + \beta_M(b) \Big( - \beta_M''(f) \sigma_{ij} \p_{v_i} f \p_{v_j} f + \eta (\beta_M'(f) f - \beta_M(f)) \Big) \\
        &\quad + \beta_M(f) \Big( (b_\ve \beta_M'(b) - \beta_M(b)) \p^2_{v_i v_j}\sigma_{ij}  - \beta_M''(b) \sigma_{ij} \p_{v_i} g \p_{v_j} g \Big) \\
        &\quad + \beta_M(f) \Big( (b \beta_M'(b) - \beta_M(b)) (\nabla_v\cdot \zeta ) + \eta (b \beta_M'(b) - \beta_M(b)) \Big) 
    \end{align*}
    in duality with $\calD'([0,T]\times\ov{\calO})$ with unique trace functions $\mathfrak{P}_{Mt} = \beta_M(f_t) \beta_M(b_t)$ and $\gamma \mathfrak{P}_M = \beta_M(\gamma f) \beta_M(\gamma b)$.
    
    (Step 4: Passing to the limit $M\to\infty$ and then $\varphi \nearrow 1$.) We now specify the test function $\varphi(x,v) = \chi_R(v)$, where $\chi_R$ is a radial function satisfying 
    \begin{equation}\label{P: E: Nice Cutoff}
        \mathbf{1}_{|v|\le R} \le \chi_R \le \mathbf{1}_{|v|\le 2R}, \quad |\nabla_v \chi_R| \le \frac{2}{R} \mathbf{1}_{R\le |v| \le 2R}, \quad |\nabla_v^2 \chi_R| \le \frac{10}{R^2} \mathbf{1}_{R \le |v|\le 2R},
    \end{equation}
    with $\chi_R \nearrow 1$ as $R\to\infty$. Note that in particular, the radial symmetry of $\chi_R$ and the dual reflection boundary conditions imply
    \begin{equation*}
        -\iint_{\Sigma} \beta_M(\gamma f) \beta_M(\gamma b) \chi_R (v\cdot n_x) \,\tnd\sigma_x \dv = 0.
    \end{equation*}
    Thus in the Green's formula satisfied by $\mathfrak{P}_M$ and $\chi_R$, the boundary integral may be discarded. Then using the bounds
    \begin{align*}
        |\beta_M(z)| \le |z|, \quad |\beta_M'(z)| \lesssim 1, \quad |\beta_M''(z)| \lesssim 1 ,
    \end{align*}
    and owing to the fact that $f,b\in L^2_{t,x}H^1_{v}$, we can apply the dominated convergence theorem of Lebesgue to pass to the limit $M\to\infty$ in this Green's formula, obtaining
    \begin{equation} \label{P: E: dual chiR}
    \begin{split}
        &\iint_{\calO} f_T b_T \chi_R - \iint_{\calO} f_0 b_0 \chi_R \\
        &= - \iiint_{\calO^T} \Big( b \sigma_{ij} \p_{v_j} f - f \p_{v_j} (\sigma_{ij} b) + \zeta_i f b \Big) \p_{v_i} \chi_R \\
        &\quad + \iiint_{\calO^T} (b \scrP f + b S - f \scrP^* b) \chi_R .
    \end{split}
    \end{equation}
    
    At last, to pass to the limit $R\to\infty$, we observe that the assumptions \eqref{P: E: Nice Cutoff} imply
    \begin{align*}
        \forall R>1, \quad |\nabla_v \chi_R(v)| \le \frac{6}{1+|v|}.
    \end{align*}
    Combining this estimate with what we have in the assumptions (W1)--(L3)--(L4), we observe that the integrands on the right-hand side of \eqref{P: E: dual chiR} can be bounded by
    \begin{align*}
        |b\sigma_{ij}\p_{v_j} f \p_{v_i} \chi_R| &\lesssim  |b||\p_{v_j} f| \lesssim |b \omega^{-1}|^2 +  | \omega \p_{v_j} f|^2 \lesssim |b \omega^{-1}|^2 + | \nabla_v (f \omega)|^2 + |f \omega|^2, \\
        |f \p_{v_j}(\sigma_{ij}b) \p_{v_i} \chi_R | &\lesssim |fb| + |f \sigma_{ij} \p_{v_j} b \p_{v_i} \chi_R| \lesssim |f\omega|^2 + |b \omega^{-1}|^2 + |\nabla_v (b\omega^{-1})|^2 , \\
        |\zeta_i f b \p_{v_i}\chi_R| &\lesssim |f\omega|^2 + |b\omega^{-1}|^2.
    \end{align*}
    For the integrand in the last line of \eqref{P: E: dual chiR}, we know that $b\in L^2_{\omega^{-1}}(\calO^T)$, $\scrP f\in L^2_{\omega}(\calO^T)$, $f\in L^2_\omega(\calO^T)$, and $\scrP^* b\in L^2_{\omega^{-1}}(\calO^T)$. Therefore, the Lebesgue dominated convergence theorem applies to \eqref{P: E: dual chiR} and we get
    \begin{align*}
        \iint_{\calO} f_T b_T - \iint_{\calO} f_0 b_0 &= \iiint_{\calO^T} Sb + \iiint_{\calO^T} (b \scrP f - f \scrP^* b) \\
        &= \iiint_{\calO^T} Sb.
    \end{align*}
    This completes the proof.
\end{proof}


\subsection{Basic facts on spatially homogeneous and radial solutions to \eqref{eq: main}}
In this section we recall some properties of spatially homogeneous and radial solutions to \eqref{eq: main}. It is easy to check that if the initial datum $g_0(v)$ is radial, then the solution $g_t(v)$ to \eqref{eq: main} is also radial at all times $t\ge 0$, and due to this symmetry the boundary condition is automatically satisfied. We may then view $g$ as a solution to \eqref{eq: main} in all of $\R_+\times \R^d$. Furthermore, owing to the conservation laws in \eqref{Eq/ F Conv}, we have that
\begin{equation*}
    \ddt \int_{\R^d} (1,v,|v|^2) g_t(v) \,\dv = 0,
\end{equation*}
and consequently the coefficients $\rho[g], u[g], \theta[g]$ in \eqref{eq: main} are constant. In particular, since $g$ is radial we obtain that $j[g] = u[g] = 0$. Thus, if $\rho[g_0]=1$, then we simply have
\begin{equation*}
    \theta[g] = \frac{E[g_0]}{d}.
\end{equation*}
Therefore, under the assumptions of Theorem \ref{Thm: Main}, $g$ is merely a solution to the linear Fokker--Planck equation
\begin{equation}\label{Eq/ Spatially hom}
    \p_t g = \left( \frac{E[g_0]}{d} \Delta_v g + \nabla_v\cdot (vg) \right) \quad \text{in }\R_+\times\R^d.
\end{equation}

In the following lemma we provide some elementary weighted $L^\infty$-estimates.

\begin{lemma} \label{Lem/ Elem Linfty}
    We assume that $\omega \in \wt{\calW}$. Let $g_t$ denote the solution to \eqref{Eq/ Spatially hom} corresponding to initial datum $g_0\in W^{2,\infty}_{\omega}(\R^d)$. Then there exists $\lambda > 0$ such that
    \begin{equation*}
        \|g_t \|_{W^{2,\infty}_{\omega}(\R^d)} \lesssim e^{\lambda t} \|g_0\|_{W^{2,\infty}_{\omega}(\R^d)}.
    \end{equation*}
\end{lemma}
\begin{proof}
    We only give the proof for the first derivative. The analogous estimate for the second derivative, and $g$ itself, follows in the same way. Furthermore, in order to simplify the notations, we will just assume without loss of generality $E[g_0] = d$. Then $g_t$ is a solution to
    \begin{equation*}
        \p_t g = \Delta_v g + \nabla_v\cdot (vg).
    \end{equation*}
    We denote $G_t := \omega \nabla_v g_t$, and after a direct computation, we find that $G$ is a solution to
    \begin{align*}
        &\p_t G = \Delta_v G + \lt(v - 2\frac{\nabla_v\omega}{\omega}\rt)\cdot \nabla_v G + c_\omega G, \\
        &c_\omega := \lt( 2\frac{|\nabla_v \omega|^2}{\omega} - \frac{\Delta_v \omega}{\omega} + 2d - v\cdot \frac{\nabla_v \omega}{\omega} \rt).
    \end{align*}
    For $\omega\in \wt{\calW}$, we can straightforwardly check that $\sup_{v\in \R^d} c_\omega(v) < \infty$, and therefore the maximum principle provides $\|G_t\|_{L^\infty(\R^d)}\le e^{c_{\omega} t} \|G_0\|_{L^\infty(\R^d)}$.
\end{proof}

Second, we have the following trend-to-equilibrium result, whose proof we refer to \cite[Section 3]{MisMou16}.

\begin{lemma}\label{Lem/ Trend to Equi}
    We assume that $\omega\in \calW_{g_0}$. Then given $g_0 \in L^\infty_{\omega}(\R^d)$ such that $\rho[g_0] = 1$, there exist $a>0$ and $C_a>0$ such that
    \begin{equation*}
        \|g_t(v) - \mu\|_{L^\infty_{\omega}(\R^d)} \le C_a e^{-at} \|g_0 - \mu\|_{L^\infty_{\omega}(\R^d)}.
    \end{equation*}
\end{lemma}


\section{On the weakly inhomogeneous regime}

\subsection{Setup of the problem} \label{Sec: Inhom: setup}

Let us derive the equation that is satisfied by $h:= f-g$. We recall $\eqref{eq: main}_1$ here,
\begin{align*}
    (\p_t + v\cdot \nabla_x) f&= \rho[f]\Big( \theta[f] \Delta_v f + \nabla_v\cdot ((v-u[f])f) \Big) \\
    &= \rho[f]\theta[f]\Delta_v f + d\rho[f] f + \rho[f] (v - u[f])\cdot \nabla_v f,
\end{align*}
and linearize the macroscopic coefficients with respect to $g$, by expanding everything up to the first order. First of all, we can write
\begin{equation*}
    \rho[f] = \rho[g_0] + \rho[h] = 1 + \rho[h].
\end{equation*}
Second, there holds
\begin{align*}
    \rho[f]\theta[f] &= \rho[g]\theta[g] + \frac{1}{d}E[h] - \frac{|j[h]|^2}{d(\rho[g_0] + \rho[h])} \\
    &= \frac{1}{d}E[g_0] +  \frac{1}{d}E[h] - \frac{|j[h]|^2}{d(1 + \rho[h])}.
\end{align*}
Third, we have
\begin{equation*}
    \rho[f]u[f] = j[f] = j[h],
\end{equation*}
thanks to the radial symmetry of $g$ (namely that $j[g] = 0$). The equation for $h$ can thus be written as
\begin{align*}
    &(\p_t + v\cdot \nabla_x) h  \\
    &= \lt(\frac{1}{d}E[g_0] + \frac{1}{d}E[h] - \frac{|j[h]|^2}{d( 1 + \rho[h])} \rt) \Delta_v h \\
    &\quad + \Big( ( 1  + \rho[h]) v - j[h] \Big) \cdot \nabla_v h 
    + d h + dg\rho[h] + dh\rho[h] \\
    &\quad + \lt(\frac{1}{d}E[h] - \frac{|j[h]|^2}{d( 1+\rho[h])} \rt) \Delta_v g
    + (\rho[h] v - j[h]) \cdot \nabla_v g .
\end{align*}

In order to simplify the notation, we define
\begin{align*}
    &\bfA[h](t,x) := \frac{|j[h]|^2}{ 1 + \rho[h]}, \quad \bfB[h](t,x,v) := \rho[h] v - j[h], 
\end{align*}
with which the equation for $h$ can then be simply written as
\begin{equation} \label{S: E: nonlin}
\begin{split}
    (\p_t + v\cdot \nabla_x)h &= \frac{1}{d}\lt(E[g_0] + E[h] - \bfA[h] \rt) \Delta_v h + ( v + \bfB[h]) \cdot \nabla_v h + d  h \\
    &\quad + d g \rho[h] +  d h \rho[h]  + \frac{1}{d} \Delta_v g \, E[h] + \bfB[h] \cdot \nabla_v g \\
    &\quad - \frac{1}{d} \Delta_v g \, \bfA[h].
\end{split}
\end{equation}

Towards the end of addressing the existence of a solution to \eqref{S: E: nonlin}, we consider first a regularized version of \eqref{S: E: nonlin}. The construction of a solution to \eqref{S: E: nonlin} is surprisingly delicate, owing to the presence of quadratic nonlinear terms in the principal part of the operator. Therefore, the method of approximation has to be chosen carefully. We first define for $\varsigma > 0$ the usual family of mollifiers $(\bfr_\va)$ and $(\bfs_\va)$, where
\begin{equation*}
\begin{split}
    &\textnormal{supp}(\bfr) \subset B_1, \quad \int_{\R^d} \bfr(x) \,\dx = 1, \quad \bfr_\va (x) := \frac{1}{\va^d} \bfr\lt(\frac{x}{\va}\rt), \\
    &\textnormal{supp}(\bfs) \subset \lt[-\frac{1}{2},\frac{1}{2}\rt] , \quad \int_\R \bfs(t) \,\dt = 1, \quad \bfs_\va(t) := \frac{1}{\va}\bfs\lt(\frac{t}{\va}\rt),
\end{split}
\end{equation*}
here $B_1$ denotes the unit ball in $\R^d$. For a given function $\bfrho:\Omega^T\to \R$, we then define its time-space regularized version $\bfrho *_t \bfs_\va \star_\va \bfr_\va : [0,T]\times \ov{\Omega} \to \R$ as
\begin{equation*}
    \bfrho *_t \bfs_\va \star_\varsigma \bfr_\varsigma (t,x) := \int_0^T \int_\Omega \bfrho(s,y) \bfs_\va(t-s) \bfr_\varsigma(x - 2\varsigma n_x - y ) \,\ds\dy.
\end{equation*}

Since the main difficulty comes from the presence of $\bfA[h]$ as part of the diffusion coefficient, we now set
\begin{equation*}
    \bfA_\varsigma[h] := \frac{ |j[h] *_t \bfs_\va \star_\varsigma \bfr_{\varsigma}|^2 }{ 1 + \rho[h] *_t \bfs_\va \star_\varsigma \bfr_{\varsigma} } .
\end{equation*}
The advantage of this formulation is that then, even if $\rho[h]$ and $j[h]$ are only known to converge for instance weakly$-*$, the regularized coefficient $\bfA_\varsigma[h]$ is at least pointwise convergent. Provided that we have some nice bounds on the macroscopic quantities, this pointwise convergence would then imply strong convergence in suitable Lebesgue spaces.

The regularized equation in consideration then reads
\begin{equation}\label{S: E: Reg Nonlin}
    \begin{cases}
        \begin{aligned}[c] 
            (\p_t + v\cdot \nabla_x)h &= \frac{1}{d}\lt(E[g_0] + E[h] - \bfA_\varsigma[h] \rt) \Delta_v h  \\
            &\quad + (v + \bfB[h])\cdot \nabla_v h + d (1 + \rho[h]) h \\
            &\quad + d g \rho[h] + \frac{1}{d}\Delta_v g E[h] + \bfB[h] \cdot \nabla_v g - \frac{1}{d} \bfA_\varsigma[h] \Delta_v g 
        \end{aligned}   &\text{in }\calO^T, \\
        h|_{t=0} = h_0  &\text{in } \calO, \\
        \gamma_- h = \gamma_+ h \circ \scrV_x  &\text{on }\Sigma_-^T.
    \end{cases}
\end{equation}
However, even the existence of a solution to \eqref{S: E: Reg Nonlin} is yet nontrivial, so we approach using a linearized version of \eqref{S: E: Reg Nonlin}:
\begin{equation}\label{S: E: Lin}
    \begin{cases}
        \begin{aligned}[c]
        (\p_t + v\cdot \nabla_x)h &= \frac{1}{d} \lt(E[g_0] + E[k] - \bfA_\varsigma[k] \rt) \Delta_v h \\
        &\quad + (v + \bfB[k])\cdot \nabla_v h + d (1 + \rho[k]) h \\
        &\quad + d g \rho[h] + \frac{1}{d}\Delta_v g E[h] + \bfB[h] \cdot \nabla_v g - \frac{1}{d} \bfA_\varsigma[k] \Delta_v g 
        \end{aligned} &\text{in }\calO^T, \\
        h|_{t=0} = h_0  &\text{in } \calO, \\
        \gamma_- h = \gamma_+ h \circ \scrV_x  &\text{on }\Sigma_-^T,
    \end{cases}
\end{equation}
where we assume that $k=k(t,x,v)$ is given from some set that we specify later. Let us remark that the solution to \eqref{S: E: Reg Nonlin} is formally realized as a fixed point of the map $k \mapsto h$ defined on \eqref{S: E: Lin}.

Towards the analysis of \eqref{S: E: Lin}, the following backwards equation (with zero source term and dual specular reflection) will also be utilized
\begin{equation}\label{S: E: Back Lin}
    \begin{cases}
        \begin{aligned}[c]
            (-\p_t - v\cdot \nabla_x) b &= \frac{1}{d} \lt( E[g_0] + E[k] - \bfA_\varsigma[k] \rt) \Delta_v b - ( v + \bfB[k]) \cdot \nabla_v b + d\rho[k] b \\
            &\quad + d \rho[bg] + \frac{1}{d}|v|^2 \rho[b\Delta_v g] + \rho[v\cdot b\nabla_v g] - v \rho[b\nabla_v g]
        \end{aligned} &\text{in }\calO^T, \\
        b|_{t=T} = b_T  &\text{in }\calO, \\
        \gamma_+ b = \gamma_- b \circ \scrV_x &\text{on }\Sigma_+^T.
\end{cases}
\end{equation}
In the weighted $L^2$ framework, the well-posedness of both equations \eqref{S: E: Lin} and \eqref{S: E: Back Lin} follows as an application of Proposition \ref{prop: WP}. For the forward equation \eqref{S: E: Lin}, let us define the following wide range of polynomial weights which will act as a \textit{reference class}:
\begin{equation*}
        \wt{\calW} := \lt\{
        \la v \ra^m, \quad  m > \frac{d}{2} + 2
        \rt\} .
    \end{equation*}
We remark that $L^2_{\wt{\omega}}\hookrightarrow L^1_{\la v \ra^2}$ for every $\wt{\omega}\in \wt{\calW}$ because $\la v \ra^2 \wt{\omega}^{-1} \in L^2(\R^d)$.

On the other hand, we remind the reader that the class of admissible weights is defined as
\begin{equation*}
    \calW_{g_0} := \lt\{
        \begin{aligned}
        \la v \ra^m, \quad & m > d + 2, \\
        \la v \ra^m e^{\gamma |v|^\tau}, \quad &\text{either } m>0, \; \gamma > 0, \text{ and } \tau \in (0,2) \\
        &\text{or } m>0, \; \gamma\in \lt(0,\frac{d}{2E[g_0]}\rt), \text{ and } \tau = 2
        \end{aligned}
        \rt\} .
\end{equation*}
It is worth noting that the computation of $\Gamma$ in \eqref{p: e: Gamma} is relatively simple for the class of weights above. For instance, let us consider a diffusion coefficient $\sigma_{ij} = \sigma \delta_{ij}$ with $\sigma = \sigma(t,x)$ real-valued, which is enough for our purpose; then setting $\scrL := \sigma \Delta_v + \zeta \cdot \nabla_v + \eta$, consider any weight of the form $\omega = \la v \ra^m e^{\gamma \la v \ra^\tau}$, with $m\ge 0$, $\gamma\ge 0$, $\tau \ge 0$ so that in particular all weights in $\wt{\calW}$ or $\calW_{g_0}$ are taken into consideration. Then with the notation $W := \log \omega$, \eqref{p: e: Gamma} reduces to
\begin{equation}\label{S: E: simp gamma}
\begin{split}
    \Gamma_{\scrL,\omega,p} &= \sigma |\nabla_v W|^2 + \lt(\frac{2}{p} - 1\rt) \sigma   \Delta_v W - \zeta \cdot \nabla_v W - \frac{1}{p} \nabla_v\cdot \zeta + \eta, \\
    \nabla_v W &= \lt(\frac{m}{\la v \ra^2} + \gamma \tau |v|^{\tau - 2} \rt) v , \\
    \Delta_v W &= m \frac{d + (d - 2)|v|^2}{\la v \ra^{4}} + \gamma \tau (d + \tau - 2) |v|^{\tau - 2}.
\end{split}
\end{equation}

For the remainder of this section, we now fix a specific weight
\begin{equation*}
    \omega \in \calW_{g_0},
\end{equation*}
and use the following functional space corresponding to this weight: for some $a>0$, we will assume that $k$ lies in
\begin{equation*}
     \calX_{T,\omega,a} := \lt\{ k \in L^\infty_\omega(\calO^T) \; \middle| \; \|k\|_{L^\infty_\omega(\calO^T)} \le a \rt\}.
\end{equation*}

\subsection{Analysis of the linear equation}

\begin{lemma}\label{l: sp: for lin WP}
    There exists $a_1>0$ small enough, and independent of $\varsigma$, such that for any $a\le a_1$ and $k\in \calX_{T,\omega,a}$, the following holds true: for any polynomial weight $\wt{\omega}\in \wt{\calW}$, if
    \begin{align*}
        h_0 \in L^2_{\wt{\omega}}(\calO) \quad \text{and} \quad g_0 \in W^{2,2}_{\la v \ra \wt{\omega}}(\R^d),
    \end{align*}
    then the linear problem \eqref{S: E: Lin} admits a unique solution $h\in C([0,T];L^2_{\wt{\omega}}(\calO))$. 
\end{lemma}

\begin{proof}
    For the sake of clarity, let us set
    \begin{equation}\label{S: E: sigma}
    \begin{split}
    \sigma_\varsigma[k] &:= \frac{1}{d} \lt( E[g_0] + E[k] - \bfA_\varsigma[k] \rt), \\
    \scrL_\varsigma[k] h &:= \sigma_\varsigma[k] \Delta_v h  + (v + \bfB[k])\cdot \nabla_v h + d( 1 + \rho[k]) h.
    \end{split}
    \end{equation}
    We choose $a_1\le 1/2$ small enough (and independent of $\varsigma$) so that for any $a\le a_1$, the membership $k\in \calX_{T,\omega,a}$ implies both
\begin{equation}\label{S: E: coeff}
\begin{cases}
    |\rho[k]| \le \frac{1}{2}, \\
     \frac{1}{2d}E[g_0] \le \sigma_\varsigma[k] \le \frac{3}{2d}E[g_0]
\end{cases}
\quad \text{in }\Omega^T.
\end{equation}
Indeed, such a choice is possible since $\rho[k]$, $j[k]$, and $E[k]$ can be controlled pointwisely as
\begin{equation*}
    \lt\| \intr \la v \ra^2 k \,\dv \rt\|_{L^\infty(\Omega^T)} \le \|\la v \ra^2 \omega^{-1}\|_{L^1(\R^d)} \|k\|_{L^\infty_\omega(\calO^T)} \le  \|\la v \ra^2 \omega^{-1}\|_{L^1(\R^d)} \, a ,
\end{equation*}
whereas $\bfA_\varsigma[k]$ can be controlled pointwisely, and uniformly in $\varsigma$, by using the relations
\begin{align*}
    \|\rho[k] *_t \bfs_\va \star \bfr_\varsigma \|_{L^\infty(\Omega^T)} \le \|\rho[k]\|_{L^\infty(\Omega^T)}, \quad \|j[k] *_t \bfs_\va \star \bfr_\varsigma \|_{L^\infty(\Omega^T)} \le \|j[k]\|_{L^\infty(\Omega^T)}. 
\end{align*}

    Now, it is enough to check that the conditions of Proposition \ref{prop: WP} are satisfied by the equation \eqref{S: E: Lin} and the operator set in \eqref{S: E: sigma}. First, the diffusion coefficient $\sigma_\varsigma[k]$ is obviously symmetric and strictly positive definite, thanks to \eqref{S: E: coeff}. Therefore (L1) is satisfied.
    
    In the notation of \eqref{S: E: simp gamma}, we can set
    \begin{equation*}
    \zeta[k] := v + \bfB[k] = (1+\rho[k])v - j[k] , \quad \eta[k] := d( 1 + \rho[k])  
    \end{equation*}
    to find that
    \begin{align*}
        \Gamma_{\scrL_\varsigma[k] ,\wt{\omega},2} = \sigma_\varsigma[k] \frac{|\nabla_v \wt{\omega}|^2}{\wt{\omega}^2} - ( v + \bfB[k])\cdot \frac{\nabla_v \wt{\omega}}{\wt{\omega}} + \frac{d}{2}  + d\rho[k].
    \end{align*}
    Owing to \eqref{S: E: coeff}, 
    \begin{align*}
        \Gamma_{\scrL_\varsigma[k] ,\wt{\omega},2}(v) \le \frac{3}{2d}E[g_0] \frac{|\nabla_v \wt{\omega}|^2}{\wt{\omega}^2} - \frac{1}{2} \frac{v\cdot \nabla_v\wt{\omega}}{\wt{\omega}} + |j[k]| \lt|\frac{\nabla_v\wt{\omega}}{\wt{\omega}}\rt| + \frac{d}{2}  + d |\rho[k]|.
    \end{align*}
    Since $j[k],\rho[k] \in L^\infty(\Omega^T)$, we deduce that for any $\wt{\omega}\in \wt{\calW}$, there exists $C>0$ independent of $\varsigma$ such that
    \begin{equation*}
        \sup_{(t,x,v)\in \calO^T} \Gamma_{\scrL_\varsigma[k],\wt{\omega},2} \le C.
    \end{equation*}
    Consequently (L2) holds. It is further straightforward to check that (L3) is satisfied by $\scrL_\varsigma[k]$.

    We now set the nonlocal operator of \eqref{S: E: Lin} as 
    \begin{equation}\label{S: E: nonloc}
    \scrP h := dg\rho[h] + \frac{1}{d}\Delta_v g E[h] + \nabla_v g \cdot \bfB[h].
    \end{equation}
    The condition (P1) can then be checked through the following computations: by using the Cauchy--Schwarz inequality,
    \begin{align*}
        \iiint_{\calO^T} |g|^2 |\rho[h]|^2 \wt{\omega}^2 &\le  \|g\|_{L^\infty(0,T;L^2_{\wt{\omega}}(\R^d))}^2 \iint_{\Omega^T} |\rho[h]|^2  \\
        &\le \|g\|_{L^\infty(0,T;L^2_{\wt{\omega}}(\R^d))}^2  \|\wt{\omega}^{-1}\|_{L^2(\R^d)}^2 \|h\|_{L^2_{\wt{\omega}}(\calO^T)}^2,
    \end{align*}
    and
    \begin{align*}
        \iiint_{\calO^T} |\Delta_v g E[h]|^2 \wt{\omega}^2 &\le \|\Delta_v g\|_{L^\infty(0,T;L^2_{\wt{\omega}}(\R^d))}^2 \iint_{\calO^T} |E[h]|^2  \\
        &\le \|\Delta_v g\|_{L^\infty(0,T;L^2_{\wt{\omega}}(\R^d))}^2  \||v|^2 \wt{\omega}^{-1} \|_{L^2(\R^d)}^2  \|h\|_{L^2_{\wt{\omega}}(\calO^T)}^2, \\
        \iiint_{\calO^T} |\nabla_v g \cdot \rho[h] v|^2 \wt{\omega}^2 &\le \|\nabla_v g\|_{L^\infty(0,T;L^2_{\la v \ra \wt{\omega}}(\R^d))}^2  \|\wt{\omega}^{-1}\|_{L^2(\R^d)}^2 \|h\|_{L^2_{\wt{\omega}}(\calO^T)}^2, \\
        \iiint_{\calO^T} |h| |\nabla_v g \cdot j[h]| &\le \|\nabla_v g\|_{L^\infty(0,T;L^2_{\la v \ra \wt{\omega}}(\R^d))}^2  \| v \wt{\omega}^{-1}\|_{L^2(\R^d)}^2 \|h\|_{L^2_{\wt{\omega}}(\calO^T)}^2. 
    \end{align*}

    Finally, we set the source term in \eqref{S: E: Lin} as
    \begin{equation*}
        S = S_1 := - \frac{1}{d} \bfA_\varsigma[k] \Delta_v g ,
    \end{equation*}
    from which we immediately have that
    \begin{equation*}
        \|S_1 \wt{\omega}\|_{L^2(\calO^T)}^2 \lesssim  \|\bfA_\varsigma[k]\|_{L^\infty(\Omega^T)}^2 \|\Delta_v g\|_{L^\infty(0,T;L^2_{\wt{\omega}}(\R^d))}^2 < +\infty.
    \end{equation*}
    This shows (S1). We conclude to the well-posedness of \eqref{S: E: Lin} thanks to Proposition \ref{prop: WP}.
\end{proof}

Similarly we have the following result for the backward equation.

\begin{lemma}\label{sp: l: back lin WP L2}
    For $a_1$ as chosen in Lemma \ref{l: sp: for lin WP}, whenever $a\le a_1$, then the following holds true: for any weight $\wt{\omega} \in \wt{\calW}$, if
    \begin{equation*}
        b_T \in L^2_{\wt{\omega}^{-1}}(\calO), \quad k \in \calX_{T,\omega,a}, \quad \text{and} \quad g_0 \in W^{2,2}_{\la v \ra \wt{\omega} }(\R^d),
    \end{equation*}
    then the backwards linear problem \eqref{S: E: Back Lin} admits a unique solution $b\in C([0,T];L^2_{\wt{\omega}^{-1}}(\calO))$.
\end{lemma}

\begin{proof}
    We note that the formal adjoint of $\scrL_\varsigma[k]$ is given by
    $$\scrL_\varsigma[k]^* b := \sigma_\varsigma[k] \Delta_v b - (v + \bfB[k]) \cdot \nabla_v b + d\rho[k] b.$$
    Therefore
    \begin{equation*}
        \Gamma_{\scrL_\varsigma[k]^*, \wt{\omega}^{-1}, 2} = \Gamma_{\scrL_\varsigma[k] , \wt{\omega}, 2}.
    \end{equation*}
    Thus, the proof of Lemma \ref{l: sp: for lin WP} tells us that (L2) is satisfied by $\Gamma_{\scrL_\varsigma[k]^*,\wt{\omega}^{-1},2}$. The remaining assertions are rather straightforward, thus let us only check that (P1) is satisfied by $\scrP^*$ and the weight $\wt{\omega}^{-1}$. Recall that the dual projection operator of \eqref{S: E: nonloc}, as written in \eqref{S: E: Back Lin}, is formally given by
    \begin{equation}\label{S: E: dual proj}
        \scrP^* b := d\rho[bg] + \frac{1}{d}|v|^2 \rho[b\Delta_v g] + \rho[v\cdot b\nabla_v g] - v \rho[b\nabla_v g].
    \end{equation}
    Each term of $\scrP^*$ can be shown to satisfy (P1) in the following manner: using the Cauchy--Schwarz inequality,
    \begin{align*}
        \iiint_{\calO^T} |\rho[bg]|^2 \wt{\omega}^{-2} &= \|\wt{\omega}^{-1}\|_{L^2(\R^d)} \iint_{\Omega^T} |\rho[bg]|^2 \\
        &\le \|\wt{\omega}^{-1}\|_{L^2(\R^d)} \|b\|_{L^2_{\wt{\omega}^{-1}} (\calO^T) }^2 \|g\|_{L^\infty(0,T;L^2_{\wt{\omega}}(\R^d) ) }^2 , \\
        \iiint_{\calO^T} |v|^4 |\rho[b\Delta_v g]|^2 \wt{\omega}^{-2}
        &= \||v|^2 \wt{\omega}^{-1}\|_{L^2(\R^d)} \iint_{\Omega^T} |\rho[b\Delta_v g]|^2  \\
        &\le \||v|^2 \wt{\omega}^{-1}\|_{L^2(\R^d)} \|b\|_{L^2_{\wt{\omega}^{-1}} (\calO^T) }^2 \|\Delta_v g\|_{L^\infty(0,T; L^2_{ \wt{\omega} } (\R^d) ) }^2 .
    \end{align*}
    In the same way,
    \begin{align*}
        \iiint_{\calO^T} |\rho[v\cdot b\nabla_v g]|^2 \wt{\omega}^{-2} &= \|\wt{\omega}^{-1}\|_{L^2(\R^d)} \iint_{\Omega^T} |\rho[v\cdot b\nabla_v g]|^2  \\
        &\le \|\wt{\omega}^{-1}\|_{L^2(\R^d)} \|b\|_{L^2_{\wt{\omega}^{-1}} (\calO^T) }^2 \| \nabla_v g \|_{L^\infty(0,T;L^2_{ \la v \ra \wt{\omega}}(\R^d)) }^2, \\
        \iiint_{\calO^T} |v \rho[b\nabla_v g]|^2 \wt{\omega}^{-2} &= \| v \wt{\omega}^{-1}\|_{L^2(\R^d)} \iint_{\Omega^T} |\rho[b\nabla_v g]|^2  \\
        &\le \| v \wt{\omega}^{-1}\|_{L^2(\R^d)} \|b\|_{L^2_{\wt{\omega}^{-1}} (\calO^T) }^2 \|\nabla_v g\|_{L^\infty(0,T; L^2_{ \wt{\omega} }(\R^d) ) }^2 .
    \end{align*}
    We thus get that (P1) is satisfied by $\scrP^*$ and the weight $\wt{\omega}^{-1}$. The well-posedness of \eqref{S: E: Back Lin} then follows from Proposition \ref{prop: WP}.
\end{proof}

Next, we transition to the $L^1$ framework of the backwards equation.

\begin{lemma}\label{s: l: b L1}
    There exists $a_2>0$ (dependent on the given weight $\omega$) such that if $a\le \min\{a_1,a_2\}$ and $k\in \calX_{T,\omega,a}$, the following holds. Suppose that
    \begin{equation}\label{S: E: l1 assu}
    g_0 \in W^{2,\infty}_{\la v \ra \omega}(\R^d),
    \end{equation}
    and for any weight $\wt{\omega}\in \wt{\calW}$ growing slowly enough so that $\wt{\omega} \omega^{-1} \in L^2(\R^d)$, let $b_T\in L^2_{\wt{\omega}^{-1}}(\calO)$, and denote by $b \in C([0,T]; L^2_{\wt{\omega}^{-1}}(\calO))$ the unique solution to the backwards equation \eqref{S: E: Back Lin} as established in Lemma \ref{sp: l: back lin WP L2}. Then $b$ satisfies the estimate
    \begin{equation}\label{S: E: L1}
        \forall t\in [0,T], \quad  \|b_t\|_{L^1_{\omega^{-1}}(\calO)} \le C \exp \Big( \lambda (T-t) \exp (C T) \Big) \|b_T\|_{L^1_{\omega^{-1}}(\calO)}
    \end{equation}
    for some constants $C>0$ and $\lambda>0$ independent of both $\varsigma$ and $T$.
\end{lemma}

\begin{remark}
    For instance, if $\omega(v) = \la v \ra^m$ where $m = d+2+\alpha$ with $\alpha>0$, we take $\wt{\omega}(v) = \la v \ra^{2 + (d+\alpha)/2} \in \wt{\calW}$.
\end{remark}

\begin{remark}
    Thanks to the choice of $\wt{\omega}$, the assumption \eqref{S: E: l1 assu} guarantees $g_0\in W^{2,2}_{\la v \ra \wt{\omega}}(\R^d)$, so that the results of Lemma \ref{sp: l: back lin WP L2} can be applied.
\end{remark}

\begin{proof}[Proof of Lemma \ref{s: l: b L1}]
    First we remark that since $\wt{\omega}\omega^{-1} \in L^2(\R^d)$, H\"older's inequality shows that the embedding
    \begin{equation*}
        L^2_{\wt{\omega}^{-1}} \hookrightarrow L^1_{\omega^{-1}}
    \end{equation*}
    is continuous. Therefore, $b_T \in L^1_{\omega^{-1}}$ and $b \in C([0,T];L^1_{\omega^{-1}})$, which means that the estimate \eqref{S: E: L1} makes sense.
    
    Let us define for $\delta>0$ the approximations $\beta_\delta(z) = \sqrt{z^2-\delta^2}-\delta$, so that $\beta_\delta(z) \nearrow |z|$ and $\beta_\delta(b)$ is a distributional solution to
    \begin{align*}
        (-\p_t - v\cdot \nabla_x) \beta_\delta(b) &= \sigma_\varsigma[k] \Delta_v \beta_\delta(b) - \sigma_\varsigma[k] \beta''(b) |\nabla_v b|^2 \\
        &\quad - (v + \bfB[k]) \cdot \nabla_v \beta_\delta(b) + d\rho[k] \beta_\delta'(b) b + \beta_\delta'(b) \scrP^* b.
    \end{align*}
    Here $\sigma[k]$ is the diffusion coefficient defined in \eqref{S: E: sigma}, and $\scrP^*$ is the dual projection operator given in \eqref{S: E: dual proj}. Performing some elementary algebraic manipulations, we point out that the above is equivalent to
    \begin{equation}\label{S: E: b L1 renorm}
    \begin{split}
        (-\p_t - v\cdot \nabla_x) \beta_\delta(b) &= \sigma_\varsigma[k]\Delta_v \beta_\delta(b) - (v + \bfB[k])\cdot \nabla_v \beta_\delta(b) + d\rho[k] \beta_\delta(b) \\
        &\quad - \sigma_\varsigma[k] \beta_\delta''(b) |\nabla_v b|^2 + d\rho[k] (\beta_\delta'(b) b - \beta_\delta(b) ) + \beta_\delta'(b) \scrP^* b \\
        &= \scrL_\varsigma[k]^* b - \sigma[k] \beta_\delta''(b) |\nabla_v b|^2 + d\rho[k] (\beta_\delta'(b) b - \beta_\delta(b) ) + \beta_\delta'(b) \scrP^* b,
    \end{split}
    \end{equation}
    here $\scrL_\varsigma[k]$ denoting the differential operator in \eqref{S: E: sigma}.
    
    To now perform $\omega^{-1}$-weighted estimates on $\beta_\delta(b)$, we observe that
    \begin{equation*}
        \scrL_\varsigma[k] \omega^{-1} = \omega^{-1} \Gamma_{\scrL_\varsigma[k] , \omega, \infty},
    \end{equation*}
    where $\Gamma_{\scrL_\varsigma[k],\omega,\infty}$ is defined as in \eqref{S: E: simp gamma}, with $\zeta[k] = v + \bfB[k] = (1+\rho[k])v - j[k]$ and $\eta[k] = d(1 + \rho[k])$. We claim that for our given choice of $\omega\in \calW$, there exists $a_2>0$ small enough so that $a\le \min\{a_1,a_2\}$ implies the existence of $C>0$ independent of $\varsigma$ such that
    \begin{equation}\label{S: E: infty defect}
        \omega^{-1} \Gamma_{\scrL_\varsigma[k],\omega,\infty} \le C \omega^{-1}.
    \end{equation}
    Indeed, let us write $\omega= \la v \ra^m e^{\gamma|v|^\tau}$. In the case where $\gamma=0$ (\textit{i.e.} when $\omega$ is a polynomial), we merely have
    \begin{equation*}
        \lim_{|v|\to\infty} \Gamma_{\scrL_\va[k],\omega,\infty}(v) = - m(1+\rho[k]) + d(1+\rho[k]) < \infty.
    \end{equation*}
    In the case when $\tau\in (0,2)$, we have $\Gamma_{\scrL_\va[k],\omega,\infty} \to -\infty$ unconditionally. In the case where $\tau=2$ and $\gamma\in (0,d/2E[g_0])$, the leading order of $\Gamma_{\scrL_\va[k],\omega,\infty}$ is given by (see \eqref{S: E: simp gamma})
    \begin{equation*}
    \begin{split}
        \Gamma_{\scrL_\va[k],\omega,\infty} &\sim \sigma_\va[k] |\nabla_v \log \omega |^2 - \zeta[k]\cdot \nabla_v \log \omega \\
        &\sim 2\gamma \,(2\gamma \sigma_\va[k] - (1+\rho[k])) |v|^2 . 
    \end{split}
    \end{equation*}
    Therefore $\Gamma_{\scrL_\va[k],\omega,\infty}\to -\infty$ if and only if
    \begin{equation*}
        \frac{\sigma_\va[k]}{1+\rho[k]} < \frac{1}{2\gamma}.
    \end{equation*}
    Since $\gamma\in(0,d/2E[g_0])$ we can write
    \begin{equation*}
        \frac{1}{2\gamma} = \frac{E[g_0]}{d} + \kappa_0
    \end{equation*}
    for some $\kappa_0 > 0$. Thus we just need to prove that, if $a_2$ is small enough, then $a\le a_2$ and $k\in \calX_{T,\omega,a}$ imply
    \begin{equation*}
        \frac{\sigma_\va[k]}{1+\rho[k]} \le \frac{E[g_0]}{d} + \frac{\kappa_0}{2}.
    \end{equation*}
    This follows by a standard computation because $\sigma_\va[k] \approx E[g_0]/d$ and $1+\rho[k]\approx 1$ for $k$ small in the $\|\cdot\|_{L^\infty_{\omega}}$-norm. Hence, in all cases we have $\Gamma_{\scrL_\va[k],\omega,\infty}\le C$ and therefore \eqref{S: E: infty defect} is proven.
    
    Coming back the equation \eqref{S: E: b L1 renorm}, we can first compute
    \begin{equation}\label{S: E: ddt b renorm}
    \begin{split}
        - \ddt \iint_{\calO} \beta_\delta(b_t) \omega^{-1} &= \iint_{\calO} \beta_\delta(b) \scrL_\varsigma[k] \omega^{-1} - \iint_{\calO} \sigma[k] \beta_\delta''(b_t) |\nabla_v b_t|^2 \omega^{-1} \\
        &\quad + \iint_{\calO} d\rho[k] (\beta_\delta'(b) b - \beta_\delta(b)) \omega^{-1} + \iint_{\calO} \beta_\delta'(b_t) \scrP^* b_t \, \omega^{-1}  \\
        &=: \rm{I} + \rm{II} + \rm{III} + \rm{IV}.
    \end{split}
    \end{equation}
    For $\rm{I}$ we just use \eqref{S: E: infty defect}, and the fact that $\beta_\delta\ge 0$, to estimate
    \begin{equation*}
        \rm{I} \le C \iint_{\calO} \beta_\delta(b) \omega^{-1}.
    \end{equation*}
    Next we throw away $\rm{II}$ by using the facts that $\beta_\delta$ is convex, and $\sigma_\varsigma[k] > 0$ (by \eqref{S: E: coeff}). For $\rm{III}$, we note that
    \begin{equation*}
        |\beta_\delta'(b) b - \beta_\delta(b)| = \lt|\frac{-\delta^2}{\sqrt{b^2+\delta^2}} + \delta \rt| \le 2\delta ,
    \end{equation*}
    thus using that $\rho[k] \in L^\infty$ and $\omega^{-1}\in L^1(\R^d)$, we can simply throw away $\rm{III}$ later by taking $\delta \searrow 0$. Finally, we need to estimate $\rm{IV}$. Since $|\beta_\delta'(z)| \le 1$, it is at least clear that
    \begin{equation}\label{S: E: IV}
        \textrm{IV} \le \iint_{\calO} |\scrP^* b| \, \omega^{-1}. 
    \end{equation}
    Now, there are four terms in the definition of $\scrP^*$ as written in \eqref{S: E: dual proj}. We will insert into the right-hand side of \eqref{S: E: IV} each of those terms. For the first term, we can estimate
    \begin{align*}
        \iint_{\calO} |\rho[bg]| \omega^{-1} &= \iint_{\calO} \lt| \int_{\R^d_{v'}} b(v') g(v') \,\dv' \rt| \,\omega^{-1} \,\dx\dv \\
        &\le C_\Omega \|g\|_{L^\infty_\omega((0,T)\times\R^d)} \|b_t\|_{L^1_{\omega^{-1}}} \|\omega^{-1}\|_{L^1(\R^d)},
    \end{align*}
    thanks to the Fubini--Tonelli theorem. Here, we emphasize that the right-hand side is well-defined since we already know that $b\in C([0,T];L^1_{\omega^{-1}})$. For the second term, in the same way, we obtain
    \begin{align*}
        \iint_{\calO} |v|^2 |\rho[b\Delta_v g]| \omega^{-1} \lesssim \|\Delta_v g\|_{L^\infty_\omega((0,T)\times\R^d)} \|b_t\|_{L^1_{\omega^{-1}}} \||v|^2 \omega^{-1}\|_{L^1(\R^d)} ,
    \end{align*}
    where indeed $|v|^2 \omega^{-1} \in L^1$ thanks to $\omega\in \calW$. The third and fourth terms also follow in a similar fashion:
    \begin{align*}
        \iint_{\calO} |\rho[v\cdot b\nabla_v g]| \omega^{-1} &\le \|\nabla_v g\|_{L^\infty_{\la v \ra \omega }((0,T)\times\R^d)} \|b_t\|_{L^1_{\omega^{-1}}} \|\omega^{-1}\|_{L^1(\R^d)}, \\
        \iint_{\calO} |v\rho[b\nabla_v g]| \omega^{-1} &\le \|\nabla_v g\|_{L^\infty_{\omega}((0,T)\times\R^d)} \|b_t\|_{L^1_{\omega^{-1}}} \|v\omega^{-1}\|_{L^1(\R^d)}.
    \end{align*}
    In all, we get from \eqref{S: E: IV} and the estimates for $g$ (Lemma \ref{Lem/ Elem Linfty} applied with the weight $\la v \ra \omega$) that there exists $\lambda >0$ for which
    \begin{equation*}
        \textrm{IV} \le C e^{\lambda T} \|b_t\|_{L^1_{\omega^{-1}}} \|\la v \ra^2 \omega^{-1}\|_{L^1(\R^d)}.
    \end{equation*}
    Collecting the estimates, we conclude from \eqref{S: E: ddt b renorm} that
    \begin{align*}
        \iint_{\calO} \beta_\delta(b_t) \omega^{-1} &\le \iint_{\calO} \beta_\delta(b_T) \omega^{-1} + O(\delta) + C e^{\lambda T} \int_t^T \|b_s\|_{L^1_{\omega^{-1}}} \,\ds.
    \end{align*}
    By the monotone convergence theorem we can pass to the limit $\delta\searrow 0$, obtaining
    \begin{align*}
        \iint_{\calO} |b_t| \omega^{-1} &\le \iint_{\calO} |b_T| \omega^{-1} + C e^{\lambda T} \int_t^T \|b_s\|_{L^1_{\omega^{-1}}(\calO)} \,\ds.
    \end{align*}
    We obtain the assertion of the lemma via Gr\"onwall's inequality.
\end{proof}

\begin{proposition}\label{S: P: infty}
    Let $a$, $k$, and $g_0$ satisfy the assumptions of Lemma \ref{s: l: b L1}. Then for each $h_0\in L^\infty_\omega(\calO)$, there exists a unique solution $h \in C([0,T]; L^\infty_\omega(\calO)-*)$ to the linear forwards problem \eqref{S: E: Lin}. Furthermore, the solution satisfies
    \begin{equation}\label{S: E: h Linfty}
        \forall t\in [0,T], \quad \|h_t\|_{L^\infty_\omega} \le C_1 \exp( \lambda T \exp(C T)) \|h_0\|_{L^\infty_\omega} + C_2 \exp(\lambda T \exp(CT)) a^2 
    \end{equation}
    for some generic constants $C,C_1,C_2>0$ which are independent of $\varsigma$ and $T$.
\end{proposition}

\begin{proof}
    Fix any weight $\wt{\omega}\in \wt{\calW}$ so that $\wt{\omega} \omega^{-1} \in L^2(\R^d)$. Then by H\"older's inequality, the embedding $L^\infty_{\omega}(\calO) \hookrightarrow L^2_{\wt{\omega}}(\calO)$ is continuous. Therefore, $h_0\in L^2_{\wt{\omega}}(\calO)$ and we deduce via Lemma \ref{l: sp: for lin WP} that there exists a unique solution $h\in C([0,T];L^2_{\wt{\omega}})$ to \eqref{S: E: Lin}. Thus, it remains for us to prove that $h_T\in L^\infty_\omega$, namely the estimate \eqref{S: E: h Linfty}, and also that $t\mapsto h_t\in L^\infty_\omega - *$ is continuous. Toward this end, let $b_T\in L^2_{\wt{\omega}^{-1}}$ and denote by $b\in C([0,T];L^2_{\wt{\omega}^{-1}})$ the unique solution to the backwards problem \eqref{S: E: Back Lin} corresponding to $b_T$. We then apply Lemma \ref{p: l: dual} to find that 
    \begin{equation}\label{S: E: hb duality}
        \iint_{\calO} h_T b_T \,\dx\dv = \iint_{\calO} h_0 b(0) - \frac{1}{d} \iiint_{\calO^T} \bfA_\varsigma[k] \Delta_v g \, b_t \,\dx\dv\dt.
    \end{equation}
    Next, we utilize the Riesz representation theorem as follows: using that $L^2_{\wt{\omega}^{-1}}$ is dense in $L^1_{\omega^{-1}}$,
    \begin{equation*}
    \begin{split}
        \|h_T\|_{L^\infty_{\omega}} &= \sup_{ \substack{ b_T \in L^2_{\wt{\omega}^{-1}} \\ \|b_T\|_{L^1_{\omega^{-1}}} \le 1 }  } \iint_{\calO} h_T b_T \,\dx\dv \\
        &= \sup_{ \substack{ b_T \in L^2_{\wt{\omega}^{-1}} \\ \|b_T\|_{L^1_{\omega^{-1}}} \le 1 }  } \lt( \iint_{\calO} h_0 b(0) \,\dx\dv - \frac{1}{d} \iiint_{\calO^T} \bfA_\varsigma[k] \Delta_v g \, b_t \,\dx\dv\dt \rt) \\
        &\le C_1 \|h_0\|_{L^\infty_{\omega}} \exp(\lambda T \exp(CT))  \\
        &\quad + C \|\bfA_\varsigma[k]\|_{L^\infty(\Omega^T)} \exp(\lambda T) \int_0^T  \exp\Big( \lambda (T-t) \exp(C T) \Big) \,\dt  \\
        &\le C_1 \|h_0\|_{L^\infty_{\omega}} \exp( \lambda T \exp(CT)) +  C_2 \exp( \lambda T \exp(CT) )  a^2 
    \end{split}
    \end{equation*}
    for some constants $C,C_1,C_2>0$. In the above, we used that $\bfA_\varsigma[k]$ is uniformly bounded in $\varsigma$:
    \begin{equation*}
        0 \le \bfA_\varsigma[k] \le \frac{ \| j[k] *_t \bfs_\va \star_\varsigma \bfr_\varsigma \|_{L^\infty}^2 }{ 1 - \| \rho[k] *_t \bfs_\va \star_\varsigma \bfr_\varsigma \|_{L^\infty}  } \le  \frac{ \|j[k]\|_{L^\infty}^2 }{ 1 - \|\rho[k]\|_{L^\infty} } \le 2 \|k\|_{L^\infty_\omega}^2 \lesssim a^2,
    \end{equation*}
    as well as Lemma \ref{Lem/ Elem Linfty} (the $L^\infty$-estimates for $\omega\Delta_v g$). We conclude that $h_T \in L^\infty_\omega$ and that \eqref{S: E: h Linfty} holds.
    
    Finally, we can easily check that $h\in C([0,T];L^\infty_\omega -*)$. Since we already know that $h\in L^\infty_\omega(\calO^T)$, let us merely prove that $t\mapsto \iint_{\calO} (h_t \, \mathfrak{b})$ is continuous for every $\mathfrak{b} \in L^2_{\wt{\omega}^{-1}}$. This is enough to conclude the claim, because $L^2_{\wt{\omega}^{-1}}$ is dense in $L^1_{\omega^{-1}}$.
    
    Let $0\le s \le t \le T$, and denote by $b\in C([0,t];L^1_{\omega^{-1}})$ the solution to the backwards problem \eqref{S: E: Back Lin} corresponding to terminal datum $b_t = \mathfrak{b}$. The duality formula \eqref{S: E: hb duality} shows that
    \begin{align*}
        \iint_{\calO} h_t \mathfrak{b} - \iint_{\calO} h_s b_s &= - \int_s^t \iint_{\calO} \bfA[k] \Delta_v g \, b_\tau \,\dx\dv\tnd\tau.
    \end{align*}
    The above is equivalent to
    \begin{align*}
        \iint_{\calO} (h_t - h_s) \mathfrak{b}  &= \iint_{\calO} h_s (b_s - \mathfrak{b}) - \int_s^t \iint_{\calO} \bfA[k] \Delta_v g \,b_\tau \,\dx\dv\tnd\tau.
    \end{align*}
    On the one hand, since $b\in C([0,T];L^1_{\omega^{-1}})$, the first integral of the right-hand side satisfies
    \begin{align*}
        \lt| \iint_{\calO} h_s (b_s - \mathfrak{b}) \rt| \le \lt(\sup_{\tau\in [0,T]} \|h_\tau\|_{L^\infty_\omega(\calO)} \rt) \|b_s - \mathfrak{b}\|_{L^1_{\omega^{-1}}(\calO)} \xrightarrow[|s-t|\to 0]{} 0.
    \end{align*}
    On the other hand, the second integral of the right-hand side is absolutely continuous in $s$ because the integrand is integrable over $[0,T]$. Therefore,
    \begin{equation*}
        \forall \mathfrak{b}\in L^2_{\wt{\omega}^{-1}}, \quad \lt|\iint_{\calO} (h_t - h_s) \mathfrak{b} \rt| \xrightarrow[|s-t|\to 0]{} 0,
    \end{equation*}
    which proves the claim.
\end{proof}

The following result is immediate

\begin{corollary}\label{e: c: inv}
    Given $T>0$, there exists $a_3>0$ (dependent on $T$) small enough such that for each $a < \min\{a_1, a_2, a_3\}$, the following holds. For $g_0\in W^{2,\infty}_{\la v \ra \omega}(\R^d)$ and $k\in \calX_{T,\omega,a}$, let $(S_k(t,0))_{t\in [0,T]}$ denote the family of evolution operators defined so that for each $h_0\in L^\infty_\omega$, the unique solution to \eqref{S: E: Lin} in $C([0,T];L^\infty - *)\cap L^\infty_\omega(\calO^T)$ is given by $S_k(t,0)h_0$. There exists $\delta = \delta(a, T)>0$ independent of $\varsigma$ such that whenever $\|h_0\|_{L^\infty_\omega} \le \delta$, the map
    \begin{equation*}
        \Psi : k \mapsto  S_k(t,0) h_0
    \end{equation*}
    leaves $\calX_{T,\omega,a}$ invariant.
\end{corollary}

\begin{proof}
    We fix $a_3$ small enough (in terms of $T$) so that for any $a \le a_3$, the last term in \eqref{S: E: h Linfty} satisfies
    \begin{align*}
        C_2 \exp(\lambda T \exp(CT)) \, a^2 < \frac{a}{2}.
    \end{align*}
    Then, for any $a < \min\{a_1,a_2,a_3\}$, we deduce from \eqref{S: E: h Linfty} that
    \begin{align*}
        \|S_k(t,0)h_0\|_{L^\infty_\omega} < \|h_0\|_{L^\infty_\omega} C_1 \exp(\lambda T \exp(CT)) + \frac{a}{2}.
    \end{align*}
    For any such given $a$, if we choose $\delta=\delta(a,T)>0$ so small that
    \begin{equation*}
        \delta \, C_1 \exp(\lambda T \exp(CT)) < \frac{a}{2},
    \end{equation*}
    then we find that $\|h_0\|_{L^\infty_\omega} \le \delta$ implies
    \begin{equation*}
        \|S_k(t,0)h_0\|_{L^\infty_\omega(\calO^T)} < a.
    \end{equation*}
\end{proof}

\subsection{Global existence of the regularized equation}

We now return to the discussion of \eqref{S: E: Reg Nonlin}.

\begin{proposition} \label{S: P: Exist Nonlin}
    Following the notations of Corollary \ref{e: c: inv}, we assume that $a < \min\{a_1,a_2,a_3\}$, $\|h_0\|_{L^\infty_\omega} \le \delta(a,T)$, and also $g_0 \in W^{2,\infty}_{\la v \ra \omega} (\R^d)$. Then for each $\varsigma>0$, there exists at least one weak and renormalized solution $h^\varsigma \in C([0,T];L^\infty - *) \cap \calX_{T,\omega,a}$ to the regularized nonlinear problem \eqref{S: E: Reg Nonlin}.
\end{proposition}

\begin{proof}
    Thanks to Corollary \ref{e: c: inv}, we know that 
    \begin{equation*}
    \Psi: k \mapsto S_k(t,0)h_0 
    \end{equation*}
    maps $\calX_{T,\omega,a}$ into itself. We endow $\calX_{T,\omega,a}$ with the $L^\infty_\omega-*$ topology, for which it is obviously compact and convex. Hence, to conclude to the proposition, it only remains prove that $\Psi$ is continuous with respect to this topology. In other words, let $k^n \in \calX_{T,\omega,a}$ denote a sequence with
    \begin{equation*}
        k^n \weakto k \quad \text{in} \quad L^\infty_\omega -*.  
    \end{equation*}
    For brevity, we write $h^n := S_{k^n}(t,0) h_0$. We need to prove that $h^n \weakto S_k(t,0)h_0$.
    
    (Step 1: Convergence of the coefficients.) 

    Let us first establish how the coefficients in $k^n$ of \eqref{S: E: Lin} (such as $E[k^n]$, $\bfA_\varsigma[k^n]$) behave with respect to weak$-*$ convergence of $k$. Using the fact that $k^n\weakto k$ in $L^\infty_\omega -*$, we can easily establish all of the following convergences
    \begin{equation} \label{S: E: Macs Lin}
        \rho[k^n] \weakto \rho[k], \quad j[k^n]\weakto j[k], \quad E[k^n] \weakto E[k] \quad \text{in} \quad L^\infty(\Omega^T) -*.
    \end{equation}
    For instance, for $\varphi\in L^1(\Omega^T)$ there holds
    \begin{align*}
        \lt| \iint_{\Omega^T} (E[k^n] - E[k])\varphi(t,x) \rt| &= \lt| \iiint_{\calO^T} |v|^2 (k^n - k) \varphi(t,x) \rt| \\
        &= \lt| \iiint_{\calO^T} \Big( (k^n - k) \omega \Big) \, \Big(|v|^2 \omega^{-1} \varphi(t,x) \Big)  \rt| \\
        &\xrightarrow[n\to\infty]{} 0.
    \end{align*}
    We can now prove the strong convergence of $\bfA_\varsigma[k^n]$. Indeed, let us observe first that the weak$-*$ convergence of $j[k^n]$ provides
    \begin{align*}
        \forall (t,x)\in \Omega^T, \quad j[k^n] *_t \bfs_\va \star_\varsigma \bfr_{\varsigma}(t,x) &:= \iint_{\Omega^T} j[k^n](s,y) \, \bfs_\va(t-s) \bfr_\varsigma(x - 2\varsigma n_x - y) \,\ds\dy \\
        &\xrightarrow[n\to\infty]{} \iint_{\Omega^T} j[k](s,y) \, \bfs_\va(t-s) \bfr_\varsigma(x-2\varsigma n_x - y) \,\ds\dy \\
        &= j[k] *_t \bfs_\va \star_\varsigma \bfr_\zeta (t,x).
    \end{align*}
    Similarly we get $\rho[k^n] *_t \bfs_\va \star_\varsigma \bfr_{\varsigma} \to \rho[k] *_t \bfs_\va \star_\varsigma \bfr_{\varsigma}$ pointwisely in $\Omega^T$. By the very definition of $\bfA_\varsigma$, we deduce
    \begin{equation*}
        \forall (t,x) \in \Omega^T, \quad \bfA_\varsigma[k^n](t,x) \to \bfA_\varsigma[k](t,x).
    \end{equation*}
    Since $(\bfA_\varsigma[k^n])_n$ is bounded in $L^\infty(\Omega^T)$, the convergence theorem of Vitali then provides
    \begin{equation}\label{S: E: bfA Lin}
        \forall p\in [1,\infty), \quad \bfA_\varsigma[k^n] \to \bfA_\varsigma[k] \quad \text{in} \quad L^p(\Omega^T).
    \end{equation}
    In similar nature, we can easily prove that
    \begin{equation}\label{S: E: bfB Lin}
        \bfB[k^n] \weakto \bfB[k] \quad \text{in} \quad \calD'(\calO^T).
    \end{equation}

    (Step 2: Passing to the limit in the PDE.) We now have enough information to pass to the limit. Indeed, consider the equation \eqref{S: E: Lin} satisfied by the $h^n$. We note that the diffusion coefficient ($=:\sigma_\varsigma[k^n]$) is strictly and uniformly lower bounded as a result of \eqref{S: E: coeff}. Consequently, noting that $L^\infty_{\omega}\hookrightarrow L^2$, we can apply the renormalized Green's formula (with the equation \eqref{S: E: Lin} satisfied by the $h^n$) to deduce that
    \begin{equation*}
        (\sigma_\va[h^n] h^n) \quad \text{is bounded in} \quad L^2_{t,x} H^1_v.
    \end{equation*}
    In particular the uniform lower bound on $\sigma_\va[k^n]$ yields that
    \begin{equation}\label{S: E: H1}
    \text{$(\nabla_v h^n)$ is bounded in $L^2_{t,x,v}$.}
    \end{equation}
    Now denote by $\chi_{R,1}(x)\in C_c^\infty(\Omega)$ a cutoff function satisfying
    \begin{equation*}
        \chi_{R,1}(x) = 1 \quad \text{in} \quad \{x\in \Omega \; | \; \textnormal{dist}(x,\Omega^c) > 1/R\}.
    \end{equation*}
    Furthermore, consider a smooth radial cutoff function $\chi_{R,2}(v) \in C_c^\infty(\R^d)$, satisfying $\mathbf{1}_{|v|\le R}(v) \le \chi_{R,2}(v) \le \mathbf{1}_{|v|\le 2R}(v)$. 
    Denoting $h^n_R(t,x,v) := h^n(t,x,v) \chi_{R,1}(x) \chi_{R,2}(v)$, we can derive from $\eqref{S: E: Lin}_1$ the PDE that is satisfied by $h^n_R$, which holds in all of $[0,T]\times\R^d_x \times\R^d_v$ because of the cutoffs. Then keeping in mind \eqref{S: E: H1}, we can apply the hypoelliptic energy estimates of \cite[Theorem 1.3]{Bou02}\footnote{Following the notation of Bouchut \cite{Bou02}, we take $p=2$, $r=0$, $\beta=1$, $\Omega=1$, and $m=2$, so that $s=1/4$.}, to deduce that $(h^n_R)_n$ is bounded in $H^{1/4}_{t,x}([0,T]\times \R^d_x \times \R^d_v)$. Since
    \begin{equation*}
        \chi_{R,1}(x) \chi_{R,2}(v) \equiv 1 \quad \text{in} \quad \calO_R := \{(x,v)\in \Omega \; | \; \textnormal{dist}(x,\Omega^c) > 1/R, \; |v| < R \},
    \end{equation*}
    we consequently obtain that $(h^n)$ is bounded in $H^{1/4}_{t,x}([0,T]\times \calO_R)$. We already mentioned that $(\nabla_v h^n)$ is bounded in $L^2_{t,x,v}$, so altogether
    \begin{equation*}
        (h^n) \quad \text{is bounded in} \quad H^{1/4}([0,T]\times \calO_R).
    \end{equation*}
    The usual diagonal argument proves that, up to some subsequence, there is a limit $\mathfrak{h}$ such that
    \begin{equation*}
        h^n \to \mathfrak{h} \quad \text{strongly in} \quad L^2_{\rm loc}([0,T]\times \calO) \quad \text{and a.e.}
    \end{equation*}
    Since $(h^n)\subset \calX_{T,\omega,a}$, an interpolation shows next that
    \begin{equation}\label{S: E: hn Lin}
        \forall p\in [1,\infty), \quad h^n \to \mathfrak{h} \quad \text{strongly in} \quad L^p_{\rm loc}([0,T]\times \calO).
    \end{equation}
    The convergences in \eqref{S: E: Macs Lin}--\eqref{S: E: bfA Lin}--\eqref{S: E: bfB Lin}--\eqref{S: E: hn Lin} allow us to pass to the limit in the equation $\eqref{S: E: Lin}_1$ satisfied by the $h^n$ in $\calD'((0,T)\times \calO)$. We deduce that the limit $\mathfrak{h}$ is a solution to $\eqref{S: E: Lin}_1$ in $\calD'((0,T)\times \calO)$ corresponding to $k$. 

    (Step 3: Recovering the traces.) We first point out that clearly, at least up to some subsequence, $h^n\weakto \mathfrak{h}$ in both $L^\infty_\omega-*$ and $L^2_{t,x}H^1_{v, {\rm loc}}$, thanks to the uniqueness of distributional limits. Consequently $\mathfrak{h}$, being a distributional solution to the linear problem $\eqref{S: E: Lin}_1$ and also a member of $L^2_{t,x}H^1_{v,{\rm loc}}$, is automatically a renormalized solution which admits traces $\gamma \mathfrak{h}$ and $\mathfrak{h}_t$, as a straightforward adaptation of \cite[Theorem 4.2]{Mi10}.

    It only remains to prove that $\gamma \mathfrak{h}$ satisfies the specular reflection boundary condition. Indeed, if so, then the well-posedness result of Proposition \ref{prop: WP} provides that $\mathfrak{h} = S_k(t,0)h_0$ in the class $L^\infty(0,T;L^2_{\wt{\omega}}(\calO))$ and therefore a.e. This would next imply, thanks to the uniqueness of $S_k(t,0)h_0$, that in fact the full sequence $(h^n)$ converges to $\mathfrak{h}=S_k(t,0)h_0$ and the continuity of $\Psi$ would be proven.
    
    To establish that the boundary condition holds for $\mathfrak{h}$, we proceed exactly as in (Step 3) of the proof of Proposition \ref{prop: WP}. Namely, recall once more that since $(h^n)$ is bounded in $L^\infty_\omega$ (being members of $\calX_{T,\omega,a}$), the family $(h^n)$ is bounded in $L^\infty(0,T;L^2_{\wt{\omega}}(\calO))$; thus, by using the test function $\chi(x,v) = (n_x\cdot v)\wt{\omega}^2 \la v \ra^{-4}$ against the equation satisfied by $|h^n|^2$, we obtain
    \begin{align*}
        \|\gamma h^n\|_{L^2(\Sigma^T, (n_x\cdot v)^2 \wt{\omega}^2 \la v \ra^{-4} \,\tnd\sigma_x \dv \dt)}^2 \le C.
    \end{align*}
    Then in the same way as in the proof of Proposition \ref{prop: WP}, we obtain that $\gamma_{\pm} h^n \weakto \gamma_{\pm} \mathfrak{h}$ weakly in $L^2_{\rm loc}(\Sigma^T_{\pm},(n_x\cdot v)^2 \tnd\sigma_x\dv\dt)$ (thanks to uniqueness of the trace). Since each $\gamma h^n$ satisfies $\gamma_- h^n = \gamma_+ h^n\circ \scrV_x$, we can write this against a test function and then pass to the limit to conclude that $\gamma_- \mathfrak{h} = \gamma_+\mathfrak{h}\circ \scrV_x$ a.e. This completes the proof.
\end{proof}

\subsection{Passing to the limit in the regularized equation}

For each $\varsigma>0$, we are armed with the existence of a solution to the regularized problem \eqref{S: E: Reg Nonlin}. In this section we will pass to the limit $\varsigma\searrow 0$ and prove that the limit point is a solution to the original nonlinear equation \eqref{S: E: nonlin}.

\begin{proposition}\label{S: P: final}
    Let any $a<\min\{a_1,a_2,a_3\}$ be given, and assume $g_0\in W^{2,\infty}_{\la v \ra \omega}(\R^d)$, $\|h_0\|_{L^\infty_\omega}\le \delta(a,T)$. Denote by $(h^\varsigma)_{\varsigma>0} \subset \calX_{T,\omega,a}$ a family of solutions to \eqref{S: E: Reg Nonlin}, as constructed in Proposition \ref{S: P: Exist Nonlin}. Then as $\varsigma\searrow 0$ there is a weak$-*$ limit point $h\in \calX_{T,\omega,a} \cap L^\infty_t L^2_{x,v} \cap L^2_{t,x}H^1_v$ which is a weak and renormalized solution to the nonlinear problem \eqref{S: E: nonlin}.
\end{proposition}
\begin{proof}
    Being realized as a fixed point of the linear problem, we know that each $h^\varsigma \in \calX_{T,\omega,a}$ is a weak and renormalized solution to \eqref{S: E: Reg Nonlin}, and clearly there is a limit $h\in \calX_{T,\omega,a}$ such that
    \begin{equation*}
        h^\varsigma \weakto h \quad \text{in} \quad L^\infty_\omega(\calO^T)-*.
    \end{equation*}

    (Step 1: Uniform-in-$\varsigma$ estimates.)
    Let us find some estimates that are uniform in $\varsigma$. From the fact that $(h^\varsigma)\subset \calX_{T,\omega,a}$, it is first clear that 
    \begin{equation}\label{S: E: va macs}
        |\rho[h^\varsigma]|, \; |j[h^\varsigma]|, \; |E[h^\varsigma]| \lesssim a.
    \end{equation}
    In particular, we can also obtain
    \begin{equation}\label{S: E: vasig A}
        0 \le \bfA_\varsigma[h^\varsigma] \le \frac{\|j[h^\varsigma] *_t \bfs_\va \star_\varsigma \bfr_\varsigma \|_{L^\infty}^2 }{ 1 - \|\rho[h^\varsigma] *_t \bfs_\va \star_\varsigma \bfr_{\varsigma} \|_{L^\infty} } \le \frac{\|j[h]\|_{L^\infty}^2}{ 1 - \|\rho[h^\varsigma]\|_{L^\infty} } \le C a^2,
    \end{equation}
    thus $(\bfA_\varsigma[h^\varsigma])$ is bounded in $L^\infty(\Omega^T)$. Similarly, it is clear that
    \begin{equation}\label{S: E: vasig B}
        (\bfB[h^\varsigma]) := (\rho[h^\varsigma] v - j[h^\varsigma]) \quad \text{is bounded in} \quad L^\infty((0,T)\times \Omega; L^\infty_{\rm loc}(\R^d)).
    \end{equation}

    Next, using the renormalized Green's formula, we observe that the following holds, as can be seen by computing for instance as in \eqref{p: e: formal}:
    \begin{align*}
        &\ddt \frac{1}{2} \iint_{\calO} |h^\varsigma|^2 + \iint_{\calO} \sigma_\varsigma[h^\varsigma] |\nabla_v h^\varsigma|^2 \\
        &= \iint_{\calO} \frac{d}{2} (1 + \rho[h^\varsigma]) |h^\varsigma|^2 \\
        &\quad + \iint_{\calO} \lt( dg\rho[h^\varsigma] + \frac{1}{d}\Delta_v g \, E[h^\varsigma] + \bfB[h^\varsigma]\cdot \nabla_v g - \frac{1}{d}\bfA_\varsigma[h^\varsigma] \Delta_v g \rt) h^\varsigma . 
    \end{align*}
    In the above, we noted that the boundary integral vanishes by symmetry, and we remind the reader that $\sigma_\varsigma$ is defined in \eqref{S: E: sigma}. It is also worth emphasizing that all of the integrals above make sense because $L^\infty_\omega \hookrightarrow L^2$. All terms in the last integral above can easily be crudely bounded. For instance
    \begin{align*}
        \iint_{\calO} g \rho[h^\varsigma] h^\varsigma &\lesssim Ca (\|g\|_{L^2}^2 + \|h^\varsigma\|_{L^2}^2),
    \end{align*}
    and
    \begin{align*}
        \iint_{\calO} \bfB[h^\varsigma] \cdot \nabla_v g \, h^\varsigma &= \iint_{\calO} (\rho[h^\varsigma] v - j[h^\varsigma]) \cdot \nabla_v g \, h^\varsigma \\
        &= - \iint_{\calO} \rho[h^\varsigma] (d h^\varsigma + v\cdot \nabla_v h^\varsigma) g \\
        &\lesssim a \|h^\varsigma\|_{L^2} \|g\|_{L^2} + a \|v g\|_{L^2} \|\nabla_v h^\varsigma\|_{L^2},
    \end{align*}
    by applying the Gauss--Ostrogradsky theorem. Note that the last term can be absorbed through Young's inequality. In conclusion, we can obtain that $(h^\varsigma)$ is bounded in $L^\infty_t L^2_{t,x} \cap L^2_{t,x} H^1_v$, especially thanks to the uniform nondegeneracy of $\sigma_\varsigma$ (see \eqref{S: E: coeff}).

    At last, using these bounds for $(h^\va)$ just established, we can obtain the same energy estimate on the trace by computing as in \eqref{P: E: Trace Egy}. Namely, using $\varphi_R(x,v) = (v\cdot n_x)\chi_R(v)$ as a test function with $\chi_R\in C_c^\infty(B_{R+1})$ and $\mathbf{1}_{|v|\le R} \le \chi_R \le \mathbf{1}_{|v|\le R+1}$, we may obtain
    \begin{equation}\label{S: E: trace loc}
        \begin{split}
        \iiint_{[0,T]\times \p\Omega\times B_R} |\gamma h^\va|^2 (n_x\cdot v)^2 
        &\le C_R \Big( \|h^\va\|_{L^\infty(0,T;L^2)}^2 + \|h_0\|_{L^2}^2 + \|h^\va\|_{L^2_{t,x} H^1_v}^2 \Big)  \\
        &\le C_R,
    \end{split}
    \end{equation}
    where the first inequality follows by resorting to the bounds in \eqref{S: E: va macs}--\eqref{S: E: vasig A}--\eqref{S: E: vasig B}.

    (Step 2: Passing to the limit and obtaining a weak solution to \eqref{S: E: nonlin}.) We now aim to pass to the limit in the Green's formula of \eqref{S: E: Reg Nonlin} satisfied by $(h^\va, \gamma h^\va)$. First we observe that just as in (Step 2) of the proof of Proposition \ref{S: P: Exist Nonlin}, we may apply the hypoelliptic estimates of \cite{Bou02} on a suitably truncated sequence of $h^\varsigma$, to deduce
    \begin{align*}
        h^\varsigma \to h \quad \text{strongly in} \quad L^2_{\rm loc}((0,T) \times \calO).
    \end{align*}
    Passing to subsequences, then arguing diagonally, we obtain a subsequence such that
    \begin{equation*}
        h^\varsigma \to h \quad \text{a.e. in} \quad (0,T)\times \calO.
    \end{equation*}
    The dominated convergence theorem of Vitali, along with the bounds in \eqref{S: E: va macs}, then show
    \begin{equation*}
        \rho[h^\varsigma] \to \rho[h], \quad j[h^\varsigma] \to j[h], \quad E[h^\varsigma] \to E[h] \quad \text{in}\quad L^p(\Omega^T), \quad \forall p\in [1,\infty).
    \end{equation*}
    It is also clear that
    \begin{align*}
        &\rho[h^\varsigma] *_t \bfs_\va \star_\varsigma \bfr_{\varsigma} - \rho[h] \\
        &= (\rho[h^\varsigma] - \rho[h]) *_t \bfs_\va \star_\varsigma \bfr_{\varsigma} + ( \rho[h] *_t \bfs_\va \star_\varsigma \bfr_{\varsigma} - \rho[h]) \\
        &\to 0 \quad \text{in} \quad L^p(\Omega^T), \quad \forall p\in [1,\infty),
    \end{align*}
    and similarly $j[h^\varsigma] *_t \bfs_\va \star_\varsigma \bfr_{\varsigma} \to j[h]$ in $L^p(\Omega^T)$ for all $p\in [1,\infty)$. Consequently, up to further subsequences,
    \begin{align*}
        \bfA_\varsigma[h^\varsigma] \to \bfA[h], \quad \bfB[h^\varsigma] \to \bfB[h] \quad \text{a.e. and in} \quad L^p((0,T)\times \Omega; L^p_{\rm loc}(\R^d)) \quad \forall p\in [1,\infty).
    \end{align*}
    The convergence of $\bfA_\varsigma[h^\varsigma]$ implies in particular that
    \begin{align*}
        \sigma_{\varsigma}[h^\varsigma] \to \sigma[h] := \frac{1}{d}\lt(E[g_0] + E[h] - \bfA[h] \rt) \quad \text{a.e. and in} \quad L^p(\Omega^T), \quad \forall p\in [1,\infty).
    \end{align*}
    On the other hand, the convergence of $\bfB[h^\va]$ implies
    \begin{align*}
        \zeta[h^\va] &:= v + \bfB[h^\va]\\
        &\to v + \bfB[h] = \zeta[h] \quad \text{in} \quad L^p_{\rm loc}([0,T]\times \ov{\calO}), \quad \forall p\in [1,\infty).
    \end{align*}
    Next, we have
    \begin{align*}
        \eta[h^\va] &:= d(1 + \rho[h^\va]) \\
        &\to d(1 + \rho[h]) = \eta[h] \quad \text{in} \quad L^p_{\rm loc}([0,T]\times \ov{\calO}), \quad \forall p\in [1,\infty).
    \end{align*}
    Similarly, denoting the remaining terms (last line) of \eqref{S: E: Reg Nonlin} as
    \begin{align*}
        P_\varsigma[h^\varsigma] := dg\rho[h^\varsigma] + \frac{1}{d}\Delta_v g E[h^\varsigma] + \bfB[h^\varsigma] \cdot \nabla_v g - \frac{1}{d}\Delta_v g \, \bfA_\varsigma [h^\varsigma],
    \end{align*}
    the previously established convergences imply
    \begin{align*}
        P_\va[h^\va] &\to P[h]\\
        &:= dg\rho[h] + \frac{1}{d}\Delta_v g E[h] + \bfB[h] \cdot \nabla_v g - \frac{1}{d}\Delta_v g \, \bfA [h] 
    \end{align*}
    in $L^p_{\rm loc}([0,T]\times \ov{\calO})$, again for all $p\in [1,\infty)$. Finally, for the trace, from the uniform estimate \eqref{S: E: trace loc} we know that there is a limit $\wt{\gamma}$ for which (up to a subsequence)
    \begin{equation*}
        \gamma h^\va \,(n_x\cdot v) \weakto \wt{\gamma} \quad \text{weakly in } L^2_{\rm loc}([0,T]\times \p\Omega \times \R^d).
    \end{equation*}
    We then define $\ov{\gamma} := \wt{\gamma}/(n_x\cdot v)$ with the convention $\ov{\gamma} = 0$ when $n_x\cdot v = 0$.
    
    Collecting all of the convergences above, we may write the Green's formula of \eqref{S: E: Reg Nonlin} that is satisfied by $(h^\va, \gamma h^\va)$ then pass to the limit, to obtain that $(h,\ov{\gamma})$ satisfies the Green's formula of \eqref{S: E: nonlin}. By the uniqueness of the trace we deduce $\ov{\gamma} = \gamma h$. It is clear that $\gamma h$, being the weak limit of $\gamma h^\va$, satisfies the specular reflection boundary condition because each $\gamma h^\va$ does.

    (Step 3: The solution is renormalized.) We consider the $(x,v)$-regularized functions $h_\ve:= h \star_\ve \bfr_\ve *_v \bfr_\ve$ as was done in the proof of Lemma \ref{p: l: dual}. Since $h\in L^2_{t,x} H^1_v \cap L^\infty_\omega$, it clearly holds that (up to a subsequence)
    \begin{equation}\label{S: E: hve itself}
    \begin{split}
        &h_{\ve t} \to h_t \quad \text{a.e. and in} \quad L^p(\calO) \quad \forall p\in [2,\infty), \\
        &h_\ve \to h \quad \text{a.e. and in} \quad L^p((0,T)\times \calO) \quad \forall p\in [2,\infty), \\
        &\nabla_v h_\ve \to \nabla_v h \quad \text{a.e. and in} \quad L^2((0,T)\times\Omega;L^2(\R^d)).
    \end{split}
    \end{equation}
    Interpolating this with the fact that $h\in \calX_{T,\omega,a}$, we obtain also that
    \begin{equation}\label{S: E: hve macs}
        \rho[h_\ve] \to \rho[h], \quad j[h_\ve] \to j[h], \quad E[h_\ve] \to E[h] \quad \text{a.e. and in} \quad L^p(\Omega^T) \quad \forall p\in [1,\infty), 
    \end{equation}
    and consequently
    \begin{equation}\label{S: E: hve coefficients}
    \begin{split}
        &\sigma[h_\ve] \to \sigma[h], \quad \eta[h_\ve] \to \eta[h] \quad \text{in} \quad L^p(\Omega^T), \quad \forall p\in [1,\infty), \\
        &\zeta[h_\ve] \to \zeta[h] \quad \text{in} \quad L^p([0,T]\times \Omega; L^p_{\rm loc}(\R^d)) \quad \forall p\in [1,\infty).
    \end{split}
    \end{equation}
    Also, we have that $h_\ve$ is a solution to
    \begin{equation}\label{S: E: hve}
    \begin{split}
        \p_t h_\ve &= -(v\cdot \nabla_x h)_\ve + \nabla_v\cdot (\sigma[h] \nabla_v h)_\ve \\
        &\quad + (\zeta[h]\cdot \nabla_v h)_\ve + (\eta[h] h)_\ve + (P[h])_\ve  \quad \text{in}\quad \calD'((0,T)\times \ov{\calO}),
    \end{split}
    \end{equation}
    and from this very equation $h_\ve \in W^{1,1}((0,T); W^{1,\infty}_{\rm loc}(\ov{\calO}))$. By Sobolev embedding we get a version of $h_\ve$ that is continuous on all of $(0,T) \times \ov{\calO}$. It is clear that $h_\ve$ thus admits traces $h_{\ve t}$ and $\gamma h_\ve$ for every $t\in (0,T)$.

    We now rewrite \eqref{S: E: hve} into the more convenient form
    \begin{equation}\label{S: E: hve weak}
    \begin{split}
        \p_t h_\ve &= - v\cdot \nabla_x h_\ve + \nabla_v\cdot (\sigma[h_\ve] \nabla_v h_\ve) \\
        &\quad + \zeta[h_\ve] \cdot \nabla_v h_\ve + \eta[h_\ve] h_\ve + (P[h])_\ve + \scrC_\ve , 
    \end{split}
    \end{equation}
    where the commutator term is
    \begin{align*}
        \scrC_\ve &:= - \Big( (v\cdot \nabla_x) h_\ve - (v\cdot \nabla_x h)_\ve \Big)
        +  (\zeta[h] \cdot \nabla_v h)_\ve  - \zeta[h_\ve] \cdot \nabla_v h_\ve  \\
        &\quad + (\eta[h] h)_\ve - \eta[h_\ve] h_\ve + \nabla_v\cdot \Big( (\sigma[h] \nabla_v h)_\ve - \sigma[h_\ve] \nabla_v h_\ve \Big).
    \end{align*}
    Let us put $\scrC_\ve = \scrC_\ve^1 + \nabla_v\cdot \scrC_\ve^2$ so that $\scrC_\ve^2 := (\sigma[h]\nabla_v h)_\ve - \sigma[h_\ve] \nabla_v h_\ve$. To make clear how $\scrC_\ve$ vanishes in the limit, we write 
    \begin{equation} \label{S: E: comm rewrite}
    \begin{split}
        \scrC_\ve^1 &= - \Big( (v\cdot \nabla_x) h_\ve - (v\cdot \nabla_x h)_\ve \Big) \\
        &\quad + \Big( (\zeta[h] \cdot \nabla_v h)_\ve - \zeta[h] \cdot \nabla_v h \Big) + \Big( \zeta[h] \cdot \nabla_v h  - \zeta[h_\ve] \cdot \nabla_v h_\ve \Big) \\
        &\quad + \Big( (\eta[h] h)_\ve - \eta[h] h \Big) + \Big( \eta[h] h - \eta[h_\ve] h_\ve \Big) .
    \end{split}
    \end{equation}
    On the right-hand side of \eqref{S: E: comm rewrite}, as $\ve\searrow 0$ we have that the first line tends to zero in $L^2_{\rm loc}([0,T]\times \ov{\calO})$ by \cite[Lemma II.1]{DiPLion89}. For the second line, its first term clearly tends to zero in $L^2_{\rm loc}([0,T]\times\ov{\calO})$ by standard property of the mollifier. For its second term, we can rewrite again
    \begin{align*}
        &\zeta[h]\cdot \nabla_v h - \zeta[h_\ve] \cdot \nabla_v h_\ve\\
        &= (\zeta[h] - \zeta[h_\ve] ) \cdot \nabla_v h + \zeta[h_\ve] \cdot (\nabla_v h - \nabla_v h_\ve),
    \end{align*}
    where on the one hand, $(\zeta[h] - \zeta[h_\ve] ) \cdot \nabla_v h \to 0$ in $L^p_{\rm loc}([0,T]\times\ov{\calO})$ for every $p\in [1,\infty)$ thanks to \eqref{S: E: hve coefficients} and the fact that $\nabla_v h \in L^2$. On the other hand $\zeta[h_\ve]\cdot (\nabla_v h - \nabla_v h_\ve)\to 0$ in $L^2_{\rm loc}([0,T]\times \ov{\calO})$ because we can combine $\eqref{S: E: hve itself}_3$ with the fact that $|\zeta[h_\ve]| = | v + \bfB[h_\ve]| \lesssim 1+|v|$. The last line of $\scrC_\ve^1$ can be checked in the same way, and we obtain that $\scrC_\ve^1 \to 0$ at least in $L^2_{\rm loc}([0,T]\times\ov{\calO})$.

    Regarding $\scrC_\ve^2$, by rewriting
    \begin{equation*}
        \scrC_\ve^2 = \Big( (\sigma[h] \nabla_v h)_\delta - \sigma[h]\nabla_v h\Big) + \Big( \sigma[h] \nabla_v h - \sigma[h_\ve] \nabla_v h_\ve \Big),
    \end{equation*}
    its first term clearly tends to zero in $L^2$, by standard property of the mollifier. The second term converges to zero in $L^q([0,T]\times \Omega; L^q_{\rm loc}(\R^d))$ for any $q\in [1,2)$, because
    \begin{align*}
        & \| \sigma[h]\nabla_v h - \sigma[h_\ve] \nabla_v h_\ve \|_{L^q_{t,x} L^{q}_{v, {\rm loc}}} \\
        &\le  \lt\| \sigma[h] \Big(\nabla_v h - \nabla_v h_\ve \Big) \rt\|_{L^q_{t,x} L^{q}_{v, {\rm loc}}} + \lt\| \Big(\sigma[h] - \sigma[h_\ve] \Big) \nabla_v h_\ve \rt\|_{L^q_{t,x} L^{q}_{v, {\rm loc}}} \\
        &\le \|\sigma[h]\|_{L^\infty_{t,x}} \|\nabla_v h - \nabla_v h_\ve\|_{L^q_{t,x} L^{q}_{v, {\rm loc}}} + \|\sigma[h_\ve] - \sigma[h]\|_{{L^{\frac{2q}{2-q}}_{t,x} } L^{\frac{2q}{2-q}}_{v, {\rm loc}} } \|\nabla_v h_\ve\|_{L^2_{t,x,v}},
    \end{align*}
    and the right-hand side tends to zero due to \eqref{S: E: hve itself}--\eqref{S: E: hve coefficients}. Altogether, we have established that
    \begin{equation*}
        \scrC_\ve \to 0 \quad \text{in} \quad L^2_{\rm loc}([0,T]\times \ov{\calO}) + L^q([0,T]\times\Omega; W^{-1,q}_{\rm loc}(\R^d)) \quad \forall q\in [1,2).
    \end{equation*}

    Now recalling the regularity of $h_\ve$, note that the chain rule applies, so for any $\beta\in C^2 \cap W^{2,\infty}$ we obtain that
    \begin{equation}\label{S: E: hve renorm}
    \begin{split}
        \p_t \beta(h_\ve) + v\cdot \nabla_x \beta(h_\ve) &= \nabla_v\cdot (\sigma[h_\ve] \nabla_v \beta(h_\ve) ) -  \beta''(h_\ve) \sigma[h_\ve] |\nabla_v h_\ve|^2 \\
        &\quad + \zeta[h_\ve]\cdot \nabla_v \beta(h_\ve) + \beta'(h_\ve) \Big( \eta[h_\ve] h_\ve + (P[h])_\ve + \scrC_\ve^1 \Big) \\
        &\quad + \nabla_v \cdot (\scrC_\ve^2 \beta'(h_\ve)) - \beta''(h_\ve) \nabla_v h_\ve \quad \text{in} \quad \calD'([0,T]\times \ov{\calO}). 
    \end{split}
    \end{equation}
    Since $\beta\in W^{2,\infty}$, \eqref{S: E: hve itself} implies
    \begin{equation} \label{S: E: beta hve conv}
        \beta(h_\ve) \to \beta(h), \quad \beta'(h_\ve)\to \beta'(h), \quad \beta''(h_\ve) \to \beta''(h) \quad \text{in} \quad L^p_{\rm loc}([0,T]\times \ov{\calO}) \quad \forall p\in [1,\infty).
    \end{equation}
    Now using the convergences in \eqref{S: E: hve itself}--\eqref{S: E: hve macs}--\eqref{S: E: hve coefficients}--\eqref{S: E: beta hve conv}, we can pass to the limit $\ve\searrow 0$ in \eqref{S: E: hve renorm}. For instance let us demonstrate the convergence of just a few of the terms in \eqref{S: E: hve renorm}. Writing
    \begin{equation} \label{S: E: beta two prime}
    \begin{split}
        &\beta''(h_\ve)\sigma[h_\ve] |\nabla_v h_\ve|^2 - \beta''(h) \sigma[h] |\nabla_v h| \\
        &= \beta''(h_\ve) \sigma[h_\ve] (|\nabla_v h_\ve|^2 - |\nabla_v h|^2) + |\nabla_v h|^2 (\beta''(h_\ve)\sigma[h_\ve] - \beta''(h) \sigma[h]) ,
    \end{split}
    \end{equation}
    the first term on the right-hand side of \eqref{S: E: beta two prime} converges to zero in $L^1_{t,x,v}$ because
    \begin{align*}
        &\iiint_{\calO^T} |\beta''(h_\ve) \sigma[h_\ve]| \Big| \,|\nabla_v h_\ve|^2 - |\nabla_v h|^2 \Big| \\
        &= \iiint_{\calO^T} |\beta''(h_\ve) \sigma[h_\ve]| (|\nabla_v h_\ve| + |\nabla_v h|) \Big| \, |\nabla_v h_\ve| - |\nabla_v h| \Big| \\
        &\lesssim \lt(\iiint_{\calO^T} (|\nabla_v h_\ve|^2 + |\nabla_v h|^2)\rt)^{1/2} \lt(\iiint_{\calO^T} \Big| \, |\nabla_v h_\ve | - |\nabla_v h| \Big|^2 \rt)^{1/2} \\
        &\le \lt(\iiint_{\calO^T} (|\nabla_v h_\ve|^2 + |\nabla_v h|^2)\rt)^{1/2} \lt(\iiint_{\calO^T} |\nabla_v h_\ve  - \nabla_v h |^2 \rt)^{1/2} \\
        &\le C \lt(\iiint_{\calO^T} |\nabla_v h_\ve  - \nabla_v h |^2 \rt)^{1/2} \\
        &\to 0.
    \end{align*}
    Here we used the fact that $\|\beta''(h_\ve) \sigma[h_\ve]\|_{L^\infty} \lesssim 1$. The second term on the right-hand side of \eqref{S: E: beta two prime} converges to zero in $\calD'((0,T)\times\calO)$ because $\beta''(h_\ve)\sigma[h_\ve] \to \beta''(h)\sigma[h]$ a.e. implies that $\beta''(h_\ve)\sigma[h_\ve] \weakto \beta''(h)\sigma[h]$ in $L^\infty-*$, which allows us to conclude using the fact that $|\nabla_v h|^2 \in L^1$. As another example, the term $\nabla_v\cdot (\scrC_\ve^2 \beta'(h_\ve))$ in \eqref{S: E: hve renorm} vanishes in the limit since for any test function $\chi\in \calD((0,T)\times\calO)$, we have
    \begin{align*}
        \iiint_{\calO^T} \nabla_v\cdot (\scrC_\ve^2 \beta'(h_\ve)) \chi &= - \iiint_{\calO^T} \scrC_\ve^2 \beta'(h_\ve) \nabla_v\cdot \chi \to 0
    \end{align*}
    owing to the fact that $(\beta'(h_\ve) \nabla_v\cdot \chi)$ is bounded in $L^\infty$. In conclusion, taking the limit yields that $\beta(h)$ solves, at least in the sense of $\calD'([0,T]\times \calO)$:
    \begin{equation}\label{S: E: beta h}
    \begin{split}
        \p_t \beta(h) + v\cdot \nabla_x \beta(h) &= \nabla_v\cdot (\sigma[h] \nabla_v \beta(h) ) - \beta''(h) \sigma[h] |\nabla_v h|^2 \\
        &\quad + \zeta[h]\cdot \nabla_v \beta(h) + \beta'(h) \Big( \eta[h] h + P[h] \Big)  - \beta''(h) \nabla_v h .
    \end{split}
    \end{equation}

    In order to recover the trace of $\beta(h)$, we briefly go back to the weak formulation \eqref{S: E: hve weak} satisfied by $h_\ve$. Our goal is to first prove the strong convergence $\gamma h_\ve \to \gamma h$. Consider $\ve_1,\ve_2>0$ and the equation satisfied by $(h_{\ve_2} - h_{\ve_1})^2$:
    \begin{equation} \label{S: E: diff squared}
    \begin{split}
        &(\p_t + v\cdot \nabla_x) (h_{\ve_2} - h_{\ve_1})^2 \\
        &= \sigma[h_{\ve_1}] \Delta_v (h_{\ve_2} - h_{\ve_1})^2 - 2\sigma[h_{\ve_1}] |\nabla_v (h_{\ve_2} - h_{\ve_1})|^2 \\
        &\quad + 2 (h_{\ve_2} - h_{\ve_1}) (\sigma[h_{\ve_2}] - \sigma[h_{\ve_1}]) \Delta_v h_{\ve_2} \\
        &\quad + \zeta[h_{\ve_1}]\cdot \nabla_v (h_{\ve_2} - h_{\ve_1})^2 + (\zeta[h_{\ve_2}] - \zeta[h_{\ve_1}]) \cdot \nabla_v h_{\ve_2} \\
        &\quad + 2 (h_{\ve_2} - h_{\ve_1}) \Big( (P[h])_{\ve_2} - (P[h])_{\ve_1} + \scrC_{\ve_2}^1 - \scrC_{\ve_1}^1 + \nabla_v\cdot (\scrC_{\ve_2}^2 - \scrC_{\ve_1}^2) \Big).
    \end{split}
    \end{equation}
    Also, let $\varphi_R(x,v) = (n_x\cdot v)\chi_R$, where $\chi_R\in C^\infty_c(B_{R+1})$ is the cutoff as introduced near \eqref{P: E: Trace Egy}. Observing that 
    \begin{align*}
        &\gamma (h_{\ve_2} - h_{\ve_1})^2 = (\gamma (h_{\ve_2} - h_{\ve_1}))^2 = (\gamma h_{\ve_2} - \gamma h_{\ve_1})^2, 
    \end{align*}
    we can use the Green's formula with $\varphi_R$ as a test function in \eqref{S: E: diff squared}, to write an estimate for
    \begin{equation}\label{S: E: elem Trace}
    \begin{split}
        &\iiint_{(0,T)\times\Omega\times B_R} |\gamma h_{\ve_2} - \gamma h_{\ve_1}|^2 (n_x\cdot v)^2 \,\tnd\sigma_x \dv \dt,
    \end{split}
    \end{equation}
    in almost the same way as we did in \eqref{P: E: Trace Egy}. Then, utilizing the convergences \eqref{S: E: hve itself}--\eqref{S: E: hve macs}--\eqref{S: E: hve coefficients}, noting that $\scrC_\ve^1, \scrC_\ve^2 \in L^2_{t,x}L^2_{v, {\rm loc}}$, and keeping in mind that $\sigma[h],\zeta[h],\eta[h],P[h]$ are members of $L^\infty_{\rm loc}$, we can use the obtained estimate to prove as in \eqref{P: E: TraceCauchy} that \eqref{S: E: elem Trace} vanishes as $\ve_1,\ve_2\to 0$. Therefore $(\gamma h_\ve)$ is Cauchy in $L^2((0,T)\times\p\Omega\times B_R,(n_x\cdot v)^2\tnd\sigma_x\dv\dt)$, and there is some limit $\gamma'$ for which up to a diagonal extraction $\gamma h_\ve \to \gamma'$ a.e. and strongly in $L^2_{\rm loc}(\Sigma^T,(n_x\cdot v)^2\tnd\sigma_x\dv\dt)$. Using this convergence, we can then pass to the limit in the Green's formula satisfied by $(h_\ve, \gamma h_\ve)$ in \eqref{S: E: hve weak}, to deduce by uniqueness of the trace that $\gamma' = \gamma h$.
    
    Now coming back to the renormalized equation for $h_\ve$, we write the Green's formula for \eqref{S: E: hve renorm} satisfied by $(\beta(h_\ve), \beta(\gamma h_\ve))$. Since $\beta\in W^{2,\infty}$, the strong convergence $\gamma h_\ve \to \gamma' = \gamma h$ implies $\beta(\gamma h_\ve) \to \beta(\gamma h)$ a.e. Hence we can pass to the limit and find that $(\beta(h), \beta(\gamma h))$ satisfies the Green's formula for \eqref{S: E: beta h}, with $\gamma \beta(h) = \beta(\gamma h)$. The proof is now complete.
\end{proof}

\subsection{Proof of Theorem \ref{Thm1}}

Writing $f = g + h$, thanks to the regularity assumption on $g$ and recalling the embedding $L^\infty_{\omega}\hookrightarrow L^2$, we immediately obtain that $f\in L^\infty_{\omega} \cap L^\infty_t L^2_{x,v} \cap L^2_{t,x}H^1_v$ and $\sup_{t\in [0,T]}\|f_t - g_t\|_{L^\infty_{\omega}(\calO)} \le a$, along with the fact that $f$ is a solution to \eqref{eq: main} in $\calD'([0,T]\times\ov{\calO})$. In particular, using that $f\in L^2_{t,x}H^1_v$ we may repeat Step 3, Proof of Proposition \ref{S: P: final} (with no changes in the proof) to deduce that $f$ is a renormalized solution to \eqref{eq: main} satisfying the Green's formula \eqref{P: E: Renorm Green Formula}. Finally, if it is the case that $f_0$ satisfies the conservation laws in (C), we may use the fact that $L^\infty_{\omega}(\calO)\hookrightarrow L^1_{\la v \ra^2}(\calO)$ to justify the admission of $(1,R(x)\cdot v,|v|^2)$ as test functions in the Green's formula for $f$. Thanks to the pointwise decay of $f$ (remember that $f\in L^\infty_{\omega}$) any applications of the Gauss--Ostrogradsky theorem are easily justified and after some basic computations, we deduce that \eqref{Eq/ Mass Egy}--\eqref{Eq/ Mom} hold for every $R\in \calR_{\Omega}$, which imply that $f$ satisfies (C).


\section{On the close-to-Maxwellian regime}

\subsection{Setup of the problem} \label{Sec/ Max Setup}

In this section, we prove Theorem \ref{Thm2}. We consider the close to equilibrium regime and set 
\begin{equation*}
    f = \mu + h.
\end{equation*}
Thus $h$ again corresponds to the perturbation as in Section \ref{Sec: Inhom: setup}, but for the particular choice of $g = \mu$. By formally computing the Gateaux derivative of the nonlinear Fokker--Planck operator $\scrF$ at $\mu$ (see Appendix \ref{App: A: Formal}), we find that the following linear operator acts as its first order approximation
\begin{equation*}
\sfC_{\rm lin} h := \Delta_v h + \nabla_v\cdot (vh) + \sum_{i=1}^d \pi_i h + 2\pi_{d+1} h,
\end{equation*}
where we use the following notations for the macroscopic projections in the space $L^2_{\mu^{-1/2}}(\R^d)$:
\begin{equation*}
\pi_i h := \begin{cases}
  \displaystyle \lt(\int_{\R^d} h(w)\dw\rt) \mu(v) = \rho[h] \mu,  & i=0, \\
  \displaystyle \lt(\int_{\R^d} w_i h(w)\dw \rt) v_i \, \mu(v) = j[h]_i v_i \mu, & i=1,2,\ldots,d, \\
  \begin{aligned}[c]
  \displaystyle \lt(\int_{\R^d} \frac{|w|^2-d}{\sqrt{2d}} h(w)\,\dw \rt) & \frac{|v|^2-d}{\sqrt{2d}} \mu(v)  \\
  & \hspace{-3cm} = \frac{1}{2d}(E[h] - d\rho[h]) (|v|^2-d)\mu, 
  \end{aligned}  & i=d+1.
\end{cases}
\end{equation*}
Now writing the equation formally satisfied by $h := f - \mu$, then subtracting the linear part $-v\cdot \nabla_x h + \scrL h$ from the equation, we find that the nonlinear remainder is given by
\begin{equation}\label{Eq/ Upsilon Def}
\begin{split}
    \Upsilon(h) &:= \scrF[f] - \sfC_{\rm lin} h \\
    &= \frac{1}{d}\lt(E[h] - \frac{|j[h]|^2}{(1+\rho[h])} \rt) \Delta_v h + (v\rho[h] - j[h])\cdot \nabla_v h \\
    &\quad + d\rho[h] h - \frac{|j[h]|^2}{d(1+\rho[h])} (|v|^2-d)\mu. 
\end{split}
\end{equation}

Towards the end of working in a setting which is compatible with the conservation laws of \eqref{eq: main} (see (C)), we now make the following observations. First, we can easily check that the same conservation properties of $\scrF$ hold true for the linearized operator $\sfC_{\rm lin}$. That is, for every $h:\R^d\to \R$ sufficiently integrable:
\begin{equation}\label{M/Eq/ Lin Con}
\int_{\R^d} (1,v,|v|^2) \sfC_{\rm lin} h \,\dv = 0.
\end{equation}
It follows as a consequence of this and \eqref{Eq/ F Conv} that
\[
\int_{\R^d} (1,v,|v|^2) \Upsilon(h) \,\dv = 0,
\]
so that $\Upsilon(h) = \leftperp{\Upsilon}(h)$, where $\leftperp{\Upsilon}(h) := (I - \pi) \Upsilon(h)$ and $\pi := \sum_{i=0}^{d+1} \pi_i$. In particular, we are tempted to write as in \cite{CarMis24}:
\begin{equation*}
    \Upsilon(h) = \leftperp{\Upsilon}[h] h,
\end{equation*}
by defining for each given function $k=k(t,x,v)$, the linear operator $\Upsilon[k]$ as
\begin{equation*} 
    \Upsilon[k] h := \frac{1}{d}\lt(E[k] - \frac{|j[k]|^2}{1+\rho[k]} \rt) \Delta_v h + (v\rho[k] - j[k])\cdot \nabla_v h + d\rho[k] h - \frac{1}{d} \frac{ |j[k]|^2 }{1+\rho[k]} (|v|^2 - d)\mu.
\end{equation*}

However, as we previously observed in Section \ref{Sec: Inhom: setup}, the quadratic diffusion coefficient is ill-behaved with respect to weak$-*$ convergence. Therefore, we are motivated to use again the trick of employing convolution operations and define
\begin{equation*}
    \Upsilon_{\va}[k] h := \frac{1}{d}\lt(E[k] - \bfA_\varsigma[k] \rt) \Delta_v h + (v\rho[k] - j[k])\cdot \nabla_v h + d\rho[k] h - \frac{1}{d} \bfA_\varsigma[k] (|v|^2-d) \mu,
\end{equation*}
where we used the same notation as in Section \ref{Sec: Inhom: setup}:
\begin{equation*}
    \bfA_\varsigma[k] := \frac{| j[k] *_t \bfs_\va \star_\varsigma \bfr_\varsigma |^2}{1 + \rho[k] *_t \bfs_\va \star_\varsigma \bfr_\varsigma }, \quad \varsigma>0.
\end{equation*}
Then, observing that
\begin{align*}
    &\rho[\Upsilon_{\va}[k] h] = 0,\\
    &j[\Upsilon_{\va}[k] h] = \rho[h]j[k] - \rho[k]j[h], \\
    &E[\Upsilon_{\va}[k] h] = 2 \lt(E[k] - \bfA_\varsigma[k] \rt) \rho[h] - 2 \rho[k] E[h] + 2j[k]\cdot j[h] - 2 \bfA_\varsigma[k],
\end{align*}
we obtain
\begin{equation}\label{eq: nonlinear perp}
    \begin{split}
    \leftperp{\Upsilon_{\va}[k]} h &:= (I - \pi) \Upsilon_{\va}[k] h \\
    &= \Upsilon_\va[k] h - j[\Upsilon_\va[k] h] \cdot v\mu - \frac{1}{2d} E[\Upsilon_\va[k] f] (|v|^2 - d) \mu \\
    &= \frac{1}{d}\lt( E[k] - \bfA_\varsigma[k] \rt) \Delta_v h + (v\rho[k] - j[k]) \cdot \nabla_v h + d\rho[k] h \\
    &\quad - (\rho[h] j[k] - \rho[k] j[h]) \cdot v\mu \\
    &\quad - \frac{1}{d} \Big( (E[k] - \bfA_\varsigma[k]) \rho[h] - \rho[k] E[h] + j[k]\cdot j[h] \Big) (|v|^2 - d)\mu.
\end{split}
\end{equation}

At last, we denote
\[ \sfF_{\varsigma}[k] := - v\cdot \nabla_x + \sfC_{\rm lin} + \leftperp{\Upsilon_{\va}[k]}, \]
whose explicit form is given by
\begin{equation*}
\begin{split}
    &\sfF_\va[k] h = -v\cdot \nabla_x h + \lt(1 + \frac{1}{d}E[k] - \frac{1}{d}\bfA_\varsigma[k] \rt) \Delta_v h \\
    &\qquad \qquad + \big( v(1+\rho[k])-j[k]\big) \cdot \nabla_v h + d(1+\rho[k]) h \\
    &\qquad \qquad + \Big( (1 + \rho[k]) j[h] - j[k] \rho[h] \Big)\cdot v\mu \\
    &\qquad\qquad + \frac{1}{d} \lt( - \Big( d + E[k] - \bfA_\varsigma[k] \Big) \rho[h] - j[k] \cdot j[h] + (1 + \rho[k]) E[h]  \rt) (|v|^2 - d) \mu ,
\end{split}
\end{equation*}
and our goal is to investigate the properties of the now linear evolution equation
\begin{equation}\label{eq: f g}
    \begin{cases}
        \p_t h = \sfF_\va[k] h &\text{in }\R_+\times \calO, \\
        h|_{t=0} = h_0 &\text{on }\calO,
    \end{cases}
\end{equation}
equipped with the specular reflection boundary condition \eqref{eq: specular}. Once all of the desired properties have been deduced from \eqref{eq: f g}, we will commence a fixed point argument to transition all of those properties to the nonlinear problem
\begin{equation}\label{Eq/ Nonlin but mollified}
    \begin{cases}
        \p_t h = \sfF_\va[h] h = -v\cdot \nabla_x h + \sfC_{\rm lin} h + \leftperp{\Upsilon_\va}[h] h &\text{in }\R_+\times \calO, \\
            h |_{t=0} = h_0 &\text{on }\calO, \\
            \gamma_- h(t,x,v) = \gamma_+ h(t,x,\scrV_x) &\text{on }\R_+\times \Sigma_-,
    \end{cases}
\end{equation}
see Proposition \ref{Prop/ reg}. Finally, in Proposition \ref{Prop: noreg}, we pass to the limit $\va\to 0$ to obtain a weak solution to
\begin{equation*}
    \begin{cases}
        \p_t h = -v\cdot \nabla_x h + \sfC_{\rm lin} h + \Upsilon(h) &\text{in }\R_+\times \calO, \\
            h |_{t=0} = h_0 &\text{on }\calO, \\
            \gamma_- h(t,x,v) = \gamma_+ h(t,x,\scrV_x) &\text{on }\R_+\times \Sigma_-,
    \end{cases}
\end{equation*}

We close this section by mentioning that from hereon, we keep a weight $\omega\in \calW_{\mu^{-1/2}}$ fixed. Furthermore, we will impose the assumption that the given function $k=k(t,x,v)$ is extracted from $\calX_{\omega,c}$, which is defined for each $c>0$ as
\begin{equation*}
    \calX_{\omega,c} := \lt\{k:(0,\infty)\times \calO \to \R \;\middle| \; \|k\|_{L^\infty_\omega((0,\infty)\times \calO)} \le c \rt\}.
\end{equation*}
In order to simplify the notation, we will also adopt the nomenclature
\begin{align*}
    &\sigma_\varsigma[k] := \lt(1 + \frac{1}{d}E[k] - \frac{1}{d} \bfA_\varsigma[k] \rt), \\
    &\zeta[k] := v(1+\rho[k]) - j[k] , \quad 
    \eta[k] := d(1+\rho[k]).
\end{align*}
In this way, we can write the collisional part of the operator $\sfF_\va[k]$ as
\begin{equation*}
    \sfL_\varsigma[k] h := \sigma_\varsigma[k] \Delta_v h + \zeta[k] \cdot \nabla_v h + \eta[k] h.
\end{equation*}
Then, setting $\sfP_\varsigma[k]$ as the projection part of $\sfF_\va[k]$:
\begin{equation}\label{Eq/ P_va}
\begin{split}
    \sfP_\varsigma[k]  h &:= \Big( (1+\rho[k])j[h] - j[k] \rho[h] \Big) \cdot v\mu \\
    &\quad + \frac{1}{d} \lt( - \Big( d + E[k] - \bfA_\varsigma[k] \Big) \rho[h] - j[k] \cdot j[h] + (1 + \rho[k]) E[h]  \rt) (|v|^2 - d) \mu ,
\end{split}
\end{equation}
we have that
\begin{equation*}
    \sfF_\va[k]  = - v\cdot \nabla_x + \sfL_{\va}[k] + \sfP_\va[k] .
\end{equation*}

\subsection{Sketch of the proof of Theorem \ref{Thm2}} \label{Sec/ Sketch2}

Motivated by the preceding works \cite{CarMis24KFP,GuaMisMou17,MisMou16}, we also factorize $\sfF_\va[k]$ in a different direction, by writing
\begin{equation*}
    \sfF_\va[k] = \sfG_\va[k] + \sfH_\va[k],
\end{equation*}
where
\begin{align*}
    \sfG_\va[k] &:= \sfP_\va[k] + M \psi_R(v), \\
    \sfH_\va[k] &:=  - v\cdot \nabla_x + \sfL_\va[k] - M \psi_R.
\end{align*}
Here, $\psi_R = \psi_R(v)$ is a smooth radial cutoff function satisfying
\begin{equation*}
    \mathbf{1}_{|v|\le R} \le \psi_R \le \mathbf{1}_{|v|\le 2R},
\end{equation*}
and $M,R>1$ are large constants to be determined later. Below is a sketch of proof of Theorem \ref{Thm2}. As aforementioned, most of the work is done on the linear problem associated to $\sfF_\va[k]$, and we pass the results to the original model later.

The ingredients of Section \ref{Sec/ WP sfH}:
\begin{enumerate}
    \item In Lemma \ref{M: L: B L2 framework WP}, we prove that the evolution equation corresponding to $\sfH_\va[k]$ is well-posed in $L^p_{\omega_1}$ whenever $\omega_1\in \calW_{\mu^{-1/2}}$ and $p\in [2,\infty]$, by manipulating the class of polynomial weights in $\wt{\calW}$ (we use the same notation as in Section \ref{Sec: Inhom: setup} here).
    \item In Proposition \ref{M:Prop: h diss}, we prove that whenever the factorization $\sfF_\va[k] = \sfG_\va[k] + \sfH_\va[k]$ is done suitably (that is, for all appropriately large choices of $M,R$), the flow corresponding to $\sfH_\va[k]$ is dissipative in $L^p_{\omega_1}$. This is done by working on the backwards equation, see Lemma \ref{Max: Lem: B Back L2 WP}, and resorting to duality results.
\end{enumerate}

The ingredients of Section \ref{S/ ultra}: the goal of this section is summarized in Proposition \ref{M/P/ Ultra}. Informally, this proposition states that for any suitable exponential weight $\bfw \in \calW_{\mu^{-1/2}}$, the flow corresponding to $\sfH_\va[k]$ satisfies an $L^2_{\bfw}\to L^\infty_{\omega_\star}$ ultracontractivity estimate, for every suitable choice of weight $\omega_\star\in \calW$. The proof of Proposition \ref{M/P/ Ultra} goes as follows:
\begin{enumerate}
    \item As in Section \ref{Sec/ WP sfH}, we work on the backwards dual equation corresponding to $\sfH_\va[k]^*$, and therefore we use a dual weight $\frakm$ such that $\frakm^{-1} = \bfw \in \calW_{\mu^{-1/2}}$. By incorporating an appropriate multiplier in the spirit of \cite{CarMis24KFP,LionsPer92,Mi10}, we adapt the estimates of \cite{CarMis24KFP} to prove a boundary penalizing estimate for the backwards dual equation in Lemma \ref{M: Lem: penal}.
    \item Next, in Lemma \ref{M: Lem: local}, we prove a local ultracontractivity estimate for the backwards equation by taking advantage of the Green's function known in the whole space.
    \item Combining these two results, we end up with an $L^2\to L^r$ estimate (with $r>2$) in Lemma \ref{M: Lem: L2 Lr} that holds \textit{up to the boundary}.
    \item At last, by using some classical estimates in the spirit of Nash, we obtain the $L^1\to L^2$ estimate for the backwards equation in Lemma \ref{M:Lem: pure L1 L2}. By duality we recover the $L^2\to L^\infty$ estimate promised in Proposition \ref{M/P/ Ultra}.
\end{enumerate}

The ingredients of Sections \ref{S/ F WP}--\ref{S/ F Hypo}: From hereon, we end our discussion of $\sfH_\va[k]$, and aim to establish properties of the full operator $\sfF_\va[k]$.

\begin{enumerate}
    \item In Proposition \ref{P/ F WP}, we prove that the evolution equation corresponding to $\sfF_\va[k]$ is well-posed in the reference space $L^2_{\omega_1}$, whenever $\omega_1\in \calW_{\mu^{-1/2}}$ and $k$ is chosen suitably small corresponding to $\omega_1$.
    \item In Proposition \ref{prop: L_g hypo} we prove that the flow corresponding to $\sfF_\va[k]$ dissipates exponentially to the global Maxwellian, in the reference space $L^2_{\mu^{-1/2}}$. This is done by adapting the $L^2$-hypocoercivity framework of \cite{BCMT23}, and by revisiting computations on the \textit{linear} Fokker--Planck operator, making sure that the estimates are refined enough so that all unwanted terms arising from the nonlinearity can be absorbed. 
\end{enumerate}

The ingredients of Section \ref{Sec/ Ext}:
\begin{enumerate}
\item Using an extension trick, we prove in the spirit of \cite{CarMis24} that the evolution system corresponding to $\sfF_\va[k]$ is hypodissipative in $L^\infty_{\omega}$. This result is summarized in Proposition \ref{Prop/ decay}.
\item The proof of Proposition \ref{Prop/ decay} is based on an iterated Duhamel formula as in \cite{GuaMisMou17,MisMou16} which is reminiscent of the Dyson series. The results of Proposition \ref{Prop/ decay} make crucial use of all of the previous results in Sections \ref{Sec/ WP sfH}--\ref{S/ ultra}--\ref{S/ F WP}--\ref{S/ F Hypo}. We take care in ensuring that all of the estimates can be made independent of the regularization parameter $\va$.
\end{enumerate}

The ingredients of Sections \ref{Sec/ Back to Nonlin}--\ref{Sec/ Thm2}: 

\begin{enumerate}
    \item In Proposition \ref{Prop/ reg}, we perform a fixed point argument $k\mapsto \calU_{\sfF_\va[k]}(t,0)h_0$, where $\calU_{\sfF_\va[k]}$ is the evolution system corresponding to $\sfF_\va[k]$, in order to prove the existence of a solution to the nonlinear (but regularized) equation \eqref{Eq/ Nonlin but mollified}. 
    \item At last, in Proposition \ref{Prop: noreg}, we pass to the limit $\va\to 0$ to prove the existence of a weak solution $h$ to the full equation $\partial_t h +v\cdot \nabla_x h = \sfC_{\rm lin}h + \Upsilon(h)$, which satisfies the decay as claimed in Theorem \ref{Thm2}. 
    \item The proof of Theorem \ref{Thm2} then follows as a consequence of all we have established so far, and we sketch it in Section \ref{Sec/ Thm2}.
\end{enumerate}

\subsection{Well-posedness and dissipativity of $\sfH_\va[k]$} \label{Sec/ WP sfH}

In this section, we investigate the well-posedness of
\begin{equation}\label{M: E: B}
    \begin{cases}
        \p_t h = \sfH_\va[k] h &\text{in }\R_+\times \calO, \\
        \gamma_-h = \gamma_+ h \circ \scrV_x &\text{on }\R_+\times \Sigma_-, \\
        h|_{t=0} = h_0 &\text{on }\calO,
    \end{cases}
\end{equation}
and establish its dissipative properties for suitably large $M,R$.

\begin{lemma}\label{M: L: B L2 framework WP}
    There exists $c_1>0$ small enough (independent of $\va$) such that for any $c\le c_1$, $k\in \calX_{\omega,c}$, and $M,R \ge 1$, the following holds true: for any $p\in [2,\infty]$, $\omega_1\in \calW_{\mu^{-1/2}}$, and $h_0\in L^p_{\omega_1}(\calO)$, the linear problem \eqref{M: E: B} associated to $\sfH_\va[k]$ admits a unique solution $h$ which lies in $C([0,T];L^2_{\wt{\omega}}(\calO))$ for all weights $\wt{\omega}\in \wt{\calW}$ such that $\wt{\omega}\omega_1^{-1} \in L^2\cap L^\infty(\R^d)$.
\end{lemma}

\begin{proof}
    The proof follows the method of Lemma \ref{l: sp: for lin WP}. We fix $\wt{\omega} \in \wt{\calW}$ growing slowly enough so that
    \begin{equation*}
        \wt{\omega}\omega_1^{-1} \in L^2 \cap L^\infty(\R^d).
    \end{equation*}
    In particular $\wt{\omega}\omega_1^{-1} \in L^q(\R^d)$ for every $q\in [2,\infty]$, and this implies that the embedding $L^p_{\omega_1} \hookrightarrow L^2_{\wt{\omega}}$ is continuous. Then provided that $c_1$ is chosen small enough so that $c\le c_1$ and $k\in \calX_{\omega,c}$ imply at least
    \begin{equation}\label{M: E: Coeff}
     (|\rho[k]| + |j[k]| + |E[k]|) \le \frac{1}{10}, \quad \frac{9}{10} \le \sigma_\va[k] \le \frac{11}{10},
    \end{equation}
    we may apply Proposition \ref{prop: WP}, to deduce that for given $h_0\in L^p_{\omega_1}(\calO)$ there is a unique corresponding solution $h \in C([0,T];L^2_{\wt{\omega}}(\calO))$ (notice that (L2) is automatically satisfied because the weight $\wt{\omega}$ is a polynomial; see for instance Lemma \ref{Lem/ Asymptote} below). 
\end{proof}

Note that the above lemma only provides the well-posedness of \eqref{M: E: B} in a weighted $L^2$-framework; it does not tell us whether the solution to \eqref{M: E: B} actually lies in $C([0,T];L^p_{\omega_1}(\calO))$. For this reason, we now commence an argument reminiscent of those in Lemmas \ref{sp: l: back lin WP L2}--\ref{s: l: b L1} and Proposition \ref{S: P: infty}. Namely, we consider the backwards dual problem
\begin{equation}\label{M: E: B Back}
    \begin{cases}
        -\p_t b = \sfH_\va[k]^* b &\text{in }\R_+\times \calO, \\
        \gamma_+ h = \gamma_- h \circ \scrV_x &\text{on }\R_+\times \Sigma_+, \\
        b|_{t=T} = b_T &\text{on }\calO,
    \end{cases}
\end{equation}
where $\sfH_\va[k]^*$ is the formal adjoint of $\sfH_\va[k]$:

\begin{align*}
    \sfH_\va[k]^* b &:= v\cdot \nabla_x b + \sigma_{\va}[k] \Delta_v b - \nabla_v\cdot (\zeta[k] b) + \eta[k] b - M \psi_R b.
\end{align*}

Prior to stating the lemma which provides the dissipative structure of \eqref{M: E: B Back}, let us briefly revisit the computations near \eqref{S: E: simp gamma} and establish the following asymptotic behavior of $\Gamma_{\sfL_\va[k],\omega_1,p}$ for $\omega_1\in \calW_{\mu^{-1/2}}$, in order to avoid repeating many arguments.

\begin{lemma}\label{Lem/ Asymptote}
    Let $\omega_1 = \la v \ra^m e^{\gamma|v|^\tau} \in \calW_{\mu^{-1/2}}$. Then
    \begin{enumerate}
        \item If $\gamma=0$ so that $\omega_1 = \la v \ra^m$, then $c\le c_1$ and $k\in \calX_{\omega,c}$ implies
        \begin{equation*}
            \sup_{\R_+\times \Omega} \sup_{p\in [1,\infty]} \limsup_{|v|\to\infty}  \Gamma_{\sfL_\va[k],\omega_1,p} \le -1 .
        \end{equation*}
        \item If $\tau\in (0,2)$, $m\ge 0$, and $\gamma>0$, then $c\le c_1$ and $k\in \calX_{\omega,c}$ implies
        \begin{equation*}
            \sup_{\R_+\times \Omega} \sup_{p\in [1,\infty]} \limsup_{|v|\to\infty}  \Gamma_{\sfL_\va[k],\omega_1,p} = -\infty.
        \end{equation*}
        \item If $\tau=2$, $m\ge 0$, and $\gamma\in \lt(0,\frac{1}{2}\rt)$, then there exists $c_{\omega_1}>0$ (independent of $\va$) such that whenever $c\le \min\{c_1,c_{\omega_1}\}$ and $k\in \calX_{\omega,c}$, 
        \begin{equation*}
            \sup_{\R_+\times \Omega} \sup_{p\in [1,\infty]} \limsup_{|v|\to\infty}  \Gamma_{\sfL_\va[k],\omega_1,p} = -\infty.
        \end{equation*}
    \end{enumerate}
    For the sake of convenience, we define the convention $c_{\omega_1} = +\infty$ if either $\gamma=0$ or $\tau\in (0,2)$.
\end{lemma}

\begin{proof}
    We merely recall the formulae in \eqref{S: E: simp gamma} and compute directly.
    
    (1) If $\gamma=0$, then by definition of $\calW_{\mu^{-1/2}}$ we have $m>d+2$, and thus for all $p\in [1,\infty]$:
    \begin{align*}
        \Gamma_{\sfH_\va[k],\omega_1,p} &\sim - \zeta[k] \cdot \nabla_v W_1 - \frac{1}{p} \nabla_v\cdot \zeta[k] + \eta[k] \\
        &\sim -m (1+\rho[k]) + \lt(1 - \frac{1}{q'}\rt) d(1+\rho[k]) \\
        &\le -(m-d) (1+\rho[k]) \\
        &\le -1 \quad \text{as }|v|\to\infty,
    \end{align*}
    because $|\rho[k]|\le 1/2$.
    
    (2) If $\gamma>0$ and $\tau\in (0,2)$, then $-\zeta[k]\cdot \nabla_v W_1$ is the only dominant term, and we can easily check that $\Gamma_{\sfH_\va[k],\omega_1,p}\to -\infty$ unconditionally.
    
    (3) If $\tau=2$ and $\gamma\in\lt(0,\frac{1}{2}\rt)$, we argue as in the proof of Lemma \ref{s: l: b L1}. We have that
    \begin{equation*}
        \Gamma_{\sfH_\va[k],\omega_1,q'} \sim 2\gamma \Big(2\gamma \sigma_\va[k] - (1+\rho[k]) \Big) |v|^2,
    \end{equation*}
    so that $\Gamma_{\sfH_\va[k],\omega_1,q'}\to -\infty$ if and only if 
    \begin{equation*}
        \frac{\sigma_\va[k]}{1+\rho[k]} < \frac{1}{2\gamma}.
    \end{equation*}
    Since $\gamma\in (0,1/2)$ we can write
    \begin{equation*}
        \frac{1}{2\gamma} = 1 + \kappa_{\omega_1}
    \end{equation*}
    for some $\kappa_{\omega_1}>0$. Next, we can prove that if $c_{\omega_1}$ is sufficiently small, then $c\le c_{\omega_1}$ implies
    \begin{equation*}
        \frac{\sigma_\va[k]}{1+\rho[k]} \le 1 + \frac{\kappa_{\omega_1}}{2} < \frac{1}{2\gamma},
    \end{equation*}
    because $\sigma_\va[k] \sim 1$ and $1+\rho[k] \sim 1$, independently of $\va$, if $k$ is small in $\|\cdot\|_{L^\infty_{\omega}}$. 
\end{proof}

\begin{lemma} \label{Max: Lem: B Back L2 WP}
    Let $M,R>1$, and suppose that $\omega_1 \in \calW_{\mu^{-1/2}}$ is given. Also we let $k\in \calX_{\omega,c}$ with $c\le \min\{c_1,c_{\omega_1}\}$, where $c_{\omega_1}>0$ is as in Lemma \ref{Lem/ Asymptote}. Consider any weight $\wt{\omega}\in \wt{\calW}$ satisfying $\wt{\omega} \omega_1^{-1} \in L^2\cap L^\infty(\R^d)$. Then for any $b_T\in L^2_{\wt{\omega}^{-1}}$, the dual backwards problem \eqref{M: E: B Back} associated to $\sfH_\va[k]^*$ admits a unique solution in $b\in C([0,T];L^2_{\wt{\omega}^{-1}}(\calO))$. As a result, for any $q\in [1,2]$ we also have $b\in C([0,T];L^q_{\omega_1^{-1}}(\calO))$. Furthermore, there exist $M_{\omega_1},R_{\omega_1}>1$ sufficiently large (and independent of $\va$ or the exponent $q$) such that for all
    \begin{equation*}
        M \ge M_{\omega_1}, \quad R \ge R_{\omega_1},
    \end{equation*}
    the dissipative estimate below holds for some constant $\lambda' > 0$:
    \begin{equation*}
        -\ddt \frac{1}{q}\iint_{\calO} |b \omega_1^{-1}|^q  + 4\frac{(q-1)}{q^2} \iint_{\calO} \sigma_\va[k] |\nabla_v (b \omega_1^{-1})^{q/2}|^2 \le - \lambda' \iint_{\calO} |b \omega_1^{-1}|^q  .
    \end{equation*}
    In particular
    \begin{equation}\label{M: E: B Back dissipative}
     \|b_t\|_{L^q_{\omega_1^{-1}}(\calO)} \le  e^{- \lambda' (T-t)} \|b_T\|_{L^q_{\omega_1^{-1}}(\calO)} \quad \forall t\in [0,T], \quad \forall q\in [1,2].
    \end{equation}
\end{lemma}

\begin{proof}
    We can easily check that the embedding $L^2_{\wt{\omega}^{-1}}\hookrightarrow L^q_{\omega_1^{-1}}$ is continuous. Thus, the well-posedness of \eqref{M: E: B Back} in $C([0,T];L^q_{\omega_1^{-1}})$ is an immediate consequence of Proposition \ref{prop: WP}.
    
    It now suffices to prove the estimate \eqref{M: E: B Back dissipative}, under the assumption that $c_{\omega_1}>0$ is small enough and $M,R$ are large enough. One has (for instance by using the renormalized Green's formula)
    \begin{equation} \label{M: E: B Back Lq}
    \begin{split}
        &- \ddt  \frac{1}{q} \iint_{\calO} |b_t|^q \omega_1^{-q} \\
        &= \iint_{\calO} |b_t|^{q-2} b_t (\sfH_\varsigma[k]^*  b_t) \omega_1^{-q} \\
        &= \iint_{\calO} |b_t|^{q-2} b_t \Big[ (\sfL_\varsigma[k]^* - M \psi_R) b_t \Big] \omega_1^{-q} \\ 
        &= \iint_{\calO} \Big( \Gamma_{\sfL_{\varsigma}[k]^*, \omega_1^{-1}, q} - M \psi_R \Big) |b_t|^q \omega_1^{-q} - \frac{4(q-1)}{q^2} \iint_{\calO} \sigma_\va[k] |\nabla_v (b_t \omega_1^{-1})^{q/2}|^2 ,
    \end{split}
    \end{equation}
    where we used \eqref{p: e: formal} and the fact that no boundary term appears thanks to the specular reflection boundary condition. Now
    \begin{equation*}
        \Gamma_{\sfL_\va[k]^*, \omega_1^{-1}, q} = \Gamma_{\sfL_\va[k], \omega_1, q'},
    \end{equation*}
    where $1/q'+1/q = 1$. Using Lemma \ref{Lem/ Asymptote}, we choose $c_{\omega_1}>0$ small enough so that $c\le \min\{c_1,c_{\omega_1}\}$ and $k\in \calX_{\omega,c}$ implies
    \begin{equation*}
        \sup_{\R_+\times \Omega} \sup_{q'\in [2,\infty]} \limsup_{|v|\to\infty}  \Gamma_{\sfL_\va[k], \omega_1, q'} (v) \le -1/2,
    \end{equation*}
    and thus for all sufficiently large $M, R>1$, it holds that
    \begin{equation*}
        \sup_{q\in [1,2]} \Big( \Gamma_{\sfL_\va[k], \omega_1, q'}(v) - M \psi_R(v) \Big) \le - 1/4 \quad \text{in }\R_+ \times \calO.
    \end{equation*}
    Going back to \eqref{M: E: B Back Lq} and then applying Gr\"onwall's inequality, we arrive at \eqref{M: E: B Back dissipative}.
\end{proof}

Using the backwards equation, we can now extend the $L^2$ results of Lemma \ref{M: L: B L2 framework WP} to a general $L^p_\omega$-framework. We will only consider the case of large $M,R$, since that is enough for our purposes later.

\begin{proposition}\label{M:Prop: h diss}
    Let $p\in[2,\infty]$, $\omega_1\in \calW_{\mu^{-1/2}}$, and suppose $M\ge M_{\omega_1},R \ge R_{\omega_1}$. We further assume that $c\le \min\{c_1,c_{\omega_1}\}$ and $k\in \calX_{\omega,c}$ (see the notations of Lemmas \ref{Lem/ Asymptote}--\ref{Max: Lem: B Back L2 WP}). Then for given $h_0\in L^p_{\omega_1}(\calO)$, for the linear problem \eqref{M: E: B} associated to $\sfH_\va[k]$, its unique solution in $h\in C([0,T];L^2_{\wt{\omega}}(\calO))$ (see Lemma \ref{M: L: B L2 framework WP}) satisfies
    \begin{equation*}
        \|h_T\|_{L^p_{\omega_1}} \le e^{-\lambda' T} \|h_0\|_{L^p_{\omega_1}} \quad \forall T>0.
    \end{equation*}
    Here $\lambda'>0$ is given as in Lemma \ref{Max: Lem: B Back L2 WP}. Furthermore, the solution lies in $h\in C([0,T];L^p_{\omega_1}(\calO))$ if $p\in [2,\infty)$, and $C([0,T];L^\infty_{\omega_1}(\calO)-*)$ if $p=\infty$; thus the linear problem \eqref{M: E: B} is well-posed in these spaces. We may therefore associate to it an evolution system (or non-autonomous semigroup) which we denote $(\calU_{\sfH_\va[k]}(t,s))_{0\le s \le t}$. This family satisfies the following dissipative estimate:
    \begin{equation}\label{Max: eq: B diss}
        \forall h_s\in L^p_{\omega_1}, \quad \|\calU_{\sfH_\va[k]}(t,s) h_s\|_{L^p_{\omega_1}} \le e^{-\lambda' (t-s)} \|h_s\|_{L^p_{\omega_1}}.
    \end{equation}
\end{proposition}

\begin{proof}
    Since $\wt{\omega} \omega_1^{-1} \in L^p(\R^d)$, we remind the reader that $h_0 \in L^p_{\omega_1}(\calO) \subset L^2_{\wt{\omega}}(\calO)$. For any given $b_T\in L^2_{\wt{\omega}^{-1}}$, we may consider the unique solution $b\in C([0,T];L^2_{\wt{\omega}^{-1}})$ to the backwards dual problem \eqref{M: E: B Back} associated to $\sfH_\va[k]^*$ (see Lemma \ref{Max: Lem: B Back L2 WP}). Owing to Lemma \ref{p: l: dual}, we have the following duality result
    \begin{align*}
        \iint_{\calO} h_T b_T \,\dx\dv = \iint_{\calO} h_0 b(0) \,\dx\dv.
    \end{align*}
    Consequently, by using that $L^2_{\wt{\omega}^{-1}}$ is dense in $L^{p'}_{\omega^{-1}}$, we can compute the $L^p_\omega$-norm of $h_T$ as follows:
    \begin{align*}
        \|h_T\|_{L^p_\omega} &= \sup_{ \substack{ b_T \in L^2_{\wt{\omega}^{-1}}, \\ \|b_T\|_{L^{p'}_{\omega^{-1}}} \le 1 }   }  \iint_{\calO} h_T b_T \,\dx\dv \\
        &= \sup_{ \substack{ b_T \in L^2_{\wt{\omega}^{-1}}, \\ \|b_T\|_{L^{p'}_{\omega^{-1}}} \le 1 }   } \iint_{\calO} h_0 b(0) \,\dx\dv \\
        &\le e^{-\lambda' T} \|h_0\|_{L^p_\omega}.
    \end{align*}
    It follows that $h_T \in L^p_\omega$ for every $T\ge 0$. The claim that $t\mapsto h_t\in L^p_\omega$ (or $L^\infty_\omega-*$) is continuous just follows as a consequence of the duality formula; we refer to the proof of Proposition \ref{S: P: infty}. Finally, the general result \eqref{Max: eq: B diss} follows by a time translation and it is thus skipped. 
\end{proof}

\subsection{Ultracontractivity property of the dissipative operator $\sfH_\va[k]$} \label{S/ ultra}

In this section we aim to prove an ultracontractivity estimate for the solution to \eqref{M: E: B} associated to $\sfH_\va[k]$. The arguments of this section follow in large what has already been demonstrated in \cite{CarMis24}. As before, most of the analysis will be performed on the backwards problem \eqref{M: E: B Back}, and we will transition later by resorting to the duality argument in Lemma \ref{p: l: dual}. We stress that all of the estimates here hold \textit{independently of $\va$}, as those in Section \ref{Sec/ WP sfH}.

The starting step is to obtain a boundary penalizing estimate for $b$. The following lemma, inspired from \cite{CarMis24KFP}, exploits a suitable twisted weight which allows us to obtain the penalized estimate, while killing the troublesome boundary contributions at the same time. The weight utilized here takes a simpler form due to the fact that we only incorporate the specular reflection boundary condition, contrary to \cite{CarMis24KFP,CarMis24} where the Maxwell boundary condition was considered.

\begin{lemma}\label{M: Lem: penal}
    We let $\mathfrak{m}:(0,\infty)\to\R^d$ be such that $\frakm^{-1} \in \calW_{\mu^{-1/2}}$ and $\frakm^{-1}$ is not a polynomial weight (in other words $\frakm^{-1}\in \calW_{\mu^{-1/2}}$ and $\frakm^{-1} = \la v \ra^m e^{\gamma |v|^\tau}$ with $\gamma > 0$, $\tau > 0$). Let $b\in C([0,T];L^2_{\frakm^{-1}}(\calO))$ be the solution to the dual equation \eqref{M: E: B Back} associated to $\sfH_\va[k]^*$ as constructed under the setting of Lemma \ref{Max: Lem: B Back L2 WP}. Then, for $c\le \min\{c_1,c_{\frakm^{-1}}\}$ and $k\in \calX_{\omega,c}$, it holds for all $M,R$ large enough (chosen dependently on $\mathfrak{m}^{-1}$) that for any $0\le \varphi\in \calD((0,T))$ and $\alpha\in (0,1)$:
    \begin{align*}
        &\iiint_{\calO^T} b_t^2 \ov{\mathfrak{m}}^2 \frac{(n_x\cdot v)^2}{\la v \ra^2 \mathfrak{d}(x)^{1-\alpha}} \varphi^2 + \int_0^T \|b \mathfrak{m} \|_{L^2_x H^1_v}^2 \varphi^2 \\
        &\lesssim \iiint_{\calO^T} b_t^2 \ov{\mathfrak{m}}^2 \varphi (\varphi')_- .
    \end{align*}
    Here $\ov{\mathfrak{m}}:\calO\to (0,\infty)$ is some twisted weight function satisfying $\mathfrak{m} \lesssim \ov{\mathfrak{m}} \lesssim \mathfrak{m}$, to be specified in the proof, and $c_{\frakm^{-1}}$ is the constant as in the third statement of Lemma \ref{Lem/ Asymptote}, with weight $\omega_1 = \frakm^{-1}$.
\end{lemma}

\begin{proof}[Proof of Lemma \ref{M: Lem: penal}]
    For any $\chi\in \calD(\ov{\calO})$, the solution $b$ satisfies
    \begin{align*}
        - \frac{1}{2} \ddt \iint_{\calO} b_t^2 \chi \,\dx\dv &= \iint_{\calO} b_t (\sfH_\va[k]^* b_t) \chi \,\dx\dv + \frac{1}{2} \iint_{\p\calO} b_t^2 \chi (v\cdot n_x) \,\tnd\sigma_x \dv,
    \end{align*}
    or equivalently
    \begin{align*}
        - \frac{1}{2} \ddt \iint_{\calO} b_t^2 \chi \,\dx\dv &= - \frac{1}{2} \iint_{\calO} b_t^2 (v\cdot \nabla_x \chi) \,\dx\dv +  \iint_{\calO} b_t (\sfH_\va[k]^* b_t) \chi \,\dx\dv \\
        &\quad + \frac{1}{2} \iint_{\p\calO} \gamma b_t^2 \chi (v\cdot n_x) \,\tnd\sigma_x \dv.
    \end{align*}
    Therefore, for any $0\le \varphi\in \calD((0,T))$ we can do a simple integration by parts to obtain
    \begin{equation} \label{Max: E: Penal 1}
    \begin{split}
        &\frac{1}{2} \iiint_{\calO^T} b_t^2 \chi \varphi \varphi'  \\
        &= - \frac{1}{2} \iiint_{\calO^T} b_t^2 (v \cdot \nabla_x \chi) \varphi^2 + \iint_{\calO^T} b_t (\sfH_\va[k]^* b_t) \chi \varphi^2 \\
        &\quad + \frac{1}{2} \iiint_{(0,T)\times \p\calO} \gamma b_t^2 \chi \varphi^2 (v\cdot n_x).
    \end{split}
    \end{equation}

    We choose the following specific test function
    \begin{equation*}
        \chi = \ov{\mathfrak{m}}^2 := \mathfrak{m}^2 \lt(1 - \frac{\mathfrak{d}(x)^{\alpha}}{4 D^{\alpha}} \frac{n_x\cdot v}{\la v \ra^2 } \rt) ,
    \end{equation*}
    which obviously satisfies the relations $\mathfrak{m}\lesssim \ov{\mathfrak{m}} \lesssim \mathfrak{m}$. With this choice, let us explicitly compute the three terms on the right-hand side of \eqref{Max: E: Penal 1}. For the first term, we have
    \begin{align*}
        &- \frac{1}{2} \iiint_{\calO^T} b_t^2 (v \cdot \nabla_x \chi) \varphi^2 \\
        &= \iiint_{\calO^T} b_t^2 \frac{\mathfrak{m}^2}{8 D^\alpha \la v \ra^2} v\cdot \nabla_x \Big( \mathfrak{d}(x)^\alpha (n_x\cdot v) \Big)\varphi^2 \\
        &= \iiint_{\calO^T} b_t^2 \frac{\mathfrak{m}^2}{8 D^\alpha \la v \ra^2} v\cdot \Big( \alpha \mathfrak{d}(x)^{\alpha-1} \nabla_x \mathfrak{d}(x) \,  (n_x\cdot v) + \mathfrak{d}(x)^\alpha \nabla_x n_x \, v  \Big) \varphi^2 \\
        &= \iiint_{\calO^T} b_t^2 \frac{\mathfrak{m}^2}{8 D^\alpha \la v \ra^2} \Big( - \alpha \mathfrak{d}(x)^{\alpha-1} |\nabla_x \mathfrak{d}(x)| (n_x\cdot v)^2 + \mathfrak{d}(x)^\alpha v\cdot \nabla_x n_x \, v \Big) \varphi^2 ,
    \end{align*}
    where we recalled that $\nabla_x \mathfrak{d}(x) = - |\nabla_x \mathfrak{d}(x)| n_x$. For the second term, we may use \eqref{p: e: formal} to obtain
    \begin{align*}
        \iiint_{\calO^T} b_t (\sfH_\va[k]^* b_t) \ov{\mathfrak{m}}^2 \varphi^2 &= \iiint_{\calO^T} \Gamma_{\sfH_\va[k]^*, \ov{\mathfrak{m}} , 2} |b_t|^2 \ov{\mathfrak{m}}^2 \varphi^2 - \iiint_{\calO^T} \sigma_\va[k] |\nabla_v (b \ov{\mathfrak{m}})|^2 \varphi^2.
    \end{align*}
    Finally, using that $\mathfrak{d}(x) = 0$ for $x\in \p\Omega$ and also the (dual) specular reflection boundary condition, we get
    \begin{align*}
        \iiint_{(0,T)\times \p\calO} (\gamma b_t)^2 \chi \varphi^2 &= \iiint_{(0,T)\times \p\calO} (\gamma b_t)^2 \mathfrak{m}^2 \varphi^2 = 0.
    \end{align*}

    Collecting everything, \eqref{Max: E: Penal 1} is now equivalent to
    \begin{equation}\label{M: E: penal2}
    \begin{split}
        &\iiint_{\calO^T} b_t^2 \mathfrak{m}^2 \frac{\alpha}{8 D^\alpha \la v \ra^2 \mathfrak{d}(x)^{1-\alpha}} |\nabla_x \mathfrak{d}(x)| (n_x\cdot v)^2 \varphi^2  + \iiint_{\calO^T} \sigma_\va[k] |\nabla_v (b\ov{\mathfrak{m}})|^2 \varphi^2 \\
        &= - \iiint_{\calO^T} b_t^2 \ov{\mathfrak{m}}^2 \lt(1 - \frac{\mathfrak{d}^\alpha}{4D^\alpha} \frac{n_x\cdot v}{\la v \ra^2}\rt) \varphi \varphi'  \\
        &\quad + \iiint_{\calO^T} b_t^2 \mathfrak{m}^2 \frac{\mathfrak{d}(x)^\alpha v\cdot \nabla_x n_x \, v}{4D^\alpha \la v \ra^2} \varphi^2
        + \iiint_{\calO^T} \Gamma_{\sfH_\va[k]^*, \ov{\mathfrak{m}},2} |b_t|^2 \ov{\mathfrak{m}}^2 \varphi^2 .
    \end{split}
    \end{equation}
    In order to conclude, we observe that $\Gamma_{\sfH_\va[k]^*, \ov{\mathfrak{m}}, 2} = \Gamma_{\sfH_\va[k], \ov{\omega}, 2}$ with $\ov{\omega}:=\ov{\mathfrak{m}}^{-1}$. Now $\ov{\mathfrak{m}}$ and $\mathfrak{m}$ have equivalent growth, so it is enough to compute $\Gamma_{\sfH_\va[k], \frakm^{-1},2}$. Since $\frakm^{-1}$ is assumed not to be a polynomial weight, we can apply either the second or third statement of Lemma \ref{Lem/ Asymptote} to find that for $c\le \min\{c_1, c_{\frakm^{-1}}\}$, it holds for all sufficiently large $M,R$ that 
    \begin{align*}
      \Gamma_{\sfH_\va[k], \ov{\frakm}^{-1} , 2} =  \Gamma_{\sfL_\va[k], \frakm^{-1},2} - M \psi_R \le - \frac{8}{3}\|\nabla_x n_x\|_{L^\infty}.
    \end{align*}
    In particular, for such choices of $M,R$, the last line of \eqref{M: E: penal2} is strictly negative, and more specifically it satisfies
    \begin{align*}
         &\iiint_{\calO^T} b_t^2 \mathfrak{m}^2 \frac{\mathfrak{d}(x)^\alpha v\cdot \nabla_x n_x \, v}{4D^\alpha \la v \ra^2} \varphi^2
        + \iiint_{\calO^T} \Gamma_{\sfH_\va[k]^*, \ov{\mathfrak{m}},2} |b_t|^2 \ov{\mathfrak{m}}^2 \varphi^2 \\
        &\le - \|\nabla_x n_x\|_{L^\infty} \iiint_{\calO^T} |b_t|^2 \mathfrak{m}^2 \varphi^2.
    \end{align*}
    Consequently, from \eqref{M: E: penal2} we deduce
    \begin{align*}
        &\iiint_{\calO^T} b_t^2 \mathfrak{m}^2 \frac{\alpha}{\la v \ra^2 \mathfrak{d}(x)^{1-\alpha}} (n_x\cdot v)^2 \varphi^2  + \iiint_{\calO^T} \Big[ |b \mathfrak{m}|^2 + \sigma_\va[k] |\nabla_v (b\ov{\mathfrak{m}})|^2 \Big] \varphi^2 \\
        &\lesssim \iiint_{\calO^T} b_t^2 \ov{\mathfrak{m}}^2 \varphi (\varphi')_- .
    \end{align*}
    The thesis of the lemma then follows by recalling that $\mathfrak{m}\lesssim \ov{\mathfrak{m}}\lesssim \mathfrak{m}$, $\sigma_\va[k] \gtrsim 1$ (see \eqref{M: E: Coeff}), and also observing that
    \begin{align*}
        \NORM{b \mathfrak{m} }_{L^2_x H^1_v} \lesssim \NORM{b \ov{\mathfrak{m}} }_{L^2_x H^1_v} \lesssim \NORM{b \mathfrak{m}}_{L^2_x H^1_v}.
    \end{align*}
\end{proof}

We then recall the following useful interpolation estimate from \cite[Lemma 5.3]{CarMis24}.

\begin{lemma}
    For any function $f:\calO\to \R$, we have
    \begin{align*}
        \|\mathfrak{d}^{-1/4} \la v \ra^{-1} f \|_{L^2(\calO)}^2 \lesssim \iint_{\calO} f^2 \frac{(n_x\cdot v)^2}{\la v \ra^2 \mathfrak{d}^{1/2}} + \|\nabla_v f\|_{L^2(\calO)}^2.
    \end{align*}
\end{lemma}
Applying this lemma with $f = b\ov{\mathfrak{m}}$, we deduce immediately that the assumptions of Lemma \ref{M: Lem: penal} imply (with the choice of $\alpha=1/2$)
\begin{equation}\label{M: E: clr penal}
    \lt\| b \ov{\mathfrak{m}} \mathfrak{d}^{-1/4} \la v \ra^{-1} \varphi \rt\|_{L^2(\calO^T)}^2 \lesssim \lt\| b \ov{\mathfrak{m}} \sqrt{\varphi (\varphi')_-} \rt\|_{L^2(\calO^T)}^2 .
\end{equation}

In order to make use of \eqref{M: E: clr penal}, we mention a rather standard gain of integrability estimate which is local in nature.

\begin{lemma}\label{M: Lem: local}
    We adopt the setting and all of the assumptions of Lemma \ref{M: Lem: penal}. Let $b$ be the solution to the dual equation \eqref{M: E: B Back} associated to $\sfH_\va[k]^*$ as constructed under the setting of Lemma \ref{Max: Lem: B Back L2 WP}. Then it holds for each $p\in \lt(2, \frac{2d+1}{d}\rt)$ and $p_1 \in \lt(p, \frac{2d+1}{d}\rt)$ that
    \begin{align*}
        &\lt\| b \frakm \mathfrak{d}^{p_1/p} \la v \ra^{-2} \varphi  \rt\|_{L^p(\calO^T)} \lesssim C_T \lt( \lt\| b \frakm \sqrt{\varphi(\varphi')_-} \rt\|_{L^2(\calO^T)} + \|b \frakm \varphi'\|_{L^2(\calO^T)} \rt) , \\
        &C_T := \lt( T^{\frac{(2d+1)(p+2)}{2p} - 2d} + T^{\frac{(2d+1)(p+2)}{2p} - 2d - \frac{1}{2}} \rt).
    \end{align*}
\end{lemma}

\begin{proof}
    Given $\nu(x)\in \calD(\Omega)$, we denote $\mathfrak{m}_0 := \la v \ra^{-2} \mathfrak{m}$ and $\mathfrak{b} := b \mathfrak{m}_0 \nu \varphi$. Then $\mathfrak{b}$ is a solution to
    \begin{align*}
        (\p_t  + v\cdot \nabla_x) \mathfrak{b} &= - \sigma_\va[k] \Delta_v \mathfrak{b} - \sigma_\va[k] \nu \varphi \Big( 2 \nabla_v b \cdot \nabla_v \mathfrak{m}_0 + b \Delta_v \mathfrak{m}_0 \Big) \\
        &\quad + \nabla_v\cdot (\zeta[k] \mathfrak{b}) - \zeta[k] \cdot b \nabla_v \mathfrak{m}_0 \, \nu \varphi \\
        &\quad - \eta[k]\mathfrak{b} + M \chi_R \mathfrak{b} + v b \mathfrak{m}_0 \varphi \cdot \nabla_x \nu + b \mathfrak{m}_0 \nu \varphi' 
    \end{align*}
    in the sense of $\calD'((0,\infty)\times \R^d\times\R^d)$. The above can be rewritten as
    \begin{align*}
        \lt(-\p_t - v\cdot \nabla_x - \Delta_v \rt) \mathfrak{b} &= T_1 + \nabla_v \cdot T_2,
    \end{align*}
    where
    \begin{align*}
        T_1 &:= - \sigma_\va[k] \nu \varphi \Big( 2 \nabla_v b \cdot \nabla_v \mathfrak{m}_0 + b \Delta_v \mathfrak{m}_0 \Big) - \zeta[k]\cdot b \nu \varphi \nabla_v \mathfrak{m}_0 \\
        &\qquad - (\eta[k] + M \chi_R) \mathfrak{b} + vb\mathfrak{m}_0\varphi\cdot \nabla_x \nu + b \mathfrak{m}_0 \nu \varphi' , \\
        T_2 &:= (\sigma_\va[k] - 1) \nabla_v \mathfrak{b} + \zeta[k] \mathfrak{b}.
    \end{align*}
    We point out that since $\mathfrak{m}^{-1} \in \calW$, it holds
    \begin{align*}
        |\nabla_v \mathfrak{m}| \lesssim \la v \ra \mathfrak{m}, 
        \quad |\nabla_v \mathfrak{m}_0| \lesssim \la v \ra^{-1} \mathfrak{m},
        \quad |\Delta_v \mathfrak{m}_0| \lesssim \mathfrak{m}.
    \end{align*}
    Consequently, using that 
    \begin{equation*}
        \sigma_\va[k] \lesssim 1, \quad \zeta[k] \lesssim \la v \ra, \quad \eta[k] \lesssim 1,
    \end{equation*}
    we obtain
    \begin{align*}
        |T_1|^2 &\lesssim \la v \ra^{-2} |\nabla_v (b \mathfrak{m})|^2 \varphi^2 + b^2 \mathfrak{m}^2 \varphi^2 + \la v \ra^{-2} b^2 \mathfrak{m}^2 \varphi^2 |\nabla_x \nu|^2  + \la v \ra^{-2} b^2 \mathfrak{m}^2 |\varphi'|^2 .
    \end{align*}
    Similarly,
    \begin{align*}
        |T_2|^2 \lesssim \la v \ra^{-4} |\nabla_v (b \mathfrak{m})|^2 \varphi^2 + \la v \ra^{-6} b^2 \mathfrak{m}^2 \varphi^2 + \la v \ra^{-2} b^2 \mathfrak{m}^2 \varphi^2.
    \end{align*}
    By estimating all terms of the form $\|b \frakm\|_{L^2_x H^1_v}^2 \varphi^2$ by using Lemma \ref{M: Lem: penal}, namely that
    \begin{equation*}
        \int_0^T \|b \frakm \|_{L^2_x H^1_v}^2 \varphi^2 \lesssim \iiint_{\calO^T} b^2 \frakm^2 \varphi(\varphi')_-,
    \end{equation*}
    we deduce (rather crudely)
    \begin{align*}
        &\|T_1\|_{L^2([0,T]\times\R^d\times\R^d)} + \|T_2\|_{L^2([0,T]\times\R^d\times\R^d)} \\
        &\lesssim \|\nu\|_{W_x^{1,\infty}(\Omega)} \lt( \lt\|b \mathfrak{m} \sqrt{\varphi (\varphi')_-} \rt\|_{L^2(\calO^T)} +  \|b \mathfrak{m} \varphi' \|_{L^2(\calO^T)} \rt).
    \end{align*}
    
    The remaining parts of the proof are then completely the same as with \cite{CarMis24}. We clean up the signs by setting $\mathfrak{B}(t,x,v) = \mathfrak{b}(T-t,x,-v)$ and similarly $S_i(t,x,v) = T_i(T-t,x,-v)$, $i=1,2$. In this way, $\mathfrak{B}$ solves
    \begin{align*}
        (\p_t + v\cdot \nabla_v - \Delta_v) \mathfrak{B} = S_1 - \nabla_v\cdot S_2
    \end{align*}
    in the sense of $\calD'((0,\infty)\times \R^d\times \R^d)$, and therefore admits the representation formula
    \begin{align*}
        \mathfrak{B}(t,x,v) &= \iiint_{t,x,v}  \mathfrak{K}(t-t', x-x'-(t-t')v', v-v') (S_1(t',x',v') - \nabla_{v'}\cdot S_2(t',x',v')) \,\dt'\dx'\dv' \\
        &= \iiint_{t,x,v} \mathfrak{K}(t-t', x-x'-(t-t')v', v-v') S_1(t',x',v') \,\dt'\dx' \dv' \\
        &\quad - \iiint_{t,x,v} \nabla_v \mathfrak{K}(t-t', x-x'-(t-t')v', v-v') S_2(t',x',v') \,\dt'\dx'\dv',
    \end{align*}
    where $\mathfrak{K}$ is the fundamental solution of the Kolmogorov equation (see for instance \cite{AncReb22} or the classical \cite{Kol34}). Namely, we have the explicit expression
    \begin{equation*}
        \frakK(t,x,v) = \frac{C_0}{t^{2d}} \exp\lt( -\frac{3}{t^3} \lt|x - \frac{t}{2}v \rt|^2 - \frac{1}{4t} |v|^2 \rt)
    \end{equation*}
    where $C_0$ is a normalizing constant, and for any $q\in \lt[1, \frac{4d+2}{4d+1} \rt)$, a direct computation shows that
    \begin{align*}
        &\|\frakK\|_{L^q([0,T]\times \R^d\times\R^d)} \lesssim T^{\frac{2d+1}{q} - 2d} , \\
        &\|\nabla_v \frakK\|_{L^q([0,T]\times \R^d\times \R^d)}  \lesssim T^{\frac{2d+1}{q} - 2d - \frac{1}{2}} .
    \end{align*}
    Hence the representation formula and Young's inequality for convolutions show
    \begin{align*}
        &\|\frakB\|_{L^p([0,T]\times\R^d\times\R^d)}  \\
        &\le \|\mathfrak{K}\|_{ L^{\frac{2p}{p+2}}([0,T]\times\R^d\times\R^d) } \|S_1\|_{L^2([0,T]\times\R^d\times\R^d)} + \|\nabla_v \mathfrak{K}\|_{L^{\frac{2p}{p+2}}([0,T]\times\R^d\times\R^d)} \|S_2\|_{L^2([0,T]\times\R^d\times\R^d)} \\
        &\lesssim C_T \|\nu\|_{W^{1,\infty}_x(\Omega)} \lt( \lt\| b \frakm \sqrt{\varphi(\varphi')_-} \rt\|_{L^2(\calO^T)} + \|b \frakm \varphi'\|_{L^2(\calO^T)} \rt)
    \end{align*}
    for every $p\in \lt[2, \frac{2d+1}{d}\rt)$.
    Note that in the above, we made use of the fact that $\|S_i\|_{L^2_{t,x,v}} = \|T_i\|_{L^2_{t,x,v}}$. In the same way, we have $\|\frakB\|_{L^p_{t,x,v}} = \|\frakb\|_{L^p_{t,x,v}}$, and so we have obtained
    \begin{align*}
        &\|\frakb\|_{L^p([0,T]\times\R^d\times\R^d)} \\
        &\lesssim C_T \|\nu\|_{W^{1,\infty}_x(\Omega)} \lt( \lt\| b \frakm \sqrt{\varphi(\varphi')_-} \rt\|_{L^2(\calO^T)} + \|b \frakm \varphi'\|_{L^2(\calO^T)} \rt).
    \end{align*}
    
    Finally, we set $\Omega_k := \lt\{x\in \Omega \; \middle| \; \mathfrak{d}(x) > 2^{-k} \rt\}$ and define $\calO^T_k := (0,T)\times \Omega_k \times \R^d$. We choose the cutoffs $\nu_k\in \calD(\Omega)$ so that $\mathbf{1}_{\Omega_{k+1}} \le \nu_k \le \mathbf{1}_{\Omega_k}$ and $\|\nu_k\|_{W^{1,\infty}_x(\Omega)} \lesssim 2^k$ for every $k\in \bbN$. From the previous arguments that we just established, we know that for any $p\in \lt[1,\frac{2d+1}{d}\rt)$, there holds
    \begin{align*}
        &\|b \frakm_0 \varphi\|_{L^p(\calO_k^T)} \lesssim C_T 2^k \lt( \lt\| b \frakm \sqrt{\varphi(\varphi')_-} \rt\|_{L^2(\calO^T)} + \|b \frakm \varphi'\|_{L^2(\calO^T)} \rt) , \\
        &C_T := T^{\frac{2d+1}{q} - 2d} + T^{\frac{2d+1}{q} - 2d - \frac{1}{2}} . 
    \end{align*}
    Then using that $\mathfrak{d}\sim 2^{-k}$ on each $\calO^T_k$, we can obtain
    \begin{align*}
        \iiint_{\calO^T} \mathfrak{d}^{p_1} |b\mathfrak{m}_0 \varphi|^p &= \sum_{k=1}^\infty \iiint_{\calO^T_{k+1}\setminus \calO^T_{k}} \mathfrak{d}^{p_1} |b \frakm_0 \varphi|^p \\
        &\lesssim \sum_{k=1}^\infty 2^{-k p_1} \iiint_{\calO^T_{k+1}\setminus \calO^T_k} |b \frakm_0 \varphi|^p \\
        &\le C_T^p \lt( \lt\| b \frakm \sqrt{\varphi(\varphi')_-} \rt\|_{L^2(\calO^T)} + \|b \frakm \varphi'\|_{L^2(\calO^T)} \rt)^p  \sum_{k=1}^\infty 2^{k(p-p_1)+1} \\
        &\lesssim C_T^p \lt( \lt\| b \frakm \sqrt{\varphi(\varphi')_-} \rt\|_{L^2(\calO^T)} + \|b \frakm \varphi'\|_{L^2(\calO^T)} \rt)^p,
    \end{align*}
    which finishes the proof.
\end{proof}

Combining the result of Lemma \ref{M: Lem: local} with what we mentioned in \eqref{M: E: clr penal}, we obtain an up-to-the-boundary $L^2\to L^r$ estimate for some $r>2$.

\begin{lemma}\label{M: Lem: L2 Lr}
    We adopt the setting and all of the assumptions of Lemma \ref{M: Lem: penal}, along with the notations of Lemma \ref{M: Lem: local}. Then there exists an exponent $r\in (2,p)$ such that 
    \begin{equation}\label{M: eq: L2Lr}
        \|b \frakm \la v \ra^{-2+\theta} \varphi\|_{L^r(\calO^T)} \lesssim C_T^{1-\theta} \lt(\lt\| b \frakm \sqrt{\varphi(\varphi')_-} \rt\|_{L^2(\calO^T)} + \|b \frakm \varphi'\|_{L^2(\calO^T)} \rt) ,
    \end{equation}
    where
    \begin{equation*}
        \theta := \frac{p_1}{p} \lt(\frac{p_1}{p} + \frac{1}{4} \rt)^{-1} \in (0,1).
    \end{equation*}
\end{lemma}

\begin{proof}
    Using H\"older's inequality, we immediately obtain
    \begin{equation*}
        \|b \frakm \la v \ra^{-2+\theta}  \varphi\|_{L^r(\calO^T)} \le \|b \frakm \delta^{-1/4} \la v \ra^{-1} \varphi\|^\theta_{L^2(\calO^T)} \|b \frakm \delta^{p_1/p} \la v \ra^{-2} \varphi \|^{1-\theta}_{L^p(\calO^T)},
    \end{equation*}
    where 
    \begin{equation*}
        \frac{1}{r} = \frac{\theta}{2} + \frac{1-\theta}{p}.
    \end{equation*}
    Thanks to \eqref{M: E: clr penal} and Lemma \ref{M: Lem: local}, we further obtain that
    \begin{equation*}
        \|b\frakm \la v \ra^{-2+\theta}\varphi\|_{L^r(\calO^T)} \lesssim C_T^{1-\theta} A_1^\theta (A_1 + A_2)^{1-\theta},
    \end{equation*}
    where $A_1 = \lt\| b \frakm \sqrt{\varphi(\varphi')_-} \rt\|_{L^2(\calO^T)}$ and $A_2 = \|b \frakm \varphi'\|_{L^2(\calO^T)}$. Since $(A_1+A_2)^{1-\theta}\le A_1^{1-\theta} + A_2^{1-\theta}$, we get 
    \begin{align*}
        \|b\frakm \la v \ra^{-2+\theta}\varphi\|_{L^r(\calO^T)} &\lesssim C_T^{1-\theta} (A_1 + A_1^\theta A_2^{1-\theta}) \\
        &\lesssim C_T^{1-\theta} (A_1 + A_2),
    \end{align*}
    that last inequality following due to Young's inequality. This concludes the proof.
\end{proof}

At last, we may establish the $L^1\to L^2$ estimate for the solution to the backwards equation, which is the cornerstone of De Giorgi's arguments.

\begin{lemma}\label{M:Lem: pure L1 L2}
    We adopt the same assumptions and notations of Lemmas \ref{M: Lem: penal}--\ref{M: Lem: local}--\ref{M: Lem: L2 Lr}. Suppose that $T\in (0,1]$, and that the terminal datum $b_T$ associated to the dual problem \eqref{M: E: B Back} satisfies $b_T\in L^1_{\frakm_1}(\calO)$, with
    \begin{equation*}
        \frakm_1 := \frakm \la v \ra^{2-\theta}.
    \end{equation*}
    Then whenever $c\le \min\{c_1,c_{\frakm^{-1}}, c_{\frakm_1^{-1}}\}$ and $M,R$ are made larger as necessary (so that in particular $M\ge M_{\frakm_1^{-1}}$ and $R\ge R_{\frakm_1^{-1}}$), there exists a constant $\vartheta>0$ such that
    \begin{align*}
        \|b_0 \|_{L^2_{\frakm}(\calO)} \lesssim \frac{1}{T^{\vartheta}} \|b_T \|_{L^1_{\frakm_1}(\calO)}.
    \end{align*}
\end{lemma}

\begin{proof}
In order to transition the result of Lemma \ref{M: Lem: L2 Lr} into an $L^1\to L^2$ estimate, we use H\"older's inequality to write each term of the right-hand side of \eqref{M: eq: L2Lr} as
\begin{align*}
    &\|b \frakm \sqrt{\varphi (\varphi')_-}\|_{L^2} \le \lt\|b \frakm \la v \ra^{2-\theta } (\varphi'/\varphi)^{1/2\Lambda} \varphi \rt\|_{L^1}^{\Lambda} \|b \frakm \la v \ra^{-2+\theta} \varphi\|_{L^r}^{1-\Lambda} , \\
    &\|b \frakm \varphi'\|_{L^2} \lesssim \lt\|b \frakm \la v \ra^{ 2-\theta } (\varphi' / \varphi)^{1/\Lambda} \varphi \rt\|^{\Lambda}_{L^1} \|b \frakm \la v \ra^{-2+\theta} \varphi\|_{L^r}^{1-\Lambda},
\end{align*}
where $\Lambda = \frac{r-2}{2(r-1)}\in (0,1)$. Inserting these bounds into the right-hand side of \eqref{M: eq: L2Lr} and then dividing both sides by $\|b \frakm \la v \ra^{-2+\theta}\varphi\|_{L^r}^{1-\Lambda}$, we obtain
\begin{equation}\label{M: Eq: L1 Lr}
    \|b \frakm \la v \ra^{-2+\theta}\varphi \|_{L^r(\calO^T)} \lesssim C_T^{\frac{1-\theta}{\Lambda}} \lt( \lt\|b \frakm \la v \ra^{ 2-\theta } (\varphi'/\varphi)^{\frac{1}{2\Lambda}} \varphi \rt\|_{L^1} + \lt\|b \frakm \la v \ra^{ 2-\theta } (\varphi'/\varphi)^{\frac{1}{\Lambda}} \varphi \rt\|_{L^1} \rt),
\end{equation}
in other words an $L^1\to L^r$ estimate.

For the sake of simplicity, let us now set
\begin{align*}
    \frakm_1 := \frakm \la v \ra^{ 2-\theta }, \quad \frakm_2 := \frakm \la v \ra^{-2+\theta},
\end{align*}
so that \eqref{M: Eq: L1 Lr} is nothing but
\begin{equation}\label{M: Eq: rest L1 Lr}
\begin{split}
    &\|b \frakm_2 \varphi\|_{L^r(\calO^T)} \lesssim C_T^{\frac{1-\theta}{\Lambda}} (B_1 + B_2), \\
    &B_1 := \lt\|b \frakm_1 (\varphi'/\varphi)^{\frac{1}{2\Lambda}} \varphi \rt\|_{L^1(\calO^T)}, \quad B_2 := \lt\|b \frakm_1 (\varphi'/\varphi)^{\frac{1}{\Lambda}} \varphi \rt\|_{L^1(\calO^T)}.
\end{split}
\end{equation}
Furthermore, let us fix $0 \le \varphi_0 \in \calD((0,1))$ and define $\varphi\in \calD((0,T))$ explicitly as $\varphi(t) := \varphi_0(t/T)$. Then using $L^p$-interpolation, we can obtain the following estimate for the $L^2$-norm:
\begin{equation}\label{M:Eq: DG1}
\begin{split}
    \|\varphi_0'\|_{L^2(0,1)}^2 \|b_0 \bfm\|_{L^2(\calO)}^2
    &= T \int_0^T (\varphi')^2 \|b_0 \bfm\|_{L^2(\calO)}^2 \\
    &\le T \int_0^T (\varphi')^2 \|b_t \frakm \|^2_{L^2(\calO)} \\
    &\le T \int_0^T (\varphi')^2 \|b_t \frakm_1\|_{L^1(\calO)}^{2\Lambda} \|b_t \frakm_2\|_{L^r}^{2(1-\Lambda)} \\
    &= T \int_0^T \Big[ (\varphi'/\varphi)^{2} \varphi^{2\Lambda} \|b_t \frakm_1\|_{L^1(\calO)}^{2\Lambda} \Big] \Big[ \varphi^{2(1-\Lambda)} \|b_t \frakm_2 \|_{L^r}^{2(1-\Lambda)} \Big] \\
    &\le T \lt(\int_0^T (\varphi'/\varphi)^{\frac{1}{\Lambda}} \varphi \|b_t \frakm_1\|_{L^1(\calO)} \rt)^{\frac{r-2}{r-1}} \lt( \int_0^T \varphi^r \|b_t \frakm_2\|^r_{L^r} \rt)^{\frac{1}{r-1}} \\
    &\le T C_T^{ \frac{2r(1-\theta)}{r-2} } B_1^{\frac{r-2}{r-1}} \lt( B_1 + B_2 \rt)^{\frac{r}{r-1}} .
\end{split}
\end{equation}
In the above, the first inequality follows by an appropriate application of Lemma \ref{Max: Lem: B Back L2 WP} (with $q=2$ and by recalling that $\frakm^{-1}\in \calW_{\mu^{-1/2}}$, $c\le c_{\frakm^{-1}}$). The second inequality is a result of $L^p$-interpolation. The second-to-last inequality is due to H\"older's inequality with the exponent $\frac{1}{2\Lambda}=\frac{r-1}{r-2}$ and its conjugate $\frac{1}{1-2\Lambda}=r-1$. The final inequality is due to \eqref{M: Eq: rest L1 Lr}. Using that $z\mapsto z^{\frac{r}{r-1}}$ is convex, we can estimate further as
\begin{align*}
    \|\varphi_0'\|_{L^2(0,1)}^2 \|b_0 \bfm\|_{L^2(\calO)}^2 &\lesssim T C_T^{ \frac{2r(1-\theta)}{r-2} } (B_1^2 + B_1^{\frac{r-2}{r-1}} B_2^{\frac{r}{r-1}} ) \\
    &\lesssim T C_T^{ \frac{2r(1-\theta)}{r-2} } (B_1^2 + B_2^2),
\end{align*}
the last line following by Young's inequality with exponent $\frac{1}{\Lambda} = \frac{2(r-1)}{r-2}$ and its conjugate. Now, for all $c\le \min\{c_1,c_{\frakm^{-1}}, c_{\frakm_1^{-1}}\}$ and $M,R$ sufficiently large, Lemma \ref{Max: Lem: B Back L2 WP} shows that
\begin{align*}
    \|b_t\|_{L^q_{\frakm_1}(\calO)} \le e^{-\lambda'(T-t)} \|b_T\|_{L^q_{\frakm_1}(\calO)} \quad \forall t\in [0,T], \quad \forall q\in [1,2], 
\end{align*}
which shows
\begin{equation}\label{M:Eq: DG2}
\begin{split}
    &T C_T^{ \frac{2r(1-\theta)}{r-2} } (B_1^2 + B_2^2) \\
    &= T C_T^{ \frac{2r(1-\theta)}{r-2} } \lt( \lt( \int_0^T (\varphi'/\varphi)^{1/2\Lambda}  \varphi \|b_t \frakm_1\|_{L^1(\calO)}  \rt)^2 + \lt( \int_0^T (\varphi'/\varphi)^{1/\Lambda} \varphi \|b_t \frakm_1\|_{L^1(\calO)} \rt)^2 \rt) \\
    &\lesssim  T C_T^{\frac{2r(1-\theta)}{r-2}} \|b_T \frakm_1\|_{L^1(\calO)}^2 \sum_{i=1}^2 \lt(\int_0^T (\varphi'/\varphi)^{1/(i\Lambda)} \varphi \rt)^2.
\end{split}
\end{equation}
Inserting \eqref{M:Eq: DG2} into \eqref{M:Eq: DG1} and then changing variables back to $\varphi_0$, we have established
\begin{align*}
    \|b_0 \bfm\|_{L^2(\calO)} &\lesssim \|\varphi_0'\|_{L^2(0,1)}^{-1}  T^{\frac12} C_T^{\frac{r(1-\theta)}{r-2}} \sqrt{\sum_{i=1}^2 T^{1 - \frac{1}{i\Lambda}} \lt(\int_0^1 (\varphi_0'/\varphi_0)^{1/(i\Lambda)} \varphi_0 \rt)^2 } \|b_T \frakm_1\|_{L^1(\calO)} \\
    &\lesssim  T^{-\vartheta} \|\varphi_0'\|_{L^2(0,1)}^{-1} \sqrt{\sum_{i=1}^2 \lt(\int_0^1 (\varphi_0'/\varphi_0)^{1/(i\Lambda)} \varphi_0 \rt)^2 } \|b_T \frakm_1\|_{L^1(\calO)} \quad \forall T\in (0,1],
\end{align*}
where
\begin{equation*}
    \vartheta = \frac{r+5}{2(r-2)} - \frac{r(1-\theta)}{r-2} \lt( \frac{(2d+1)(p+2)}{2p} - 2d - \frac{1}{2} \rt) > 0.
\end{equation*}
We conclude by choosing any $0\le \varphi_0\in \calD((0,1))$ appropriately so that
\begin{equation*}
    \|\varphi_0'\|_{L^2(0,1)} > 0, \quad \int_0^1 (\varphi_0'/\varphi_0)^{1/(i\Lambda)} \varphi_0 < \infty, \quad \forall i\in \{1,2\}.
\end{equation*}
\end{proof}

Restating everything in terms of the forward equation, we now have the $L^2\to L^\infty$ ultracontractivity property of the solution to \eqref{M: E: B} as follows.

\begin{proposition} \label{M/P/ Ultra}
    We adopt the assumptions and notations of Lemma \ref{M:Lem: pure L1 L2}, so that in particular $c\le \min\{c_1, c_{\frakm^{-1}}, c_{\frakm_1^{-1}}\}$. Denote $\bfw := \frakm^{-1}$. Assume that $\omega_\star \in \calW_{\mu^{-1/2}}$ is another weight satisfying
    \begin{equation*}
        \lt\| \frac{\omega_\star }{\bfw} \la v \ra^{2-\theta} \rt\|_{L^\infty(\R^d)} =: C_{\omega_\star ,\bfw} < \infty.
    \end{equation*}
    Then for every $0\le s < t$,
    \begin{equation}\label{M/Eq/ L2 Linfty}
        \| \calU_{\sfH_\va[k]} (t,s)  h_s \|_{L^\infty_{\omega_\star }(\calO)} \lesssim C_{\omega_\star , \bfw} \frac{e^{-\lambda' (t-s)}}{\min\{(t-s)^{\vartheta}, 1\}} \|h_s \|_{L^2_{\bfw}(\calO)}.
    \end{equation}
\end{proposition}

\begin{proof}
    All of the results can be obtained for general $0\le s < t$ by time translation, so we will only consider the case $s=0$ and $t=T$.
    
    First, we discuss the case $t-s\le 1$, \textit{i.e.} $T\le 1$. Choosing any $\wt{\omega}\in \wt{\calW}$ with $\wt{\omega}\omega_\star ^{-1}\in L^2(\R^d)$, and using the fact that that $L^2_{\wt{\omega}^{-1}}$ is dense in $L^1_{\omega_\star ^{-1}}(\calO)$, we can compute the $L^\infty_{\omega_\star }$-norm of $h_T$ as follows
    \begin{equation}
    \label{M/Eq/ Duality}
    \begin{split}
        \|h_T\|_{L^\infty_{\omega_\star }} &= \sup_{ \substack{ b_T \in L^2_{\wt{\omega}}(\calO), \\ \|b_T\|_{L^1_{\omega_\star ^{-1}}} \le 1   } } \iint_{\calO} h_T b_T  \\
        &= \sup_{ \substack{ b_T \in L^2_{\wt{\omega}}(\calO), \\ \|b_T\|_{L^1_{\omega_\star ^{-1}}} \le 1   } } \iint_{\calO} h_0 b_0 \\
        &\le \|h_0\|_{L^2_{\bfw}}  \sup_{ \substack{ b_T \in L^2_{\wt{\omega}}(\calO), \\ \|b_T\|_{L^1_{\omega_\star ^{-1}}} \le 1   } }  \|b_0\|_{L^2_{\frakm}} \\
        &\lesssim \frac{1}{ \min\{ T^{\vartheta} , 1\} } \|h_0\|_{L^2_{\bfw}}  \sup_{ \substack{ b_T \in L^2_{\wt{\omega}}(\calO), \\ \|b_T\|_{L^1_{\omega_\star ^{-1}}} \le 1   } }  \|b_T \|_{L^1_{\frakm_1}(\calO)}.
    \end{split}
    \end{equation}
    In the above, the second line is due to the duality formula in Lemma \ref{p: l: dual}. The last line follows by applying Lemma \ref{M:Lem: pure L1 L2}. From the conclusion of \eqref{M/Eq/ Duality}, we deduce \eqref{M/Eq/ L2 Linfty} for every $T\in (0,1]$, because in this range we have $1\lesssim e^{-\lambda' T}$.

    For the general case of $T>1$, we write
    \begin{align*}
        &\|\calU_{\sfH_\va[k]}(T,0) h_0 \|_{L^\infty_{\omega_\star }} \\
        &= \|\calU_{\sfH_\va[k]}(T,T-1) \, \calU_{\sfH_\va[k]}(T-1,0) \, h_0\|_{L^\infty_{\omega_\star }} \\
        &\le \|\calU_{\sfH_\va[k]}(T,T-1)\|_{ \scrB(L^2_{\bfw}, L^\infty_{\omega_\star})} \|\calU_{\sfH_\va[k]}(T-1,0) \, h_0\|_{L^2_{\bfw}} \\
        &\lesssim C_{\omega_\star ,\bfw} e^{-\lambda'} \|\calU_{\sfH_\va[k]}(T-1,0) \, h_0\|_{L^2_{\bfw}}  \\
        &\le C_{\omega_\star , \bfw} e^{-\lambda' T} \|h_0\|_{L^2_{\bfw}},
    \end{align*}
    where we used \eqref{M/Eq/ L2 Linfty} for the case $t-s\le 1$ to estimate $\|\calU_{\sfH_\va[k]}(T,T-1)\|_{\scrB(L^2_{\bfw}, L^\infty_{\omega_\star})}$;
    the last inequality is due to Proposition \ref{M:Prop: h diss} with $p=2$ (recall that we assume $c\le c_{\frakm^{-1}}$). This completes the proof.
\end{proof}

\subsection{Well-posedness of the full operator $\sfF_\va[k]$ in a reference space} \label{S/ F WP}

In this section and those to follow, we return to the discussion of the full operator $\sfF_\va[k] = \sfG_\va[k] + \sfH_\va[k]$. Here, let us now mention the well-posedness of \eqref{eq: f g}, namely the evolution equation associated to $\sfF_\va[k]$, which we recall here:

\begin{equation}\label{M/Eq/ Full Evol Eq}
    \begin{cases}
        \p_t h = \sfF_\va[k] h &\text{in }\R_+\times \calO, \\
        h|_{t=0} = h_0 &\text{on }\calO, \\
        \gamma_- h = \gamma_+ h \circ \scrV_x &\text{on } \R_+ \times \Sigma_-.
    \end{cases}
\end{equation}
We also remind the reader that $\sfF_\va[k] = -v\cdot \nabla_x + \sfL_\va[k] + \sfP_\va[k]$, where
    \begin{align*}
        \sfL_\va[k] h &:= \sigma_\va[k] \Delta_v h + \zeta[k]\cdot \nabla_v h + \eta[k] h, \\
        \sigma_\va[k] &:= 1 + \frac{1}{d}E[k] - \frac{1}{d}\bfA_\va[k], \quad \bfA_\va[k] := \frac{|j[k]| *_t \bfs_\va \star_\va \bfr_\va|^2}{1 + \rho[k] *_t \bfs_\va \star_\va \bfr_\va }, \\
        \zeta[k] &:= v(1+\rho[k]) - j[k], \quad \eta[k] := d(1+\rho[k]), \\
        \sfP_\va[k] h &:= \Big( (1+\rho[k])j[h] - j[k] \rho[h] \Big) \cdot v\mu \\
        &\quad + \frac{1}{d} \lt( - \Big( d + E[k] - \bfA_\varsigma[k] \Big) \rho[h] - j[k] \cdot j[h] + (1 + \rho[k]) E[h]  \rt) (|v|^2 - d) \mu .
    \end{align*}

\begin{proposition}\label{P/ F WP}
    Let $\omega_1 \in \calW_{\mu^{-1/2}}$. We assume that $k\in \calX_{\omega,c}$ where $c\le \min\{c_1,c_{\omega_1}\}$ (see Lemmas \ref{M: L: B L2 framework WP}--\ref{Lem/ Asymptote}). Then for each $h_0 \in L^2_{\omega_1}(\calO)$, the system \eqref{M/Eq/ Full Evol Eq} is well-posed in $C([0,T];L^2_{\omega_1}(\calO))$. Consequently we may associate to \eqref{M/Eq/ Full Evol Eq} an evolution system $(\calU_{\sfF_\va[k]}(t,s))$ for which $\calU_{\sfF_\va[k]}(t,s)h_s = h_t$ for every $0\le s \le t$ and $h_s\in L^2_{\omega_1}(\calO)$. Furthermore, the associated solution $h$ satisfies the conservation laws in the following sense: denoting
    \begin{align*}
        &\calC := \{h_0 \in L^1(\calO, \la v \ra^2 \,\dx\dv) \; | \; \doublebrackets{h_0} = \doublebrackets{|v|^2 h_0} = \doublebrackets{(R(x)\cdot v) h_0} = 0 \quad \forall R\in \calR_\Omega \}, \\
        &\doublebrackets{h_0} := \iint_{\calO} h_0 \,\dx\dv,
    \end{align*}
    and writing $\Pi: L^2_{\mu^{-1/2}}(\calO) \to \calC$ as the orthogonal projection operator, we have
    \begin{equation*}
        (I - \Pi) \, \calU_{\sfF_\va[k]}(t,s) h_s = \calU_{\sfF_\va[k]}(t,s) (I - \Pi) h_s \quad \forall h_s\in L^2_{\omega_1}(\calO).
    \end{equation*}
    Consequently, provided that $h_0\in \calC$, we have $h_t \in \calC$ for every $t\ge 0$.
\end{proposition}

\begin{proof}
    The well-posedness just follows by checking the assumptions of Proposition \ref{prop: WP}. Under the assumption of $k\in \calX_{\omega,c}$ and $c\le \min\{c_1,c_{\omega_1}\}$, the conditions (L1)--(L2)--(L3) are easy to check, and in particular (L2) follows by the choice of $c_{\omega_1}$ (see Lemma \ref{Lem/ Asymptote}). We will therefore only verify (P1). Namely, we need to prove that 
    \begin{equation}\label{M/Eq/ P1 ETS}
        \iint_{\calO} |\sfP_\va[k]|^2 \omega_1^2 \,\dx\dv \lesssim \iint_{\calO} |f|^2 \omega_1^2 \,\dx\dv.
    \end{equation}
    Since $k\in \calX_{\omega,c}$, we observe that every coefficient in the definition of $\sfP_\va[k]h$ may be bounded as
    \begin{align*}
        &\Big| (1+\rho[k])j[h] - j[k]\rho[h] \Big| \lesssim_c |j[h]| + |\rho[h]| , \\
        &\lt| -(d + E[k] - \bfA_\va[k]) \rho[h] -  j[k]\cdot j[h] + (1+\rho[k])E[h] \rt| \lesssim_c |\rho[h]| + |j[h]| + |E[h]|.
    \end{align*}
    Henceforth,
    \begin{align*}
        \iint_{\calO} |\sfP_\va[k]|^2 \omega_1^2 &\lesssim \iint_{\calO} (|\rho[h]| + |j[h]| + |E[h]|)^2 \la v \ra^4 \mu^2 \omega_1^2  \\
        &= \lt(\int_{\R^d} \la v \ra^4 \mu^2 \omega_1^2 \rt) \lt(\int_{\Omega} (|\rho[h]| + |j[h]| + |E[h]|)^2 \rt) \\
        &\le C \int_{\Omega} (|\rho[h]| + |j[h]| + |E[h]|)^2 \\
        &\le C \int_\Omega \lt(\int_{\R^d} \la v \ra^2 h \rt)^2 \\
        &\le C  \int_\Omega \lt( \int_{\R^d} h^2 \omega_1^2 \rt) \lt( \int_{\R^d} \la v \ra^4 \omega_1^{-2} \rt)  \\
        &= C \iint_{\calO} h^2 \omega_1^2 .
    \end{align*}
    In the above, we took advantage of the fact that $\la v \ra^2 \mu \omega_1, \la v \ra^2 \omega_1^{-1} \in L^2(\R^d)$. This proves \eqref{M/Eq/ P1 ETS} and thus (P1).

    Next, we verify that the conservation laws are satisfied by $h$, \textit{i.e.} $h\in \calC$. It is enough to observe that since $L^\infty_{\omega_1} \hookrightarrow L^1_{\la v \ra^2}$, any $\chi\in \calD(\ov{\calO})$ satisfying $\chi \lesssim \la v \ra^2$ is admissible as a test function in the Green's formula for $h$. By construction (see \eqref{M/Eq/ Lin Con}--\eqref{eq: nonlinear perp}), we have $\iint_{\calO}(1,v,|v|^2) (\sfC_{\rm lin} + \leftperp{\Upsilon_{\va}[k]}) h = 0$. Recalling that $\sfF_\va[k] = -v\cdot \nabla_x + \sfC_{\rm lin} + \leftperp{\Upsilon_{\va}[k]}$, the conservation laws follow immediately.
\end{proof}

\subsection{Hypocoercivity of the full operator $\sfF_\va[k]$ in a reference space} \label{S/ F Hypo}

In this section we establish the hypocoercivity property of the operator $\sfF_\va[k]$ in the most typical $L^2$-space which will serve as a reference point later. Let us remind the reader that
\begin{equation*}
    \calC := \{h_0 \in L^1(\calO, \la v \ra^2 \,\dx\dv) \; | \; \doublebrackets{h_0} = \doublebrackets{|v|^2 h_0} = \doublebrackets{(R(x)\cdot v) h_0} = 0 \quad \forall R\in \calR_\Omega \}.
\end{equation*}

\begin{proposition}\label{prop: L_g hypo}
    There exists $c_2>0$ small enough such that the following assertion is true: if $k\in \calX_{\omega,c}$ with $c\le \min\{c_1,c_2\}$, then there exists a scalar product $\doubleparantheses{\cdot}{\cdot}_{L^2_{\mu^{-1/2}}}$ on $L^2_{\mu^{-1/2}}(\calO)$ such that the induced norm $\NORMMM{\cdot}_{L^2_{\mu^{-1/2}}}$ is equivalent to $\NORM{\cdot}_{L^2_{\mu^{-1/2}}}$, and moreover the Dirichlet form corresponding to $\sfF_\va[k]$ satisfies
    \begin{equation*}
        \doubleparantheses{-\sfF_\va[k] f}{f}_{L^2_{\mu^{-1/2}}} \gtrsim \|f\|_{L^2_{\mu^{-1/2}}}^2
    \end{equation*}
    for every $f\in \textnormal{Dom}(\sfF_\va[k]) \cap \calC$ satisfying the specular reflection boundary condition \eqref{eq: specular}.
\end{proposition}

Toward the proof of this proposition we introduce some handy notations, following the work of \cite{BCMT23}. After appropriately translating as necessary, we assume without loss of generality that
\begin{equation*}
    \int_{\Omega} x \,\dx = 0,
\end{equation*}
so that by the arguments in \cite[Section 1.2]{BCMT23} we see that $\calR_\Omega$ now reduces to the \textit{centered} infinitesimal rigid displacement fields:
\begin{equation*}
    \calR_\Omega = \left\{ U\in \R^{d\times d} \; \middle | \; \textnormal{$U$ is skew-symmetric and $Ux \cdot n_x=0$ $\forall x\in \p\Omega$} \right\}.
\end{equation*}

Then, for a given source term $S\in L^2(\Omega)$, we denote by $u[S]\in H^1(\Omega)$ the solution to the Poisson equation
\begin{equation}\label{M/Eq/ Poi}
    \begin{cases}
    -\Delta_x u[S] = S &\text{in } \Omega, \\
    \nabla_x u[S] \cdot n_x = 0 &\text{on } \p\Omega.
    \end{cases}
\end{equation}
We remark that whenever the compatibility condition $\int_\Omega S = 0$ holds, the classical Lax--Milgram theorem ensures the existence of $u[S]\in H^1$ which solves \eqref{M/Eq/ Poi} in the variational sense and satisfies $\int_\Omega u[S] = 0$. An immediate consequence of the Lax--Milgram theorem is that
\begin{equation}\label{eq: poi lax}
    \|u[S]\|_{H^1(\Omega)} \lesssim \|S\|_{L^2(\Omega)},
\end{equation}
but it is well-known that the regularity can be improved:
\begin{equation}\label{M/Eq/ Poi H2}
    \|u[S]\|_{H^2(\Omega)} \lesssim \|S\|_{L^2(\Omega)}.
\end{equation}

In similar nature, for $S\in L^2(\Omega;\R^d)$ we denote by $\wt{u}[S] \in H^1(\Omega)$ the solution to the elliptic equation
\begin{equation}\label{M/Eq/ Lame}
    \begin{cases}
    -\nabla_x\cdot (\nabla_x^s \wt{u}[S]) = S &\text{in }\Omega, \\
    \wt{u}[S]\cdot n_x = 0 &\text{on }\p\Omega, \\
    \nabla_x^s \wt{u}[S] n_x - (\nabla_x^s \wt{u}[S] : n_x\otimes n_x) n_x = 0 &\text{on }\p\Omega.
    \end{cases}
\end{equation}
where $\nabla_x^s$ stands for the symmetric gradient $(\nabla_x^s U)_{ij} := \frac{1}{2}(\p_{x_j} U_i + \p_{x_i} U_j)$. Now writing the skew-symmetric gradient as $(\nabla_x^a)_{ij}:= \frac{1}{2}(\p_{x_j}U_i - \p_{x_i} U_j)$, we define the Hilbert space
\begin{align*}
    \calV := \lt\{ W \in H^1(\Omega;\R^d) \; \middle| \; W \cdot n = 0 \text{ on }\p\Omega \text{ and } \Pi_\Omega \int_\Omega \nabla^a W \,\dx = 0 \rt\},
\end{align*}
equipped with the $H^1(\Omega)$-norm, where $\Pi_\Omega$ denotes the orthogonal projection onto $\calR_\Omega$. From the results of \cite{BCMT23}, for each $S\in L^2(\Omega;\R^d)$ we know the existence and uniqueness of $\wt{u}[S]\in \calV$ solving \eqref{M/Eq/ Lame} in the variational sense, with (again thanks to the Lax--Milgram theorem)
\begin{equation}\label{eq: lame lax}
    \|\wt{u}[S]\|_{H^1(\Omega)} \lesssim \|S\|_{L^2(\Omega)}.    
\end{equation}
If additionally $S$ is such that $\int_\Omega S\cdot Ux = 0$ for every $U\in \calR_\Omega$, then the following regularity estimate holds (see \cite[Theorem 2.11]{BCMT23})
\begin{equation}\label{M/Eq/ Lame H2}
    \|\wt{u}[S]\|_{H^2(\Omega)} \lesssim \|S\|_{L^2(\Omega)}.
\end{equation}

We then define the scalar product $\doubleparantheses{\cdot}{\cdot}$ as
\begin{equation*}
    \begin{split}
    \doubleparantheses{f}{h}_{L^2_{\mu^{-1/2}}} :&= (f,h)_{L^2_{\mu^{-1/2}}} \\
    &\quad + \delta_1 (-\nabla_x u[\rho[f]], j[h])_{L^2(\Omega)} + \delta_1 (-\nabla_x u[\rho[h]], j[f])_{L^2(\Omega)} \\
    &\quad + \delta_2 (-\nabla_x^s \wt{u}[j[f]], \wt{M}[h])_{L^2(\Omega)} + \delta_2 (-\nabla_x^s \wt{u}[j[h]], \wt{M}[f])_{L^2(\Omega)} \\
    &\quad + \delta_3 (-\nabla_x u[\tau[f]], M[h])_{L^2(\Omega)} + \delta_3 (-\nabla_x u[\tau[h]], M[f])_{L^2(\Omega)},
    \end{split}
\end{equation*}
where $\delta_i>0$ are constants to be chosen later, and
\begin{equation*}
\begin{split}
    &\tau[f](x,v) := \int_{\R^d} \frac{|v|^2-d}{\sqrt{2d}} f(x,v) \,\dv, \\
    &M[f] := \frac{1}{\sqrt{2d}} \int_{\R^d} v(|v|^2-d-2) f \,\dv ,\\
    &\wt{M}[f] := \int_{\R^d} (v\otimes v - I_{d\times d}) f \,\dv.
\end{split}
\end{equation*}
In the above, $I_{d\times d}\in \R^{d\times d}$ is the identity matrix, and $(\cdot,\cdot)_{L^2}$ refers to the usual $L^2$ scalar product. Here we point out that if $f$ satisfies the specular reflection boundary condition \eqref{eq: specular}, then there is no flux of mass through the boundary, \textit{i.e.}
\begin{equation*}
    \forall x\in \p\Omega, \quad \int_{\R^d}  \gamma f \, (n_x\cdot v) \,\dv = 0,
\end{equation*}
and therefore $j[f]\cdot n_x = 0$ whenever $x\in \p\Omega$. Furthermore, if $f\in \calC$ then by its definition we have $\iint_{\calO} Ux\cdot vf \,\dx\dv = 0$ for every $U\in \calR_{\Omega}$, which translates to $\int_{\Omega} Ux\cdot j[f] \,\dx = 0$ for every $U\in \calR_{\Omega}$. This ensures that the regularity estimate \eqref{M/Eq/ Lame H2} applies to $\wt{u}[j[f]]$. In similar nature, if $f\in \calC$ then $\doublebrackets{f} = \int_{\Omega} \rho[f] = 0$ and so $u[\rho[f]]$ obeys the estimate in \eqref{M/Eq/ Poi H2}.

Prior to making computations on $\doubleparantheses{\cdot}{\cdot}_{L^2_{\mu^{-1/2}}}$ it is worth mentioning that the ``linearized operator part'' of $\sfF_\va[k]$, namely $-v\cdot \nabla_x + \sfC_{\rm lin}$, is quite well-behaved and falls into the scope of \cite{BCMT23}. In particular, the main result of \cite{BCMT23} immediately guarantees that there exist suitable choices of $\delta_1,\delta_2,\delta_3>0$ such that the scalar product $\doubleparantheses{\cdot}{\cdot}_{L^2_{\mu^{-1/2}}}$ satisfies
    \begin{equation}\label{eq: hypo of L}
     \doubleparantheses{ (-v\cdot \nabla_x + \sfC_{\rm lin}) f}{f}_{L^2_{\mu^{-1/2}}} \lesssim - \|f\|_{L^2_{\mu^{-1/2}}}^2 
     \end{equation}
for every $f$ which satisfies the specular reflection boundary condition. Unfortunately, the result \eqref{eq: hypo of L} as it stands is not enough to incorporate the terms which arise from the ``nonlinear part'' $\Upsilon^\perp_{k,\va}$. Indeed, we cannot expect the right-hand side of \eqref{eq: hypo of L} to be able to absorb all of the terms which arise from it, especially due to the fact that $\Upsilon^\perp_{k,\va}$ is part of the principal operator of $\sfF_\va[k]$. Thus, it is necessary for us to go back to the level of basic computations regarding $\sfC_{\rm lin}$ where derivatives are still present.

Now, we establish some properties of the linear operators in discussion, which will aid in the forthcoming analysis. Let us denote by $\sfC_{\rm FP}$ the linear Fokker--Planck operator:
\begin{equation*}
\begin{split}
    \sfC_{\rm FP} f &:= \Delta_v f + \nabla_v\cdot (vf), \\
\end{split}
\end{equation*}
so that $\sfC_{\rm lin}$ is nothing but
\begin{equation*}
    \sfC_{\rm lin} f = \sfC_{\rm FP} f + \sum_{i=1}^d \pi_i f + 2 \pi_{d+1} f .
\end{equation*}
When the transport operator is also taken into effect, we will denote 
\begin{equation*}
    \sfC_{\rm full} f := -v\cdot \nabla_x f+ \sfC_{\rm lin} f .
\end{equation*}

\begin{lemma} \label{lem: elem}
    The following properties hold.
    \begin{enumerate}
        \item The operator $\sfC_{\rm lin}$ is self-adjoint and its kernel in $L^2_{\mu^{-1/2}}(\R^d_v)$ is given by
        $$K:= \textnormal{Span}\{\mu,v_1\mu,\ldots,v_d\mu, |v|^2\mu\},$$ so that $\sfC_{\rm lin} \pi = 0$.
        \item For any $f\in  \textnormal{Dom}(\sfC_{\rm full})$ satisfying the specular reflection boundary condition \eqref{eq: specular},
        \begin{equation*}
            (\sfC_{\rm full} f, f)_{L^2_{\mu^{-1/2}}} = (\sfC_{\rm FP}\leftperp{f}, \leftperp{f})_{L^2_{\mu^{-1/2}}},
        \end{equation*}
        where $(\cdot,\cdot)_{L^2_{\mu^{-1/2}}}$ is the scalar product defined on $L^2_{\mu^{-1/2}}(\calO)$:
        \begin{equation*}
            (f,h)_{L^2_{\mu^{-1/2}}} := \iint_{\calO} f h \mu^{-1} \,\dx\dv,
        \end{equation*}
        and $\leftperp{f} := (I - \pi) f$.
    \end{enumerate}
\end{lemma}
\begin{proof}
A straightforward computation shows that $K\subset \textnormal{Ker}(\sfC_{\rm lin})$. The opposite inclusion also follows easily. Indeed, let $f\in \textnormal{Dom}(\sfC_{\rm FP}) \subset L^2_{\mu^{-1/2}}(\R^d_v)$ and observe that
\begin{equation*}
    \sfC_{\rm lin} f = 0 \implies \sfC_{\rm FP}f \in K,
\end{equation*}
because $\pi_i f$ are trivially members of $K$. We can then explicitly solve this differential inclusion to prove that $f \in K$.

To see the second assertion, we compute
\begin{align*}
        (\sfC_{\rm full} f, f)_{L^2_{\mu^{-1/2}}} &= \lt( \sfC_{\rm lin} f, f \rt)_{L^2_{\mu^{-1/2}}} \\
        &= \lt( \sfC_{\rm lin} \leftperp{f}  , \pi f + \leftperp{f} \rt)_{L^2_{\mu^{-1/2}}} \\
        &= (\sfC_{\rm lin} \leftperp{f}, \leftperp{f} )_{L^2_{\mu^{-1/2}}} \\
        &= (\sfC_{\rm FP} \leftperp{f} , \leftperp{f} )_{L^2_{\mu^{-1/2}}} .
    \end{align*}
The right-hand side of the first line follows from the fact that the transport term vanishes, owing to the Gauss--Ostrogradsky theorem and the specular reflection boundary condition. To obtain the second line, we used the fact that $\sfC_{\rm lin} \pi f = 0$. The third line is due to the fact that $\sfC_{\rm lin}$ is self-adjoint and again that $\sfC_{\rm lin}\pi = 0$. Finally the last line comes from $\pi_i (\sfI - \pi) = 0$ for all $i=0,1,\ldots,d+1$.
\end{proof}

A direct consequence of Lemma \ref{lem: elem} is the following. Although the proof follows more or less classical estimates regarding the linear Fokker--Planck operator, we fully include the arguments for the sake of completeness.
\begin{lemma} \label{lem: L coercive}
    For any $f\in \textnormal{Dom}(\sfF_\va[k]) \cap \calC$ satisfying \eqref{eq: specular}, we have the coercivity estimate
    \begin{equation*}
        \begin{split}
            (-\sfC_{\rm full} f, f)_{L^2_{\mu^{-1/2}}} \gtrsim \|\nabla_v (\mu^{-1} \leftperp{f})\|_{L^2_{\mu^{+1/2}}}^2 +  \|\nabla_v \leftperp{f}\|_{L^2_{\mu^{-1/2}}}^2 + \|v \, \leftperp{f}\|_{L^2_{\mu^{-1/2}}}^2 + \|\leftperp{f}\|_{L^2_{\mu^{-1/2}}}^2.
        \end{split}
    \end{equation*}
    
\end{lemma}
\begin{proof}
Thanks to Lemma \ref{lem: elem}, we have the explicit identity
\begin{equation}\label{eq: Coercive 0}
    (-\sfC_{\rm full} f, f)_{L^2_{\mu^{-1/2}}} =  (-\sfC_{\rm FP} \leftperp{f}, \leftperp{f})_{L^2_{\mu^{-1/2}}}  = \iint_{\calO} |\nabla_v (\mu^{-1} \leftperp{f} )|^2 \, \mu \,\dx\dv.
\end{equation}
Expanding the square and the gradient, we observe that the integral of the right-hand side is equal to
\begin{equation} \label{eq: Expansion Poincare}
\begin{split}
    \iint_{\calO} \lt|\nabla_v(\mu^{-1} \leftperp{f}) \rt|^2 \mu \,\dx\dv &= \iint_{\calO} \lt| v \leftperp{f} +  \nabla_v \leftperp{f}  \rt|^2 \mu^{-1} \,\dx\dv \\
    &= \iint_{\calO} |v|^2 |\leftperp{f}|^2 \mu^{-1} \,\dx\dv + \int |\nabla_v \leftperp{f}|^2 \mu^{-1} \,\dx\dv \\
    &\quad + 2 \iint_{\calO} (v\mu^{-1} \leftperp{f})\cdot (\mu^{-1} \nabla_v \leftperp{f}) \mu \,\dx\dv \\
    &= \iint_{\calO} \Big( -d|\leftperp{f}|^2 + |\nabla_v \leftperp{f}|^2 \Big) \mu^{-1} \,\dx\dv,
\end{split}
\end{equation}
where the integral $\iint_{\calO} |v|^2 |\leftperp{f}|^2 \mu^{-1}$ vanished because the crossing term satisfies:
\begin{equation}\label{eq: Crossing Term}
\begin{split}
    2\iint_{\calO} (v\mu^{-1}\leftperp{f}) \cdot (\mu^{-1} \nabla_v \leftperp{f}) \mu \,\dx\dv &= 2 \iint_{\calO} v\leftperp{f} \cdot \nabla_v \leftperp{f} \,\mu^{-1}\,\dx\dv \\
    &= \iint_{\calO} v \cdot \nabla_v |\leftperp{f}|^2 \mu^{-1}\,\dx\dv \\
    &= - \iint_{\calO} |v|^2 |\leftperp{f}|^2 \mu^{-1} \,\dv - d \iint_{\calO} |\leftperp{f}|^2 \mu^{-1} \,\dv .
\end{split}
\end{equation}

Hence, from \eqref{eq: Coercive 0} and \eqref{eq: Expansion Poincare} we currently have
\begin{equation}\label{eq: Coercive 1}
(-\sfC_{\rm full} f, f)_{L^2_{\mu^{-1/2}}} \gtrsim \|\nabla_v(\mu^{-1}\leftperp{f})\|^2_{L^2_{\mu^{+1/2}}} + \|\nabla_v \leftperp{f}\|_{L^2_{\mu^{-1/2}}}^2 - \|\leftperp{f}\|_{L^2_{\mu^{-1/2}}}^2. 
\end{equation}
That last negative term does not create any problem, because the Poincar\'e inequality applied to the Gaussian measure gives
\begin{equation}\label{eq: Coercive 3}
    (-\sfC_{\rm full} f, f) \gtrsim \|\leftperp{f}\|_{L^2_{\mu^{-1/2}}}^2.
\end{equation}

It thus remains for us to somehow restore the term $\| v \, \leftperp{f}\|_{L^2_{\mu^{-1/2}}}$ to the right-hand side. Let us take advantage of the presence of $\|\nabla_v \leftperp{f}\|_{L^2_{\mu^{-1/2}}}$ in \eqref{eq: Coercive 1}. In order to obtain an identity for $\|\nabla_v \leftperp{f}\|_{L^2_{\mu^{-1/2}}}$ which is not as damaging as in \eqref{eq: Crossing Term}, we introduce a slightly variational term as follows: for any $\alpha, \beta >0$
\begin{equation*}
\begin{split}
    \iint_{\calO} \lt| \alpha v \leftperp{f} + \beta \nabla_v \leftperp{f} \rt|^2 \mu^{-1} \,\dx\dv &= \alpha^2 \| v\leftperp{f} \|_{L^2_{\mu^{-1/2}}}^2 + \beta^2 \|\nabla_v \leftperp{f}\|_{L^2_{\mu^{-1/2}}}^2 \\
    &\quad + 2\alpha\beta \int v\leftperp{f} \cdot \nabla_v \leftperp{f} \, \mu^{-1} \,\dx\dv \\
    &= \beta^2 \|\nabla_v \leftperp{f}\|_{L^2_{\mu^{-1/2}}}^2 + \alpha \lt(\alpha - \beta \rt) \| v\leftperp{f} \|_{L^2_{\mu^{-1/2}}}^2 - d\alpha\beta \|\leftperp{f}\|_{L^2_{\mu^{-1/2}}}^2.
    \end{split}
\end{equation*}
Equivalently we have
\begin{align*}
    \beta^2 \|\nabla_v \leftperp{f}\|_{L^2_{\mu^{-1/2}}}^2 &= \lt\| \alpha v\leftperp{f} + \beta \nabla_v \leftperp{f}\rt\|_{L^2_{\mu^{-1/2}}}^2 - \alpha \lt(\alpha - \beta \rt) \|v\leftperp{f}\|_{L^2_{\mu^{-1/2}}}^2 + d\alpha\beta \|\leftperp{f}\|_{L^2_{\mu^{-1/2}}}^2.
\end{align*}
Choosing $\beta=1$ and $\alpha=1/2$ , we get
\begin{equation*}
    \|\nabla_v \leftperp{f}\|_{L^2_{\mu^{-1/2}}}^2 = \lt\| \alpha v\leftperp{f} + \beta \nabla_v \leftperp{f}\rt\|_{L^2_{\mu^{-1/2}}}^2 + \frac{1}{4} \|v \leftperp{f}\|_{L^2_{\mu^{-1/2}}}^2 + \frac{d}{2} \|\leftperp{f}\|_{L^2_{\mu^{-1/2}}}^2,
\end{equation*}
so in particular
\begin{equation*}
    \|\nabla_v \leftperp{f}\|_{L^2_{\mu^{-1/2}}}^2 \gtrsim \|v \leftperp{f}\|_{L^2_{\mu^{-1/2}}}^2 .
\end{equation*}
Inserting this into \eqref{eq: Coercive 1} we deduce
\begin{equation}\label{eq: Coercive 2}
    \begin{split}
    (-\sfC_{\rm full} f, f) &\gtrsim \|v\leftperp{f}\|_{L^2_{\mu^{-1/2}}}^2 - \|\leftperp{f}\|_{L^2_{\mu^{-1/2}}}^2.
    \end{split}
\end{equation}

Combining \eqref{eq: Coercive 1}--\eqref{eq: Coercive 3}--\eqref{eq: Coercive 2} we complete the proof.
\end{proof}

As is typical with kinetic collisional operators, at the moment we only have a coercivity estimate with respect to the orthogonal complement. We can proceed further by using the arguments from \cite{BCMT23}, then combining them with Lemma \ref{lem: L coercive}, in order to obtain full coercivity of the linearized operator $\sfL$ with respect to the twisted product $\doubleparantheses{\cdot}{\cdot}_{L^2_{\mu^{-1/2}}}$. Note that below, we keep the derivative term, and also the weighted norm $\|v\leftperp{f}\|_{L^2_{\mu^{-1/2}}}$, as aforementioned, they are both crucial in absorbing the terms which arise from $\leftperp{\Upsilon_\va[k]}$ (see Lemmas \ref{lem: non sma11}--\ref{lem: non sma22} below).

\begin{lemma}\label{M/Lem/ Cfull coer}
    There exist suitable choices of $\delta_1,\delta_2,\delta_3>0$ such that for any $f\in \textnormal{Dom}(\sfF_\va[k]) \cap \calC$ satisfying \eqref{eq: specular},
    \begin{equation*}
        \doubleparantheses{-\sfC_{\rm full} f}{f}_{L^2_{\mu^{-1/2}}} \gtrsim \|f\|_{L^2_{\mu^{-1/2}}}^2 + \|\nabla_v(\leftperp{f}\mu^{-1})\|_{L^2_{\mu^{+1/2}}}^2 + \|v\leftperp{f}\|_{L^2_{\mu^{-1/2}}}^2.
    \end{equation*}
\end{lemma}

\begin{proof}
    We only sketch the proof and refer to \cite{BCMT23} for the details. Thanks to Lemma \ref{lem: L coercive} and the results of \cite[Section 3]{BCMT23} (namely \eqref{M/Eq/ Poi H2}--\eqref{M/Eq/ Lame H2}), there exist $C_i>0$ and $C>0$ such that
    {\small \begin{equation*}
        \begin{split}
            &\doubleparantheses{-\sfC_{\rm full} f}{f}_{L^2_{\mu^{-1/2}}} \\
            &\ge C_0 \lt( \|\leftperp{f}\|_{L^2_{\mu^{-1/2}}}^2 + \|\nabla_v(\mu^{-1}\leftperp{f})\|_{L^2_{\mu^{+1/2}}}^2 + \|v\leftperp{f}\|_{L^2_{\mu^{-1/2}}}^2 \rt) \\
            &\quad + \delta_1 \lt( C_1 \|\rho[f]\|_{L^2(\Omega)}^2 - C \Big( \|j[f]\|_{L^2(\Omega)}^2 + \|\tau[f]\|_{L^2(\Omega)}^2 + \|\leftperp{f}\|_{L^2_{\mu^{-1/2}}}^2 \Big) \rt) \\
            &\quad + \delta_2 \lt( C_2 \|j[f]\|_{L^2(\Omega)}^2 - C  \Big(\|\leftperp{f}\|_{L^2_{\mu^{-1/2}}} \|\rho\|_{L^2(\Omega)} - \|\tau[f]\|_{L^2(\Omega)} \|\rho\|_{L^2(\Omega)} - \|\tau[f]\|_{L^2(\Omega)} - \|\leftperp{f}\|_{L^2_{\mu^{-1/2}}} \Big) \rt) \\
            &\quad + \delta_3 \lt( C_3 \|\tau[f]\|_{L^2(\Omega)} - C \Big( \|j[f]\|_{L^2(\Omega)} \|\leftperp{f}\|_{L^2_{\mu^{-1/2}}} - \|\leftperp{f}\|_{L^2_{\mu^{-1/2}}}^2 \Big) \rt) .
        \end{split}
    \end{equation*}
    }
    Here Lemma \ref{lem: L coercive} was used to obtain the first line. Proceeding as in \cite[Proof of Theorem 1.1]{BCMT23}, we can then choose $0\ll \delta_1 \ll \delta_2 \ll \delta_3 \ll 1$ appropriately small enough so that
    \begin{equation*}
    \begin{split}
        \doubleparantheses{-\sfC_{\rm full} f}{f}_{L^2_{\mu^{-1/2}}} &\gtrsim \Big(\|\leftperp{f}\|_{L^2_{\mu^{-1/2}}}^2 + \|\rho[f]\|_{L^2(\Omega)}^2 + \|j[f]\|_{L^2(\Omega)}^2 + \|\tau[f]\|_{L^2(\Omega)}^2 \Big) \\
        &\quad + \|\nabla_v (\mu^{-1} \leftperp{f})\|_{L^2_{\mu^{+1/2}}}^2 + \|v\leftperp{f}\|_{L^2_{\mu^{-1/2}}}^2 .
    \end{split}
    \end{equation*}
    The right-hand side of the first line is equivalent to $\|f\|_{L^2_{\mu^{-1/2}}}^2$, and we therefore obtain the desired result.
\end{proof}

In order to extend to result to $\doubleparantheses{-\sfF_\va[k]f}{f}_{L^2_{\mu^{-1/2}}}$, it is now enough to prove that the nonlinear remainder does not alter the scalar product too much. We recall that the nonlinear remainder satisfies (by definition)
\begin{equation*}
    \rho[\leftperp{\Upsilon_{\va}[k]}  f] = j[\leftperp{\Upsilon_{\va}[k]}  f] = \tau[\leftperp{\Upsilon_{\va}[k]}  f] = 0,
\end{equation*}
which allows us to write
\begin{equation*}
    \begin{split}
        \doubleparantheses{\sfF_\va[k] f}{f}_{L^2_{\mu^{-1/2}}} &= \doubleparantheses{\sfC_{\rm full} f}{f}_{L^2_{\mu^{-1/2}}} + (\leftperp{\Upsilon_{\va}[k]}  f, f)_{L^2_{\mu^{-1/2})}} \\
        &\quad + \delta_2 (-\nabla_x^s \wt{u}[j[f]], \wt{M}[\leftperp{\Upsilon_{\va}[k]}  f])_{L^2(\Omega)} \\
        &\quad + \delta_3 (-\nabla_x u[\tau[f]], M[\leftperp{\Upsilon_{\va}[k]}  f])_{L^2(\Omega)}.
    \end{split}
\end{equation*}
We proceed to estimate the remaining terms on the right-hand side, one by one.
\begin{lemma}\label{lem: non sma11}
    We assume that $\|k\|_{L^\infty_\omega}\ll 1$. Then for any $f\in \textnormal{Dom}(\leftperp{\Upsilon_{\va}[k]} )$, the following estimate holds:
    \begin{equation*}
    \begin{split}
        &\lt|(\leftperp{\Upsilon_{\va}[k]}  f, f)_{L^2_{\mu^{-1/2}}}\rt| \\
        &\le C\|k\|_{L^\infty_\omega} \lt(\|\nabla_v(\leftperp{f} \mu^{-1})\|_{L^2_{\mu^{+1/2}}}^2 + \|f\|_{L^2_{\mu^{-1/2}}}^2 + \|v\leftperp{f}\|_{L^2_{\mu^{-1/2}}}^2 \rt).
    \end{split}        
    \end{equation*}
\end{lemma}
\begin{proof}
    We first perform the computation
    \begin{align}
        (\leftperp{\Upsilon_{\va}[k]}  f, f)_{L^2_{\mu^{-1/2}}} &= (\leftperp{\Upsilon_{\va}[k]}  f, \leftperp{f} )_{L^2_{\mu^{-1/2}}} \nonumber \\
        &= \iint_{\calO} \leftperp{f} \lt( \rho[k]\sfC_{\rm FP}f + \frac{1}{d} \lt(E[k] - \bfA_\va[k] \rt) \Delta_v f - j[k]\cdot \nabla_v f \rt) \mu^{-1} \,\dx\dv, \label{eq: N_g Dirichlet}
    \end{align}
    where the second line is due to the fact that $\pi\leftperp{\Upsilon_{\va}[k]}  = 0$; the last equality then follows since $\pi \leftperp{f} = 0$. We now estimate each term in the last line \eqref{eq: N_g Dirichlet}.
    
    To compute and estimate the first term in \eqref{eq: N_g Dirichlet}, we decompose $f = \leftperp{f} + \rho[f]\mu + j[f]\cdot v\mu + \tau[f] \frac{|v|^2-d}{\sqrt{2d}}\mu$, then use the eigenvalue identities
    \begin{equation*}
        \sfC_{\rm FP} (\mu, v_i\mu, (|v|^2-d)\mu) = (0, -v_i\mu, -2(|v|^2-d)\mu ) \\
    \end{equation*}
    in order to write
    \begin{equation}\label{eq: N_g Dirichlet Comp1}
    \begin{split}
        &\iint_{\calO} \leftperp{f} \rho[k] \sfC_{\rm FP}f \,\mu^{-1} \,\dx\dv \\
        &= \iint_{\calO} \rho[k] \lt( \leftperp{f} \sfC_{\rm FP} \leftperp{f} \, \mu^{-1} - \leftperp{f} j[f]\cdot v -2 \leftperp{f} \, \tau[f] \frac{|v|^2-d}{\sqrt{2d}} \rt) \,\dx\dv .
    \end{split}
    \end{equation}
    On one hand, by standard computations on the Fokker--Planck operator, we can immediately compute
    \begin{align*}
        \iint_{\calO} \rho[k] \leftperp{f} \sfC_{\rm FP} \leftperp{f} \, \mu^{-1} \,\dx\dv &= - \iint_{\calO} \rho[k] |\nabla_v (\leftperp{f} \mu^{-1})|^2 \mu \,\dx\dv  \\
        &\lesssim \|k\|_{L^\infty_\omega} \|\nabla_v (\leftperp{f} \mu^{-1})\|_{L^2_{\mu^{+1/2}}}^2 .
    \end{align*}
    On the other hand, the remaining terms on the right-hand side of \eqref{eq: N_g Dirichlet Comp1} vanish thanks to the identities
    \begin{equation}\label{eq: perp no moments}
        j[\leftperp{f}] = \tau[\leftperp{f}] = 0.
    \end{equation}
    Therefore, the first term in the right-hand side of \eqref{eq: N_g Dirichlet} is simply bounded as
    \begin{equation}\label{eq: N_g Dirichlet Comp1 Complete}
        \lt|\iint_{\calO} \leftperp{f} \rho[k] \sfC_{\rm FP} f \,\mu^{-1} \,\dx\dv \rt| \lesssim \|k\|_{L^\infty_\omega} \|\nabla_v(\leftperp{f} \mu^{-1})\|_{L^2_{\mu^{+1/2}}}^2 .
    \end{equation}

    Next, we move on to compute the term in \eqref{eq: N_g Dirichlet} containing $\Delta_v f$. We observe that
    \begin{equation} \label{eq: N_g Dirichlet Comp2}
    \begin{split}
        &\int_{\R^d} \leftperp{f} \Delta_v f \mu^{-1} \,\dv
        \\
        &= - \int_{\R^d} \nabla_v (\leftperp{f}\mu^{-1}) \cdot \nabla_v f\,\dv \\
        &= \int_{\R^d} \nabla_v (\leftperp{f} \mu^{-1}) \cdot ( \nabla_v (f\mu^{-1}) - v f \mu^{-1} ) \mu \,\dv \\
        &= \int_{\R^d} \nabla_v (\leftperp{f} \mu^{-1}) \cdot \lt( \nabla_v (\leftperp{f} \mu^{-1}) + j[f] + \sqrt{\frac{2}{d}} \tau[f] v \rt) \mu \,\dv - \int_{\R^d} \nabla_v(\leftperp{f} \mu^{-1}) \cdot (vf) \,\dv \\
        &= \|\nabla_v (\leftperp{f} \mu^{-1})\|_{L^2_{\mu^{+1/2}}}^2 - \int_{\R^d} (\leftperp{f}\mu^{-1}) \cdot \lt( j[f] v \mu + \sqrt{\frac{2}{d}} \tau[f] (d-|v|^2) \mu \rt) \\
        &\quad + \int_{\R^d} (\leftperp{f} \mu^{-1}) \nabla_v\cdot (vf) \,\dv \\
        &= \|\nabla_v (\leftperp{f} \mu^{-1})\|_{L^2_{\mu^{+1/2}}}^2 + j[f]\cdot \int_{\R^d} v\leftperp{f} \,\dv + \sqrt{\frac{2}{d}} \tau[f] \int_{\R^d} (|v|^2-d) \leftperp{f} \,\dv \\
        &\quad + \int_{\R^d} \leftperp{f} \mu^{-1} (df + v \cdot \nabla_v f) \,\dv \\
        &= \|\nabla_v(\leftperp{f} \mu^{-1})\|_{L^2_{\mu^{+1/2}}}^2 + d\|\leftperp{f}\|_{L^2_{\mu^{-1/2}}}^2 \\
        &\quad + \int_{\R^d} \leftperp{f} \mu^{-1} \lt( v\cdot \nabla_v \lt(\leftperp{f} + \rho[f] \mu + j[f]\cdot v\mu + \tau[f] \frac{|v|^2-d}{\sqrt{2d}} \mu \rt) \rt) \,\dv,
    \end{split}
    \end{equation}
    the last line following again from \eqref{eq: perp no moments}. It remains to estimate the integral in the last line. A direct computation first gives
    \begin{equation}\label{eq: N_g Dirichlet Comp 2_1}
    \begin{split}
        \lt|\int_{\R^d} \leftperp{f} \mu^{-1} v\cdot \nabla_v \leftperp{f} \,\dv \rt| &= \lt| \int_{\R^d} \nabla_v\cdot(v\mu^{-1}) |\leftperp{f}|^2 \,\dv \rt| \\
        &\le \int_{\R^d} (|v|^2 + d) |\leftperp{f}|^2 \mu^{-1} \,\dv \\
        &\lesssim \|\leftperp{f}\|_{L^2_{\mu^{-1/2}}(\R^d_v)}^2 + \|v\leftperp{f}\|_{L^2_{\mu^{-1/2}}(\R^d_v)}^2,
    \end{split}
    \end{equation}
    and for the remaining terms containing the macroscopic variables, we can estimate in the following way. For instance:
    \begin{equation} \label{eq: N_g Dirichlet Comp 2_2}
    \begin{split}
        &\int_{\R^d} \leftperp{f} \mu^{-1} v\cdot \nabla_v \lt(\tau[f] \frac{|v|^2-d}{\sqrt{2d}}\mu \rt) \,\dv \\
        &= \tau[f] \int_{\R^d} \leftperp{f} |v|^2 \lt( \sqrt{\frac{2}{d}}   - \frac{|v|^2 - d}{\sqrt{2d}}  \rt) \,\dv \\
        &\le |\tau[f]|  \|\leftperp{f}\|_{L^2_{\mu^{-1/2}}(\R^d_v)} \lt(\int_{\R^d} |v|^4 \lt( \sqrt{\frac{2}{d}}   - \frac{|v|^2 - d}{\sqrt{2d}}  \rt)^2 \mu\,\dv \rt)^{1/2} \\
        &\lesssim |\tau[f]|^2 + \|\leftperp{f}\|_{L^2_{\mu^{-1/2}}(\R^d_v)},
    \end{split}
    \end{equation}
    where we used the Cauchy--Schwarz inequality. Therefore, by inserting \eqref{eq: N_g Dirichlet Comp 2_1}--\eqref{eq: N_g Dirichlet Comp 2_2} into \eqref{eq: N_g Dirichlet Comp2}, we deduce that the second term in \eqref{eq: N_g Dirichlet} is bounded as
    \begin{equation}\label{eq: N_g Dirichlet Comp2 Complete}
        \begin{split}
            &\iint_{\calO} \leftperp{f} \lt(E[k] - \bfA_\va[k] \rt) \Delta_v f \, \mu^{-1} \,\dx\dv \\
            &\lesssim \|k\|_{L^\infty_\omega} (1+\|k\|_{L^\infty_\omega}) \|k\|_{L^\infty_\omega} \Bigg( \|\nabla_v (\leftperp{f} \mu^{-1})\|_{L^2_{\mu^{+1/2}}} + \|\leftperp{f}\|_{L^2_{\mu^{-1/2}}}^2 + \|v\leftperp{f} \|_{L^2_{\mu^{-1/2}}} \\
            &\hspace{6cm} + \|\rho[f]\|_{L^2(\Omega)}^2 + \|j[f]\|_{L^2(\Omega)}^2 + \|\tau[f]\|_{L^2(\Omega)}^2 \Bigg).
        \end{split}
    \end{equation}

    We now estimate the final term in \eqref{eq: N_g Dirichlet}. A direct computation shows
    \begin{align*}
        \iint_{\calO} \leftperp{f} j[k]\cdot \nabla_v f \,\mu^{-1} \,\dx\dv &= \iint_{\calO} \leftperp{f} j[k]\cdot \nabla_v \lt( \leftperp{f} + \rho[f]\mu + j[f]\cdot v\mu + \tau[f] \frac{|v|^2-d}{\sqrt{2d}}\mu \rt) \mu^{-1} \,\dx\dv \\
        &= \iint_{\calO} \frac{1}{2} j[k]\cdot \nabla_v |\leftperp{f}|^2\,\mu^{-1} \,\dx\dv \\
        &\quad + \iint_{\calO} \leftperp{f} j[k]\cdot \nabla_v \lt( \rho[f] \mu + j[f]\cdot v\mu + \tau[f] \frac{|v|^2-d}{\sqrt{2d}} \mu \rt) \mu^{-1} \,\dx\dv \\
        &=: I_1 + I_2.
    \end{align*}
    The Gauss--Ostrogradsky theorem and Young's inequality show
    \begin{equation*}
        |I_1| \lesssim \|k\|_{L^\infty_\omega} \iint_{\calO} (1+|v|^2) |\leftperp{f}|^2 \mu^{-1} \,\dx\dv \le \|k\|_{L^\infty_\omega} \lt( \|f\|_{L^2_{\mu^{-1/2}}}^2 + \|vf\|_{L^2_{\mu^{-1/2}}}^2 \rt).
    \end{equation*}
    For the terms in $I_2$ we estimate with almost the same method as in \eqref{eq: N_g Dirichlet Comp 2_2}, to establish
    \begin{equation*}
        |I_2| \lesssim \|k\|_{L^\infty_\omega} \lt( \|\rho[f]\|_{L^2(\Omega)}^2 +  \|j[f]\|_{L^2(\Omega)}^2 + \|\tau[f]\|_{L^2(\Omega)}^2  + \|\leftperp{f}\|_{L^2_{\mu^{-1/2}}}^2  \rt).
    \end{equation*}
    In all, it follows that the last term in \eqref{eq: N_g Dirichlet} is bounded as
    \begin{equation}\label{eq: N_g Dirichlet Comp3 Complete}
    \begin{split}
        &\lt| \iint_{\calO} \leftperp{f} j[k]\cdot \nabla_v f \, \mu^{-1} \,\dx\dv \rt| \\
        &\lesssim \|k\|_{L^\infty_\omega} \lt(\|f\|_{L^2_{\mu^{-1/2}}}^2 + \|vf\|_{L^2_{\mu^{-1/2}}}^2 + \|\rho[f]\|_{L^2(\Omega)}^2 +  \|j[f]\|_{L^2(\Omega)}^2  + \|\tau[f]\|_{L^2(\Omega)}^2 \rt).
    \end{split}
    \end{equation}
    Collecting \eqref{eq: N_g Dirichlet Comp1 Complete}--\eqref{eq: N_g Dirichlet Comp2 Complete}--\eqref{eq: N_g Dirichlet Comp3 Complete}, and keeping in mind $\|\rho[f],j[f],\tau[f]\|_{L^2(\Omega)}\lesssim \|f\|_{L^2_{\mu^{-1/2}}(\calO)}$, we obtain the estimate announced in the lemma.
\end{proof}

\begin{lemma}\label{lem: non sma22}
    For any $f\in \textnormal{Dom}(\leftperp{\Upsilon_\va[k]})$, if $\|k\|_{L^\infty}\ll 1$, then the following estimate holds:
    \begin{equation*}
    \begin{split}
        &(-\nabla_x^s \wt{u}[j[f]], \wt{M}(\leftperp{\Upsilon_{\va}[k]}  f) )_{L^2(\Omega)} \lesssim \|k\|_{L^\infty_\omega} \|f\|_{L^2_{\mu^{-1/2}}}^2, \\
        &(-\nabla_x u[\tau[f]], M[\leftperp{\Upsilon_{\va}[k]}  f])_{L^2(\Omega)} \lesssim \|k\|_{L^\infty_\omega} \|f\|_{L^2_{\mu^{-1/2}}}^2.
    \end{split}
    \end{equation*}
\end{lemma}
\begin{proof}
    We consider any $\omega_{\bullet}:\R^d\to \R_+$ such that $\nabla_v\omega_{\bullet}$ and $\Delta_v\omega_{\bullet}$ are of polynomial growth. Then we compute
    \begin{align*}
        &\intr \omega_{\bullet}(v) \leftperp{\Upsilon_{\va}[k]}  f \,\dv \\
        &= \intr \omega_{\bullet}(v) \lt( \frac{1}{d}\lt(E[k] - \bfA_\va[k] \rt) \Delta_v f - j[k]\cdot \nabla_v f + \rho[k] \nabla_v\cdot (vf) \rt) \,\dv \\
        &\quad + \intr \omega_{\bullet}(v) (\rho[k] j[f] - \rho[f] j[k]) \cdot v\mu \,\dv \\
        &\quad - \intr \omega_{\bullet}(v) \frac1d \lt( \lt(E[k] - \bfA_\va[k] \rt) \rho[f] - \rho[k] E[f] + j[k]\cdot j[f] \rt) (|v|^2-d) \mu  \,\dv \\
        &=: I_1 + I_2 + I_3.
    \end{align*} 
    For $I_1$, the Gauss--Ostrogradsky formula gives
    \begin{align*}
        |I_1| &= \lt|\intr -\nabla_v \omega_{\bullet}(v) \cdot \lt(\frac{1}{d}\lt(E[k] - \bfA_\va[k] \rt) \nabla_v f - j[k] f + \rho[k] vf \rt) \,\dv \rt|  \\
        &= \lt| \intr \Delta_v \omega_{\bullet} (v) \frac{1}{d}\lt(E[k]-\bfA_\va[k]\rt) f + \nabla_v \omega_{\bullet}(v) \cdot f (j[k] - v\rho[k]) \,\dv \rt| \\
        &\le C_{\omega_{\bullet},d} \|k\|_{L^\infty_\omega} ( 1 +  \|k\|_{L^\infty_\omega}) \|f\|_{L^2_v(\mu^{-1/2})}.
    \end{align*}
    For $I_2$, we simply have
    \begin{align*}
        |I_2| &= \lt|\intr \omega_{\bullet}(v) (\rho[k] j[f] - \rho[f] j[k]) \cdot v\mu \,\dv\rt| \\
        &\le C_{\omega_{\bullet}} \|k\|_{L^\infty_\omega} \Big(|j[f]| + |\rho[f]| \Big) \\
        &\le C_{\omega_{\bullet}} \|k\|_{L^\infty_\omega} \|f\|_{L^2_{\mu^{-1/2}}},
    \end{align*}
    and the estimate for $I_3$ is analogous. Combining these lead to
    \begin{equation*}
        \|\wt{M}[\leftperp{\Upsilon_{\va}[k]}  f]\|_{L^2(\Omega)} + \|M[\leftperp{\Upsilon_{\va}[k]}  f]\|_{L^2(\Omega)} \lesssim \|k\|_{L^\infty_\omega} \|f\|_{L^2_{\mu^{-1/2}}}.
    \end{equation*}
    We deduce, using \eqref{eq: poi lax}--\eqref{eq: lame lax}, that
    \begin{align*}
        (-\nabla_x^s \wt{u}[j[f]], \wt{M}(\leftperp{\Upsilon_{\va}[k]} f) )_{L^2(\Omega)} &\lesssim \|\nabla_x^s \wt{u}[j[f]]\|_{L^2(\Omega)} \|\wt{M}[\leftperp{\Upsilon_{\va}[k]} f]\|_{L^2(\Omega)} \\
        &\lesssim \|j[f]\|_{L^2(\Omega)}  \|k\|_{L^\infty_\omega} \|f\|_{L^2_{\mu^{-1/2}}} \\
        &\lesssim \|k\|_{L^\infty_\omega} \|f\|_{L^2_{\mu^{-1/2}}}^2, \\
        (-\nabla_x u[\tau[f]], M[\leftperp{\Upsilon_{\va}[k]} f])_{L^2(\Omega)} &\lesssim \|\nabla_x u[\tau[f]]\|_{L^2(\Omega)} \|M[\leftperp{\Upsilon_{\va}[k]} f]\|_{L^2(\Omega)} \\
        &\lesssim \|\tau[f]\|_{L^2(\Omega)} \|k\|_{L^\infty_\omega} \|f\|_{L^2_{\mu^{-1/2}}} \\
        &\lesssim \|k\|_{L^\infty_\omega} \|f\|_{L^2_{\mu^{-1/2}}}^2.
    \end{align*}
\end{proof}


\subsubsection{Proof of Proposition \ref{prop: L_g hypo}}

Gathering Lemmas \ref{M/Lem/ Cfull coer}--\ref{lem: non sma11}--\ref{lem: non sma22}, we obtain that
\begin{align*}
    \doubleparantheses{-\sfF_\va[k] f}{f} &\gtrsim \|f\|_{L^2_{\mu^{-1/2}}}^2 + \|\nabla_v (\leftperp{f} \mu^{-1})\|_{L^2_{\mu^{+1/2}}} + \|v \leftperp{f}\|_{L^2_{\mu^{-1/2}}}  \\
    &\quad - C\|k\|_{L^\infty_\omega} \Big( \|f\|_{L^2_{\mu^{-1/2}}}^2 + \|\nabla_v (\leftperp{f} \mu^{-1})\|_{L^2_{\mu^{+1/2}}} + \|v \leftperp{f}\|_{L^2_{\mu^{-1/2}}} \Big). 
\end{align*}
for every $\|k\|_{L^\infty_\omega}\ll 1$ and $f\in \textnormal{Dom}(\sfF_\va[k])\cap \calC$ satisfying the boundary condition \eqref{eq: specular}. We just choose $c_2>0$ so small that $c\le c_2$ and $k\in \calX_{\omega,c}$ imply the negative terms can be absorbed. This concludes the proof.


\subsection{Extension of the evolution system $\calU_{\sfF_\va[k]}$ to $L^\infty_\omega$} \label{Sec/ Ext}

Gathering the results of Sections \ref{S/ F WP}--\ref{S/ F Hypo}, we now establish the extension of the evolution system $\calU_{\sfF_\va[k]}$ to $L^\infty_\omega$ and also its decay.

\begin{proposition} \label{Prop/ decay}
    There exists $c_\star>0$ small enough so that if $c\le c_\star$ and $k\in \calX_{\omega,c}$, then the evolution system as constructed in Proposition \ref{P/ F WP} satisfies the following decay estimate
    \begin{equation*}
        \forall h_s \in L^\infty_\omega \cap \calC, \quad  \|\calU_{\sfF_\va[k]}(t,s) h_s\|_{L^\infty_\omega} \lesssim  e^{-\iota'(t-s)} \|h_s\|_{L^\infty_\omega},
    \end{equation*}
    up to constructive constants and some $\iota'>0$ which are independent of $\va$.
\end{proposition}

The proof of Proposition \ref{Prop/ decay} is quite long and so we split it into several lemmas which provide the appropriate necessary estimates.

\subsubsection{Proof of Proposition \ref{Prop/ decay}}

(Step 1: Collecting previous estimates.) At this moment, we know from Proposition \ref{P/ F WP} that for $c\le \min\{c_1, c_{\mu^{-1/2}}\}$, the evolution equation corresponding to $\sfF_\va[k]$ is well-posed in $C([0,T];L^2_{\mu^{-1/2}}(\calO))$.

Consequently, combining with Proposition \ref{prop: L_g hypo}, we have that for all $c\le \min\{c_1, c_{\mu^{-1/2}} ,c_2\}$, for every $\mathfrak{f}_s\in L^2_{\mu^{-1/2}}\cap \calC$ there holds
\begin{equation}
\label{Eq/ Ref Hypo}
\begin{split}
    \| \calU_{\sfF_\va[k]}(t,s) \mathfrak{f}_s \|_{L^2_{\mu^{-1/2}}(\calO)} \lesssim \frakp_{\lambda}(t-s) \|\mathfrak{f}_s\|_{L^2_{\mu^{-1/2}}(\calO)},
\end{split}
\end{equation}
where for $\lambda>0$ we denoted
\begin{equation*}
    \frakp_{\lambda}(t) := e^{-\lambda t}.
\end{equation*}

On the other hand, in regard of the dissipative part of the operator $\sfH_\va[k]$, we know from Proposition \ref{M:Prop: h diss} that for each $\omega_1\in \calW$, there exists $c_{\omega_1}>0$ and $M_{\omega_1},R_{\omega_1}$ such that $c\le \min\{c_1,c_{\omega_1}\}$, $M\ge M_{\omega_1}$, and $R\ge R_{\omega_1}$ imply for some $\lambda'>0$
\begin{equation} \label{Eq/ H dissipative}
    \forall p\in [2,\infty], \quad  \|\calU_{\sfH_\va[k]}(t,s)\|_{\scrB(L^p_{\omega_1})} \lesssim \frakp_{\lambda'}(t-s).
\end{equation}

In similar nature, by applying 
\eqref{M/Eq/ L2 Linfty} of Proposition \ref{M/P/ Ultra} with the specific choice of $\bfw = \mu^{-1/2} \approx e^{|v|^2/4}$, we know that for $c\le \min\{c_1, c_{\mu^{-1/2}}, c_{\mu^{-1/2} \la v \ra^{\theta-2} }\}$, it holds for any weight $\omega_{\star}\in \calW$ that
\begin{equation} \label{Eq/ Ultra}
\begin{split}
    &\| \calU_{\sfH_\va[k]}(t,s) \|_{\scrB(L^2_{\mu^{-1/2}},L^\infty_{\omega_\star})} \lesssim \lt\|\omega_\star \mu^{1/2} \la v \ra^{2-\theta} \rt\|_{L^\infty(\R^d)} \frakq(t-s) ,\\
    &\frakq(t) := \frac{ \frakp_{\lambda'} (t) }{ \min\{t^{\vartheta}, 1\} } = e^{- \lambda' t}  \max\{t^{-\vartheta} , 1\},
\end{split}
\end{equation}
at least when $M,R$ are sufficiently large (with respect to the choice of $\mu^{-1/2}$).

From hereon, we fix the exponential weight $\omega_\star$, say for instance $\omega_\star := e^{|v|^2/8}$, so that the right-hand side of $\eqref{Eq/ Ultra}_1$ is finite. In this way, we can fix $M,R>1$ so large that, whenever $c\le c_\star$ and $k\in \calX_{\omega,c}$, (i) \eqref{Eq/ H dissipative} holds for all three choices of $\omega_1 = \omega, \mu^{-1/2},\omega_\star$ and (ii) \eqref{Eq/ Ultra} holds. Finally, we denote $\iota := \min\{\lambda,\lambda'\}$ and consider arbitrary numbers $\iota' \in (0,\iota)$ and $\iota''\in (\iota',\iota)$.

Now denoting again by $\Pi$ the orthogonal projection onto $\calC$ (in the space $L^2_{\mu^{-1/2}}(\calO)$), we introduce the notations
\begin{align*}
    \calU_{\sfF_\va[k]}^\perp &:= \calU_{\sfF_\va[k]} (I - \Pi) = (I - \Pi) \calU_{\sfF_\va[k]} =: \leftperp{\calU_{\sfF_\va[k]}}, \\
    \calU_{\sfH_\va[k]}^\perp &= \calU_{\sfH_\va[k]} (I - \Pi), \\
    \leftperp{\calU_{\sfH_\va[k]}} &:= (I - \Pi) \calU_{\sfH_\va[k]}.
\end{align*}
It is easy to observe that $\Pi\in \scrB(L^\infty_{\omega_1})$ for every $\omega_1\in \calW_{\mu^{-1/2}}$, and therefore \eqref{Eq/ H dissipative} also implies
\begin{equation}\label{Eq/ Perps dissipative}
     \|\leftperp{\calU_{\sfH_\va[k]}}\|_{\scrB(L^\infty_{\omega_1})} + \|\calU_{\sfH_\va[k]}^{\perp}\|_{\scrB(L^\infty_{\omega_1})} \lesssim \frakp_{\lambda}(t-s), \quad \forall \omega_1 \in \lt\{\omega, \omega_\star\rt\}.
\end{equation}

(Step 2: The Duhamel formula and iteration.) For given two-parameter families of operators $(\calU_1(t,s))_{0\le s \le t}$ and $(\calU_2(t,s))_{0\le s \le t}$, we define
\begin{equation*}
    \calU_1 \star \calU_2 (t,s) := \int_s^t \calU_1(t,\Theta) \calU_2(\Theta,s) \,\tnd\Theta.
\end{equation*}
In addition, for a two-parameter family of operators $(\calU(t,s))_{0\le s \le t}$ and a one-parameter family of operators $(\sfG(t))_{t\ge 0}$, we define the multiplication operation
\begin{align*}
    (\calU \sfG)(t,s) := \calU(t,s) \sfG(s), \quad (\sfG \calU)(t,s) := \sfG(t) \calU(t,s).
\end{align*}
Then Duhamel's formula provides that
\begin{equation}\label{Eq/ Duhamel}
\begin{split}
    \calU_{\sfF_\va[k]}^\perp &= \calU_{\sfH_\va[k]}^\perp + (\calU_{\sfH_\va[k]} \sfG_\va[k]) \star \calU_{\sfF_\va[k]}^\perp \\
    &= \leftperp{\calU_{\sfH_\va[k]}} + \calU_{\sfF_\va[k]}^\perp \star (\sfG_\va[k] \calU_{\sfH_\va[k]}).
\end{split}
\end{equation}
Iterating \eqref{Eq/ Duhamel}, we obtain for every $n\in \bbN$ that
\begin{equation}\label{Eq/ Iter 1}
    \calU_{\sfF_\va[k]}^\perp = \calU_{\sfH_\va[k]}^\perp + \sum_{i=1}^{n-1} (\calU_{\sfH_\va[k]} \sfG_\va[k])^{\star i} \star \calU_{\sfH_\va[k]}^\perp + (\calU_{\sfH_\va[k]} \sfG_\va[k])^{\star n} \star \calU_{\sfF_\va[k]}^\perp ,
\end{equation}
where we denoted the $n$-fold convolution as $(\calU)^{\star n} := \calU \star \cdots \star \calU$. We also have the alternate expression from \eqref{Eq/ Duhamel}, for every $m\in \bbN$
\begin{equation}\label{Eq/ Iter 2}
    \calU_{\sfF_\va[k]}^\perp
    = \leftperp{\calU_{\sfH_\va[k]}} + \sum_{i=1}^{m-1} \leftperp{\calU_{\sfH_\va[k]}} \star (\sfG_{\va[k]} \calU_{\sfH_\va[k]})^{\star i} + \calU_{\sfF_\va[k]}^\perp \star (\sfG_\va[k] \calU_{\sfH_\va[k]})^{\star m}.
\end{equation}
Putting \eqref{Eq/ Iter 1}--\eqref{Eq/ Iter 2} together with the choice of $m=2$ (which is enough for our purposes), we obtain the Dyson-like series
\begin{equation} \label{Eq/ Final Iter}
\begin{split}
    \calU_{\sfF_\va[k]}^\perp &= \calU_{\sfH_\va[k]} + \sum_{i=1}^{n-1} (\calU_{\sfH_\va[k]}\sfG_\va[k])^{\star i} \star \calU_{\sfH_\va[k]}^\perp
    + (\calU_{\sfH_\va[k]} \sfG_\va[k])^{\star n} \star \leftperp{\calU_{\sfH_\va[k]}} \\
    &\quad + (\calU_{\sfH_\va[k]} \sfG_\va[k])^{\star n} \star \leftperp{\calU_{\sfH_\va}[k]} \star (\sfG_\va[k] \calU_{\sfH_\va}[k]) \\
    &\quad + (\calU_{\sfH_\va[k]} \sfG_\va[k])^{\star n} \star \calU_{\sfF_\va[k]}^\perp \star (\sfG_\va[k] \calU_{\sfH_\va[k]})^{\star 2}.
\end{split}
\end{equation}

(Step 3: $L^\infty_\omega \to L^\infty_\omega$ estimates.) From \eqref{Eq/ Final Iter} it obviously holds
\begin{align*}
    \|\calU_{\sfF_\va[k]}^\perp\|_{\scrB(L^\infty_\omega)} &\le \|\calU_{\sfH_\va[k]}\|_{\scrB(L^\infty_\omega)} + \sum_{i=1}^{n-1} \lt\|(\calU_{\sfH_\va[k]} \sfG_\va[k])^{\star i} \star \calU_{\sfH_\va[k]}^\perp \rt\|_{\scrB(L^\infty_\omega)} \\
    &\quad + \lt\| (\calU_{\sfH_\va[k]} \sfG_\va[k])^{\star n} \star \leftperp{\calU_{\sfH_\va[k]}} \rt\|_{\scrB(L^\infty_\omega)} \\
    &\quad + \lt\| (\calU_{\sfH_\va[k]} \sfG_\va[k])^{\star n} \star \leftperp{\calU_{\sfH_\va}[k]} \star (\sfG_\va[k] \calU_{\sfH_\va}[k]) \rt\|_{\scrB(L^\infty_\omega)} \\
    &\quad + \lt\|(\calU_{\sfH_\va[k]} \sfG_\va[k])^{\star n} \star \calU_{\sfF_\va[k]}^\perp \star (\sfG_\va[k] \calU_{\sfH_\va[k]})^{\star 2} \rt\|_{\scrB(L^\infty_\omega)} \\
    &=: I_1 + I_2 + I_3 + I_4 + I_5.
\end{align*}

Now beginning with $I_1$, thanks to \eqref{Eq/ H dissipative} it is immediate that
\begin{equation*}
    I_1 \lesssim \frakp_{\iota}(t-s).
\end{equation*}

For $I_2,I_3,I_4$, we claim
\begin{lemma}
    Under the assumptions of Proposition \ref{Prop/ decay}, there holds
    \begin{equation*}
        I_2 + I_3 + I_4 \lesssim \frakp_{\iota'}(t-s).
    \end{equation*}
\end{lemma}
\begin{proof}
We begin by recalling the mere definition
\begin{equation}\label{Eq/ Mere}
    \sfG_\va[k] = \sfP_\va[k] + M \psi_R (v),
\end{equation}
where $\sfP_\va[k]$ is the projection part as written in \eqref{Eq/ P_va}. On the one hand, thanks to $\|k\|_{L^\infty_\omega}\ll 1$, the very definition of $\calW$, and using the fact that $L^\infty_\omega(\R^d) \hookrightarrow L^1_{\la v \ra^2}(\R^d)$, we can easily establish that
\begin{equation*}
    \|\sfP_\va[k]\|_{\scrB(L^\infty_\omega)} \lesssim 1.
\end{equation*}
This inequality holds up to a constant that is independent of $\va$, thanks to the fact that $|\bfA_\va[k]| \lesssim \|k\|^2_{L^\infty_\omega}$. On the other hand, having now fixed $M,R$ sufficiently large, $M\psi_R$ is nothing but a compactly supported and bounded function, and so
\begin{equation}\label{Eq/ G bounded}
    \|\sfG_\va[k]\|_{\scrB(L^\infty_\omega)} \lesssim 1.
\end{equation}
Using \eqref{Eq/ G bounded}, we may therefore deduce that for every $i\in \bbN$ there holds
\begin{align*}
    &\lt\|(\calU_{\sfH_\va[k]} \sfG_\va[k])^{\star i} \star \calU_{\sfH_\va[k]}^\perp (t,s) \rt\|_{\scrB(L^\infty_\omega)}  \\
    &\lesssim \int_s^t \lt\| (\calU_{\sfH_\va[k]} \sfG_\va[k])^{\star i} (t,\Theta) \rt\|_{\scrB(L^\infty_\omega)}  \|\calU_{\sfH_\va[k]}^\perp (\Theta,s)\|_{\scrB(L^\infty_\omega)} \,\tnd\Theta \\
    &\lesssim \int_s^t \frakp_{\iota}(t-\Theta) \frakp_{\iota}(\Theta-s) \,\tnd\Theta \\
    &\lesssim \frakp_{\iota'}(t-s),
\end{align*}
by utilizing \eqref{Eq/ Perps dissipative}. Hence
\begin{equation*}
    I_2 \lesssim \frakp_{\iota'}(t-s).
\end{equation*}
In basically the same way, we obtain
\begin{equation*}
    I_3 + I_4 \lesssim \frakp_{\iota'}(t-s).
\end{equation*}
\end{proof}

At last, we consider $I_5$. We begin by listing some more basic estimates for the operator $\sfG_\va[k]$:

\begin{lemma}
    Under the assumptions of Proposition \ref{Prop/ decay}, the operator $\sfG_\va[k]$ satisfies
    \begin{align}
        &\|\sfG_\va[k]\|_{\scrB( L^\infty_{\omega}, L^2_{\mu^{-1/2}} )} \lesssim 1,  \label{Eq/ omg mu}\\
        &\|\sfG_\va[k]\|_{\scrB(L^\infty_{\omega_{\star}} , L^\infty_{\omega} )} \lesssim 1, \label{Eq/ omg star omg} \\
        &\|\sfG_\va[k]\|_{\scrB(L^\infty_{\omega_\star})} \lesssim 1 \label{Eq/ omg star}, \\
        &\|\sfG_\va[k]\|_{\scrB( L^2_{\mu^{-1/2}} )} \lesssim 1, \label{Eq/ mu} ,
    \end{align}
    up to some constants independent of time and $\va$.
\end{lemma}
\begin{proof}
    Recalling \eqref{Eq/ Mere}, we note that obviously $M\psi_R\in \scrB(L^\infty_\omega, L^2_{\mu^{-1/2}})$, whereas it holds $\|\sfP_\va[k]\|_{\scrB(L^\infty_\omega,L^2_{\mu^{-1/2}})} \lesssim 1$ independently of $\va$ because $|\bfA_\va[k]|\le C\|k\|_{L^\infty_\omega}^2$. Consequently \eqref{Eq/ omg mu} holds. 

    For \eqref{Eq/ omg star omg}, we use the very definition of $\calW$, and in particular observe that
\begin{equation*}
    \la v \ra^2 \mu \omega \in L^\infty(\R^d).
\end{equation*}
Using also the fact that (since $\omega_\star$ is exponential)
\begin{equation*}
    |\rho[h]| + |j[h]| + |E[h]| \lesssim \|h\|_{L^\infty_{\omega\star}},
\end{equation*}
we can straightforwardly prove that the projection part of the operator $\sfP_\va[k]$ in \eqref{Eq/ P_va} satisfies
\begin{equation*}
    \|\sfP_\va[k]\|_{\scrB(L^\infty_{\omega_\star}, L^\infty_{\omega})} \lesssim 1.
\end{equation*}
Since $M\psi_R$ merely truncates the support, we also clearly have 
\begin{equation*}
    \forall h\in L^\infty_{\omega_\star}, \quad |M \psi_R h \omega| \lesssim |h \omega_\star| \mathbf{1}_{|v|\le 2R} \frac{\omega}{\omega_\star} \lesssim_R \|h\|_{L^\infty_{\omega_\star}} .    
\end{equation*}
Thus, we conclude to
\begin{equation*}
    \|\sfG_\va[k]\|_{\scrB(L^\infty_{\omega_\star}, L^\infty_{\omega})} \lesssim 1,    
\end{equation*}
proving \eqref{Eq/ omg star omg}. The proof of \eqref{Eq/ omg star} is almost the same and straightforward.

At last, the estimate \eqref{Eq/ mu} also follows similarly. For instance, looking at just the very last term of $\sfP_\va[k] h$ in \eqref{Eq/ P_va}, we can compute
\begin{align*}
    \| (1+\rho[k])E[h] (|v|^2-d)\mu \|_{L^2_{\mu^{-1/2}}} &\lesssim \iint_{\calO} |E[h]|^2 \la v \ra^4 \mu \,\dx\dv \\
    &= \int_{\Omega} |E[h]|^2 \lt(\int_{\R^d} \la v \ra^4 \mu \,\dv \rt) \dx \\
    &= C \int_{\Omega} |E[h]|^2 \,\dx \\
    &\lesssim \|h\|_{L^2_{\mu^{-1/2}}}^2,
\end{align*}
the last line following by Cauchy--Schwarz. All of the other terms in \eqref{Eq/ P_va} can be handled in the same way, and we obtain that
\begin{equation*}
    \|\sfP_\va[k]\|_{\scrB(L^2_{\mu^{-1/2}})}\lesssim 1.
\end{equation*}
We easily conclude to \eqref{Eq/ mu} and this completes the proof of the lemma.
\end{proof}

Now, we break down the convolution in $I_5$ into
\begin{equation}\label{Eq/ Three}
    \lt\| (\calU_{\sfH_\va[k]} \sfG_\va[k])^{\star n} \rt\|_{\scrB(L^2_{\mu^{-1/2}}, L^\infty_\omega)} , \quad 
    \lt\| \calU_{\sfF_\va[k]}^\perp \rt\|_{\scrB(L^2_{\mu^{-1/2}}) } , \quad \lt\| (\sfG_\va[k] \calU_{\sfH_\va[k]})^{\star 2} \rt\|_{\scrB(L^\infty_\omega, L^2_{\mu^{-1/2}})} 
\end{equation}
and look at each portion separately. Starting from the right, we put \eqref{Eq/ omg mu} and \eqref{Eq/ H dissipative} together to deduce
\begin{align*}
    \lt\| \sfG_\va[k] \calU_{\sfH_\va[k]} \rt\|_{\scrB(L^\infty_\omega, L^2_{\mu^{-1/2}})} &\lesssim \frakp_{\iota}(t-s).
\end{align*}
However we can also put together \eqref{Eq/ G bounded} with \eqref{Eq/ H dissipative}, and thus
\begin{equation*}
    \lt\| \sfG_\va[k] \calU_{\sfH_\va[k]} \rt\|_{\scrB(L^\infty_\omega)} \lesssim \frakp_{\iota}(t-s).
\end{equation*}
In conclusion,
\begin{equation}\label{Eq/ projector bound}
\begin{split}
    \lt\| (\sfG_\va[k] \calU_{\sfH_\va[k]})^{\star 2} \rt\|_{\scrB(L^\infty_\omega, L^2_{\mu^{-1/2}})} &\lesssim \frakp_{\iota''}(t-s).
\end{split}
\end{equation}

For the middle term in \eqref{Eq/ Three}, we just use the bound in \eqref{Eq/ Ref Hypo}:
\begin{equation}\label{Eq/ Hypo again}
    \lt\| \calU_{\sfF_\va[k]}^\perp \rt\|_{\scrB(L^2_{\mu^{-1/2}})} \lesssim \frakp_{\lambda}(t-s) \lesssim \frakp_{\iota}(t-s).
\end{equation}
For the remaining left term in \eqref{Eq/ Three}, 
in the spirit of \cite{CarMis24,GuaMisMou17,MisMou16} we shall prove that there exists $n\in \bbN$ large enough so that 
\begin{equation}\label{Eq/ real n conv}
    \lt\|(\calU_{\sfH_\va[k]} \sfG_\va[k])^{\star n} \rt\|_{\scrB(L^2_{\mu^{-1/2}}, L^\infty_\omega)} \lesssim \frakp_{\iota''}(t-s).
\end{equation}
Toward this end, we can first write (using that $\|\sfG_\va[k]\|_{\scrB(L^\infty_\omega)} \lesssim 1$)
\begin{equation}\label{Eq/ nconv First}
\begin{split}
    &\lt\|(\calU_{\sfH_\va[k]} \sfG_\va[k])^{\star n} \rt\|_{\scrB(L^2_{\mu^{-1/2}}, L^\infty_\omega)} \\
    &\lesssim \lt\| (\sfU_{\sfH_\va[k]} \sfG_\va[k]) \star (\sfU_{\sfH_\va[k]} \sfG_\va[k])^{\star (n-2)} \star \calU_{\sfH_\va[k]}  \rt\|_{\scrB(L^2_{\mu^{-1/2}}, L^\infty_\omega )} \\
    &\lesssim \int_s^t \int_s^{\Theta} \lt\|(\calU_{\sfH_\va[k]} \sfG_\va[k])(t,\Theta)\rt\|_{\scrB( L^\infty_{\omega_\star} , L^\infty_{\omega} )}  \lt\| (\calU_{\sfH_\va[k]} \sfG_\va[k])^{\star(n-2)} (\Theta,\Theta')\rt\|_{\scrB(L^2_{\mu^{-1/2}} , L^\infty_{\omega_\star} ) } \\
    &\hspace{2cm} \times \lt\|\calU_{\sfH_\va[k]} (\Theta',s) \rt\|_{\scrB(L^2_{\mu^{-1/2}})} \,\tnd\Theta' \tnd\Theta. 
\end{split}
\end{equation}
Thanks to \eqref{Eq/ omg star omg} and \eqref{Eq/ H dissipative}, the first term on the right-hand side of \eqref{Eq/ nconv First} may be bounded as
\begin{align*}
    \|(\calU_{\sfH_\va[k]} \sfG_\va[k])(t,\Theta)\|_{\scrB(L^\infty_{\omega_\star}, L^\infty_{\omega})} \lesssim \|\calU_{\sfH_\va[k]}(t,\Theta)\|_{\scrB(L^\infty_{\omega})} \lesssim \frakp_{\iota}(t-\Theta).
\end{align*}
For the second term on the right-hand side of \eqref{Eq/ nconv First}, we claim that:

\begin{lemma}\label{Lem/ n-2}
There exists $n\in \bbN$ large enough so that
\begin{equation}\label{Eq/ n conv goal}
    \lt\|(\calU_{\sfH_\va[k]} \sfG_\va[k])^{\star (n-2)} \star \calU_{\sfH_\va[k]} \rt\|_{\scrB(L^2_{\mu^{-1/2}}, L^\infty_{\omega_\star})} \lesssim \frakp_{\iota''}(t-s) .
\end{equation}
\end{lemma}

\begin{proof}
Let us begin with an estimate for the case $n=3$. We write
\begin{equation}\label{Eq/ split}
\begin{split}
    &(\calU_{\sfH_\va[k]} \sfG_\va[k]) \star \calU_{\sfH_\va}[k]  \\
    &= \int_{s}^{\frac{t+s}{2}} \calU_{\sfH_\va[k]}(t,\Theta) \sfG_\va[k](\Theta) \calU_{\sfH_\va[k]}(\Theta,s) \,\tnd \Theta + \int_{\frac{t+s}{2}}^{t} \calU_{\sfH_\va[k]}(t,\Theta) \sfG_\va[k](\Theta) \calU_{\sfH_\va[k]}(\Theta,s) \,\tnd \Theta ,
\end{split}
\end{equation}
in order to isolate the possibly non-integrable factor arising in the ultracontractive estimate \eqref{Eq/ Ultra}. For the first integral here, we may estimate
\begin{align*}
    &\lt\| \int_{s}^{\frac{t+s}{2}} \calU_{\sfH_\va[k]}(t,\Theta) \sfG_\va[k](\Theta) \calU_{\sfH_\va[k]}(\Theta,s) \,\tnd \Theta \rt\|_{\scrB(L^2_{\mu^{-1/2}}, L^\infty_{\omega_\star})}   \\
    &\lesssim \int_s^{\frac{t+s}{2}} \|\calU_{\sfH_\va[k]}(t,\Theta)\|_{\scrB(L^2_{\mu^{-1/2}} , L^\infty_{\omega_\star} ) } \|\sfG_\va[k](\Theta) \|_{\scrB( L^2_{\mu^{-1/2}}  ) }  \|\calU_{\sfH_\va[k]} (\Theta, s) \|_{\scrB(L^2_{\mu^{-1/2}})} \, \tnd \Theta \\
    &\lesssim \int_s^{\frac{t+s}{2}} \frakq (t-\Theta) \frakp_{\iota}(\Theta-s) \,\tnd\Theta \\
    &= \frakp_{\iota}(t-s) \int_s^{\frac{t+s}{2}} \max\{ (t-\Theta)^{-\vartheta} ,1\} \,\tnd\Theta .
\end{align*}
If $|t-s|\le 2$ then
\begin{align*}
    \int_{s}^{\frac{t+s}{2}} \max \{ (t-\Theta)^{-\vartheta} , 1 \}  \, \tnd \Theta &= \int_{s}^{t-1} \frac{1}{(t-\Theta)^{\vartheta} } \,\tnd\Theta + 1 \\
    &= \frac{1}{\vartheta - 1} \lt( \frac{1}{(t-s)^{\vartheta - 1}} - 1 \rt) + 1 \\
    &\lesssim \frac{1}{(t-s)^{\vartheta - 1}} + 1,
\end{align*}
whereas if $|t-s|>2$ then
\begin{equation*}
    \int_{s}^{\frac{t+s}{2}} \max \{ (t-\Theta)^{-\vartheta} , 1 \}  \, \tnd \Theta = \int_s^{\frac{t+s}{2}} \,\tnd\Theta = \frac{t-s}{2},
\end{equation*}
and therefore we conclude to
\begin{align*}
    &\lt\| \int_{s}^{\frac{t+s}{2}} \calU_{\sfH_\va[k]}(t,\Theta) \sfG_\va[k](\Theta) \calU_{\sfH_\va[k]}(\Theta,s) \,\tnd \Theta \rt\|_{\scrB(L^2_{\mu^{-1/2}}, L^\infty_{\omega_\star})}   \\
    &\lesssim \frakp_{\iota}(t-s) \lt(1 + \max\{ (t-s)^{-(\vartheta - 1)} , 2 \}  \rt) \lesssim \frakp_{\iota}(t-s) \lt(1 + \max\{ (t-s)^{-(\vartheta - 1)} , 1 \}  \rt).
\end{align*}
For the second integral in \eqref{Eq/ split}, we reverse the order to write (using \eqref{Eq/ omg star})
\begin{align*}
    &\lt\| \int_{\frac{t+s}{2}}^t \calU_{\sfH_\va[k]}(t,\Theta) \sfG_\va[k](\Theta) \calU_{\sfH_\va[k]}(\Theta,s) \,\tnd \Theta \rt\|_{\scrB(L^2_{\mu^{-1/2}}, L^\infty_{\omega_\star})}   \\
    &\lesssim \int_{\frac{t+s}{2}}^t \|\calU_{\sfH_\va[k]}(t,\Theta)\|_{\scrB( L^\infty_{\omega_\star} )} \|\sfG_\va[k](\Theta)\|_{\scrB(L^\infty_{\omega_\star} )} \|\calU_{\sfH_\va[k]}(\Theta,s)\|_{\scrB(L^2_{\mu^{-1/2}}, L^\infty_{\omega_\star}) } \,\tnd\Theta  \\
    &\lesssim \int_{\frac{t+s}{2}}^t \frakp_{\iota}(t-\Theta) \frakq(\Theta-s) \,\tnd\Theta \\
    &= \frakp_{\iota}(t-s) \int_{\frac{t+s}{2}}^t \max\{(\Theta-s)^{-\vartheta} , 1 \}  \,\tnd\Theta, 
\end{align*}
and the same argument provides
\begin{align*}
    &\lt\| \int_{\frac{t+s}{2}}^t \calU_{\sfH_\va[k]}(t,\Theta) \sfG_\va[k](\Theta) \calU_{\sfH_\va[k]}(\Theta,s) \,\tnd \Theta \rt\|_{\scrB(L^2_{\mu^{-1/2}}, L^\infty_\omega)}  \\
    &\lesssim \frakp_{\iota}(t-s) \lt(1 + \max\{(t-s)^{-(\vartheta - 1)}, 1 \} \rt).
\end{align*}
In all we have shown that for the case $n=3$, there holds
\begin{equation}\label{Eq/ n=3}
    \lt\|(\calU_{\sfH_\va[k]} \sfG_\va[k]) \star \calU_{\sfH_\va[k]} \rt\|_{\scrB(L^2_{\mu^{-1/2}}, L^\infty_{\omega_\star})} \lesssim \frakp_{\iota}(t-s) \lt(1 + \max\{(t-s)^{-(\vartheta - 1)}, 1\} \rt).
\end{equation}

If $\vartheta - 1 > 0$, we repeat this argument for the case $n=4$. We write similarly
\begin{align*}
    &\|(\calU_{\sfH_\va[k]}\sfG_\va[k])^{\star 2} \star \calU_{\sfH_\va[k]}\|_{\scrB(L^2_{\mu^{-1/2}}, L^\infty_{\omega_\star})} \\
    &\lesssim \int_{s}^{\frac{t+s}{2}} \|(\calU_{\sfH_\va[k]}\sfG_\va[k]) \star \calU_{\sfH_\va[k]} (t,\Theta)\|_{\scrB(L^2_{\mu^{-1/2}} , L^\infty_{\omega_\star} ) } \|(\sfG_\va[k] \, \calU_{\sfH_\va[k]}) (\Theta,s) \|_{\scrB(L^2_{\mu^{-1/2}}) } \,\tnd\Theta \\
    &\quad + \int_{\frac{t+s}{2}}^t \|(\calU_{\sfH_\va[k]} \sfG_\va[k]) (t, \Theta)\|_{\scrB(L^\infty_{\omega_\star}, L^\infty_{\omega}) } \|(\calU_{\sfH_\va[k]}\sfG_\va[k]) \star \calU_{\sfH_\va[k]} (\Theta, s)\|_{\scrB(L^2_{\mu^{-1/2}}, L^\infty_{\omega_\star})} \,\tnd\Theta \\
    &\lesssim \int_{s}^{\frac{t+s}{2}} \frakp_{\iota}(t-\Theta) \lt(1 + \max\{(t-\Theta)^{-(\vartheta-1)}, 1\} \rt) \frakp_{\iota}(\Theta-s) \,\tnd\Theta \\
    &\quad + \int_{\frac{t+s}{2}}^t \frakp_{\iota}(t-\Theta) \frakp_{\iota}(\Theta-s) \lt(1 + \max\{(\Theta-s)^{-(\vartheta-1)} , 1\} \rt) \, \tnd\Theta \\
    &\lesssim \frakp_{\iota}(t-s)  \int_{s}^{\frac{t+s}{2}} \lt(1 + \max\{(t-\Theta)^{-(\vartheta - 1)}, 1\} \rt) \tnd\Theta \\
    &\lesssim \frakp_{\iota}(t-s) \lt(1 + \max\{(t-s)^{-(\vartheta - 2)}, 1 \} \rt),
\end{align*}
by using \eqref{Eq/ n=3}, \eqref{Eq/ mu}--\eqref{Eq/ omg star omg}, and \eqref{Eq/ H dissipative} with either $\omega_1 = \mu^{-1/2}$ or $\omega_1 = \omega_{\star}$. Iterating until $\vartheta - (n-2) < 0$, we deduce
\begin{align*}
    \|(\calU_{\sfH_\va[k]}\sfG_\va[k])^{\star(n-2)} \star \calU_{\sfH_\va[k]} \|_{\scrB(L^2_{\mu^{-1/2}}, L^\infty_{\omega_\star})} &\lesssim \frakp_{\iota}(t-s) \lt(1 + \max\{(t-s)^{-(\vartheta - (n-2))} , 1\} \rt) \\
    &\lesssim \frakp_{\iota''}(t-s),
\end{align*}
which proves \eqref{Eq/ n conv goal}.
\end{proof}

As a consequence of Lemma \ref{Lem/ n-2}, we may now prove \eqref{Eq/ real n conv}. Indeed, returning to \eqref{Eq/ nconv First}, we have
\begin{align*}
    &\lt\|(\calU_{\sfH_\va[k]} \sfG_\va[k])^{\star n} \rt\|_{\scrB(L^2_{\mu^{-1/2}}, L^\infty_\omega)} \\
    &\lesssim \int_s^t \int_s^{\Theta} \frakp_{\iota}(t-\Theta) \frakp_{\iota''}(\Theta-\Theta') \frakp_{\iota}(\Theta'-s) \,\tnd\Theta' \tnd\Theta \\
    &\lesssim \frakp_{\iota''}(t-s),
\end{align*}
thanks to \eqref{Eq/ H dissipative}--\eqref{Eq/ omg star omg} and Lemma \ref{Lem/ n-2}. At last, gathering \eqref{Eq/ projector bound}--\eqref{Eq/ Hypo again}--\eqref{Eq/ real n conv} together, we have obtained all the estimates required for the three terms in \eqref{Eq/ Three}, and namely we conclude to
\begin{equation*}
    I_5 \lesssim \int_s^t \int_{s}^{\Theta} \frakp_{\iota''}(t-\Theta) \frakp_{\iota}(\Theta-\Theta') \frakp_{\iota''}(\Theta'-s) \,\tnd\Theta' \tnd\Theta \lesssim (t-s)^2 \frakp_{\iota''}(t-s) \lesssim \frakp_{\iota'}(t-s).
\end{equation*}
We collect the estimates for all of the $I_i$, $1\le i \le 5$ to complete the proof.

\subsection{Back to the nonlinear problem} \label{Sec/ Back to Nonlin}

Our business with the linear problem is now done, and so we now aim to return to the nonlinear problem. We remind the reader that the operator $\leftperp{\Upsilon_\va}[k]$ is defined in \eqref{eq: nonlinear perp} as
\begin{align*}
    \leftperp{\Upsilon_\va}[k] h &= \frac{1}{d}\lt( E[k] - \bfA_\varsigma[k] \rt) \Delta_v h + (v\rho[k] - j[k]) \cdot \nabla_v h + d\rho[k] h \\
    &\quad - (\rho[h] j[k] - \rho[k] j[h]) \cdot v\mu \\
    &\quad - \lt( (E[k] - \bfA_\varsigma[k]) \rho[h] - \frac{1}{d} \rho[k] E[h] + \frac{1}{d}j[k]\cdot j[h] \rt) (|v|^2 - d)\mu.
\end{align*}

\begin{proposition}\label{Prop/ reg}
    We assume that $c\le c_\star$. There exists some $a_0>0$ small enough such that for every $h_0\in \calC$ satisfying $\|h_0\|_{L^\infty_{\omega}}\le a_0$, there exists a corresponding weak and renormalized solution $h^\va \in \calX_{\omega,c} \cap \calC$ to the nonlinear PDE
    \begin{equation}\label{Eq/ Nonlin/ Reg}
        \begin{cases}
            \p_t h^\va = -v\cdot \nabla_x h^\va + \sfC_{\rm lin} h^\va + \leftperp{\Upsilon_\va}[h^\va] h^\va &\text{in }\R_+\times \calO, \\
            h^\va |_{t=0} = h_0  &\text{on }\calO, \\
            \gamma_- h^\va(t,x,v) = \gamma_+ h^\va(t,x,\scrV_x) &\text{on }\R_+\times \Sigma_-.
        \end{cases}
    \end{equation}
    This one satisfies the decay estimate
    \begin{equation*}
        \|h^\va_t\|_{L^\infty_{\omega}(\calO)} \le C' e^{-\iota' t} \|h_0\|_{L^\infty_{\omega}}
    \end{equation*}
    for some constants $C',\iota'>0$ which are all independent of $\va$.
\end{proposition}

\begin{proof}
    For each $k\in \calX_{\omega,c}$ with $c\le c_\star$, we define the map $\Psi: \calX_{\omega,c} \to L^\infty_{\omega}$ via
    \begin{equation*}
        \Psi : k \mapsto \calU_{\sfF_\va[k]}(t,0) h_0,
    \end{equation*}
    so that $\Psi(k)$ is the solution to the linear problem
    \begin{equation*}
        \begin{cases}
            \p_t h = -v\cdot \nabla_x h + \sfC_{\rm lin} h + \leftperp{\Upsilon_\va}[k] h, \\
            h |_{t=0} = h_0, \\
            \gamma_- h(t,x,v) = \gamma_+ h(t,x,\scrV_x).
        \end{cases}
    \end{equation*}
    Since we can always embed $L^\infty_{\omega}\hookrightarrow L^2_{\wt{\omega}}$ for some $\wt{\omega}$ growing slowly enough, Proposition \ref{P/ F WP} tells us that $\calU_{\sfF_\va[k]}(t,0) h_0$ is the unique solution to this equation which corresponds to initial datum $h_0$. Also, according to Proposition \ref{Prop/ decay}, we know that
    \begin{equation}\label{Eq/ decay again}
        \| \calU_{\sfF_\va[k]}(t,0) h_0\|_{L^\infty_{\omega}} \le C' e^{-\iota'(t-s)} \|h_0\|_{L^\infty_{\omega}}
    \end{equation}
    for $C', \iota' >0$ independent of $\va$. Therefore, $\Psi$ is well-defined. Furthermore, if
    \begin{equation*}
        \|h_0\|_{L^\infty_{\omega}} \le a_0 := \frac{c}{C'},
    \end{equation*}
    then \eqref{Eq/ decay again} tells us that $\Psi$ leaves $\calX_{\omega,c}$ invariant. Hence, endowing $\calX_{\omega,c}$ with the $L^\infty_{\omega}-*$ topology, we just need to prove that $\Psi:\calX_{\omega,c}\to \calX_{\omega,c}$ is continuous in order to prove the existence of a fixed point of $\Psi$. The required arguments are completely the same as with Proposition \ref{S: P: Exist Nonlin} and thus skipped. We just point out that the convolutions in the definition of $\bfA_\va[\cdot]$ allow us to establish strong stability of this term, even if we only know that its argument converges weakly.
\end{proof}

\begin{proposition}\label{Prop: noreg}
    We assume that $h_0\in \calC$ and that $\|h_0\|_{L^\infty_\omega}\le a_0$, with $a_0$ chosen as in Proposition \ref{Prop/ reg}. Denote by $(h^\varsigma)_{\varsigma>0} \subset \calX_{\omega,c} \cap \calC$ a family of solutions to \eqref{Eq/ Nonlin/ Reg} (again as constructed in Proposition \ref{Prop/ reg}). Then as $\varsigma\searrow 0$ there is a weak$-*$ limit point $h\in \calX_{\omega,c} \cap \calC \cap L^\infty_t L^2_{x,v} \cap L^2_{t,x} H^1_v$ which is a weak and renormalized solution to the nonlinear problem 
    \begin{equation}\label{Eq/ Fully Nonlin}
        \begin{cases}
            \p_t h = -v\cdot \nabla_x h + \sfC_{\rm lin} h + \Upsilon(h), \\
            h|_{t=0} = h_0, \\
            \gamma_-h(t,x,v) = \gamma_+h(t,x,\scrV_x v).
        \end{cases}
    \end{equation}
    This one satisfies also the estimate
    \begin{equation*}
        \|h_t\|_{L^\infty_{\omega}(\calO)} \le C' e^{-\iota' t} \|h_0\|_{L^\infty_{\omega}(\calO)}.
    \end{equation*}
\end{proposition}
\begin{proof}
    The proof is a repetition of the one provided in Proposition \ref{S: P: final}, and we therefore only sketch it. We define the weighted space
    \begin{equation*}
        \calY := \left\{ h : \R_+\times \Omega\times \R^d\to \R \; \middle| \; \lt\| e^{\iota' t} h \omega \rt\|_{L^\infty(\R_+\times \Omega\times \R^d)} \le C' \|h_0\|_{L^\infty_\omega} \right\}.
    \end{equation*}
    Thanks to Proposition \ref{Prop/ reg}, we know that $(h^\va)$ is bounded in $\calY$, and thus the Banach--Alaoglu theorem provides a limit $\mathfrak{h}$ satisfying
    \begin{equation*}
    \begin{split}
        &e^{\iota' t} h \omega \weakto \mathfrak{h} \quad \text{in } L^\infty(\R_+\times\Omega\times\R^d)-*, \\
        &\|\frakh\|_{L^\infty(\R_+\times\Omega\times\R^d)} \le C' \|h_0\|_{L^\infty_{\omega}}.
    \end{split}
    \end{equation*}
    Defining now $h := e^{-\iota' t} \omega^{-1} \mathfrak{h}$, we obtain that (at least) $h^\va \weakto h$ in the sense of $\calD'(\R_+\times \ov{\calO})$ and $\|h_t\|_{L^\infty_{\omega}(\calO)} \le C' e^{-\iota' t} \|h_0\|_{L^\infty_{\omega}}$.
    
    From the equation \eqref{Eq/ Nonlin/ Reg}, we observe that it is enough to check whether we can pass to the limit of the term $\leftperp{\Upsilon_\va}[h^\va]h^\va$ in $\calD'((0,\infty)\times \calO)$. Indeed, the linear terms can be passed to the limit easily; after that, the renormalization properties, traces, and boundary condition can be recovered in the same way as in the proof of Proposition \ref{S: P: final}.
    
    Towards the end of analyzing the limit of $\leftperp{\Upsilon_\va}[h^\va]h^\va$, we write its explicit form here
    \begin{equation}\label{Eq/ Upsilon hh}
    \begin{split}
        \leftperp{\Upsilon_\va}[h^\va] h^\va &= \frac{1}{d}(E[h^\va] - \bfA_\va[h^\va]) \Delta_v h^\va + (v\rho[h] - j[h])\cdot \nabla_v h^\va + d\rho[h^\va] h^\va \\
        &\quad - \frac{1}{d} \lt( (E[h^\va] - \bfA_\va[h^\va]) \rho[h^\va] - \rho[h^\va] E[h^\va] +  |j[h^\va]|^2 \rt) (|v|^2 - d)\mu.
    \end{split}
    \end{equation}
    First, we remark that again using $L^\infty_{\omega}\hookrightarrow L^2$, and resorting once more to hypoelliptic estimates, we may assume that $h^\va \to h$ in $L^2_{\rm loc}(\R_+\times \calO)$ and a.e., as well as $\nabla_v h^\va \weakto \nabla_v h$ weakly in $L^2_{t,x,v}$ (see the proof of Proposition \ref{S: P: final}). Since $\rho[\cdot],j[\cdot],E[\cdot]$ are linear in their arguments, it follows that
    \begin{equation}\label{Eq/ Macs PW}
        \rho[h^\va] \to \rho[h], \quad j[h^\va]\to j[h], \quad E[h^\va]\to E[h] \quad \text{a.e. in }\R_+\times \Omega.
    \end{equation}
    On the other hand, since $|\rho[h^\va]|, |j[h^\va]|, |E[h^\va]| \lesssim \|h^\va\|_{L^\infty_{\omega}(\R_+\times \calO)} \lesssim 1$, we deduce by Vitali's theorem that the convergences in \eqref{Eq/ Macs PW} hold in $L^p_{\rm loc}(\R_+ \times \ov{\Omega})$ for every $p\in [1,\infty)$. We also know that
    \begin{equation*}
        0 \le \bfA_\va[h^\va] \le \frac{\|j[h^\va] *_t \bfs_\va \star_\va \bfr_\va \|_{L^\infty(\R_+\times \Omega)}^2 }{1 - \|\rho[k] *_t \bfs_\va \star_\va \bfr_\va \|_{L^\infty(\R_+\times \Omega)}} \le \frac{\|j[h^\va]\|_{L^\infty(\R_+\times\Omega)}^2}{1 - \|\rho[k]\|_{L^\infty(\R_+\times \Omega)}} \lesssim \|h^\va\|_{L^\infty_{\omega}}^2 \lesssim 1,
    \end{equation*}
    and therefore 
    \begin{equation*}
        \bfA_\va[h^\ve] \to \bfA[h] := \frac{|j[h]|^2}{1 + \rho[k]} \quad \text{a.e. and in }L^p_{\rm loc}(\R_+\times \ov{\Omega}), \quad \forall p\in [1,\infty).
    \end{equation*}
    In similar manner,
    \begin{equation*}
        |j[h^\va]|^2 \to |j[h]|^2 \quad \text{a.e. and in }L^p_{\rm loc}(\R_+\times \ov{\Omega}).
    \end{equation*}
    Therefore there is no problem with passing to the limit of $\leftperp{\Upsilon}_\va[h^\va] h^\va$ in the sense of $\calD'((0,\infty)\times \calO)$. We note in particular that the last line of \eqref{Eq/ Upsilon hh} converges to 
    \begin{equation*}
    \begin{split}
        &-\frac{1}{d}\lt( (E[h^\va] - \bfA_\va[h^\va]) \rho[h^\va] - \rho[h^\va] E[h^\va] +  |j[h^\va]|^2 \rt) (|v|^2 - d)\mu \\
        &\to - \frac{1}{d} \lt( \lt(E[h] - \frac{|j[h]|^2}{1+\rho[h]} \rt) \rho[h] - \rho[h] E[h] + |j[h]|^2 \rt) (|v|^2-d)\mu \\
        &= - \frac{1}{d} \frac{|j[h]|^2}{1+\rho[h]} (|v|^2 - d)\mu ,
    \end{split}
    \end{equation*}
    and so we have established that $h$ solves \eqref{Eq/ Fully Nonlin} in $\calD'((0,\infty)\times \calO)$ with $\Upsilon(h)$ as given in \eqref{Eq/ Upsilon Def}.
    
    Finally, one can check that $h\in \calC$ as follows: we note that, owing to the embedding $L^\infty_{\omega}\hookrightarrow L^1_{\la v \ra^2}$, we can admit for any $R\in \calR_{\Omega}$ the functions $1,R(x)\cdot v, |v|^2$ as test functions in the Green's formula for $h$ via a standard approximation procedure. For any of these test functions, we can deduce from their Green's formulae that
    \begin{equation*}
        t\mapsto \iint_{\calO} h_t (1, R(x)\cdot v, |v|^2) \,\dx\dv 
    \end{equation*}
    is absolutely continuous for every $R\in \calR_{\Omega}$. Thanks to $h\in L^\infty_{\omega}$ we may apply the Gauss--Ostrogradsky theorem freely to deduce, after some straightforward computations, that
    \begin{equation*}
        \ddt \iint_{\calO} h_t (1, R(x)\cdot v, |v|^2) \,\dx\dv = 0
    \end{equation*}
    and we therefore obtain that $h_t \in \calC$ because $h_0\in \calC$.
\end{proof}

\subsection{Proof of Theorem \ref{Thm2}} \label{Sec/ Thm2}

The proof here is a repetition of that of Theorem \ref{Thm1}. Writing $f := \mu + h$ where $h$ is as in Proposition \ref{Prop: noreg}, it is clear that (i) $f\in L^\infty_{\omega} \cap L^\infty_t L^2_{x,v} \cap L^2_{t,x}H^1_v$ (recalling that $L^\infty_{\omega}\hookrightarrow L^2$), (ii) $f$ satisfies the compatibility condition (C) (because $h\in \calC$), and (iii) $f$ is a solution to \eqref{eq: main} in $\calD'([0,t]\times\ov{\calO})$ for every $t > 0$. We can then check, by repeating the proof of Proposition \ref{S: P: final}, that $f$ is a renormalized solution to \eqref{eq: main}; finally, the decay \eqref{Eq/ hypodissipative} is an immediate consequence of Proposition \ref{Prop: noreg}.

\section{Acknowledgments}

The author thanks Young-Pil Choi for his advice and support. This work received funding from NRF grants no. 2022R1A2C1002820 and RS-2024-00406821. The author also acknowledges support from the Presidential Science Scholarship provided by the Korea Student Aid Foundation.

\appendix

\section{Formal computation of the linearized operator} \label{App: A: Formal}
Let $\scrF$ denote the nonlinear Fokker--Planck operator in \eqref{eq: main}. In this section, we compute the linearization of $\scrF$ based at the point $\mu$, and explain why $\sfC_{\rm lin}$ takes its form. In other words, we wish to find the following first order approximation of $\scrF$:
\[ \scrF[F] \approx \scrF [\mu] + D\scrF[\mu] (F-\mu) = D\scrF[\mu] (F-\mu) =: \sfC_{\rm lin}(F-\mu) . \]
In the above, we denoted by $D\scrF$ the Gateaux derivative, and we noted that $\scrF[\mu] = 0$.

Toward this end, we write
\begin{align*}
    G = \mu + \ve g,
\end{align*}
and aim to collect all terms that are formally up to first order in $\ve$. In this way, a straightforward computation allows us to check that
\begin{align*}
    \rho[G] &= 1 + \ve \rho[g], \\
    u[G] &= \frac{\ve j[g]}{1 + \ve \rho[g]}, \quad j[g] := \intr vf\,\dv.
\end{align*}
Using the expansion $\frac{1}{1+x} = 1 - x + x^2 - \cdots$, we find
\[ u[G] = \ve j[g] (1 - \ve \rho[g] + o(\ve)) = \ve j[g] + (\textnormal{$\ve$-higher order terms}). \]
Similarly, setting $E[g] := \int |v|^2 g \dv$, we can find
\begin{equation*}
    E[G] = d + \ve E[g],
\end{equation*}
so that
\[
\begin{split}
\theta[G] &:= \frac{1}{d\rho[G]} \Big( E[G] - \rho[G]|u[G]|^2 \Big) \\
&= \frac{1}{d(1+\ve\rho[g])} \lt( d + \ve E[g] - (1+\ve\rho[g]) \lt|\frac{\ve j[g]}{1+\ve \rho[g]}\rt|^2 \rt)  \\
&= \frac{1 + \ve E[g]/d}{1 + \ve \rho[g]} + o(\ve), 
\end{split}
\]
and again using the expansion of $\frac{1}{1+x}$, we get
\[ \theta[G] = 1 + \frac{\ve}{d}E[g] - \ve \rho[g] + (\textnormal{$\ve$-higher order terms}). \]
That is, up to first order in $\ve$:
\[ \theta[G] = 1 + \frac{\ve}{d}\intr (|v|^2-d)g\,\dv. \]

Therefore,
\begin{align*}
    \scrF[\mu + \ve g] &= \scrF[G] \\
    &= \rho[G] (\theta[G] \Delta_v G + \nabla_v\cdot ( (v - u[G]) G) ) \\
    &\approx (1 + \ve \rho[g]) \lt(  \lt(1 + \frac{\ve}{d} \intr (|v|^2-d)g \,\dv\rt) (\Delta_v \mu + \ve \Delta_v g) \rt) \\
    &\quad + (1+\ve \rho[g]) \lt( d (\mu + \ve g) + \lt(v-\ve j[g] \rt) \cdot \nabla_v(\mu + \ve g) \rt) \\
    &\approx  \underbrace{ \Delta_v \mu + d\mu + v\cdot \nabla_v \mu}_{=0}  \\
    &\quad + \ve \rho[g] (\underbrace{ \Delta_v \mu + d\mu + v\cdot \nabla_v \mu}_{=0} ) \\
    &\quad + \ve \lt[ (\Delta_v g + \nabla_v\cdot (vg))  - j[g] \cdot \nabla_v \mu + \lt(\frac{1}{d} \intr (|v|^2-d) g \,\dv\rt) \Delta_v\mu \rt] \\
    &\quad + (\textnormal{$\ve$-higher order terms}) \\
    &\approx \ve \lt[ \Delta_v g + \nabla_v\cdot (vg) + j[f] \cdot (v \mu) + \lt(\frac{1}{d}\intr (|w|^2-d)g\,\dw \rt) (|v|^2-d) \mu \rt].
\end{align*}
Hence, using that $\scrF[\mu]=0$:
\begin{align*}
&\lim_{\ve\to 0} \frac{\scrF[\mu + \ve g] - \scrF[\mu]}{\ve} \\
&=  \underbrace{\Delta_v g + \nabla_v\cdot (vg)}_{=\sfC_{\rm FP} g} + \underbrace{j[g] \cdot (v \mu) + \lt(\frac{1}{d}\intr (|w|^2-d)g\,\dw \rt) (|v|^2-d) \mu}_{=\sum_{i=1}^d \pi_i g + 2 \pi_{d+1} g } \\
&= \sfC_{\rm lin} g.
\end{align*}

\newpage
\bibliographystyle{siam}
\bibliography{weakinhoVPFP}

@article {Cho16,
    AUTHOR = {Choi, Young-Pil},
     TITLE = {Global classical solutions of the {V}lasov-{F}okker-{P}lanck
              equation with local alignment forces},
   JOURNAL = {Nonlinearity},
  FJOURNAL = {Nonlinearity},
    VOLUME = {29},
      YEAR = {2016},
    NUMBER = {7},
     PAGES = {1887--1916},
      ISSN = {0951-7715,1361-6544},
   MRCLASS = {35Q83 (35A09 35B40 35Q84)},
  MRNUMBER = {3521632},
MRREVIEWER = {Wolf-Patrick\ D\"{u}ll},
       DOI = {10.1088/0951-7715/29/7/1887},
       URL = {https://doi.org/10.1088/0951-7715/29/7/1887},
}

@article {BloLeD01,
    AUTHOR = {Blouza, Adel and Le Dret, Herv\'{e}},
     TITLE = {An up-to-the-boundary version of {F}riedrichs's lemma and
              applications to the linear {K}oiter shell model},
   JOURNAL = {SIAM J. Math. Anal.},
  FJOURNAL = {SIAM Journal on Mathematical Analysis},
    VOLUME = {33},
      YEAR = {2001},
    NUMBER = {4},
     PAGES = {877--895},
      ISSN = {0036-1410,1095-7154},
   MRCLASS = {35Q72 (35A15 35A25 74G40 74K25)},
  MRNUMBER = {1884727},
MRREVIEWER = {J\'{a}n\ Lov\'{\i}\v{s}ek},
       DOI = {10.1137/S0036141000380012},
       URL = {https://doi.org/10.1137/S0036141000380012},
}

@book {Lions61,
    AUTHOR = {Lions, J.-L.},
     TITLE = {\'{E}quations diff\'{e}rentielles op\'{e}rationnelles et
              probl\`emes aux limites},
    SERIES = {Die Grundlehren der mathematischen Wissenschaften, Band 111},
 PUBLISHER = {Springer-Verlag, Berlin-G\"{o}ttingen-Heidelberg},
      YEAR = {1961},
     PAGES = {ix+292},
   MRCLASS = {35.00 (34.95)},
  MRNUMBER = {153974},
MRREVIEWER = {S.\ Zaidman},
}

@article {LionsPer92,
    AUTHOR = {Lions, Pierre-Louis and Perthame, Beno\^{i}t},
     TITLE = {Lemmes de moments, de moyenne et de dispersion},
   JOURNAL = {C. R. Acad. Sci. Paris S\'{e}r. I Math.},
  FJOURNAL = {Comptes Rendus de l'Acad\'{e}mie des Sciences. S\'{e}rie I.
              Math\'{e}matique},
    VOLUME = {314},
      YEAR = {1992},
    NUMBER = {11},
     PAGES = {801--806},
      ISSN = {0764-4442},
   MRCLASS = {35Q55 (82C70)},
  MRNUMBER = {1166050},
MRREVIEWER = {Andrzej\ Palczewski},
}

@article {Zhu24,
    AUTHOR = {Zhu, Yuzhe},
     TITLE = {Regularity of kinetic {F}okker-{P}lanck equations in bounded
              domains},
   JOURNAL = {Ann. H. Lebesgue},
  FJOURNAL = {Annales Henri Lebesgue},
    VOLUME = {7},
      YEAR = {2024},
     PAGES = {1323--1366},
      ISSN = {2644-9463},
   MRCLASS = {35Q84 (35B65 35H10)},
  MRNUMBER = {4883871},
MRREVIEWER = {Felix\ X.-F.\ Ye},
       DOI = {10.5802/ahl.221},
       URL = {https://doi.org/10.5802/ahl.221},
}

@article {ImbMou21,
    AUTHOR = {Imbert, Cyril and Mouhot, Cl\'{e}ment},
     TITLE = {The {S}chauder estimate in kinetic theory with application to
              a toy nonlinear model},
   JOURNAL = {Ann. H. Lebesgue},
  FJOURNAL = {Annales Henri Lebesgue},
    VOLUME = {4},
      YEAR = {2021},
     PAGES = {369--405},
      ISSN = {2644-9463},
   MRCLASS = {35Q84 (35B65 82C40)},
  MRNUMBER = {4275241},
       DOI = {10.5802/ahl.75},
       URL = {https://doi.org/10.5802/ahl.75},
}

@article {GIMV19,
    AUTHOR = {Golse, Fran\c{c}ois and Imbert, Cyril and Mouhot, Cl\'{e}ment
              and Vasseur, Alexis F.},
     TITLE = {Harnack inequality for kinetic {F}okker-{P}lanck equations
              with rough coefficients and application to the {L}andau
              equation},
   JOURNAL = {Ann. Sc. Norm. Super. Pisa Cl. Sci. (5)},
  FJOURNAL = {Annali della Scuola Normale Superiore di Pisa. Classe di
              Scienze. Serie V},
    VOLUME = {19},
      YEAR = {2019},
    NUMBER = {1},
     PAGES = {253--295},
      ISSN = {0391-173X,2036-2145},
   MRCLASS = {35Q84 (35B45 35B65)},
  MRNUMBER = {3923847},
MRREVIEWER = {Barbara\ Brandolini},
}

@article {Hor67,
    AUTHOR = {H\"{o}rmander, Lars},
     TITLE = {Hypoelliptic second order differential equations},
   JOURNAL = {Acta Math.},
  FJOURNAL = {Acta Mathematica},
    VOLUME = {119},
      YEAR = {1967},
     PAGES = {147--171},
      ISSN = {0001-5962,1871-2509},
   MRCLASS = {35.48 (47.00)},
  MRNUMBER = {222474},
MRREVIEWER = {Joel\ Smoller},
       DOI = {10.1007/BF02392081},
       URL = {https://doi.org/10.1007/BF02392081},
}

@article {AddDolbLiTay21,
    AUTHOR = {Addala, Lanoir and Dolbeault, Jean and Li, Xingyu and Tayeb,
              M. Lazhar},
     TITLE = {{${\rm L}^2$}-hypocoercivity and large time asymptotics of
              the linearized {V}lasov-{P}oisson-{F}okker-{P}lanck system},
   JOURNAL = {J. Stat. Phys.},
  FJOURNAL = {Journal of Statistical Physics},
    VOLUME = {184},
      YEAR = {2021},
    NUMBER = {1},
     PAGES = {Paper No. 4, 34},
      ISSN = {0022-4715,1572-9613},
   MRCLASS = {82C40 (35H10 35P15 47G20 82D10)},
  MRNUMBER = {4277286},
MRREVIEWER = {Yuxi\ Zheng},
       DOI = {10.1007/s10955-021-02784-4},
       URL = {https://doi.org/10.1007/s10955-021-02784-4},
}

@article {DolLi18,
    AUTHOR = {Dolbeault, Jean and Li, Xingyu},
     TITLE = {{$\varphi$}-entropies: convexity, coercivity and
              hypocoercivity for {F}okker-{P}lanck and kinetic
              {F}okker-{P}lanck equations},
   JOURNAL = {Math. Models Methods Appl. Sci.},
  FJOURNAL = {Mathematical Models and Methods in Applied Sciences},
    VOLUME = {28},
      YEAR = {2018},
    NUMBER = {13},
     PAGES = {2637--2666},
      ISSN = {0218-2025,1793-6314},
   MRCLASS = {82C31 (35H10 35K65 35P15 35Q83 35Q84 76P05)},
  MRNUMBER = {3884261},
       DOI = {10.1142/S0218202518500574},
       URL = {https://doi.org/10.1142/S0218202518500574},
}

@book {Jun16,
    AUTHOR = {J\"{u}ngel, Ansgar},
     TITLE = {Entropy methods for diffusive partial differential equations},
    SERIES = {SpringerBriefs in Mathematics},
 PUBLISHER = {Springer, [Cham]},
      YEAR = {2016},
     PAGES = {viii+139},
      ISBN = {978-3-319-34218-4; 978-3-319-34219-1},
   MRCLASS = {35-02 (35K57 35Q84 35Q92)},
  MRNUMBER = {3497125},
MRREVIEWER = {Gaetano\ Siciliano},
       DOI = {10.1007/978-3-319-34219-1},
       URL = {https://doi.org/10.1007/978-3-319-34219-1},
}

@article {BriGuo16,
    AUTHOR = {Briant, Marc and Guo, Yan},
     TITLE = {Asymptotic stability of the {B}oltzmann equation with
              {M}axwell boundary conditions},
   JOURNAL = {J. Differential Equations},
  FJOURNAL = {Journal of Differential Equations},
    VOLUME = {261},
      YEAR = {2016},
    NUMBER = {12},
     PAGES = {7000--7079},
      ISSN = {0022-0396,1090-2732},
   MRCLASS = {35Q20 (35B35)},
  MRNUMBER = {3562318},
MRREVIEWER = {Bertrand\ Lods},
       DOI = {10.1016/j.jde.2016.09.014},
       URL = {https://doi.org/10.1016/j.jde.2016.09.014},
}

@article {Guo10,
    AUTHOR = {Guo, Yan},
     TITLE = {Decay and continuity of the {B}oltzmann equation in bounded
              domains},
   JOURNAL = {Arch. Ration. Mech. Anal.},
  FJOURNAL = {Archive for Rational Mechanics and Analysis},
    VOLUME = {197},
      YEAR = {2010},
    NUMBER = {3},
     PAGES = {713--809},
      ISSN = {0003-9527,1432-0673},
   MRCLASS = {35Q20 (35B40 35B65 76N15 76P05 82C40)},
  MRNUMBER = {2679358},
MRREVIEWER = {Mark\ Thompson},
       DOI = {10.1007/s00205-009-0285-y},
       URL = {https://doi.org/10.1007/s00205-009-0285-y},
}

@article {KimLee18B,
    AUTHOR = {Kim, Chanwoo and Lee, Donghyun},
     TITLE = {Decay of the {B}oltzmann equation with the specular boundary
              condition in non-convex cylindrical domains},
   JOURNAL = {Arch. Ration. Mech. Anal.},
  FJOURNAL = {Archive for Rational Mechanics and Analysis},
    VOLUME = {230},
      YEAR = {2018},
    NUMBER = {1},
     PAGES = {49--123},
      ISSN = {0003-9527,1432-0673},
   MRCLASS = {35Q20 (35B40 76P05 82C40)},
  MRNUMBER = {3840911},
MRREVIEWER = {Luisa\ Arlotti},
       DOI = {10.1007/s00205-018-1241-5},
       URL = {https://doi.org/10.1007/s00205-018-1241-5},
}

@article {KimLee18A,
    AUTHOR = {Kim, Chanwoo and Lee, Donghyun},
     TITLE = {The {B}oltzmann equation with specular boundary condition in
              convex domains},
   JOURNAL = {Comm. Pure Appl. Math.},
  FJOURNAL = {Communications on Pure and Applied Mathematics},
    VOLUME = {71},
      YEAR = {2018},
    NUMBER = {3},
     PAGES = {411--504},
      ISSN = {0010-3640,1097-0312},
   MRCLASS = {35Q20 (35B30 82C40 82D10)},
  MRNUMBER = {3762275},
MRREVIEWER = {Juhi\ Jang},
       DOI = {10.1002/cpa.21705},
       URL = {https://doi.org/10.1002/cpa.21705},
}

@article {StrGuo08,
    AUTHOR = {Strain, Robert M. and Guo, Yan},
     TITLE = {Exponential decay for soft potentials near {M}axwellian},
   JOURNAL = {Arch. Ration. Mech. Anal.},
  FJOURNAL = {Archive for Rational Mechanics and Analysis},
    VOLUME = {187},
      YEAR = {2008},
    NUMBER = {2},
     PAGES = {287--339},
      ISSN = {0003-9527,1432-0673},
   MRCLASS = {82B05 (35B45 35F25)},
  MRNUMBER = {2366140},
MRREVIEWER = {Simone\ Calogero},
       DOI = {10.1007/s00205-007-0067-3},
       URL = {https://doi.org/10.1007/s00205-007-0067-3},
}

@article {CarMis17,
    AUTHOR = {Carrapatoso, K. and Mischler, S.},
     TITLE = {Landau equation for very soft and {C}oulomb potentials near
              {M}axwellians},
   JOURNAL = {Ann. PDE},
  FJOURNAL = {Annals of PDE. Journal Dedicated to the Analysis of Problems
              from Physical Sciences},
    VOLUME = {3},
      YEAR = {2017},
    NUMBER = {1},
     PAGES = {Paper No. 1, 65},
      ISSN = {2524-5317,2199-2576},
   MRCLASS = {35F25 (35B35 35B40 35Q20 47D06 47H20)},
  MRNUMBER = {3625186},
MRREVIEWER = {Hidetoshi\ Tahara},
       DOI = {10.1007/s40818-017-0021-0},
       URL = {https://doi.org/10.1007/s40818-017-0021-0},
}

@article {DesVil05,
    AUTHOR = {Desvillettes, L. and Villani, C.},
     TITLE = {On the trend to global equilibrium for spatially inhomogeneous
              kinetic systems: the {B}oltzmann equation},
   JOURNAL = {Invent. Math.},
  FJOURNAL = {Inventiones Mathematicae},
    VOLUME = {159},
      YEAR = {2005},
    NUMBER = {2},
     PAGES = {245--316},
      ISSN = {0020-9910,1432-1297},
   MRCLASS = {82C40 (35B40 35F20)},
  MRNUMBER = {2116276},
MRREVIEWER = {Manuel\ Portilheiro},
       DOI = {10.1007/s00222-004-0389-9},
       URL = {https://doi.org/10.1007/s00222-004-0389-9},
}

@article {DolMouSch09,
    AUTHOR = {Dolbeault, Jean and Mouhot, Cl\'{e}ment and Schmeiser,
              Christian},
     TITLE = {Hypocoercivity for kinetic equations with linear relaxation
              terms},
   JOURNAL = {C. R. Math. Acad. Sci. Paris},
  FJOURNAL = {Comptes Rendus Math\'{e}matique. Acad\'{e}mie des Sciences.
              Paris},
    VOLUME = {347},
      YEAR = {2009},
    NUMBER = {9-10},
     PAGES = {511--516},
      ISSN = {1631-073X,1778-3569},
   MRCLASS = {35F20 (82C40)},
  MRNUMBER = {2576899},
       DOI = {10.1016/j.crma.2009.02.025},
       URL = {https://doi.org/10.1016/j.crma.2009.02.025},
}

@article {DolMouSch15,
    AUTHOR = {Dolbeault, Jean and Mouhot, Cl\'{e}ment and Schmeiser,
              Christian},
     TITLE = {Hypocoercivity for linear kinetic equations conserving mass},
   JOURNAL = {Trans. Amer. Math. Soc.},
  FJOURNAL = {Transactions of the American Mathematical Society},
    VOLUME = {367},
      YEAR = {2015},
    NUMBER = {6},
     PAGES = {3807--3828},
      ISSN = {0002-9947,1088-6850},
   MRCLASS = {35F10 (35B40 35H10 82C31)},
  MRNUMBER = {3324910},
MRREVIEWER = {Ingrid\ Alma\ Belti\c{t}\u{a}},
       DOI = {10.1090/S0002-9947-2015-06012-7},
       URL = {https://doi.org/10.1090/S0002-9947-2015-06012-7},
}

@article {GuoHwangJangOu20,
    AUTHOR = {Guo, Yan and Hwang, Hyung Ju and Jang, Jin Woo and Ouyang,
              Zhimeng},
     TITLE = {The {L}andau equation with the specular reflection boundary
              condition},
   JOURNAL = {Arch. Ration. Mech. Anal.},
  FJOURNAL = {Archive for Rational Mechanics and Analysis},
    VOLUME = {236},
      YEAR = {2020},
    NUMBER = {3},
     PAGES = {1389--1454},
      ISSN = {0003-9527,1432-0673},
   MRCLASS = {35Q82 (35Q20)},
  MRNUMBER = {4076068},
       DOI = {10.1007/s00205-020-01496-5},
       URL = {https://doi.org/10.1007/s00205-020-01496-5},
}

@article {KimGuoHwang20,
    AUTHOR = {Kim, Jinoh and Guo, Yan and Hwang, Hyung Ju},
     TITLE = {An {$L^2$} to {$L^\infty$} framework for the {L}andau
              equation},
   JOURNAL = {Peking Math. J.},
  FJOURNAL = {Peking Mathematical Journal},
    VOLUME = {3},
      YEAR = {2020},
    NUMBER = {2},
     PAGES = {131--202},
      ISSN = {2096-6075,2524-7182},
   MRCLASS = {35Qxx},
  MRNUMBER = {4171912},
       DOI = {10.1007/s42543-019-00018-x},
       URL = {https://doi.org/10.1007/s42543-019-00018-x},
}

@article {Guo02,
    AUTHOR = {Guo, Yan},
     TITLE = {The {L}andau equation in a periodic box},
   JOURNAL = {Comm. Math. Phys.},
  FJOURNAL = {Communications in Mathematical Physics},
    VOLUME = {231},
      YEAR = {2002},
    NUMBER = {3},
     PAGES = {391--434},
      ISSN = {0010-3616,1432-0916},
   MRCLASS = {82D10 (82C40)},
  MRNUMBER = {1946444},
MRREVIEWER = {C\'{e}dric\ Villani},
       DOI = {10.1007/s00220-002-0729-9},
       URL = {https://doi.org/10.1007/s00220-002-0729-9},
}

@article {LiaoYang21,
    AUTHOR = {Liao, Jie and Yang, Xiongfeng},
     TITLE = {Stability of global {M}axwellian for fully nonlinear
              {F}okker-{P}lanck equations},
   JOURNAL = {J. Stat. Phys.},
  FJOURNAL = {Journal of Statistical Physics},
    VOLUME = {185},
      YEAR = {2021},
    NUMBER = {3},
     PAGES = {Paper No. 23, 27},
      ISSN = {0022-4715,1572-9613},
   MRCLASS = {35H10 (76P99 82C21 82C31)},
  MRNUMBER = {4338696},
       DOI = {10.1007/s10955-021-02844-9},
       URL = {https://doi.org/10.1007/s10955-021-02844-9},
}

@article {Vil09,
    AUTHOR = {Villani, C\'{e}dric},
     TITLE = {Hypocoercivity},
   JOURNAL = {Mem. Amer. Math. Soc.},
  FJOURNAL = {Memoirs of the American Mathematical Society},
    VOLUME = {202},
      YEAR = {2009},
    NUMBER = {950},
     PAGES = {iv+141},
      ISSN = {0065-9266,1947-6221},
      ISBN = {978-0-8218-4498-4},
   MRCLASS = {35Q84 (35H10 76N10 76P05 82C70)},
  MRNUMBER = {2562709},
MRREVIEWER = {Andr\'{a}s\ Domokos},
       DOI = {10.1090/S0065-9266-09-00567-5},
       URL = {https://doi.org/10.1090/S0065-9266-09-00567-5},
}

@article {Kol34,
    AUTHOR = {Kolmogoroff, A.},
     TITLE = {Zuf\"{a}llige {B}ewegungen (zur {T}heorie der {B}rownschen
              {B}ewegung)},
   JOURNAL = {Ann. of Math. (2)},
  FJOURNAL = {Annals of Mathematics. Second Series},
    VOLUME = {35},
      YEAR = {1934},
    NUMBER = {1},
     PAGES = {116--117},
      ISSN = {0003-486X,1939-8980},
   MRCLASS = {DML},
  MRNUMBER = {1503147},
       DOI = {10.2307/1968123},
       URL = {https://doi.org/10.2307/1968123},
}

@article{AncReb22,
	author = {Anceschi, Francesca and Rebucci, Annalaura},
	date = {2023/06/01},
	doi = {10.1007/s41808-022-00191-8},
	id = {Anceschi2023},
	isbn = {2296-9039},
	journal = {Journal of Elliptic and Parabolic Equations},
	number = {1},
	pages = {63--92},
	title = {On the fundamental solution for degenerate {K}olmogorov equations with rough coefficients},
	url = {https://doi.org/10.1007/s41808-022-00191-8},
	volume = {9},
	year = {2023}}

@article {CarMis24KFP,
    AUTHOR = {Carrapatoso, K. and Mischler, S.},
     TITLE = {The kinetic {F}okker-{P}lanck equation in a domain:
              ultracontractivity, hypocoercivity, and long-time asymptotic
              behavior},
   JOURNAL = {Atti Accad. Naz. Lincei Rend. Lincei Mat. Appl.},
  FJOURNAL = {Atti della Accademia Nazionale dei Lincei. Rendiconti Lincei.
              Matematica e Applicazioni},
    VOLUME = {35},
      YEAR = {2024},
    NUMBER = {4},
     PAGES = {643--680},
      ISSN = {1120-6330,1720-0768},
   MRCLASS = {35Q84 (35B40 47D06)},
  MRNUMBER = {4929973},
MRREVIEWER = {Francesca\ Anceschi},
       DOI = {10.4171/rlm/1055},
       URL = {https://doi.org/10.4171/rlm/1055},
}

@article {GuaMisMou17,
    AUTHOR = {Gualdani, M. P. and Mischler, S. and Mouhot, C.},
     TITLE = {Factorization of non-symmetric operators and exponential
              {$H$}-theorem},
   JOURNAL = {M\'{e}m. Soc. Math. Fr. (N.S.)},
  FJOURNAL = {M\'{e}moires de la Soci\'{e}t\'{e} Math\'{e}matique de France.
              Nouvelle S\'{e}rie},
    NUMBER = {153},
      YEAR = {2017},
     PAGES = {137},
      ISSN = {0249-633X,2275-3230},
      ISBN = {978-2-85629-874-9},
   MRCLASS = {47D06 (34G10 35Q20 35Q84 82C31)},
  MRNUMBER = {3779780},
MRREVIEWER = {Rodica\ Luca},
       DOI = {10.24033/msmf.461},
       URL = {https://doi.org/10.24033/msmf.461},
}

@article {Bou02,
    AUTHOR = {Bouchut, F.},
     TITLE = {Hypoelliptic regularity in kinetic equations},
   JOURNAL = {J. Math. Pures Appl. (9)},
  FJOURNAL = {Journal de Math\'{e}matiques Pures et Appliqu\'{e}es.
              Neuvi\`eme S\'{e}rie},
    VOLUME = {81},
      YEAR = {2002},
    NUMBER = {11},
     PAGES = {1135--1159},
      ISSN = {0021-7824},
   MRCLASS = {82C40 (35B65 35H10)},
  MRNUMBER = {1949176},
MRREVIEWER = {Florent\ Berthelin},
       DOI = {10.1016/S0021-7824(02)01264-3},
       URL = {https://doi.org/10.1016/S0021-7824(02)01264-3},
}

@article {Mis00,
    AUTHOR = {Mischler, St\'{e}phane},
     TITLE = {On the trace problem for solutions of the {V}lasov equation},
   JOURNAL = {Comm. Partial Differential Equations},
  FJOURNAL = {Communications in Partial Differential Equations},
    VOLUME = {25},
      YEAR = {2000},
    NUMBER = {7-8},
     PAGES = {1415--1443},
      ISSN = {0360-5302,1532-4133},
   MRCLASS = {82D10 (35Q80)},
  MRNUMBER = {1765137},
MRREVIEWER = {Mohamed\ Boulanouar},
       DOI = {10.1080/03605300008821554},
       URL = {https://doi.org/10.1080/03605300008821554},
}

@article {DiPLion89,
    AUTHOR = {DiPerna, R. J. and Lions, P.-L.},
     TITLE = {Ordinary differential equations, transport theory and
              {S}obolev spaces},
   JOURNAL = {Invent. Math.},
  FJOURNAL = {Inventiones Mathematicae},
    VOLUME = {98},
      YEAR = {1989},
    NUMBER = {3},
     PAGES = {511--547},
      ISSN = {0020-9910,1432-1297},
   MRCLASS = {34A10 (34D20 35Q20 58D25 82A70)},
  MRNUMBER = {1022305},
MRREVIEWER = {B.\ G.\ Pachpatte},
       DOI = {10.1007/BF01393835},
       URL = {https://doi.org/10.1007/BF01393835},
}

@article {Car98,
    AUTHOR = {Carrillo, Jos\'{e} A.},
     TITLE = {Global weak solutions for the initial-boundary-value problems
              to the {V}lasov-{P}oisson-{F}okker-{P}lanck system},
   JOURNAL = {Math. Methods Appl. Sci.},
  FJOURNAL = {Mathematical Methods in the Applied Sciences},
    VOLUME = {21},
      YEAR = {1998},
    NUMBER = {10},
     PAGES = {907--938},
      ISSN = {0170-4214,1099-1476},
   MRCLASS = {35D10 (35Q99 76X05 82C31 82D10)},
  MRNUMBER = {1634851},
MRREVIEWER = {R.\ Glassey},
       DOI =
              {10.1002/(SICI)1099-1476(19980710)21:10<907::AID-MMA977>3.3.CO;2-N},
       URL =
              {https://doi.org/10.1002/(SICI)1099-1476(19980710)21:10<907::AID-MMA977>3.3.CO;2-N},
}

@misc{CarMis24,
      title={The {L}andau equation in a domain}, 
      author={Carrapatoso, K. and Mischler, S.},
      year={preprint, arxiv:2407.09031},
      eprint={2407.09031},
      archivePrefix={arXiv},
      primaryClass={math.AP},
      url={https://arxiv.org/abs/2407.09031}, 
}

@misc{ChoiSong25,
      title={Global weak solutions to nonlinear kinetic {F}okker--{P}lanck equations in bounded domains under physical initial data}, 
      author={Young-Pil Choi and Sihyun Song},
      year={preprint, arxiv:2510.06656},
      eprint={2510.06656},
      archivePrefix={arXiv},
      primaryClass={math.AP},
      url={https://arxiv.org/abs/2510.06656}, 
}

@article {MisMou16,
    AUTHOR = {Mischler, S. and Mouhot, C.},
     TITLE = {Exponential stability of slowly decaying solutions to the
              kinetic-{F}okker-{P}lanck equation},
   JOURNAL = {Arch. Ration. Mech. Anal.},
  FJOURNAL = {Archive for Rational Mechanics and Analysis},
    VOLUME = {221},
      YEAR = {2016},
    NUMBER = {2},
     PAGES = {677--723},
      ISSN = {0003-9527,1432-0673},
   MRCLASS = {35Q84 (35B35)},
  MRNUMBER = {3488535},
MRREVIEWER = {Arnaud\ F.\ Heibig},
       DOI = {10.1007/s00205-016-0972-4},
       URL = {https://doi.org/10.1007/s00205-016-0972-4},
}

@article {GuoLiu17,
    AUTHOR = {Guo, Yan and Liu, Shuangqian},
     TITLE = {The {B}oltzmann equation with weakly inhomogeneous data in
              bounded domain},
   JOURNAL = {J. Funct. Anal.},
  FJOURNAL = {Journal of Functional Analysis},
    VOLUME = {272},
      YEAR = {2017},
    NUMBER = {5},
     PAGES = {2038--2057},
      ISSN = {0022-1236,1096-0783},
   MRCLASS = {35Q20 (35B07 35B40)},
  MRNUMBER = {3596715},
MRREVIEWER = {Cecil\ Pompiliu\ Gr\"{u}nfeld},
       DOI = {10.1016/j.jfa.2016.08.017},
       URL = {https://doi.org/10.1016/j.jfa.2016.08.017},
}

@article {ArkEspPul87,
    AUTHOR = {Arkeryd, L. and Esposito, R. and Pulvirenti, M.},
     TITLE = {The {B}oltzmann equation for weakly inhomogeneous data},
   JOURNAL = {Comm. Math. Phys.},
  FJOURNAL = {Communications in Mathematical Physics},
    VOLUME = {111},
      YEAR = {1987},
    NUMBER = {3},
     PAGES = {393--407},
      ISSN = {0010-3616,1432-0916},
   MRCLASS = {82A40 (76P05 82A05)},
  MRNUMBER = {900501},
MRREVIEWER = {Carlo\ Cercignani},
       URL = {http://projecteuclid.org/euclid.cmp/1104159637},
}

@article {ChoiHwaYoo25,
    AUTHOR = {Choi, Young-Pil and Hwang, Byung-Hoon and Yoo, Yeongseok},
     TITLE = {Global existence of weak solutions to the nonlinear
              {V}lasov-{F}okker-{P}lanck equation},
   JOURNAL = {J. Differential Equations},
  FJOURNAL = {Journal of Differential Equations},
    VOLUME = {444},
      YEAR = {2025},
     PAGES = {Paper No. 113573, 53},
      ISSN = {0022-0396,1090-2732},
   MRCLASS = {35Q84 (35D30 35Q83)},
  MRNUMBER = {4924963},
MRREVIEWER = {Benny\ Avelin},
       DOI = {10.1016/j.jde.2025.113573},
       URL = {https://doi.org/10.1016/j.jde.2025.113573},
}

@article {CarGan04,
    AUTHOR = {Carlen, E. A. and Gangbo, W.},
     TITLE = {Solution of a model {B}oltzmann equation via steepest descent
              in the 2-{W}asserstein metric},
   JOURNAL = {Arch. Ration. Mech. Anal.},
  FJOURNAL = {Archive for Rational Mechanics and Analysis},
    VOLUME = {172},
      YEAR = {2004},
    NUMBER = {1},
     PAGES = {21--64},
      ISSN = {0003-9527,1432-0673},
   MRCLASS = {82C40 (35D05 35K15 82C31)},
  MRNUMBER = {2048566},
MRREVIEWER = {Laurent\ E.\ Gosse},
       DOI = {10.1007/s00205-003-0296-z},
       URL = {https://doi.org/10.1007/s00205-003-0296-z},
}

@article {SR02,
    AUTHOR = {Saint-Raymond, Laure},
     TITLE = {Discrete time {N}avier-{S}tokes limit for the {BGK}
              {B}oltzmann equation},
   JOURNAL = {Comm. Partial Differential Equations},
  FJOURNAL = {Communications in Partial Differential Equations},
    VOLUME = {27},
      YEAR = {2002},
    NUMBER = {1-2},
     PAGES = {149--184},
      ISSN = {0360-5302,1532-4133},
   MRCLASS = {35Q30 (45K05 76A02 76D05 76P05)},
  MRNUMBER = {1886958},
MRREVIEWER = {Giuseppe\ Toscani},
       DOI = {10.1081/PDE-120002785},
       URL = {https://doi.org/10.1081/PDE-120002785},
}

@book {Gre52,
    AUTHOR = {Green, Herbert S.},
     TITLE = {The molecular theory of fluids},
      NOTE = {Deformation and Flow: Monographs on the Rheological Behaviour
              of Natural and Synthetic Products},
 PUBLISHER = {North-Holland Publishing Co., Amsterdam; Interscience
              Publishers, Inc., New York},
      YEAR = {1952},
     PAGES = {viii+264},
   MRCLASS = {82.00},
  MRNUMBER = {134269},
MRREVIEWER = {R.\ Balescu},
}

@article{Kir46,
    author = {Kirkwood, John G.},
    title = {The Statistical Mechanical Theory of Transport Processes I. General Theory},
    journal = {The Journal of Chemical Physics},
    volume = {14},
    number = {3},
    pages = {180-201},
    year = {1946},
    month = {03},
    issn = {0021-9606},
    doi = {10.1063/1.1724117},
    url = {https://doi.org/10.1063/1.1724117},
    eprint = {https://pubs.aip.org/aip/jcp/article-pdf/14/3/180/18793903/180_1_online.pdf},
}

@article {BCMT23,
    AUTHOR = {Bernou, Armand and Carrapatoso, Kleber and Mischler,
              St\'{e}phane and Tristani, Isabelle},
     TITLE = {Hypocoercivity for kinetic linear equations in bounded domains
              with general {M}axwell boundary condition},
   JOURNAL = {Ann. Inst. H. Poincar\'{e} C Anal. Non Lin\'{e}aire},
  FJOURNAL = {Annales de l'Institut Henri Poincar\'{e} C. Analyse Non
              Lin\'{e}aire},
    VOLUME = {40},
      YEAR = {2023},
    NUMBER = {2},
     PAGES = {287--338},
      ISSN = {0294-1449,1873-1430},
   MRCLASS = {35Q20 (76P05 82C40)},
  MRNUMBER = {4581432},
       DOI = {10.4171/aihpc/44},
       URL = {https://doi.org/10.4171/aihpc/44},
}

@article {Mi10,
    AUTHOR = {Mischler, St\'{e}phane},
     TITLE = {Kinetic equations with {M}axwell boundary conditions},
   JOURNAL = {Ann. Sci. \'{E}c. Norm. Sup\'{e}r. (4)},
  FJOURNAL = {Annales Scientifiques de l'\'{E}cole Normale Sup\'{e}rieure.
              Quatri\`eme S\'{e}rie},
    VOLUME = {43},
      YEAR = {2010},
    NUMBER = {5},
     PAGES = {719--760},
      ISSN = {0012-9593,1873-2151},
   MRCLASS = {35F30 (35B35)},
  MRNUMBER = {2721875},
       DOI = {10.24033/asens.2132},
       URL = {https://doi.org/10.24033/asens.2132},
}

\end{document}